\documentclass{mylms}
\usepackage[curve,matrix,arrow,frame,tips]{xy}
\usepackage{mathrsfs}
\usepackage{amsmath}
\usepackage{enumerate}
\usepackage{amssymb}
\usepackage{stmaryrd}
\usepackage[pdfborder={0 0 0}, pdftitle={Global algebraic K-theory},pdfauthor={Stefan Schwede},pdfstartview={FitH}]{hyperref}

\newtheorem{theorem}[equation]{Theorem} 
     
\newtheorem{cor}[equation]{Corollary}
\newtheorem{prop}[equation]{Proposition}

\newnumbered{defn}[equation]{Definition}
\newnumbered{rk}[equation]{Remark}      
\newnumbered{eg}[equation]{Example}
\newnumbered{con}[equation]{Construction}

\newcommand{\mN}{{\mathbb N}}
\newcommand{\mP}{{\mathbb P}}
\newcommand{\mR}{{\mathbb R}}
\newcommand{\mS}{{\mathbb S}}
\newcommand{\mZ}{{\mathbb Z}}
\newcommand{\Ac}{{\mathcal A}}
\newcommand{\Fin}{{\mathcal F}in}
\newcommand{\Ab}{{\mathcal A}b}
\newcommand{\Bc}{{\mathcal B}}
\newcommand{\Cc}{{\mathcal C}}
\newcommand{\Dc}{{\mathcal D}}
\newcommand{\Ec}{{\mathcal E}}
\newcommand{\Fc}{{\mathcal F}}
\newcommand{\Gc}{{\mathcal G}}
\newcommand{\Ic}{{\mathcal I}}
\newcommand{\Jc}{{\mathcal J}}
\newcommand{\M}{{\mathcal M}}
\newcommand{\Oc}{{\mathcal O}}
\newcommand{\Pc}{{\mathcal P}}
\newcommand{\Uc}{{\mathcal U}}
\newcommand{\Vc}{{\mathcal V}}

\newcommand{\bA}{{\mathbf A}}
\newcommand{\bB}{{\mathbf B}}
\newcommand{\bI}{{\mathbf I}}

\newcommand{\bK}{{\mathbf K}}
\newcommand{\bL}{{\mathbf L}}
\newcommand{\bT}{{\mathbf T}}

\newcommand{\bSigma}{{\mathbf{\Sigma}}}
\newcommand{\bset}{\mathbf{set}}
\newcommand{\spec}{{\mathcal S}p}
\newcommand{\iso}{\cong}
\newcommand{\sm}{\wedge}
\newcommand{\tensor}{\otimes}
\newcommand{\xra}{\xrightarrow}
\newcommand{\xla}{\xleftarrow}

\newcommand{\un}{\underline}

\newcommand{\upi}{{\underline \pi}}
\newcommand{\td}[1]{\langle #1\rangle}
\renewcommand{\to}{\longrightarrow}

\newcommand{\parsumcat}{{\bf{ParSumCat}}}

\providecommand{\colim}{\mathop{\rm colim}\nolimits}
\providecommand{\gl}{\mathop{\rm gl}\nolimits}

\providecommand{\Hom}{\mathop{\rm Hom}\nolimits}
\providecommand{\Wirth}{\mathop{\rm Wirth}\nolimits}
\providecommand{\map}{\mathop{\rm map}\nolimits}
\providecommand{\free}{\mathop{\rm free}\nolimits}
\providecommand{\supp}{\mathop{\rm supp}\nolimits}
\providecommand{\Id}{\mathop{\rm Id}\nolimits}

\providecommand{\op}{\mathop{\rm op}\nolimits}
\providecommand{\cat}{\mathop{\rm \mathbf{cat}}\nolimits}
\providecommand{\sat}{\mathop{\rm sat}\nolimits}
\providecommand{\tr}{\mathop{\rm tr}\nolimits}
\providecommand{\triv}{\mathop{\rm triv}\nolimits}
\providecommand{\Sym}{\mathop{\rm Sym}\nolimits}
\providecommand{\res}{\mathop{\rm res}\nolimits}
\providecommand{\sat}{\mathop{\rm sat}\nolimits}

\title[Global algebraic K-theory]%
{Global algebraic K-theory}
 
\author{Stefan Schwede}

\classno{19D, 55P (primary), 19D23, 55P48, 55P91, 55Q91 (secondary)}

\extraline{
The research in this paper was supported by the DFG Schwerpunktprogramm 1786
`Homotopy Theory and Algebraic Geometry' (project ID SCHW 860/1-1)
and by the Hausdorff Center for Mathematics
at the University of Bonn (DFG GZ 2047/1, project ID 390685813).
{\em Date}: \today}

\begin{document}
\maketitle

\begin{abstract}
  We introduce a global equivariant refinement of algebraic K-theory;
  here `global equivariant' refers to simultaneous and compatible actions
  of all finite groups. 
  Our construction turns a specific kind of categorical input data
  into a global $\Omega$-spectrum that keeps track of genuine $G$-equivariant infinite loop spaces,
  for all finite groups $G$.
  The resulting global algebraic K-theory spectrum is a rigid
  way of packaging the representation K-theory, or `Swan K-theory'
  into one highly structured object.
\end{abstract}

\tableofcontents

\section*{Introduction}

In this paper we propose a global equivariant refinement of algebraic K-theory;
here `global equivariant' refers to simultaneous and compatible actions
of all finite groups. 
Our construction turns a specific kind of categorical input data
into a global $\Omega$-spectrum that keeps track of genuine $G$-equivariant infinite loop spaces,
for all finite groups $G$.
The resulting global algebraic K-theory spectrum is an extremely rigid
way of packaging the representation K-theory, or `Swan K-theory'
into one highly structured object.
We give a more precise summary of our main results below, after a quick review
of some aspects of algebraic K-theory.

Algebraic K-theory started out as the study of the K-group
(Grothendieck group) of projective modules over a ring,
or of vector bundles over a scheme.
After some attempts to extend the Grothendieck group to
higher K-groups by purely algebraic means,
several people proposed 
definitions of higher K-groups as homotopy groups of certain K-theory spaces
at the end of the 60's;
in retrospect, the most successful definition is due to Quillen,
originally introduced via the so called `plus construction' \cite{quillen:cohomology_groups}.
Several excellent survey papers on the early days of K-theory are available
\cite{grayson:quillens_work,weibel:early_K},
and we refer to these references for more details and proper credit.
Higher algebraic K-groups are powerful invariants
(albeit notoriously difficult to compute)
that contain arithmetic information about rings,
and geometric information about symmetries of high-dimensional manifolds.

By its very nature as homotopy groups of certain spaces,
algebraic K-theory is intimately related to homotopy theory.
A key conceptual insight was that Quillen's K-theory spaces
are naturally {\em infinite loop spaces}; as far as I know, the first
reference that advocates this perspective is Segal's paper
\cite{segal:cat coho} who acknowledges that it `(...) is one possible formulation
of Quillen's ideas about algebraic $K$-theory.'
On a conceptual level, and from a revisionist's perspective,
K-theory is the composite of two constructions:
one step turns specific kinds of categorical input data
(such as for example symmetric monoidal categories, permutative categories, exact categories,
or categories with cofibrations and weak equivalences)
into a kind of coherently homotopy-commutative and homotopy-associative H-spaces
(such as for example $E_\infty$-spaces or special $\Gamma$-spaces).
The second step is a delooping construction for $E_\infty$-spaces,
which implements the equivalence between the homotopy theories of
group-like $E_\infty$-spaces and connective spectra.

In this paper we propose a global equivariant extension of algebraic K-theory;
here `global equivariant' refers to simultaneous and compatible actions of all finite groups. 
In Definition \ref{def:parsummable_category} we introduce
a specific kind of extra structure on a category
that we call a {\em parsummable category}; 
the adjective `parsummable' stands for `partially summable'.
This structure can be interpreted in at least two ways:
on the one hand, parsummable categories are to symmetric monoidal categories
what partial abelian monoids are to $E_\infty$-spaces:
the monoidal operation of a parsummable category is only partially defined,
but it is strictly associative and commutative on its domain of definition.
On the other hand, parsummable categories are also algebras over
the {\em injection operad}, a specific categorical operad made
from injections between countably infinite sets, see Remark \ref{rk:I as global E_infty}.
The categories in this operad have contractible nerves, so
the injection operad is in particular an $E_\infty$-operad.
Our results suggest to view the injection operad as a `global $E_\infty$-operad',
because it encodes $E_\infty$-$G$-operads for all finite groups $G$.
From this perspective, our K-theory construction is a global analog of the
equivariant delooping machine of Guillou and May \cite{guillou-may}
that applies to a fixed finite group $G$.\smallskip

Our main construction assigns to a parsummable category $\Cc$ its
{\em global K-theory spectrum}  $\bK_{\gl}\Cc$, see Definition \ref{def:define K_gl}.
This is a symmetric spectrum, and as such it represents a global stable homotopy type,
based on finite groups, by the work of Hausmann \cite{hausmann:global_finite}.
Our main results about the global K-theory construction are the following.

\begin{theorem*}[(Global deloopings)]
For every parsummable category $\Cc$, 
the symmetric spectrum $\bK_{\gl}\Cc$ is a
restricted global $\Omega$-spectrum.  
\end{theorem*}

The global delooping theorem  will be proved as part of Theorem \ref{thm:K_gl C is global Omega}.
Roughly speaking, it means that evaluating  $\bK_{\gl}\Cc$ at a transitive free $G$-set yields
a genuine $G$-equivariant infinite loop space, for every finite group $G$.
Since these deloopings arise from a single symmetric spectrum,
they are automatically consistent for varying groups $H$.
We refer to Remark \ref{rk:X<G>} for a more detailed explanation.
The upshot is that $\bK_{\gl}\Cc$ provides a `global delooping' of the parsummable category $\Cc$.

As we explain in Proposition \ref{prop:symmon_from_parsumcat},
the underlying category of every parsummable category $\Cc$ 
affords a symmetric monoidal structure $\varphi^*(\Cc)$,
depending on a choice of injection $\varphi:\{1,2\}\times\omega\to\omega$,
where $\omega=\{0,1,2,\dots\}$.
The symmetric monoidal structure is independent of the injection
up to contractible choice, see Remark \ref{rk:constractible choice}.
The following consistency result is part of Theorem \ref{thm:compare 2K}.

\begin{theorem*}[(Consistency)]
  For every parsummable category $\Cc$, the underlying non-equivariant stable homotopy type
of $\bK_{\gl}\Cc$ is that of the K-theory of the symmetric monoidal category $\varphi^*(\Cc)$.
\end{theorem*}

Corollary \ref{cor:saturated implies global} shows that for every finite group $G$,
the genuine $G$-fixed point spectrum of $\bK_{\gl}\Cc$ 
is equivalent to the K-theory spectrum of the category of $G$-objects in $\Cc$,
relative to a certain symmetric monoidal structure arising from the
structure as parsummable category, at least under a mild technical hypothesis that we call {\em saturation},
see Definition \ref{def:saturated}.
So the symmetric spectrum $\bK_{\gl}\Cc$ is a compact and very rigid way 
of packaging the information that is contained in the `representation K-theory'
(or the `Swan K-theory') of the input category $\Cc$.
A related result is Theorem \ref{thm:upi_0 tilde beta}, showing that  
the global Mackey functor made from the 0-th equivariant homotopy groups
of the spectrum $\bK_{\gl}\Cc$ is isomorphic to
the `Swan K-theory' functor arising from the parsummable category $\Cc$:

\begin{theorem*}[(Rigidification of representation K-theory)]
    For every parsummable category $\Cc$, the homotopy group global functor
  $\upi_0(\bK_{\gl}\Cc)$ is isomorphic to the K-group global functor of the input category $\Cc$.
\end{theorem*}

By definition, parsummable categories are $\M$-categories with additional structure,
where $\M$ is the contractible groupoid with objects all self-injections of
the set $\omega=\{0,1,2,\dots\}$.
The {\em free parsummable category} $\mP\Bc$ generated by an $\M$-category $\Bc$
is the initial example of a parsummable category equipped with a morphism
of $\M$-categories from $\Bc$, compare Example \ref{eg:free_parsummable}.
As we explain in Construction \ref{con:M2bI}, every $\M$-category $\Bc$
gives rise to an {\em $\bI$-space} $\rho(\Bc)$, i.e., a functor from the
category $\bI$ of finite sets and injections to the category of spaces.
The $\bI$-space $\rho(\Bc)$ represents the unstable global homotopy type
corresponding to $\Bc$.

\begin{theorem*}[(Global Barratt-Priddy-Quillen theorem)]
  Let $\Bc$ be a tame $\M$-category without objects with empty support.
  Then the symmetric spectrum $\bK_{\gl}(\mP\Bc)$ is globally equivalent
  to the unreduced suspension spectrum of the $\bI$-space $\rho(\Bc)$ associated with $\Bc$. 
\end{theorem*}

We state the global Barratt-Priddy-Quillen theorem as Theorem \ref{thm:global BPQ}.

In addition to a satisfying theory, our formalism also incorporates many examples.
Specific parsummable categories that we discuss are:
the parsummable category of an abelian monoid (Example \ref{eg:abelian monoid}),
the parsummable category of finite sets (Example \ref{eg:Fc}),
free parsummable categories (Example \ref{eg:free_parsummable}),
the parsummable category of finite $G$-sets for a group $G$,
possibly infinite (Example \ref{eg:G Fc}),
the parsummable category of finitely generated projective modules over a ring (Construction \ref{con:P(R)}),
and the parsummable category associated with a permutative category
(Construction \ref{con:permutative2parsummable}).
The category of parsummable categories also has all limits and colimits,
and is closed under passage to $G$-objects (Example \ref{eg:objects with action}),
$G$-fixed points (Construction \ref{con:F^G C}),
and $G$-homotopy fixed points (Construction \ref{con:I structure on F^hG C}),
for every finite group $G$.
This provides many ways to build new examples of parsummable categories from existing ones.
\smallskip

{\bf Organization.}
As explained in Hausmann's paper \cite{hausmann:global_finite},
the category of symmetric spectra in the sense of Hovey, Shipley and Smith \cite{HSS}
is a model for the global stable homotopy theory based on finite groups,
relative to a certain class of `global equivalences'.
Since this is the model we use in the present paper, we start in Section \ref{sec:symspec}
by reviewing symmetric spectra, $G$-symmetric spectra and global equivalences.
We also discuss `restricted global $\Omega$-spectra',
a type of symmetric spectra that encodes compatible
equivariant infinite loop spaces for all finite groups, see Definition \ref{def:global Omega}.

In Section \ref{sec:Mcat} we introduce and study {\em $\M$-categories}.
Here $\M$ is the contractible groupoid with objects all self-injections of
the set $\omega=\{0,1,2,\dots\}$.
The category $\M$ is a strict monoidal category under composition of injections;
an {\em $\M$-category} is a category equipped with a strict action
of the monoidal category $\M$.
An $\M$-action on a category $\Cc$ gives rise to a notion of `support'
for objects of $\Cc$, see Definition \ref{def:support}.
The support is a subset of the set $\omega$,
and of fundamental importance for everything else we do in this paper.
An $\M$-category is {\em tame} if the support of every object is finite.

In Section \ref{sec:Gamma-M_to_sym} we introduce a global equivariant variation
of Segal's formalism \cite{segal:cat coho} to produce spectra from $\Gamma$-categories,
see Construction \ref{con:spectrum from Gamma-M}.
Our input is a $\Gamma$-$\M$-category,
i.e., a functor from $\Gamma$ to the category of $\M$-categories;
or, equivalently, a $\Gamma$-category equipped with a strict action
of the monoidal category $\M$.
While $\Gamma$-categories give rise to non-equivariant stable homotopy types,
and $\Gamma$-$G$-categories yield $G$-equivariant stable homotopy types,
our $\Gamma$-$\M$-categories produce global stable homotopy types.
In Definition \ref{def:globally special Gamma-cat}
we introduce the notion of `global specialness' for $\Gamma$-$\M$-categories;
our main results of Section \ref{sec:Gamma-M_to_sym}
are the facts that globally special $\Gamma$-$\M$-categories
give rise to restricted global $\Omega$-spectra (Theorem \ref{thm:Y(S) is global Omega}),
and that the underlying $G$-homotopy type agrees with that
of a naturally associated $\Gamma$-$G$-category (Theorem \ref{thm:a-b-maps equivalence}).

Section \ref{sec:K of I} contains the central construction of this paper,
the global K-theory spectrum of a parsummable category, see Definition \ref{def:define K_gl}.
A parsummable category is a tame $\M$-category equipped with
a sum operation that is partially defined on disjointly supported objects and morphisms,
see Definition \ref{def:parsummable_category}.
In a sense, parsummable categories are to symmetric monoidal categories
what partial abelian monoids are to $E_\infty$-spaces:
in a parsummable category, the `addition' is not defined everywhere,
but it is strictly unital, associative and commutative whenever it is defined.
In a symmetric monoidal category, the addition is defined everywhere,
but it is only unital, associative and commutative up to specified coherence data.
Some examples of parsummable categories are:
the parsummable category of an abelian monoid (Example \ref{eg:abelian monoid}),
the parsummable category of finite sets (Example \ref{eg:Fc}),
free parsummable categories (Example \ref{eg:free_parsummable}),
the parsummable category of finite $G$-sets for a group $G$,
possibly infinite (Example \ref{eg:G Fc}),
the parsummable category of finitely generated projective modules over a ring (Construction \ref{con:P(R)}),
and the parsummable category associated with a permutative category
(Construction \ref{con:permutative2parsummable}).

The definition of the global K-theory spectrum of a parsummable category proceeds in two steps:
the parsummable category $\Cc$ gives rise to a globally special $\Gamma$-$\M$-category $\gamma(\Cc)$
by taking powers under the $\boxtimes$-product of tame $\M$-categories,
see Construction \ref{con:Gamma from I}.
The global K-theory spectrum $\bK_{\gl}\Cc$ is
the symmetric spectrum associated to the $\Gamma$-$\M$-category $\gamma(\Cc)$
as in Section \ref{sec:Gamma-M_to_sym}.
The $\Gamma$-$\M$-category $\gamma(\Cc)$ is globally special
by Theorem \ref{thm:special G Gamma}, so the global K-theory spectrum is
a restricted global $\Omega$-spectrum, see Theorem \ref{thm:K_gl C is global Omega}.
We also establish some basic invariance properties of global K-theory:
by Theorem \ref{thm:equivalence2equivalence}, 
global equivalences of parsummable categories
induce global equivalences of K-theory spectra;
by Theorem \ref{thm:K_gl additive}, global K-theory takes box products
and products of parsummable categories to products of K-theory spectra,
up to global equivalence.
Corollary \ref{cor:G-fixed} identifies the $G$-fixed point spectrum
of the global homotopy type of $\bK_{\gl}\Cc$, for a finite group $G$:
the $G$-fixed category $F^G\Cc$ is naturally again a parsummable category,
and the genuine $G$-fixed point spectrum $F^G(\bK_{\gl}\Cc)$
receives a natural equivalence from the K-theory spectrum of $F^G\Cc$.

Section~\ref{sec:parsumcat_versus_symmetric} contains a reality check: every parsummable category
$\Cc$ gives rise to a symmetric monoidal category~$\varphi^*(\Cc)$,
see Proposition \ref{prop:symmon_from_parsumcat}. 
We show in Theorem \ref{thm:compare 2K}
that the underlying non-equivariant homotopy type of $\bK_{\gl}\Cc$
is that of the algebraic K-theory spectrum of the
symmetric monoidal category $\varphi^*(\Cc)$, i.e., the canonical
infinite delooping of the group completion of the classifying space of~$\varphi^*(\Cc)$.

In Section \ref{sec:Swan} we identify the collection of the 0-th equivariant homotopy
groups of the global K-theory spectrum $\bK_{\gl}\Cc$
in terms of the parsummable category $\Cc$ that serves as the input.
The main result is Theorem \ref{thm:upi_0 tilde beta} that establishes an
isomorphism of global functors between $\upi_0(\bK_{\gl}\Cc)$
and the global functor of Swan K-groups of $\Cc$.
More concretely this means that the group $\pi_0^G(\bK_{\gl}\Cc)$
is a group completion of the abelian monoid $\pi_0(F^G\Cc)$,
and the homotopy theoretic restriction and transfer maps
have an explicit categorical description.

Section \ref{sec:saturation} is devoted to the phenomenon of {\em saturation},
which roughly means that there are `enough fixed objects'.
For every finite group $G$ and parsummable category $\Cc$, the $G$-fixed category $F^G\Cc$
embeds fully faithfully into the category $G\Cc$ of $G$-objects in $\Cc$.
This embedding need not be essentially surjective, so it is not necessarily an equivalence.
The parsummable category $\Cc$ is {\em saturated} if the comparison functor
$F^G\Cc\to G\Cc$ is an equivalence of categories for every finite group $G$,
see Corollary \ref{cor:saturation characterizations}.
Whenever this happens, the $G$-fixed point spectrum $F^G(\bK_{\gl}\Cc)$
of the global K-theory spectrum is equivalent to the K-theory
spectrum of $G$-objects in $\Cc$, see Corollary \ref{cor:saturated implies global}.
Saturation can always be arranged in the following sense:
there is a saturation functor for parsummable categories
and a natural morphism of parsummable categories $s:\Cc\to C^{\sat}$
that is an equivalence of underlying categories, see Theorem \ref{thm:saturation M-version}.

The last four sections are devoted to examples.
In Section \ref{sec:K of free I} we establish a global equivariant generalization
of the Barratt-Priddy-Quillen theorem:
Theorem \ref{thm:global BPQ} identifies the global K-theory spectrum
of a free parsummable category as a suspension spectrum.
In Theorem \ref{thm:global F} we use the global Barratt-Priddy-Quillen theorem
to recognize the global K-theory of finite sets as the global sphere spectrum.
In Section~\ref{sec:G-sets} we discuss the global K-theory of finite $G$-sets,
where $G$ is a group, possibly infinite.
The upshot is Corollary \ref{cor:K(GF) final},
where we identify the global K-theory spectrum of finite $G$-sets
with the wedge, indexed by conjugacy classes of finite index subgroups,
of the suspension spectra of the global classifying spaces of
the Weyl groups.
In Section~\ref{sec:K(R)} we discuss the global K-theory of rings,
and in Section \ref{sec:permutative category}
we define and study the global K-theory of permutative categories.\smallskip

{\bf Relation to equivariant algebraic K-theory.}
While our paper provides the first global equivariant approach to algebraic K-theory,
several authors have introduced equivariant K-theory constructions for
individual finite groups that produce genuine equivariant $G$-spectra;
we comment specifically on the relation between our construction and
the work of Shimakawa \cite{shimakawa}, Guillou-May \cite{guillou-may},
Merling~\cite{merling} and Barwick
and Barwick-Glasman-Shah \cite{barwick:spectral_mackey_I,barwick-glasman-shah:spectral_mackey_II}.
These constructions use different categorical input data and produce different kinds of output,
but I expect certain connections that I sketch now.

Shimakawa \cite{shimakawa} uses {\em $G$-graded monoidal categories},
a relative version of symmetric monoidal categories where all the structure
is over the category with one object and $G$ as endomorphisms.
Most of his examples arise from {\em symmetric monoidal $G$-categories},
i.e., symmetric monoidal categories equipped with a strict $G$-action through strict
symmetric monoidal functors.
From this data, Shimakawa produces an equivariant K-theory spectrum
that deloops the underlying $E_\infty$-$G$-space, has the desired
behavior on fixed points and the expected equivariant homotopy groups
in dimension 0, see Theorem~A and the Proposition on page 242 of \cite{shimakawa}.
Shimakawa's equivariant K-theory ought to be related to our global
K-theory as follows. For a parsummable category $\Cc$ and a finite group $G$,
the category $\Cc[\omega^G]$ obtained by reparameterization
is naturally a $G$-parsummable category,
i.e., a parsummable category equipped with a
strict $G$-action through morphisms of parsummable categories;
we refer to Remark \ref{rk:parsumcat and GSymMonCat}
for more details on the construction.
The passage from parsummable categories to symmetric monoidal categories
described in Proposition \ref{prop:symmon_from_parsumcat} is functorial,
so it turns the $G$-parsummable category $\Cc[\omega^G]$
into a symmetric monoidal $G$-category $\varphi^*(\Cc[\omega^G])$.
I expect that the underlying genuine $G$-homotopy type of $\bK_{\gl}\Cc$
agrees with the $G$-homotopy type obtained by applying
Shimakawa's construction to the symmetric monoidal $G$-category $\varphi^*(\Cc[\omega^G])$.
Our Theorem \ref{thm:compare 2K} verifies that this is indeed the
case when the group $G$ is trivial, and moreover,
the two constructions have equivalent fixed point spectra for all subgroups of $G$.

Guillou and May \cite{guillou-may} start from an $E_\infty$-$G$-category.
This concept is based on the notion of an $E_\infty$-operad of $G$-categories,
defined as an operad $\Oc$ in the cartesian closed category of small $G$-categories,
such that the geometric realization $|\Oc(n)|$ of the category
of $n$-ary operations is a universal $(G\times\Sigma_n)$-space
for principal $\Sigma_n$-bundles over $G$-spaces, compare \cite[Definition 3.11]{guillou-may}.
An {\em $E_\infty$-$G$-category} is then a small $G$-category
equipped with the action of an  $E_\infty$-operad of $G$-categories.
In \cite[Definition 4.12]{guillou-may}, Guillou and May associate to an
$E_\infty$-$G$-category $\Ac$ an equivariant K-theory spectrum $\mathbb K_G(\Ac)$.
The paper \cite{may-merling-osorno} by May, Merling and Osorno is
devoted to comparing the operadic approach of \cite{guillou-may}
to the Segal-Shimakawa approach of \cite{shimakawa}.
I expect that for a parsummable category $\Cc$,
the underlying genuine $G$-homotopy type of $\bK_{\gl}\Cc$
agrees with the $G$-homotopy type obtained by applying
the Guillou-May K-theory construction \cite[Definition 4.12]{guillou-may}
to the  $E_\infty$-$G$-category $\Cc[\omega^G]$;
I refer to Remark \ref{rk:I as global E_infty} for more details.

As far as I know, equivariant algebraic K-theory for rings was first considered
by Fiedorowicz, Hauschild and May in \cite{fiedorowicz-hauschild-may},
and for exact categories by Dress and Kuku \cite{dress-kuku};
both papers produce $G$-equivariant K-theory {\em spaces}.
In the ring case, Merling \cite{merling} offers a spectrum level extension and
allows for non-trivial group actions:
in \cite[Definition 5.23]{merling}, she associates to a ring with an action
of a finite group $G$ an orthogonal $G$-spectrum $\bK_G(R)$,
the {\em equivariant algebraic K-theory spectrum} of the $G$-ring $R$;
the construction is an application of the Guillou-May operadic delooping machine
outlined in the previous paragraph. 
Another innovation of Merling's paper is that she can process actions
and functors that are only `pseudo equivariant' (as opposed to strictly equivariant).
Every ring $R$ can be endowed with the trivial $G$-action;
I expect that the orthogonal $G$-spectrum $\bK_G(R^{\triv})$ 
agrees with the underlying $G$-homotopy type of
our $\bK_{\gl}R$ of Definition \ref{def:K(R)}.
More precisely, the underlying $G$-symmetric spectra of
$\bK_G(R^{\triv})$ and $\bK_{\gl}R$ ought to be $G$-stably equivalent.

In \cite{barwick:spectral_mackey_I}, Barwick introduces
a very general framework of {\em spectral Mackey functors}.
He works in the setting of quasi-categories
and defines spectral Mackey functors as additive functors from suitable
$\infty$-categories of spans to the $\infty$-category of spectra.
When applied to finite $G$-sets for a finite group $G$, 
this yields a model for genuine $G$-equivariant stable homotopy theory.
Replacing finite $G$-sets by finite groupoids ought
to give a model for $\Fin$-global stable homotopy theory
as spectral Mackey functors; here the understanding is that we consider spans
with arbitrary functors in the restriction direction,
but with only {\em faithful} functors in the transfer direction.
However, I am not aware of a place where the details of such
`global spectral Mackey functors' have been worked out,
much less compared to the symmetric spectrum model of $\Fin$-global stable homotopy theory.

Under the expected equivalence between the stable $\infty$-category
underlying the global model structure on symmetric spectra and
global spectral Mackey functors, our global K-theory construction ought
to relate to Barwick's approach as follows.
As we explain in Proposition \ref{prop:symmon_from_parsumcat},
every parsummable category $\Cc$ gives rise to a symmetric monoidal structure
on the underlying category. The symmetric monoidal structure
depends on a choice, but the choices can be parameterized by a contractible category,
see Remark \ref{rk:constractible choice}.
In symmetric monoidal categories, one can restrict actions along
arbitrary group homomorphisms, and one can induce actions along inclusions
between finite groups.
So as $\Gc$ varies over all finite groupoids,
the nerves of the functor categories $\cat(\Gc,\Cc)$ ought to extend
to an additive functor from the span  $\infty$-category of finite groupoids
to symmetric monoidal $\infty$-categories.
If the parsummable category $\Cc$ is saturated,
then postcomposing with any $\infty$-categorical delooping functor
ought to give the spectral Mackey functor counterpart of $\bK_{\gl}\Cc$.\smallskip

{\bf Outlook: parsummable categories model connective global homotopy theory.}
Thomason showed in \cite{thomason:symmetric_monoidal}
that every connective spectrum is stably equivalent
to the K-theory spectrum of a symmetric monoidal category.
Equivalently, every infinite loop space is the group completion, as an $E_\infty$-space,
of the nerve of a permutative category. Even more is true: Theorem 5.1 of
\cite{thomason:symmetric_monoidal} shows that the K-theory functor
induces an equivalence from the homotopy category of the category of
symmetric monoidal categories and lax symmetric monoidal functors,
localized at the class of K-theory equivalences,
to the homotopy category of connective spectra.
A different proof of Thomason's theorem was later given by Mandell \cite{mandell:inverse},
who also provided an `un-group-complete' version.

In the first version of this paper, posted on the \verb!arXiv! in 2019,
I had speculated about a global refinement of Thomason's and Mandell's results.
While this paper was being refereed, Tobias Lenz proved this
conjecture as part of his PhD thesis; he showed in \cite[Theorem B]{lenz:G-global}
that the global K-theory functor
\[ \bK_{\gl} \ : \ \parsumcat \ \to \ \spec \]
induces an equivalence of $\infty$-categories between the 
quasi-categorical localization of parsummable categories
at the $\bK_{\gl}$-equivalences,
and globally connective symmetric spectra at the global equivalences.
In particular, the functor
descends to an equivalence between the respective homotopy categories,
and every connective $\Fin$-global stable homotopy type
arises as the global K-theory spectrum of some parsummable category.
Lenz' results are in fact a lot more general:
he shows in \cite[Theorem A]{lenz:G-global} that for every discrete group $G$,
the $\infty$-category of $G$-global connective spectra is a quasi-categorical
localization of the category of parsummable categories with $G$-action;
and he also provides an un-group-complete version of this statement.

\section{A review of symmetric spectra}\label{sec:symspec}

Our global K-theory construction produces symmetric spectra.
Hence we start by recalling the definition of symmetric spectra, due to Jeff Smith,
and first published in the paper \cite{HSS} by Hovey, Shipley and Smith. 
Symmetric spectra were originally introduced as a convenient model
for the non-equivariant stable homotopy category with a compatible smash product.
Later, Mandell \cite{mandell:equivariant symmetric} and Hausmann \cite{hausmann:G-symmetric}
extended the work of Hovey, Shipley and Smith and
introduced different symmetric spectrum models for genuine $G$-spectra, where $G$ is a finite group.
Hausmann's model is particularly convenient for global purposes:
he endows $G$-objects internal to symmetric spectra with a suitable equivariant stable model structure;
hence non-equivariant symmetric spectra, endowed with trivial actions,
model genuine $G$-homotopy types for all finite groups $G$.
Hausmann subsequently showed in \cite{hausmann:global_finite}
that with respect to a specific notion of {\em global equivalence},
symmetric spectra model global stable homotopy theory based on finite groups;
we review the relevant definitions in this section.

We work with symmetric spectra in spaces 
(as opposed to simplicial sets as in~\cite{HSS}),
where we use the convention that a {\em space} is 
a compactly generated space in the sense of~\cite{mccord},
i.e., a $k$-space (also called {\em Kelley space})
that satisfies the weak Hausdorff condition.
We write $\bT$ for the category of compactly generated spaces and continuous maps.
Moreover, we use a slightly more invariant version of symmetric spectra,
where the terms are indexed by finite sets (as opposed to natural numbers).
So our definition below is not identical with the one given in~\cite{HSS},
but it defines an equivalent category.

For a finite set~$A$ we denote by $\mR[A]$ the $\mR$-vector space of functions from
$A$ to $\mR$, and by  $S^A$ the one-point compactification of $\mR[A]$,
based at infinity. We write $A+B$ for the disjoint union of two sets $A$ and $B$.
The canonical linear isomorphism 
\[ \mR[A]\oplus\mR[B]\ \iso \ \mR[A+B] \]
induced by the inclusions of $A$ and $B$ into $A+B$
compactifies to a homeomorphism $S^A\sm S^B\iso S^{A+B}$
that we will often use without explicit mentioning to identify $S^A\sm S^B$ and $S^{A+B}$.

\begin{defn}
A {\em symmetric spectrum} $X$ consists of a based space $X(A)$ for every finite set $A$,
equipped with continuous based maps 
\[  i_* \ : \ X(A)\sm S^{B\setminus i(A)}\ \to \ S^B\]
for every injective map $i:A\to B$ between finite sets.
This data is required to satisfy the following conditions:
\begin{enumerate}[(a)]
\item For every set $A$, the following composite is the identity:
\[ X(A)\ \iso \ X(A)\sm S^\emptyset\ \xrightarrow{(\Id_A)_*} \ X(A)\]
\item If $i:A\to B$ and $j:B\to C$ are composable injections between finite sets,
then the following diagram commutes:
\[\xymatrix@C=18mm{ 
X(A)\sm S^{ B\setminus i(A) }\sm S^{ C\setminus j(B) }\ar[r]^-{i_*\sm S^{C\setminus j(B)} }\ar[d]_{X(A)\sm \iso} &
X(B)\sm S^{C\setminus j(B)}\ar[d]^{j_*}\\
X(A)\sm S^{C\setminus j(i(A))}\ar[r]_{(j i)_*} &X(C) }  \]
\end{enumerate}
The unnamed isomorphism between
$S^{ B\setminus i(A) }\sm S^{ C\setminus j(B) }$ and $S^{ C\setminus j(i(A))}$
is given by $j$ on the coordinates in $B\setminus i(A)$,
and by the inclusion  $C\setminus j(B)\to C\setminus j(i(A))$
on the remaining coordinates.

A {\em morphism} $f:X\to Y$ of symmetric spectra consists of
based continuous maps $f(A):X(A)\to Y(A)$ for every finite set $A$ such that
for every injection $i:A\to B$ between finite sets the following square commutes:
\[\xymatrix{ 
X(A)\sm S^{ B\setminus i(A) }\ar[r]^-{i_*}\ar[d]_{f(A)\sm S^{ B\setminus i(A) }} &X(B)\ar[d]^{f(B)}\\
Y(A)\sm S^{B\setminus i(A)}\ar[r]_-{i_*} & Y(B) }  \]
\end{defn}

The data of a symmetric spectrum $X$ in particular provides an action
of the symmetric group $\Sigma_A$ of self bijections of $A$ in the
space $X(A)$, where we let a bijection $\sigma:A\to A$ act by the composite
\[ X(A)\ \iso \ X(A)\sm S^\emptyset\ \xra{\sigma_*} \ X(A)\ .\]
Also, the inclusion $A\to A+B$ of a summand into the disjoint union
of two finite sets produces a structure map
$\sigma_{A,B}:X(A)\sm S^B\to X(A+B)$, 
where we implicitly identify $B$ with the complement of the embedding $A\to A+B$.

In \cite[Definition 3.1.3]{HSS}, Hovey, Shipley and Smith define stable equivalences of symmetric spectra;
they show in \cite[Theorem 3.4.4]{HSS} that the stable equivalences participate in a model structure
on the category of symmetric spectra whose homotopy category is equivalent to the traditional stable
homotopy category. 
It is an unfortunate fact of life that stable equivalences of symmetric spectra {\em cannot}
be defined as the morphisms that induce isomorphisms of the naively defined stable homotopy groups,
see \cite[Definition 3.1.9]{HSS}.
By \cite[Theorem 3.1.11]{HSS}, every $\pi_*$-isomorphism of symmetric spectra is a stable
equivalence, but the converse is not true generally.
In  \cite[Definition 5.6.1]{HSS}, Hovey, Shipley and Smith introduce the notion of {\em semistability}
and show in \cite[Proposition 5.6.5]{HSS} that stable equivalences between semistable symmetric spectra
are already $\pi_*$-isomorphisms.

In \cite{hausmann:G-symmetric},
Hausmann generalizes the work of Hovey, Shipley and Smith to $G$-symmetric spectra, i.e.,
symmetric spectra equipped with an action of a finite group $G$.
Hausmann works relative to a $G$-set universe $\Uc$, i.e., a countably infinite $G$-set that
is isomorphic to the disjoint union of two copies of itself. We restrict our attention
the special case when $\Uc=\Uc_G$ is a universal $G$-set,
i.e., a $G$-set universe that contains an isomorphic copy of every finite $G$-set;
we often omit the universal $G$-set from part of the notation.
Hausmann introduces a notion of $G$-stable equivalences \cite[Definition 2.3.5]{hausmann:G-symmetric}
and complements them into a stable model structure
on the category of $G$-symmetric spectra
\cite[Theorem 4.8]{hausmann:G-symmetric}.
By \cite[Theorem 7.4, 7.5]{hausmann:G-symmetric},
Hausmann's model structure is Quillen equivalent
to the category of orthogonal $G$-spectra with either
the stable model structure of Mandell and May \cite[Theorem 4.2]{mandell-may},
Stolz \cite{stolz-thesis} or Hill, Hopkins and Ravenel \cite[Proposition B.63]{HHR-Kervaire}.

In \cite[Definition 3.1]{hausmann:G-symmetric} Hausmann introduces the
naive $G$-equivariant homotopy group  $\pi_k^{G,\Uc_G}(X)$ of a $G$-symmetric spectrum as
\begin{equation}\label{eq:define pi^G}
\pi_k^{G,\Uc_G}(X) \ = \ \colim_{A\subset \Uc_G}[S^{k+ A}, X(A)]^G\ ;
\end{equation}
here the colimit is taken over the filtered poset of finite $G$-subsets of $\Uc_G$, formed along
specific structure maps.
Strictly speaking, the definition above only makes sense for $k\geq 0$;
for negative $k$, the precise interpretation is explained
in \cite[3.1]{hausmann:G-symmetric}.
A morphism of $G$-symmetric spectra is a
{\em $\upi_*^{\Uc}$-isomorphism} if the induced map on $\pi_k^{H,\Uc_G}$
is an isomorphism for every integer $k$ and every subgroup $H$ of $G$.
By \cite[Theorem 3.36]{hausmann:G-symmetric}, every $\upi_*^{\Uc}$-isomorphism
of $G$-symmetric spectra is a $G$-stable equivalence.
Hausmann defines a notion of $G$-semistability in \cite[Definition 3.22]{hausmann:G-symmetric},
and shows in \cite[Corollary 3.37]{hausmann:G-symmetric}
that $G$-stable equivalences between $G$-semistable symmetric spectra
are already $\upi_*^{\Uc}$-isomorphisms.

Symmetric spectra (without any additional actions)
can be viewed as encoding `global stable homotopy types'; 
loosely speaking, one can think of this as a collection of 
compatible $G$-equivariant stable homotopy types for every finite group $G$. 
We briefly sketch how to formalize this approach to global stable homotopy theory;
the details of the theory are developed in Hausmann's paper \cite{hausmann:global_finite}.
There is a version with orthogonal spectra instead of symmetric spectra, 
where finite groups are generalized to compact Lie groups, compare~\cite{schwede:global}.

Every symmetric spectrum can be considered
as a $G$-symmetric spectrum by letting $G$ act trivially;
we call this the {\em underlying $G$-spectrum}
of a symmetric spectrum. A morphism of symmetric spectra is a {\em global equivalence}
\cite[Definition 2.10]{hausmann:global_finite} if it is a $G$-stable equivalence of underlying $G$-spectra
for every finite group $G$.
There is an alternative way to define global equivalences of symmetric spectra,
as follows. We call a morphism $f:Y\to Z$ of symmetric spectra
a {\em global level equivalence} if for every $n\geq 0$ the map
$f_n:Y_n\to Z_n$ is a $\Sigma_n$-weak equivalence (i.e., restricts to a
weak equivalence on $H$-fixed points for all subgroups $H\leq \Sigma_n$),
compare \cite[Definition 2.3]{hausmann:global_finite}.
The global level equivalences and the $S$-cofibrations of~\cite[Definition 5.3.6]{HSS}
(or rather the analog for symmetric spectra of spaces, which are called
{\em flat cofibrations} in \cite[Definition 2.3]{hausmann:global_finite})
determine a {\em global level model structure} on the category of symmetric spectra,
see \cite[Proposition 2.6]{hausmann:global_finite}.
A morphism  $f:Y\to Z$ of symmetric spectra is a global equivalence
if and only if for every global $\Omega$-spectrum~$X$ the induced map 
$[f,X]^{\rm str}:[Z,X]^{\rm str}\to[Y,X]^{\rm str}$
is bijective, where $[-,-]^{\rm str}$ is the set of morphisms 
in the global level homotopy category (the localization of symmetric spectra
with respect to the class of global level equivalences),
see the remark at the end of Section 2.2 of \cite{hausmann:global_finite}.

A morphism of symmetric spectra is a {\em global $\upi_*$-isomorphism}
\cite[Definition 4.2]{hausmann:global_finite} if it is a $\upi_*^{\Uc}$-isomorphism
of underlying $G$-spectra for every finite group $G$.
A symmetric spectrum is {\em globally semistable} in the sense of \cite[Definition 4.11]{hausmann:global_finite}
precisely when the underlying $G$-spectrum is $G$-semistable for every finite group $G$,
by \cite[Proposition 4.13 (i)]{hausmann:global_finite}.
So every global $\upi_*$-isomorphism
is a global equivalence \cite[Proposition 4.5]{hausmann:global_finite},
and every global equivalence between globally semistable symmetric spectra is 
a global $\upi_*$-isomorphism \cite[Proposition 4.13 (vi)]{hausmann:global_finite}.\medskip

In Section~\ref{sec:K of I} we will explain how to
turn a parsummable category $\Cc$ into a symmetric spectrum $\bK_{\gl}\Cc$,
the {\em global K-theory spectrum} of~$\Cc$; the symmetric spectrum $\bK_{\gl}\Cc$
has the remarkable property of being a `restricted global $\Omega$-spectrum',
a slight generalization of the
notion of global $\Omega$-spectrum of \cite[Definition 2.13]{hausmann:global_finite}.
We introduce this concept now and explain how a restricted global $\Omega$-spectrum
encodes compatible equivariant infinite loop spaces for all finite groups.

We let~$X$ be a symmetric spectrum and~$G$ a group.
Then for all finite $G$-sets $A$ and $B$, the spaces
$X(A)$, $S^B$ and $X(A+ B)$ inherit a $G$-action, and  
the structure map $\sigma_{A,B}:X(A)\sm S^B\to X(A+B)$ is $G$-equivariant.
The adjoint
\[ \tilde\sigma_{A,B} \ : \ X(A)\ \to \ \map_*(S^B,X(A+ B))  \]
is then $G$-equivariant for the conjugation action on the target,
the space of based continuous maps.
We will say that a $G$-set {\em has a free orbit} if it contains an element with trivial isotropy group.
We recall that a continuous $G$-equivariant map $f:K\to L$ 
between $G$-spaces is a {\em $G$-weak equivalence}
if the induced map $f^H:K^H\to L^H$ on $H$-fixed point spaces
is a weak homotopy equivalence for every subgroup $H$ of $G$.

\begin{defn}\label{def:global Omega} 
A symmetric spectrum $X$ is a {\em restricted global $\Omega$-spectrum}
if for every finite group $G$, every finite $G$-set $A$ with a free orbit,
and all finite $G$-sets $B$, the adjoint structure map 
$\tilde\sigma_{A,B}: X(A)\to\map_*(S^B,X(A+ B))$
is a $G$-weak equivalence.
\end{defn}

A symmetric spectrum is a {\em global $\Omega$-spectrum}
if the adjoint structure map $\tilde\sigma_{A,B}$
is a $G$-weak equivalence for all finite $G$-sets $A$ and $B$ 
such that $G$ acts faithfully on $A$, see \cite[Definition 2.13]{hausmann:global_finite}.
If $A$ contains an element with trivial isotropy group, then $G$ must in
particular act faithfully. So every global $\Omega$-spectrum
is in particular a restricted global $\Omega$-spectrum;
the `restricted' objects are somewhat analogous to positive $\Omega$-spectra in
the non-equivariant context.
The global $\Omega$-spectra are the fibrant objects in the global model structure
on the category of symmetric spectra, see \cite[Theorem 2.18]{hausmann:global_finite}.

For restricted global $\Omega$-spectra $X$,
all maps in the colimit system defining $\pi_0^G(X)$
that start at a $G$-set with a free orbit are isomorphisms.
Since $G$-sets with a free orbit are cofinal in the poset of
finite $G$-subsets of $\Uc_G$, the canonical map
\[  [S^G, X(G)]^G\ \to \ \pi_0^G(X)  \]
is bijective.

\begin{rk}\label{rk:X<G>} 
A restricted global $\Omega$-spectrum $X$ is a very rich kind of structure: 
for every finite group~$G$, the $G$-space 
\[  X\td{G}\ =\ \map_*(S^G,X(G))    \]
is an equivariant infinite loop space, indexed on a complete $G$-universe;
here $G$ acts on itself by left translation, and by conjugation
on the mapping space.
In other words, for every $G$-representation $V$, there is another
$G$-space $X\td{G,V}$ and a $G$-weak equivalence
\[ X\td{G}\ \to \ \map_*(S^V,X\td{G,V})\ . \]
To construct such an equivalence,
 we observe that $\mR[G\times\mathbf{m}]$ is isomorphic to
a sum of $m$ copies of the regular representation of $G$. 
Here we write $\mathbf{m}=\{1,2,\dots,m\}$,
and $G$ acts on $G\times\mathbf{m}$ by left translation on the first factor.
So we can choose a $G$-equivariant linear isometric embedding $i:V\to \mR[G\times\mathbf m]$ 
for some $m\geq 1$.
Then we let $V^\perp$ be the orthogonal complement of the image of $i$, and we set
\[ X\td{G,V}\ = \ \map_*(S^{V^\perp}, X(G\times\mathbf{m})) \ ,\]
with $G$ acting by conjugation on the mapping space.
Because $X$ is a restricted global $\Omega$-spectrum, the adjoint
structure map 
\[ \tilde\sigma_{G,G\times(\mathbf{m-1})} \ : \ 
X(G)\ \to \ \map_*(S^{G\times(\mathbf{m-1})}, X(G\times\mathbf{m}))\]
is a $G$-weak equivalence.
Applying $\map_*(S^G,-)$ gives the desired $G$-weak equivalence from
$X\td{G}=\map_*(S^G,X(G))$ to 
\begin{align*}
  \map_*(S^G,\map_*(S^{G\times(\mathbf{m-1})},\
  &X(G\times\mathbf{m}))) \ \iso \ 
  \map_*(S^G\sm S^{G\times(\mathbf{m-1})}, X(G\times\mathbf{m})) \\
  &\iso \ \map_*(S^{G\times\mathbf{m}}, X(G\times\mathbf{m})) \
  \iso \ \map_*(S^V\sm S^{V^\perp}, X(G\times\mathbf{m})) \\
&\iso \ 
\map_*(S^V,\map_*(S^{V^\perp}, X(G\times\mathbf{m})))\ = \
\map_*(S^V,X\td{G,V})\ .
\end{align*}
As $G$ varies, the equivariant infinite loop spaces $X\td{G}$ are 
closely related to each other. For example, if $H$ is a subgroup of $G$, then
$X\td{H}$ is $H$-weakly equivalent to the restriction of the 
$G$-equivariant infinite loop space $X\td{G}$.
Indeed, the underlying $H$-set of $G$ decomposes as the internal disjoint union
\[ \res^G_H(G)\ = \ H\, + \, (G\setminus H) \ .\]
Because $X$ is a restricted global $\Omega$-spectrum, the adjoint structure map 
\[ \tilde\sigma_{H,G\setminus H} \ : \  X(H)\ \to \ \map_*(S^{G\setminus H}, X(G))\]
is an $H$-weak equivalence. Applying $\map_*(S^H,-)$
gives an $H$-weak equivalence
\[ X\td{H}\ =\ \map_*(S^H,X(H)) \ \to \ 
\map_*(S^H,\map_*(S^{G\setminus H}, X(G)))\ \iso\ \res^G_H(X\td{G}) \ .\]
\end{rk}

\begin{con}[(Fixed point symmetric spectrum)]\label{con:G fixed point}
  Given a symmetric spectrum $X$ and a finite group  $G$,
  we construct  another symmetric spectrum $F^G X$, the {\em $G$-fixed point spectrum}.
  We let 
  \[ \bar\rho_G\ = \ \{ \sum \lambda_g \cdot g \in \rho_G \ : \ \sum \lambda_g =0 \} \]
  denote the reduced regular representation of $G$, i.e.,
  the kernel of the augmentation $\rho_G=\mR[G]\to \mR$.
  Several times below -- and sometimes without explicit mentioning --
  we shall use the $G$-equivariant isometry
  \begin{equation}\label{eq:fix_of_rho}
  \mR\oplus\bar\rho_G \ \iso \ \rho_G \ ,\quad
    (\lambda,x)\ \to \ (\frac{\lambda}{\sqrt{|G|}}{\sum}_{g\in G}\ g) + x \ .    
  \end{equation}
  For a finite set~$A$ we set
  \[ (F^G X)(A) \ = \ \map_*^G(S^{\bar\rho_G\tensor \mR[A]}, X(G\times A))\ . \]
  As before, $S^{\bar\rho_G\tensor\mR[A]}$ is the one-point compactification of the
  tensor product of $\bar\rho_G$ with $\mR[A]$,
  and $\map_*^G(-,-)$ is the space of $G$-equivariant continuous based maps.
  To define the structure maps we let $i:A\to B$ be an injection between finite sets.
  We define the structure map $i_*:(F^G X)(A)\sm S^{B\setminus i(A)}\to (F^G X)(B)$ as the composite
  \begin{align*}
    \map_*^G(S^{\bar\rho_G\tensor\mR[A]},& X(G\times A))\sm S^{B\setminus i(A)}\
    \xra{\rm assembly} \  \map_*^G(S^{\bar\rho_G\tensor \mR[A]},
      X(G\times A)\sm S^{B\setminus i(A)})\\
    \to \
    &\map_*^G(S^{\bar\rho_G\tensor \mR[A]}\sm S^{\bar\rho_G\tensor \mR[B\setminus i(A)]},
      X(G\times A)\sm S^{B\setminus i(A)}\sm S^{\bar\rho_G\tensor \mR[B\setminus i(A)]}) \\
    \xra{\ \iso\ }    \  &\map_*^G(S^{\bar\rho_G\tensor \mR[B]},
                           X(G\times A)\sm S^{G\times(B\setminus i(A))})\\
    \to \  &\map_*^G(S^{\bar\rho_G\tensor \mR[B]}, X(G\times B))\ .
  \end{align*}
  The second map smashes from the right with the sphere
  $S^{\bar\rho_G\tensor\mR[B\setminus i(A)]}$.
  The third map is induced by the isometry
  \[ (\bar\rho_G\tensor \mR[A])\oplus (\bar\rho_G\tensor \mR[B\setminus i(A)])\ \iso    \
    \bar\rho_G\tensor  \mR[B]   \ , \]
  given by $i:A\to B$ on the first summand,
  and by the inclusion $i:B\setminus i(A)\to B$ on the second summand,
  and by the isometry
  \[ \mR[B\setminus i(A)]\oplus (\bar\rho_G\tensor \mR[B\setminus i(A)])\ \iso\
    (\mR\oplus \bar\rho_G)\tensor \mR[B\setminus i(A)] \ \iso_{\eqref{eq:fix_of_rho}}\
    \rho_G\tensor \mR[B\setminus i(A)]\ . \]
  The fourth map is induced by the structure map 
  \[ (G\times i)_*\ : \  X(G\times A)\sm S^{G\times(B\setminus i(A))}\ \to\  X(G\times B) \]
  of the symmetric spectrum $X$.
  We note that when $G=e$ is a trivial group, the fixed point spectrum
  $F^G X$ is naturally isomorphic to~$X$.
\end{con}

We warn the reader that the above fixed point construction is
{\em not} fully homotopical. In other words, if $f:X\to Y$ is a $G$-stable
equivalence of $G$-symmetric spectra, then $F^G f:F^G X\to F^G Y$
need not be a non-equivariant stable equivalence without further hypotheses on $X$ and~$Y$.

\section{\texorpdfstring{$\M$}{M}-categories}\label{sec:Mcat}

In this section we introduce and study {\em $\M$-categories},
i.e., categories equipped with a strict action of a particular strict monoidal category
made from self-injections of a countably infinite set, see Definition \ref{def:M-cat}.
These $\M$-categories underlie the more highly structured parsummable categories,
the input data for our global K-theory.
An action of the monoidal category $\M$ on a category $\Cc$ gives rise to a notion of `support'
for objects of $\Cc$, see Definition \ref{def:support}.
The support is a countable set, possibly infinite,
and of fundamental importance for everything in this paper.
An $\M$-category is {\em tame} if the support of every object is finite.

Construction \ref{con:F^G for M-cat} introduces the $G$-fixed $\M$-category
$F^G\Cc$ associated with an $\M$-category $\Cc$ and a finite group $G$.
By Proposition \ref{prop:lambda_sharp}, the $G$-fixed category $F^G\Cc$
embeds fully faithfully into the category of $G$-objects in $\Cc$.
A morphism of $\M$-categories is a {\em global equivalence}
if the induced functor on $G$-fixed categories is a weak equivalence for
all finite groups, see Definition \ref{def:global equiv Gamma-M}.

The final topic of this section is the {\em box product} of two $\M$-categories,
defined as the full subcategory of the product
spanned by the pairs of disjointly supported objects,
see Definition \ref{def:box}.
The inclusion of the box product into the product is always a global
equivalence by Theorem \ref{thm:box2times},
and the box product restricts
to a symmetric monoidal product on the category of tame $\M$-categories,
see Proposition \ref{prop:box symmetric monoidal}.

\begin{con}[(The monoidal category of injections)]
We let $M$ denote the monoid, under composition,
of injective self maps of the set $\omega=\{0,1,2,\dots\}$.

Given a set $X$, the category $E X$ has object set $X$
and a unique morphism between any pair of objects.
More formally, the morphism set of the category~$E X$ is $X\times X$,
and the source and target maps are the projections to the second respectively 
the first factor. In other words, the pair~$(y,x)$ is the unique morphism from~$x$ to~$y$.
Composition is then forced to be $(z,y)\circ(y,x)=(z,x)$.
The category $E X$ is sometimes called the {\em chaotic category},
or the {\em indiscrete category} with object set $X$.

The functor $E$ is in fact right adjoint to the object functor
(as a functor from small categories to sets); 
hence $E$ preserves limits, in particular products,
and so it takes monoids to strict monoidal categories.
So we obtain a strict monoidal category $\M=E M$ 
whose monoid of objects is the injection monoid $M$.
\end{con}

\begin{defn}\label{def:M-cat}
An {\em $\M$-category} is a small category
equipped with a strict action of the strict monoidal category $\M$.
A {\em morphism} of $\M$-categories is a functor $F:\Cc\to \Dc$
such that the square of categories and functors
  \[ \xymatrix@C=15mm{
      \M\times \Cc
      \ar[r]^-{\M\times F}\ar[d]_{\text{act}} &
      \M\times \Dc\ar[d]^{\text{act}} \\
      \Cc\ar[r]_-F & \Dc} \]
  commutes on the nose.
\end{defn}

We emphasize that $\M$ is a {\em strict} monoidal category, i.e.,
the associativity and unit diagrams of categories and functors
commute strictly (and not just up to natural isomorphism).
Correspondingly we are looking at {\em strict} actions of $\M$ on a category,
and {\em strict} morphisms of $\M$-categories.

\begin{rk}[(Explicit structure)]
  We make the structure provided by an $\M$-action on a category $\Cc$ more explicit;
  at the same time, we also fix notation and spell out some specific relations
  that we will later use.

  Given two injections $u,v\in M$, the pair $(v,u):u\to v$
  is a morphism in the category $\M=E M$.
  Given an object $x$ in an $\M$-category $\Cc$, we write
  \[ [v,u]^x \ = \ (v,u)\diamond 1_x\ : \   u_*(x) \ \to \ v_*(x)\ .\]
  Here $\diamond:\M\times \Cc\to\Cc$ is the
  structure functor of the $\M$-action.
  As the object $x$ varies, these morphisms form a natural transformation
  from the functor $u_*:\Cc\to\Cc$ to the functor $v_*$.
  This transformation enjoys another level of naturality: if $F:\Cc\to\Dc$ is
  a morphism of $\M$-categories, then
  \[ [v,u]^{F(x)} \ = \ F\left([v,u]^x\right)\ : \ u_*(F(x)) = F(u_*(x)) \ \to \ v_*(F(x))= F(v_*(x)) \ . \]
  
  The transformations $[v,u]$ enjoy a number of properties, some of which we spell out now.
  Firstly, the relations $(u,u)=1_u$ and
  $(w,v)\circ(v,u)=(w,u)$ in the category $\M$ imply
  \[ [u,u]^x\ = \ 1_{u_*(x)}\text{\qquad and\qquad}  [w,v]^x\circ [v,u]^x\ = \  [w,u]^x\ . \]
  For $\lambda\in M$, the associativity property of the $\M$-action provides the relations
  \begin{equation} \label{eq:assoc_2}
    [v,u]^{\lambda_*(x)} \ = \  [v\lambda,u\lambda]^x  \text{\qquad and\qquad}
    \lambda_*\left( [v,u]^x\right)\ = \ [\lambda v,\lambda u]^x \ .   
  \end{equation}
\end{rk}

We claim that an $\M$-action on a category $\Cc$ is determined
by part of its structure, namely:
\begin{enumerate}[(a)]
\item the action of the injection monoid $M$ on the set of objects of $\Cc$, and
\item the isomorphisms
  \[ u_\circ^x\ =\ [u,1]^x\ :\ x\to u_*(x) \]
\end{enumerate}
for all $u\in M$ and all objects $x$ of $\Cc$.
Moreover, these pieces of data satisfy the relation
\begin{equation}\label{eq:basic_relation}
  v_\circ^{u_*(x)}  \circ  u_\circ^x  \ =\   (v u)_\circ^x
\end{equation}
for all $u$ and $v$ in $M$ and all objects $x$ of $\Cc.$
The next proposition simplifies the construction of $\M$-actions
in concrete examples, as it frees us from having to specify redundant data.

\begin{prop}\label{prop:minimal M-axioms}
  Let $\Cc$ be a category equipped with an action of the injection monoid $M$
  on the set of objects of $\Cc$, and with isomorphisms
  $u_\circ^x:x\to u_*(x)$ for all $u\in M$ and all objects $x$ of $\Cc$.
  Suppose moreover that the relation \eqref{eq:basic_relation} holds
  for all $u$ and $v$ in $M$ and all objects $x$ of $\Cc$.
  Then there is a unique extension to an $\M$-action on $\Cc$
  such that $u_\circ^x=[u,1]^x$ for all $u\in M$ and all objects $x$ of $\Cc$.
\end{prop}
\begin{proof}
  We start with the uniqueness.
  The naturality of $u_\circ=[u,1]$ is equivalent to the relation
  \begin{equation}\label{eq:U_*_on_morphisms}
    u_*(f)\ = \ u_\circ^y \circ f\circ (u_\circ^x)^{-1}    
  \end{equation}
  for every $\Cc$-morphism $f:x\to y$.
  So the effect of $u_*$ on morphisms is determined by the data (a) and (b).
  Also, the relation
  \[ [v,u]\ = \ [v,1] \circ [1,u]\ = \ v_\circ\circ \left( u_\circ\right)^{-1}\]
  shows that the natural isomorphism $[v,u]:u_*\Longrightarrow v_*$ is determined by (a) and (b).
  This completes the proof of uniqueness.

  For the existence, we define the rest of the structure as required by the
  uniqueness argument above. So we define the functor $u_*$
  on morphisms by the relation \eqref{eq:U_*_on_morphisms},
  and this visibly ensures that $u_*$ preserves identities and composition,
  so it is indeed a functor. 
  
  The relation $1_*=\Id_\Cc$ holds on objects by hypothesis.
  For $u=v=1$, the relation \eqref{eq:basic_relation} specializes to
  $1_\circ^x\circ  1_\circ^x =1_\circ^x$. Because $1_\circ^x$ is an isomorphism,
  we can conclude that $1_\circ^x$ is the identity of $x$.
  This yields that $1_*(f)=f$, i.e., $1_*$ is the identity functor.
  For $u$ and $v$ in $M$, the relation $v_*\circ u_*=(v u)_*$
  holds on objects by hypothesis,
  and on morphisms by \eqref{eq:basic_relation}.
  This shows that the action of $M$ on objects of $\Cc$
  extends uniquely to an action of $M$ on the whole category $\Cc$.
  
  Now we define the isomorphisms $[v,u]^x$ by $v_\circ^x\circ \left( u_\circ^x\right)^{-1}$,
  as we must. The naturality relation $v_*(f)\circ[v,u]^x=[v,u]^y\circ u_*(f)$
  for a morphism $f:x\to y$ then holds by definition, so there is a unique
  extension of the action of $M$ on $\Cc$ to a functor $\diamond:\M\times\Cc\to\Cc$.
  
  The final check is to verify that this action functor is associative and unital,
  where we already know this for the monoid $M$ of objects of $\M$.
  The remaining relations \eqref{eq:assoc_2} hold because
  \begin{align*}
    [v,u]^{\lambda_*(x)} \
    &= \  v_\circ^{\lambda_*(x)}\circ \left( u_\circ^{\lambda_*(x)}\right)^{-1}\\
    _\eqref{eq:basic_relation}
    &= \  (v\lambda)_\circ^x\circ \left(\lambda_\circ^x \right)^{-1}\circ
      \left( (u\lambda)_\circ^x\circ\left( \lambda_\circ^x\right)^{-1}\right)^{-1}\\
    &= \  (v\lambda)_\circ^x \circ \left( ( u\lambda)_\circ^x\right)^{-1}\
      = \     [v\lambda,u\lambda]^x
  \end{align*}
  and
  \begin{align*}
    \lambda_*\left( [v,u]^x\right)\
    &= \ \lambda_*\left(v_\circ^x\circ \left( u_\circ^x\right)^{-1} \right)\
      = \ \lambda_*\left(v_\circ^x\right)\circ \left(\lambda_*( u_\circ^x)\right)^{-1}\\
    &= \  (\lambda_\circ^{v_*(x)}\circ v_\circ^x\circ (\lambda_\circ^x)^{-1}) \circ
      (\lambda_\circ^{u_*(x)}\circ u_\circ^x\circ (\lambda_\circ^x)^{-1})^{-1}     \\
    &= \ (\lambda_\circ^{v_*(x)}\circ v_\circ^x)\circ 
      (\lambda_\circ^{u_*(x)}\circ u_\circ^x)^{-1}     \\
    _\eqref{eq:basic_relation}
    &= \  (\lambda v)_\circ^x\circ ((\lambda u_\circ)^x)^{-1} \
      =  \ [\lambda v,\lambda u]^x \ .    
  \end{align*}
\end{proof}

Now we note that $\M$-categories are closed under various kinds of constructions.

\begin{rk}[(Full subcategories)]
Let $\bar\Cc$ be a full subcategory of an $\M$-category 
closed under the action of the injection monoid $M$. More precisely, we
suppose that for every injection $u\in M$ the composite functor
\[ \bar\Cc \ \xra{\text{incl}}\ \Cc \ \xra{\ u_*\ } \ \Cc \]
has image in the full subcategory $\bar\Cc$. Then $\bar\Cc$ is an
$\M$-category in its own right, by restriction of structure.
The inclusion $\bar\Cc\to\Cc$ is a morphism of $\M$-categories.
\end{rk}

\begin{eg}[(Opposite $\M$-categories)]\label{eg:opposite M-version}
If $\Cc$ is an $\M$-category, then the opposite category $\Cc^{\op}$
inherits a canonical structure of $\M$-category.
Indeed, since the category $\M$ is a groupoid, it has an anti-automorphism
$(-)^{-1}:\M\to \M^{\op}$ that is the identity on objects
and sends every morphism to its inverse.
If $\diamond:\M\times \Cc\to\Cc$ is the action of $\M$ on $\Cc$,
then the action of $\M$ on $\Cc^{\op}$ is the composite
\[ \M\times \Cc^{\op}\ \xra{(-)^{-1}\times \Id}\
\M^{\op}\times \Cc^{\op}\ = \
(\M\times \Cc)^{\op} \  \xra{\diamond^{\op}} \ \Cc^{\op}\ . \]
More explicitly, this means that the structure functor of $\Cc^{\op}$
associated with an injection~$u\in M$ is the functor
\[ u^{\op}\ : \ \Cc^{\op}\ \to\  \Cc^{\op} \ ,\]
and the value of the natural isomorphism $[v,u]$ in $\Cc^{\op}$
at an object~$x$ is the {\em inverse} of $[v,u]^x$ in $\Cc$.
\end{eg}

\begin{eg}[(Limits and colimits of $\M$-categories)]\label{eg:(co)limits Mcat}
  The category $\cat$ of small categories is complete and cocomplete.
  Limits of small categories
  are calculated as limits of objects and limits of morphisms;
  colimits of categories are typically more difficult to describe.

  The forgetful functor $\M\cat \to \cat$
  has both a left adjoint and a right adjoint.
  The left adjoint $\cat\to \M\cat$ takes a small category $X$
  to the category $\M\times X$, with $\M$ action by multiplication on the first factor.
  The right adjoint $\cat\to \M\cat$ takes $X$
  to the category $\cat(\M,X)$ of functors from $\M$ to $X$, with natural
  transformation as morphisms.
  The $\M$-action on $\cat(\M,X)$ is adjoint to the composite
  \begin{align*}
    \M\times \cat(\M,X) \times \M\ \xra{\M\times\text{twist}} \
    &\M\times\M\times \cat(\M,X)\\
    \xra{\diamond\times\cat(\M,\Cc)}\
    &\M\times \cat(\M,X)\ \xra{\text{evaluation}}\ X \ .
  \end{align*}
  Since the forgetful functor is a left and a right adjoint,
  it creates limits and colimits.
  So the category $\M\cat$ of $\M$-categories is complete and cocomplete,
  and limits and colimits can be calculated on underlying categories.

  Similarly, for every $\M$-category $\Cc$ and every small category $J$,
  the category $\cat(J,\Cc)$ of functors from $J$ to~$\Cc$
  has a preferred structure of $\M$-category: the action functor
  is the composite
  \[     \M\times \cat(J,\Cc)\ \to \  \cat(J,\M\times \Cc)\ \xra{\ \cat(J,\diamond)} \  \cat(J,\Cc)\ . \]
  The first functor is of `assembly type' and sends an object
  $(u,F)$ of $\M\times \cat(J,\Cc)$ to the functor $(1_u,F):J\to\M\times\Cc$.
\end{eg}

The structure of an $\M$-category gives rise to an intrinsic finiteness
condition for objects, as well as an intrinsic notion of `disjointness' 
(or `orthogonality') for pairs of objects.
The finiteness and disjointness conditions both rely on the concept
of `support' of an object in an $\M$-category that we discuss now.

\begin{defn}
  Let $\Cc$ be an $\M$-category.
  An object $x$ of $\Cc$ is {\em supported} on a subset $A$ of $\omega$
  if the following condition holds: for every injection $u\in M$
  that is the identity on $A$, the relation $u_*(x)=x$ holds.
  An object $x$ is {\em finitely supported} if it is supported
  on some finite subset of $\omega$.
  The $\M$-category $\Cc$ is {\em tame} if all its objects are finitely supported.
  We write $\M\cat^\tau$ for the category of tame $\M$-categories and
  $\M$-equivariant functors.
\end{defn}

Clearly, if $x$ is supported on $A$ and $A\subseteq B\subseteq \omega$,
then $x$ is supported on $B$. Every object is supported on all of~$\omega$.
An object $x$ is supported on the empty set if and only if $u_*(x)=x$ for all $u\in M$.
So the objects supported on the empty set are precisely the $M$-fixed objects.

\begin{defn}\label{def:support} 
Let $x$ be an object of an $\M$-category.
The {\em support} of $x$ is the intersection of all finite subsets of $\omega$
on which $x$ is supported.  
\end{defn}

We write $\supp(x)$ for the support of an object $x$.
If $x$ is not finitely supported, then we agree that $\supp(x)=\omega$. 
It is important that in Definition~\ref{def:support} 
the intersection is only over {\em finite}
supporting subsets. Indeed every object is supported on the
set $\omega-\{j\}$ for every $j\in\omega$, 
because the only injection that fixes $\omega-\{j\}$ elementwise
is the identity. So without the finiteness condition, the intersection 
in Definition~\ref{def:support} would always be empty.

\begin{prop}\label{prop:finite support} 
Let $x$ be an object of an $\M$-category $\Cc$.
  \begin{enumerate}[\em (i)]
  \item The object $x$ is supported on its support $\supp(x)$.
  \item If two injections $v,\bar v\in M$ agree on $\supp(x)$, 
    then $v_*(x)=\bar v_*(x)$, $[v,u]^x=[\bar v,u]^x$  and $[u,v]^x=[u,\bar v]^x$ for all $u\in M$.
  \item Suppose that  $x$ is supported on a subset $A$ of $\omega$.
    Then for every injection $v\in M$, the object $v_*(x)$ is supported on the set $v(A)$.
    Moreover, if $x$ is finitely supported, then
    $\supp(v_*(x))=v(\supp(x))$, and $v_*(x)$ is also finitely supported.
  \item Let $f:x\to y$ be a $\Cc$-morphism,
    and suppose that $u,v\in M$ agree on $\supp(x)\cup \supp(y)$.
    Then $u_*(f)=v_*(f)$.
  \item For every morphism of $\M$-categories $F:\Cc\to\Dc$, the relation
    \[ \supp(F(x))\ \subseteq \ \supp(x) \]
    holds.
  \end{enumerate}
\end{prop}
\begin{proof}
  (i) There is nothing to show if $x$ is not finitely supported.
  Otherwise, there is a finite subset of $\omega$ on which $x$ is supported,
  and hence $\supp(x)$ is itself finite.
  While the support is defined as the intersection of infinitely many sets,
  it being finite means that we can express it as a finite intersection
  \[ \supp(x)\ = \ B_1\cap \dots\cap B_k \]
  of finite subset $B_1,\dots,B_k$ of $\omega$
  such that $x$ is supported on each $B_i$.
  By induction on $k$, it thus suffices to show the following claim:
  if $x$ is supported on two finite subsets $A$ and $B$ 
  of $\omega$, then $x$ is supported on the intersection $A\cap B$.

  We let $u\in M$ be an injection that fixes $A\cap B$ elementwise.
  We let $m$ be the maximum of the finite set $A\cup B\cup u(A)$ 
  and define $\sigma\in M$ as the involution that interchanges 
  $j$ with $j+m$ for all $j\in B\setminus A$, i.e.,
  \[ \sigma(j)\ = \
    \begin{cases}
      j+m & \text{\ for $j\in B\setminus A$, }\\
      j-m & \text{\ for $j\in (B\setminus A)+m$, and}\\
      j  & \text{\ for $j\not\in (B\setminus A)\cup ((B\setminus A)+m)$.}
    \end{cases}\]
  In particular, the map $\sigma$ fixes the set $A$ elementwise.
  Since $A$ and $u(A)$ are both disjoint from $B+m$, 
  we can choose a bijection $\gamma\in M$ such that
  \[ \gamma(j)\ = \
    \begin{cases}
      u(j) & \text{\ if $j\in A$, and}\\
      \ j  & \text{\ for $j\in B+m$.}
    \end{cases}\]
  Then $u$ can be written as the composition
  \[ u \ = \  \sigma(\sigma \gamma\sigma)(\sigma \gamma^{-1} u)\ .\]
  In this decomposition the factors $\sigma$ and $\sigma \gamma^{-1} u$
  fix $A$ pointwise, and the factor $\sigma \gamma\sigma$
  fixes $B$ pointwise. So 
  \[\sigma_*(x)\ = \ (\sigma \gamma\sigma)_*(x) \ = \ (\sigma \gamma^{-1} u)_*(x)
    \ = \ x \]
  because $x$ is supported on $A$ and on $B$.
  This gives
  \[ u_*(x) \ = \ \sigma_*( (\sigma \gamma \sigma)_*( (\sigma \gamma^{-1} u)_*(x))) \ =\  x\  .  \]
  Since $u$ was any injection fixing $A\cap B$ elementwise,
  the object $x$ is supported on $A\cap B$.
  
  (ii)
  If $x$ is not finitely supported, then $\supp(x)=\omega$,
  so $v=\bar v$, and there is nothing to show.
  So we can assume that the support of $x$ is finite.
  We start with a special case: we consider an injection $v\in M$ that is
  the identity on $\supp(x)$; we show that then the automorphism $v_\circ^x=[v,1]^x$
  of $x=v_*(x)$ is the identity. 

  We choose two injections $s,t\in M$ that are the identity on $\supp(x)$
  and whose images intersect only in $\supp(x)$.
  We define another injection $u\in M$ by
  \[ u(i) \ =\ \begin{cases}
      t(v(t^{-1}(i))) & \text{\ if $i$ is in the image of $t$, and}\\
      \quad i & \text{\ if $i$ is not in the image of $t$.}
    \end{cases} \]
  Then the relations 
  \[
    u s\ = \ s    \text{\quad and\quad} u t\ =\ t v
  \]
  hold in the monoid $M$. Because $s,t,u$ and $v$ are the identity on $\supp(x)$,
  we have $s_*(x)=t_*(x)=u_*(x)=v_*(x)=x$.
  The relation \eqref{eq:basic_relation} yields
  \[ u_\circ^x\circ s_\circ^x \ = \ (u s)_\circ ^x\ = \ s_\circ^x\ .  \]
  Because $s_\circ^x$ is invertible, we deduce that $u_\circ^x=\Id_x$.
  Moreover,
  \[ t_\circ^x\ = \ u_\circ^x\circ t_\circ^x\ =_\eqref{eq:basic_relation} \ (u t)_\circ^x\ = \
   (t v)_\circ^x\ =_\eqref{eq:basic_relation} \ \ t_\circ^x\circ v_\circ^x\ ;\]
  canceling the isomorphism $t_\circ^x$ yields $v_\circ^x=1_x$.

  Now we treat the general case. We let $v,\bar v\in M$ agree on $\supp(x)$.
  We choose a bijection $\gamma\in M$
  such that $\gamma v$ and $\gamma \bar v$ are the identity on $\supp(x)$.
  Then
  \[ v_*(x)\ = \ \gamma^{-1}_*((\gamma v)_*(x))\ = \ \gamma^{-1}_*(x)\ = \
    \gamma^{-1}_*((\gamma \bar v)_*(x))\ = \ \bar v_*(x)\ .  \]
  Moreover, $(\gamma v)_\circ^x=(\gamma\bar v)_\circ^x=1_x$ by the special case above.
  Now we let $u\in M$ be any injection. Then
  \begin{align*}
    [v,u]^x\
    &=\ [v,\bar v]^x\circ [\bar v,u]^x \
      =\ \gamma^{-1}_*([\gamma v,\gamma \bar v]^x)\circ [\bar v,u]^x \\
    &=\ \gamma^{-1}_*( (\gamma v)_\circ^x\circ ( (\gamma\bar v)_\circ^x)^{-1})\circ [\bar v,u]^x \
      = \ [\bar v, u]^x  \ .
  \end{align*}
  The proof of the relation  $[u,v]^x=[u,\bar v]^x$ is analogous.
  
  (iii) 
  We let $u\in M$ be the identity on $v(A)$.
  Then $u v$ and $v$ agree on $A$, and hence
  \[ u_*(v_*(x))\ = \ (u v)_*(x)\ = \ v_*(x) \]
  by part (ii), because $x$ is supported on $A$. 
  This shows that $v_*(x)$ is supported on $v(A)$.
  
  If $x$ is finitely supported, then $v_*(x)$ is supported on the finite set $v(\supp(x))$,
  so $v_*(x)$ is finitely supported and $\supp(v_*(x))\subseteq v(\supp(x))$.
  For the reverse inclusion we choose $h\in M$
  such that $h v$ fixes $\supp(x)$ elementwise; then $(h v)_*(x)=x$. 
  Applying the argument to $h$ and $v_*(x)$ (instead of $v$ and $x$) gives
  \[ \supp(x)\ = \ \supp(h_*(v_*(x)))\ \subseteq\ h(\supp(v_*(x))) \ ,\]
  and thus
  \[ v(\supp(x))\ \subseteq\ v(h(\supp(v_*(x)))) \  = \ 
    (v h)(\supp(v_*(x))) \  = \ \supp(v_*(x))\ . \]
  The last equation uses that $v h$ is the identity of $v(\supp(x))$,
  hence also the identity on the subset $\supp(v_*(x))$. 
  This proves the desired relation when $x$ is finitely supported.
  
  (iv) By part (ii) we have $u_*(x)=v_*(x)$, as well as $[v,u]^x=[u,u]^x=\Id_{u_*(x)}$
  and $[v,u]^y=[u,u]^y=\Id_{u_*(y)}$. So
  \[ u_*(f)\ = \ [v,u]^y\circ u_*(f) \ = \ v_*(f)\circ [v,u]^x \ = \ v_*(f) \]
   by naturality of $[v,u]$.

   (v) We let $u\in M$ be the identity on $\supp(x)$.
   Then
   \[  u_*(F(x)) \ = \ F(u_*(x)) \ = \ F(x) \ .\]
   So the $\Dc$-object $F(x)$ is supported on the support of $x$,
   and hence $\supp(F(x))\subseteq\supp(x)$.
 \end{proof}

\begin{eg}[(The $\M$-category of finite sets)]\label{eg:Fc as M-category}
  We introduce the $\M$-category $\Fc$ of finite sets.
  The objects of the category $\Fc$ are all finite subsets of the
  countably infinite set $\omega$. Morphisms in $\Fc$ are all bijections of sets.
  The functor
  \[ u_*\ : \ \Fc\ \to \ \Fc \]
  associated to an injection $u:\omega\to \omega$ is given on objects by
  \[ u_*(P) \ = \  u(P)\ , \]
  the image under $u$ of the given set.
  This clearly defines an action of the injection monoid $M$ on the objects set of $\Fc$.
  We define a bijection by
  \[ u_\circ^P\ = \ u|_P \ : \ P \ \to \ u(P)\ , \]
  the restriction of the injection $u$ to the finite set $P$.
  Then the relation \eqref{eq:basic_relation} holds,
  so there is a unique extension of this data to an $\M$-action on the
  category $\Fc$, compare Proposition \ref{prop:minimal M-axioms}.
  
  If $u\in M$ is the identity on the set $P$, then $u_*(P)=P$.
  If $Q$ is a proper subset of $P$,
  then we can choose an injection $u\in M$ that is the identity on $Q$,
  but such that $u_*(P)=u(P)\ne P$. So the support of an object $P$ of $\Fc$ is
  the set $P$ itself.
  
  A minimal variation of the construction produces a non-tame $\M$-category.
  Indeed, if we drop the finiteness condition on the objects in the definition of $\Fc$,
  we still obtain an $\M$-action by exactly the same formulas.
  In this larger $\M$-category $\bar\Fc$,
  the infinite sets are not finitely supported, and hence $\bar\Fc$ is not tame.
  The $\M$-category $\Fc$ is precisely the full $\M$-subcategory
  of $\bar\Fc$ consisting of finitely supported objects,
  and hence it is the maximal tame $\M$-subcategory of $\bar\Fc$.
\end{eg}

 \begin{eg}[(Limits and colimits of tame $\M$-categories)]\label{eg:(co)limits tame Mcat}
  As we explained in Example \ref{eg:(co)limits Mcat}, the category
  of $\M$-categories is complete and cocomplete;
  the same is true for the full subcategory of tame $\M$-categories.
  Indeed, for every $\M$-category $\Cc$,
  the full subcategory $\Cc^\tau$ of finitely supported objects is
  closed under the $\M$-action by Proposition \ref{prop:finite support} (iii).
  So $\Cc^\tau$ is an $\M$-category in its own right.
  Moreover, every morphism of $\M$-categories $F:\Dc\to\Cc$ whose source
  is tame automatically takes values in the subcategory $\Cc^\tau$, by
  Proposition \ref{prop:finite support} (v).
  So the functor sending an $\M$-category $\Cc$ to
  its full $\M$-subcategory $\Cc^\tau$ is right adjoint to the inclusion 
  \[ \M\cat^\tau \ \to \ \M\cat \ .\]
  Said differently: $\M\cat^\tau$ is a coreflective subcategory of $\M\cat$,
  and so the inclusion $\M\cat^\tau\to\M\cat$ creates colimits.
  Since $\M\cat$ is cocomplete, so is $\M\cat^\tau$, and the inclusion
  preserves colimits.
  Since $\M\cat$ is complete, so is $\M\cat^\tau$; limits in
  $\M\cat^\tau$ can be calculated by forming limits in the ambient
  category $\M\cat$, and then taking the full subcategory of finitely supported objects.  
\end{eg}

\begin{defn}
  Let $G$ be a finite group. A {\em universal $G$-set} is a countable $G$-set
  such that every subgroup of $G$ occurs as the stabilizer of infinitely many elements.
\end{defn}

The proof of the following proposition is straightforward, and we omit it.

\begin{prop}
  Let $G$ be a finite group.
  \begin{enumerate}[\em (i)]
  \item A countable $G$-set $U$ is a universal $G$-set if and only if
    every finite $G$-set admits a $G$-equivariant injection into $U$.
  \item Any two universal $G$-sets are $G$-equivariantly isomorphic.
  \item If $U'$ is a countable $G$-set and $U\subset U'$ a $G$-subset that is
    a universal $G$-set, then $U'$ is a universal $G$-set.
  \item For every subgroup $H$ of $G$, the underlying $H$-set of every universal $G$-set
    is a universal $H$-set.
  \end{enumerate}
\end{prop}

\begin{eg}
  We let $G$ be a finite group. Then the $G$-set
  \[ U \ = \ \coprod_{H}\ \mN\times G/H \]
  is a universal $G$-set, where the disjoint union runs over all subgroups of $G$.
  We also get a universal $G$-set by letting the union run over representatives
  of the conjugacy classes of subgroups of $G$.
\end{eg}

We let $G$ be a group and $A$ a finite $G$-set.
Then the set $\omega^A$
of functions from $A$ to $\omega=\{0,1,2\dots,\}$ becomes a $G$-set via
\[ (g\cdot f)(a)\ =\  f(g^{-1} a) \]
for $(g,a)\in G\times A$ and $f:A\to \omega$.

\begin{prop}\label{prop:universal G-sets}
  Let $G$  be a finite group,
  and let $A$ be a finite $G$-set with a free orbit.
  Then $\omega^A$ is a universal $G$-set.
\end{prop}
\begin{proof}
  Since $A$ has a free $G$-orbit  we may assume that $A=G\cup B$
  for some finite $G$-set $B$. Then the map
  \[ \omega^G \ \to \ \omega^{G\cup B} \ = \ \omega^A \]
  that extends a map by sending all of $B$ to 0 is a $G$-equivariant injection.
  If we can show that $\omega^G$ is a  universal $G$-set,
  then so is $\omega^A$. So we can assume without loss of generality that $A=G$.

  Now we let $H$ be  any subgroup of $G$. We choose infinitely many
  injections $\alpha_i:G/H\to \omega$, $i=1,2,\dots,$ with disjoint images.
  We define $f_i:G\to \omega$ by
  \[ f_i(g)\ = \ \alpha_i(g^{-1}H)  \ .  \]
  These maps $f_i$ are infinitely many distinct elements of 
  the $G$-set $\omega^G$ whose stabilizer group is $H$.
\end{proof}

\begin{con}[(Reparameterization of $M$-objects and $\M$-actions)]\label{con:extend M-action}
  We recall that  $M$ is the monoid, under composition,
  of injective self maps of the set $\omega=\{0,1,2,\dots\}$.
  In the rest of the paper we will often need to extend an $M$-object in
  some category to a functor defined on the category $J$ of countably infinite sets and
  injections. We will often refer to this process as {\em reparameterization}.

  Since every countably infinite set bijects with the set
  $\omega$, the inclusion of the full subcategory with only object $\omega$
  into $J$ is an equivalence. Since $M$ is the endomorphism monoid
  of $\omega$ in $J$, every $M$-object can be extended to a 
  functor on $F:J\to\Cc$, and for any two extensions $F,F'$
  there is a unique natural isomorphism $F\iso F'$ that is the identity on $\omega$.
  It will be convenient later to extend $M$-objects $X$ in $\Cc$ to functors 
  $X[-]:J\to \Cc$ in a specific way that we now explain. 
  For each countably infinite set $U$ we choose, once and for all, a bijection
  $\kappa_U:\omega\to U$, subject only to the requirement that $\kappa_\omega$ be
  the identity of $\omega$. Then we define a functor $X[-]:J\to\Cc$ on objects by 
  $X[U]= X$, the underlying $\Cc$-object that we started with.
  The value of $X[j]$ on an injection $j:U\to V$ is defined as
  \[ X[j]\ = \  (\kappa_V^{-1}\circ j \circ \kappa_U) _* \ :  \ X \ \to \ X\ ,\]
  the action of the injection $\kappa_V^{-1}\circ j \circ \kappa_U\in M$.
  This is clearly a functor and $X[\omega]=X$ as $M$-objects.
  
  One advantage of this specific way of extending $M$-objects is
  that it commutes, on the nose, with functors. More precisely, if
  $X$ is an $M$-object in a category $\Cc$ and $F:\Cc\to\Dc$
  a functor, then we view $F(X)$ as an $M$-object in $\Dc$ through $F$,
  and get an {\em equality} 
  \[ F(X)[-] \ = \ F\circ X[-]\ : \ J \ \to \ \Dc \]
  (and not just a natural isomorphism) of functors.

  In much the same way we can also extend actions of the monoidal category $\M$.
  If $U$ and $V$ are countably infinite sets, we write $\Jc(U,V)=E J(U,V)$
  for the chaotic category with object set $J(V,U)$.
  Composition of injections extends uniquely to a functor
  \[ \circ \ : \ \Jc(V,W)\times \Jc(U,V)\ \to \ \Jc(U,W) \ . \]
  This data defines a 2-category whose underlying 1-category is
  the category $J$ of countably infinite sets and injections, and such that
  there is a unique 2-morphism, necessarily invertible,
  between any pair of parallel 1-morphisms.

  We can now mimic the above extension procedure one category level higher:
  we extend an $\M$-category $\Cc$ to a strict 2-functor
  $\Jc\to\cat$ as follows. On objects we set $\Cc[U]=\Cc$.
  For countably infinite sets $U$ and $V$, we define the action functor
  \[ \Jc(U,V)\times \Cc[U]\ \to \ \Cc[V] \]
  as the composite
  \[ \Jc(U,V)\times \Cc \ \xra{(\kappa_V^{-1}\circ -\circ \kappa_U)\times\Cc}\
    \Jc(\omega,\omega)\times \Cc \ =\ \M\times \Cc \ \xra{\ \text{act}\ }\ \Cc\ . \]
  The notions of `finitely supported objects' and of `support' generalize to this context,
  by replacing subset of $\omega$ by subsets of general countably infinite sets $U$. 
\end{con}

\begin{con}[(Fixed $\M$-categories)]\label{con:F^G for M-cat}
  Given an $\M$-category $\Cc$ and a finite group $G$,
  we define a new $\M$-category $F^G \Cc$ as follows. 
  We denote by $\omega^G$ the set of maps from $G$ to $\omega$,
  on which~$G$ acts by $(g\cdot f)(h)=f(g^{-1} h)$,
  for $g,h\in G$ and $f:G\to\omega$.
  Now we let $\Cc$ be an $\M$-category. The underlying category of $F^G \Cc$
  is then given by
  \[ F^G \Cc\ =  \ \Cc[\omega^G]^G\ , \]
  the $G$-fixed category of the $G$-category $\Cc[\omega^G]$.
  The $\M$-action on $\Cc$ induces a natural action on $F^G \Cc$
  as follows.
  The injection monoid $M$ acts on $\omega^G$ by postcomposition, i.e.,
    \[  (u\cdot f)(h)\ = \ u(f(h)) \]
  for $u\in M$, $f:G\to\omega$ and $h\in G$.
  Then 
  \[ (u\cdot -)_*\ : \ \Cc[\omega^G] \ \to \ \Cc[\omega^G] \]
  is a $G$-equivariant functor, so it restricts to
  a functor on $G$-fixed subcategories
  \[ (F^G u)_* \ = \  ((u\cdot -)_*)^G : \
    F^G \Cc\ = \ \Cc[\omega^G]^G \ \to \ \Cc[\omega^G]^G \ = \ F^G \Cc \ .\]
  Given another injection $v\in M$,
  we define the natural isomorphism $[v,u]$ at an object $x$ of $F^G \Cc$ as
  \[ [v,u]_{F^G \Cc}^x\ = \ [v\cdot -,u\cdot -]^x\ . \]
  More precisely, $v\cdot -, u\cdot -:\omega^G\to\omega^G$ are two injections,
  and hence objects of the category $\Jc(\omega^G,\omega^G)$; and $(v\cdot -,u\cdot -)$
  is the unique morphism in $\Jc(\omega^G,\omega^G)$ from $u\cdot -$ to $v\cdot -$.
  The strict monoidal category $\Jc(\omega^G,\omega^G)$ acts on $\Cc[\omega^G]=\Cc$
  via reparameterization of the given $\M$-action
  (i.e., via restriction along the monoidal functor $\Jc(\omega^G,\omega^G)\to\M$
  given by conjugation by the bijection $\kappa_{\omega^G}:\omega\to\omega^G)$,
  and $[v,u]_{F^G\Cc}$ is the restriction to $F^G\Cc=\Cc[\omega^G]^G$
  of the natural transformation specified by $(v\cdot -,u\cdot -)$.
  So if we were to fully expand all definitions, we would discover that
  $[v,u]^x_{F^G\Cc}=[\kappa_{\omega^G}^{-1}(v\cdot-)\kappa_{\omega^G},\kappa_{\omega^G}^{-1}(u\cdot-)\kappa_{\omega^G}]^x$. The above shorthand notation $[v\cdot-,u\cdot-]^x$ is a slight abuse, but more suggestive,
  and we'll use it in what follows.
  
  We must show that $[v,u]_{F^G \Cc}^x$ is a morphism in the category
  $F^G \Cc$, i.e., that it is $G$-fixed whenever the object $x$ is.
  For every group element $g\in G$ we let $l_g:\omega^G\to\omega^G$ denote left multiplication
  by $g$. Then
  \begin{align*}
    l^g_*\left([v,u]_{F^G \Cc}^x\right)\
  &= \ [l^g (v\cdot -),l^g(u\cdot -)]^x\
  = \ [(v\cdot -)l^g,(u\cdot -)l^g]^x\\
  &= \ [v\cdot -,u\cdot -]^{l^g_*(x)}\
  = \ [v\cdot -,u\cdot -]^x\ = \ [v,u]_{F^G \Cc}^x\ .
\end{align*}
  The second equation is the fact that the $M$-action on $\omega^G$ commutes with the $G$-action.
\end{con}

\begin{prop}\label{prop:F^G_preserves_tame}
  Let $\Cc$ be an $\M$-category and $G$ a finite group.
  If $\Cc$ is tame, then the fixed point $\M$-category $F^G\Cc$ is tame.
\end{prop}
\begin{proof}
  We let $x$ be any object of $\Cc[\omega^G]$. Because the $\M$-category $\Cc$ is tame,
  the object $x$ is supported on some finite subset $T$ of $\omega^G$.
  We define
  \[ I(T)\ = \ \bigcup_{\alpha\in T}\text{image}(\alpha) \ , \]
  which is a finite subset of $\omega$.

  Now we suppose that $x$ is $G$-fixed, and hence belongs to $F^G\Cc=\Cc[\omega^G]^G$. 
  We claim that with respect to the $\M$-action on $F^G\Cc$, the object $x$ is supported on $I(T)$.
  To this end we let $u\in M$ be an injection that is the identity on $I(T)$.
  Then for all $(\alpha,g)\in T\times G$ we have
  \[ (u\cdot \alpha)(g) \ = \ u(\alpha(g))\ = \ \alpha(g)  \]
  because $\alpha(g)\in I(T)$. Hence $u\cdot\alpha=\alpha$,
  and so $u\cdot-:\omega^G\to\omega^G$ is the identity on the set $T$. Thus 
  $ u^{F^G\Cc}_*(x) = x$.
  This completes the proof that $x$ is supported on the finite set $I(T)$.  
  So the $\M$-category $F^G\Cc$ is tame.
\end{proof}

\begin{con}[($G$-fixed objects versus $G$-objects)]
  We let $\Cc$ be an $\M$-category and $G$ a finite group.
  A {\em $G$-object} in $\Cc$ is an object $x$ of $\Cc$ equipped
  with a $G$-action, i.e., a monoid homomorphism $\rho:G\to\Cc(x,x)$
  to the endomorphism monoid. 
  We denote by $G\Cc$ the category of $G$-objects in $\Cc$ 
  with $G$-equivariant $\Cc$-morphisms.

  We shall now explain that the $G$-fixed $\M$-category $F^G\Cc$,
  defined in Construction \ref{con:F^G for M-cat}, embeds fully faithfully
  into the category of $G$-objects in $\Cc$. This embedding is often -- but not always --
  essentially surjective (and hence an equivalence of categories).
  In Section \ref{sec:saturation} we look more closely at the {\em saturated} $\M$-categories,
  i.e., those for which the embeddings $F^G\Cc\to G\Cc$ are equivalences
  for all finite groups.

  We choose an injection $\lambda:\omega^G\to\omega$.
  Based on this choice, we define a functor 
  \begin{equation}\label{eq:lambda sharp}
    \lambda_\flat\ : \ F^G\Cc\ = \ \Cc[\omega^G]^G \ \to \ G\Cc\ ,
  \end{equation}
  a refinement of the functor $\lambda_*:\Cc[\omega^G]\to \Cc$.
  As we shall see in Proposition \ref{prop:lambda_sharp} below,
  different choices of injections yield canonically isomorphic functors.
  Given a $G$-fixed object $x$ of $\Cc[\omega^G]$ and an element $g\in G$, the morphism
  \[ [\lambda l^g,\lambda]^x \ : \ \lambda_*(x)\ \to \ \lambda_*(l^g_*(x))\ =\lambda_*(x)  \]
  is an endomorphism of $\lambda_*(x)$, where $l^g:\omega^G\to\omega^G$ is left translation by $g$.
  If $h\in G$ is another group element, then
  \begin{align*}
    [\lambda l^g,\lambda]^x\circ [\lambda l^h,\lambda]^x\
    &= \ [\lambda l^g,\lambda]^{l^h_*(x)}\circ [\lambda l^h,\lambda]^x \\
    &= \ [\lambda l^g l^h,\lambda l^h]^x\circ [\lambda l^h,\lambda]^x \
    = \ [\lambda l^g l^h,\lambda]^x \ = \ [\lambda l^{g h},\lambda]^x \ .
  \end{align*}
  So as $g$ varies, the morphisms $[\lambda l^g,\lambda]^x$ define a $G$-action on $\lambda_*(x)$,
  and we write $\lambda_\flat(x)$ for this $G$-object in $\Cc$.
  Now we let $f:x\to y$ be a $G$-fixed morphism between $G$-fixed objects of $\Cc[\omega^G]$.
  Then
  \begin{align*}
    [\lambda l^g,\lambda]^y \circ \lambda_*(f)  
    &= \      \lambda_*([l^g,1]^y)\circ \lambda_*(f) \
      = \     \lambda_*([l^g,1]^y\circ f) \
      = \     \lambda_*(l^g_*(f)\circ [l^g,1]^x) \\
    &= \     \lambda_*(f\circ [l^g,1]^x) \
      = \     \lambda_*(f)\circ \lambda_*([l^g,1]^x)\
      = \ \lambda_*(f)\circ [\lambda l^g,\lambda]^x \ .
  \end{align*}
  So $\lambda_*(f)$ is $G$-equivariant, and we set $\lambda_\flat(f)=\lambda_*(f)$
  on morphisms.
\end{con}

\begin{prop}\label{prop:lambda_sharp}
  Let $\Cc$ be an $\M$-category and $G$ a finite group.
  \begin{enumerate}[\em (i)]
  \item For every injection $\lambda:\omega^G\to\omega$, the functor
    $\lambda_\flat:F^G\Cc\to G\Cc$ is fully faithful.
  \item If $\mu:\omega^G\to\omega$ is another injection, then the morphisms
    $[\mu,\lambda]^x:\lambda_*(x)\to\mu_*(x)$ form a natural isomorphism
    from $\lambda_\flat$ to $\mu_\flat$.
  \end{enumerate}
\end{prop}
\begin{proof}
  (i)  The functor $\lambda_*:\Cc[\omega^G]\to\Cc$ is an equivalence of categories, so
  for all objects $x$ and $y$ of $\Cc[\omega^G]$, the induced map of morphism sets
  \[ \lambda_* \ : \  \Cc[\omega^G](x,y)\ \to\ \Cc(\lambda_*(x),\lambda_*(y)) \]
  is bijective.
  
  Now we let $x$ and $y$ be $G$-fixed objects of $\Cc[\omega^G]$.
  As we already argued above, the relations
  \[     [\lambda l^g,\lambda]^y \circ \lambda_*(f)  
    = \     \lambda_*(l^g_*(f)\circ [l^g,1]^x) \text{\qquad and \qquad}
    \lambda_*(f)\circ [\lambda l^g,\lambda]^x \ = \  \lambda_*(f\circ [l^g,1]^x)  \]
  hold for every $\Cc[\omega^G]$-morphism $f:x\to y$.
  Because $\lambda_*$ is faithful and $[l^g,1]^x$ is an isomorphism,
  these two morphisms are equal if and only if $l^g_*(f)=f$.
  This shows that $f:x\to y$ is $G$-fixed if and only if the morphism
  $\lambda_*(f):\lambda_*(x)\to\lambda_*(y)$  is $G$-equivariant
  with respect to the $G$-actions above.
  So the functor $\lambda_\flat$ is fully faithful.

  (ii) We let $x$ be a $G$-fixed object of $\Cc[\omega^G]$ and $g\in G$. Then
  \begin{align*}
      [\mu,\lambda]^x\circ [\lambda l^g,\lambda]^x
    &= \ [\mu,\lambda]^{l^g_*(x)}\circ [\lambda l^g,\lambda]^x\\
    &= \ [\mu l^g,\lambda l^g]^x\circ [\lambda l^g,\lambda]^x\
    = \ [\mu l^g,\lambda]^x\
    = \ [\mu l^g,\mu]^x \circ [\mu,\lambda]^x\ .  
  \end{align*}
 So the natural $\Cc$-isomorphism $[\mu,\lambda]^x$ is also $G$-equivariant,
 and hence a natural isomorphism from $\lambda_\flat(x)$ to $\mu_\flat(x)$ in $G\Cc$.
\end{proof}

Now we come to another key definition, that of {\em global equivalences}
of $\M$-categories.
We call a functor $F:\Ac\to\Bc$ between small categories a {\em weak equivalence}
if the induced morphism of nerves $N F:N\Ac\to N\Bc$
is a weak equivalence of simplicial sets; equivalently, the
induced continuous map of geometric realizations $|N F|:|N\Ac|\to |N\Bc|$
must be a weak equivalence (or, equivalently, a homotopy equivalence) of spaces.
We recall that for a finite group $G$, any two universal $G$-sets are $G$-equivariantly
isomorphic, and moreover the $G$-set $\omega^G$ of functions from $G$ to $\omega=\{0,1,2,\dots\}$
is universal by Proposition \ref{prop:universal G-sets}.
This shows the equivalence of conditions (a), (b) and (c) in the following definition.

\begin{defn}\label{def:global equiv Gamma-M}
  A morphism of $\M$-categories $\Phi:\Cc\to \Dc$ is a {\em global equivalence}
  if for every finite group $G$, the following equivalent conditions hold:
  \begin{enumerate}[(a)]
  \item For every universal $G$-set $\Uc$, the functor
    $\Phi[\Uc]^G: \Cc[\Uc]^G\to \Dc[\Uc]^G$
    is a weak equivalence of categories.
  \item For some universal $G$-set $\Uc$, the functor
    $\Phi[\Uc]^G: \Cc[\Uc]^G\to \Dc[\Uc]^G$
    is a weak equivalence of categories.
  \item The functor
    $F^G\Phi:F^G\Cc\to F^G\Dc$  is a weak equivalence of categories.
  \end{enumerate}
\end{defn}

\begin{con}
  We consider two finite groups $G$ and $K$ and an $\M$-category $\Cc$.
  The $(K\times G)$-categories
  $\Cc[\omega^G][\omega^K]$ and $\Cc[(\omega^K)^G]$
  both have $\Cc$ as their underlying category, and they come with specific
  (and typically different) commuting actions of $K$ and $G$,
  through the reparameterization procedure of Construction \ref{con:extend M-action}.
  We will now specify  a $(K\times G)$-equivariant isomorphism of categories   
  from $\Cc[\omega^G][\omega^K]$ to $\Cc[(\omega^K)^G]$.
  This isomorphism, and its effect on various fixed subcategories, will show up several times
  in the remaining part of the paper; so we explain the construction in detail.
  
  The reparameterization procedure in Construction \ref{con:extend M-action}
  involves a choice of equivalence between the category $J$ of countably infinite sets
  and injections and its full subcategory spanned by the object $\omega$.
  This equivalence is controlled by chosen bijections
  $\kappa_U:\omega\to U$ for all countably infinite sets $U$, subject
  to the only requirement that $\kappa_\omega$ is the identity.
  In particular, we have independently made such choices for the countably infinite sets
  $\omega^G$, $\omega^K$ and $(\omega^K)^G$.
  We define the {\em intertwiner} $\Im:\omega\to\omega$  as the composite bijection
  \begin{equation}\label{eq:intertwiner} 
    \omega \ \xra{\kappa_{\omega^G}} \ \omega^G \ \xra{(\kappa_{\omega^K})^G}\ (\omega^K)^G \
        \xra{\kappa_{(\omega^K)^G}^{-1}} \ \omega \ .    
  \end{equation}
  The associated action functor $\Im_*:\Cc\to\Cc$ is then an isomorphism of categories.
\end{con}

\begin{prop}
  Let $\Cc$ be an $\M$-category, and let $K$ and $G$ be finite groups.
  Then the action of the intertwiner \eqref{eq:intertwiner} is 
  a $(K\times G)$-equivariant isomorphism    
  \begin{equation}\label{eq:reparametrize}
    \Im_* \ : \ \Cc[\omega^G][\omega^K] \  \iso \ \Cc[ (\omega^K)^G]
  \end{equation}      
  with respect to the reparameterized actions.
\end{prop}  
\begin{proof}
  Both $\Cc[\omega^G][\omega^K]$ and $\Cc[(\omega^K)^G]$ have $\Cc$ as their underlying category,
  and the $(K\times G)$-actions are by iterated and one-step reparameterization, respectively.
  The $(\M\times G)$-action on $\Cc[\omega^G]$ is obtained from the
  original $\M$-action by restriction along the strict monoidal functor
  \[ \M\times G\ \to \ \M  \]
  that sends an object $(u,g)\in M\times G$ to the composite injection
  \[ \omega \ \xra{\kappa_{\omega^G}} \ \omega^G \ \xra{(u,g)\cdot-}\ \omega^G \ \xra{\kappa_{\omega^G}^{-1}}\ \omega\ ,
  \]
  and is uniquely extended to morphisms.
  The second map is the action of the element $(u,g)$ on $\omega^G$, i.e.,
  \[ ((u,g)\cdot f)(h)\ = \ u(f(g^{-1} h)) \]
  for $f\in \omega^G$ and $h\in G$.
  Iterating this, the $(K\times G)$-action on $\Cc[\omega^G][\omega^K]$ is obtained from the
  original $\M$-action by restriction along the strict monoidal functor
  \[ K\times G\  \to \ \M  \]
  that sends $(k,g)\in K\times G$ to the composite injection
  \[ \omega \ \xra{\kappa_{\omega^G}} \ \omega^G \ \xra{(\kappa_{\omega^K})^G} \   (\omega^K)^G \
    \xra{ (k,g)\cdot-}\ (\omega^K)^G \ \xra{(\kappa_{\omega^K}^{-1})^G}\ \omega^G\ 
    \xra{\kappa_{\omega^G}^{-1}}\ \omega\ .  \]
  The two bijections
  \[  (\kappa_{\omega^K})^G\circ\kappa_{\omega^G}\ , \ \kappa_{(\omega^K)^G}\ : \
    \omega \ \to \ (\omega^K)^G \]
  need not be related in any way, which is why the two $(K\times G)$-actions will
  typically be different (unless one of the two groups is trivial).
  But by design, the intertwiner
  $\Im = \kappa_{(\omega^K)^G}^{-1}\circ (\kappa_{\omega^K})^G\circ\kappa_{\omega^G}$
  accounts for the difference between these two unrelated bijections.
  So the action of the induced isomorphism of categories
  $\Im_*:\Cc[\omega^G][\omega^K] \to\Cc[(\omega^K)^G]$
  mediates between the two actions, as claimed.  
\end{proof}

\begin{prop}\label{prop:F^G preserves global M}
  Let $\Phi:\Cc\to \Dc$ be a global equivalence of $\M$-categories.
  Then for every finite group $G$, the morphism
  $F^G \Phi:F^G \Cc\to F^G \Dc$ is a global equivalence of $\M$-categories.
\end{prop}
\begin{proof}
  We let $K$ be another finite group.
  The restriction
  of the $(K\times G)$-equivariant isomorphism $\Im_*:\Cc[\omega^G][\omega^K]\iso \Cc[(\omega^K)^G]$ from 
  \eqref{eq:reparametrize} to $(K\times G)$-fixed subcategories is an isomorphism
  \[  \Im_*^{K\times G}\ : \
    F^K(F^G\Cc)\ = \ ( \Cc[\omega^G][\omega^K])^{K\times G}\ \iso\ \Cc[(\omega^K)^G]^{K\times G}\ .  \]
  Since $\Im_*$ is natural for morphisms of $\M$-categories,
  the following square of categories and functors commutes:
  \[ \xymatrix@C=18mm{
      F^K(F^G \Cc)\ar[r]^-{\Im_*^{K\times G}}_\iso \ar[d]_{F^K(F^G\Phi)} &
      \Cc[(\omega^K)^G]^{K\times G}\ar[d]^{\Phi[(\omega^K)^G]^{K\times G}}\\
      F^K(F^G \Dc)\ar[r]_-{\Im_*^{K\times G}}^\iso  & \Dc[(\omega^K)^G]^{K\times G}
    } \]
  Since $(\omega^K)^G$ is a universal $(K\times G)$-set and $\Phi$ is a global equivalence,
  the right vertical functor is a weak equivalence of categories.
  The horizontal functors are isomorphisms, so the
  left vertical functor $F^K(F^G\Phi)$ is a weak equivalence of categories.
  Since $K$ was any finite group, this shows that $F^G\Phi$ is a global equivalence.
\end{proof}

The final topic of this  section is the {\em box product},
a certain symmetric monoidal product for tame $\M$-categories.
In Section \ref{sec:K of I} will then define {\em parsummable categories}
as the tame $\M$-categories equipped with a commutative multiplication
with respect to  the box product.

The diagonal $\M$-action makes the product $\Cc\times\Dc$
of two $\M$-categories $\Cc$ and $\Dc$ into an $\M$-category.
With respect to this diagonal $\M$-action,
$\Cc\times\Dc$ is moreover a product of $\Cc$ and $\Dc$
in the category of $\M$-categories and strict morphisms,
compare Example \ref{eg:(co)limits Mcat}.

\begin{defn}\label{def:box}
  Let $\Cc$ and $\Dc$ be $\M$-categories. An object $(c,d)$ of $\Cc\times\Dc$ is
  {\em disjointly supported} if there are disjoint subsets $A$ and $B$ of $\omega$
  such that $c$ is supported on $A$, and $d$ is supported on $B$.
  The {\em box product} $\Cc\boxtimes\Dc$ of two $\M$-categories 
  is the full subcategory of the product $\M$-category $\Cc\times\Dc$
  generated by the disjointly supported objects.
\end{defn}

In the previous definition, we do not insist that the objects $c$ and $d$ must be finitely supported.
However, I doubt that the construction is particularly useful in this generality,
and we will mostly be interested in the box product for {\em tame} $\M$-categories.
If $c$ and $d$ are finitely supported objects of $\Cc$ and $\Dc$, respectively,
then $(c,d)$ is disjointly supported if and only if $\supp(c)\cap\supp(d)=\emptyset$.
Proposition \ref{prop:finite support} (iii) shows that the subcategory
$\Cc\boxtimes\Dc$ of $\Cc\times\Dc$ is invariant under the diagonal $\M$-action;
hence the product $\M$-action on $\Cc\times\Dc$ restricts
to an $\M$-action on $\Cc\boxtimes\Dc$.

The following theorem makes precise that assuming disjoint support
is no loss of generality, even globally.

\begin{theorem}\label{thm:box2times}
  For all $\M$-categories $\Cc$ and $\Dc$,
  the inclusion $\Cc\boxtimes\Dc\to\Cc\times\Dc$
  is a global equivalence of $\M$-categories.
\end{theorem}
\begin{proof}
  We let $G$ be a finite group.
  Since  $(\Cc\boxtimes\Dc)[\omega^G]$ is a full subcategory
  of $(\Cc\times\Dc)[\omega^G]=\Cc[\omega^G]\times\Dc[\omega^G]$,
  the inclusion restricts to a fully faithful functor 
  $F^G(\Cc\boxtimes\Dc)\to F^G(\Cc\times\Dc)=(F^G\Cc)\times(F^G\Dc)$ on $G$-fixed subcategories.
  We show that this restricted functor is also dense, and hence an equivalence of categories.
  This shows in particular that the functor is a weak equivalence of categories.

  To prove the claim we consider $G$-fixed objects $c$ of $\Cc[\omega^G]$ and $d$ of $\Dc[\omega^G]$.
  Since $\omega^G$ is a universal $G$-set, we can choose  $G$-equivariant injections
  $u,v:\omega^G\to \omega^G$ with disjoint images.
  Then
  \[ l^g_*(u_*(c))\ = \ (l^g u)_*(c)\ = \
    (u l^g)_*(c)\ = \ u_*( l^g_*(c))\ = \ u_*(c)\ ,  \]
  so $u_*(c)$ is another $G$-fixed object of $\Cc[\omega^G]$.
  Moreover,
  \[ l^g_*([u,1]^c)\ = \ [l^g u,l^g]^c\ = \ [u l^g,l^g]^c\ = \
    [u,1]^{l^g_*(c)}\ = \  [u,1]^c\ ,  \]
  i.e., the $\Cc[\omega^G]$-isomorphism $[u,1]^c:c\to u_*(c)$ is $G$-fixed.
  Similarly, $v_*(d)$ is $G$-fixed and
  the isomorphism $[v,1]^d:d\to v_*(d)$ is $G$-fixed.
  So $(c,d)$ is isomorphic in $F^G(\Cc\times\Dc)$ to the object $(u_*(c),v_*(d))$.
  Proposition \ref{prop:finite support} (iii) shows that $u_*(c)$ is supported on the image of $u$,
  and $v_*(d)$ is supported on the image of $v$. So the pair $(u_*(c),v_*(d))$
  is disjointly supported, and thus an object of the subcategory $F^G(\Cc\boxtimes\Dc)$.
\end{proof}

An object $(c,d)$ of $\Cc\times\Dc$ is supported on the set $\supp(c)\cup\supp(d)$.
So we conclude:

\begin{cor} If the $\M$-categories $\Cc$ and $\Dc$ are tame,
  then the $\M$-category $\Cc\boxtimes\Dc$ is tame.
\end{cor}

Given two morphisms between $\M$-categories $F:\Cc\to \Cc'$ and $G:\Dc\to\Dc'$,
the product functor $F:\Cc\times\Dc\to \Cc'\times\Dc'$ takes the subcategory
$\Cc\boxtimes\Dc$ to the subcategory $\Cc'\boxtimes\Dc'$,
by Proposition \ref{prop:finite support} (v).
So $F\times G$ restricts to a morphism of $\M$-categories
\[ F\boxtimes G\ : \ \Cc\boxtimes\Dc\ \to \ \Cc'\boxtimes\Dc'\ . \]
This makes the box product of tame $\M$-categories a functor
\[ \boxtimes \ : \ \M\cat^\tau\times\M\cat^\tau\ \to \ \M\cat^\tau \ .\]

The associativity, symmetry and unit isomorphisms of the cartesian product
of $\M$-categories clearly restrict to the box product;
hence they inherit the coherence conditions required for a symmetric monoidal product.
For example, the associativity isomorphism is given by
\[ (\Cc\boxtimes\Dc)\boxtimes\Ec \ \iso \
  \Cc\boxtimes(\Dc\boxtimes\Ec) \ ,\quad ((f,g),h)\ \longmapsto\ (f,(g,h))\ .
\]
We thus conclude:

\begin{prop}\label{prop:box symmetric monoidal}
  The box product is a symmetric monoidal structure on the category of tame $\M$-categories
  with respect to the associativity, symmetry and unit isomorphisms inherited
  from the cartesian product. The terminal $\M$-category is a unit object for the
  box product.
\end{prop}

\section{From \texorpdfstring{$\Gamma$-$\M$}{Gamma-M}-categories
  to symmetric spectra}
\label{sec:Gamma-M_to_sym}

In this section we introduce and study
a global equivariant variation of Segal's construction \cite{segal:cat coho}
that turns $\Gamma$-categories into spectra,
where $\Gamma$ is the category of finite based sets.
Our version in Construction \ref{con:spectrum from Gamma-M}
accepts a $\Gamma$-$\M$-category $Y$ as input,
and it returns a  symmetric spectrum $Y\td{\mS}$.
All the key qualitative properties of Segal's machine
have global equivariant generalizations:
the symmetric spectrum $Y\td{\mS}$ is globally connective and
globally semistable by Proposition \ref{prop:Y<S> globally connective-semistable};
if the $\Gamma$-$\M$-category $Y$ is {\em globally special}
in the sense of Definition \ref{def:globally special Gamma-cat} below,
then $Y\td{\mS}$ is a restricted global $\Omega$-spectrum,
see Theorem \ref{thm:Y(S) is global Omega}.
The main case of interest for us will be the
$\Gamma$-$\M$-category constructed from a parsummable category,
see Construction~\ref{con:Gamma from I} below,
as this yields the global K-theory spectrum.

A key feature of our delooping machine is its sensitivity to global equivariant structure.
For example, for every finite group $G$,
the underlying $G$-symmetric spectrum of $Y\td{\mS}$
is $G$-stably equivalent to the $G$-symmetric spectrum obtained by evaluating
a specific special $\Gamma$-$G$-category on spheres, see Theorem \ref{thm:a-b-maps equivalence}
for the precise statement.
Also, the assignment $Y\mapsto Y\td{\mS}$ commutes with $G$-fixed points
in the following sense: we show in Theorem \ref{thm:G-fixed Gamma-cat}
that the $G$-fixed point spectrum of $Y\td{\mS}$ is globally equivalent
to the result of applying our delooping to the $G$-fixed $\Gamma$-$\M$-category $F^G Y$.

\begin{con}[(Prolongation of $\Gamma$-spaces)]
We write $\Gamma$ for the category whose objects are the
finite pointed sets $n_+=\{0,1,\dots,n\}$ for $n\geq 0$, with 0 being the basepoint.
Morphisms in $\Gamma$ are all basepoint preserving maps.
A {\em $\Gamma$-object} in a category $\Dc$ is a functor $X:\Gamma\to\Dc$
that is reduced in the sense that $X(0_+)$ is a terminal object of $\Dc$.
A morphism of $\Gamma$-objects is a natural transformation of functors.
Cases of particular interest are when $\Dc$ is the category $\cat$ of small categories,
the category $\bT$ of spaces, or equivariant variations of these where a finite
group or the monoidal category $\M$ acts.

In particular, a {\em $\Gamma$-space} is a reduced functor $X:\Gamma\to\bT$ from $\Gamma$
to the category of spaces,
i.e., the value $X(0_+)$ is a one-point space.
We may then view a $\Gamma$-space as a functor to {\em pointed} spaces,
where $X(n_+)$ is pointed by the image of the map $X(0_+)\to X(n_+)$
induced by the unique morphism $0_+\to n_+$ in $\Gamma$.
A $\Gamma$-space $X$ can be extended to a continuous functor on the category
of based spaces by a coend construction. 
If $K$ is a pointed space, 
the value of the extended functor on $K$ is given by
\[
 X(K)\ = \ \int^{n_+\in\Gamma} X(n_+) \times K^n\,
\ = \ \left({\coprod}_{n\geq 0}\,   X(n_+)\times K^n\right) /\sim \ , 
  \]
where we use that $K^n=\map_*(n_+,K)$ is contravariantly functorial in $n_+$.
In more detail, $X(K)$ is the quotient space of the disjoint union
of the spaces $X(n_+)\times K^n$ by the equivalence relation generated by
\[ (X(\alpha)(x); k_1,\dots,k_n)\ \sim\ (x; k_{\alpha(1)},\dots,k_{\alpha(m)}) \]
for all morphisms $\alpha:m_+\to n_+$ in $\Gamma$, all $x\in X(m_+)$, and all 
$(k_1,\dots,k_n)\in K^n$.
Here $k_{\alpha(i)}$ is to be interpreted as the basepoint of $K$
whenever $\alpha(i)=0$. In general, quotient spaces of weak Hausdorff spaces
need not be weak Hausdorff, so it is not completely obvious that 
the space $X(K)$ is again compactly generated. However, we show in
\cite[Proposition B.26 (i)]{schwede:global} that this is the case.

We will not distinguish notationally between the original $\Gamma$-space
and its extension. We write $[x; k_1,\dots,k_n]$
for the equivalence class in $X(K)$ 
of a tuple $(x;k_1,\dots,k_n)\in X(n_+) \times K^n$.
The assignment $(X,K)\mapsto X(K)$ is functorial in the $\Gamma$-space $X$
and the based space $K$. In particular, if~$X$ comes with an action by
a group $G$ and $K$ is equipped with an action by another group $H$,
then $G\times H$ acts on $X(K)$ by
\[ (g,h)\cdot  [x; k_1,\dots,k_n]\ = \ [g x; h k_1,\dots,h k_n]\ . \]
We will often be interested in the case where $G=H$, 
i.e., we evaluate a $\Gamma$-$G$-space $X$ on a $G$-space $K$,
and then we usually restrict to the diagonal $G$-action.
The extended functor is continuous and comes with a
continuous, based {\em assembly map}
\begin{equation}\label{eq:assembly} 
 \alpha\ : \ X(K)\sm L \ \to \ X(K\sm L) \ , \quad
 \alpha([x; k_1,\dots,k_n]\sm l)\ = \ [x; k_1\sm l,\dots,k_n\sm l]\ . \end{equation}
The assembly map is associative, unital and natural in all three variables.
\end{con}

We can turn a $\Gamma$-category $\Dc:\Gamma\to\cat$ into a $\Gamma$-space
by taking nerve and geometric realization objectwise. 
To simplify the notation we suppress the nerve functor in the notation, write
$|\Dc|$ for the composite functor
\[ \Gamma \ \xra{\ \Dc\ } \ \cat \ \xra{\text{nerve}} \ 
\text{(simplicial sets)} \
\xra{\ |-|\ }\ \bT \]
and refer to $|\Dc|$ as the {\em realization} of the $\Gamma$-category $\Dc$.
If the original $\Gamma$-category comes with an action of
a monoid, group or monoidal category, then
the realization inherits an action of the same monoid or group,
or of the geometric realization of the monoidal category, respectively.
In particular, realization turns
$\Gamma$-$\M$-categories into $\Gamma$-$|\M|$-spaces.
The main example we have in mind is $\Dc=\gamma(\Cc)$,
the $\Gamma$-$\M$-category associated with a parsummable category~$\Cc$,
to be discussed in Construction~\ref{con:Gamma from I} below.

\begin{con}\label{con:spectrum from Gamma-M}
We now define the {\em associated symmetric spectrum}
$Y\td{\mS}$ of a $\Gamma$-$\M$-category $Y$, i.e.,
a $\Gamma$-category equipped with
a left action of the monoidal category $\M=E M$.
For a finite set~$A$ we continue to denote by $\omega^A$ the set of maps from $A$ to $\omega$.
We explained in Construction \ref{con:extend M-action}
how to extend an object with an action of the monoid $M$
to a functor defined on countably infinite sets and injections.
The value of $Y\td{\mS}$ at a non-empty finite set~$A$ is
\[ Y\td{\mS}(A) \ = \ |Y[\omega^A]|(S^A) \ ,\]
the value of the $\Gamma$-space $|Y[\omega^A]|$ on the $A$-sphere.
For the empty set we declare
\[  Y\td{\mS}(\emptyset) \ = \  |Y(1_+)^{\supp=\emptyset}| \ , \]
the realization of the full subcategory of $Y(1_+)$
on the objects supported on the empty set.

To define the structure map associated with an injection $i:A\to B$ we let
\[ i_!\ :\ \omega^A\ \to\ \omega^B \] 
be the map given by
\begin{equation}\label{eq:extension_by_zero}
i_!(f)(b)\ = \ 
\begin{cases}
f(a)  & \text{ if $b=i(a)$, and}\\
\ 0 & \text{ if $b\not\in i(A)$.}
\end{cases}
\end{equation}
So if $i$ is bijective, then $i_!$ is precomposition with $i^{-1}$.
On the other hand, if $i$ is the inclusion of a subset,
then $i_!$ is extension by 0.

The structure map
\[ i_* \ : \  Y\td{\mS}(A)\sm S^{B\setminus i(A)}\ \to\ Y\td{\mS}(B)\]
is now defined as the diagonal composite in the commutative diagram:
\[ \xymatrix@C=35mm{
|Y[\omega^A]|(S^A)\sm S^{B\setminus i(A)}\ar[r]^-{|Y[i_!]|(S^A)\sm S^{B\setminus i(A)}} \ar[d]_{\eqref{eq:assembly}} &
|Y[\omega^B]|(S^A)\sm S^{B\setminus i(A)} \ar[d]^{\eqref{eq:assembly}} \\
|Y[\omega^A]|(S^A\sm S^{B\setminus i(A)})\ar[r]_-{|Y[i_!]|(S^A\sm S^{B\setminus i(A)})} \ar[d]_\iso &
|Y[\omega^B]|(S^A\sm S^{B\setminus i(A)}) \ar[d]^\iso \\
|Y[\omega^A]|(S^B)\ar[r]_-{|Y[i_!]|(S^B)} &  |Y[\omega^B]|(S^B) } \]
Here the upper vertical maps are assembly maps \eqref{eq:assembly} 
of the $\Gamma$-spaces $|Y[\omega^A]|$ and $|Y[\omega^B]|$, respectively, 
and the lower vertical maps are the effects of these two $\Gamma$-spaces
on the homeomorphism $S^A\sm S^{B\setminus i(A)}\iso S^B$ given by $i$ on $A$.
In the special case where $A=\emptyset$ is empty,
the morphism of $\Gamma$-categories
$Y[i_!]:Y[\omega^A]\to Y[\omega^B]$ is to be interpreted as
the inclusion $Y(1_+)^{\supp=\emptyset}\to Y(1_+)=Y[\omega^B](1_+)$
of the subcategory of objects supported on the empty set.
\end{con}

\begin{rk} Our global K-theory machinery produces restricted global $\Omega$-spectra,
  which in particular means that the value at the empty set does not have any homotopical significance.
  Still, the definition of $Y\td{\mS}(\emptyset)$ can be motivated by the requirement that
  as a {\em symmetric} spectrum,
  the structure maps $Y\td{\mS}(\emptyset)\sm S^B\to Y\td{\mS}(B)$
  must be $\Sigma_B$-equivariant for the trivial action on $Y\td{\mS}(\emptyset)$;
  if we want these structure maps to arise from the assembly map of the $\Gamma$-space $|Y|$,
  then the only general way to arrange the necessary equivariance
  is to let $Y\td{\mS}(\emptyset)$ be a subspace of the realization of the $\M$-fixed
  subcategory of $Y(1_+)$. Our definition just uses the maximal choice, the full $\M$-fixed
  subcategory of $Y(1_+)$. 
\end{rk}

As we already mentioned, 
our construction of a symmetric spectrum from a $\Gamma$-$\M$-category
is a variation of Segal's construction 
of a spectrum from a $\Gamma$-category \cite{segal:cat coho}.
Moreover, Theorem \ref{thm:a-b-maps equivalence} below,
applied to the trivial group,
shows that the underlying non-equivariant stable homotopy type
of our construction agrees with Segal's delooping
of the underlying $\Gamma$-category of a $\Gamma$-$\M$-category.
We emphasize that the $\M$-action on $Y$ enters into the definition
of $Y\td{\mS}$ in a crucial way,
via the $\Sigma_A$-action on the $A$-th level, and via the structure maps. 
The extra flexibility coming from the $\M$-action is
the key to the good equivariance properties.\medskip

We let $Y$ be a $\Gamma$-$\M$-category
and $u:I\to J$ an injection between countably infinite sets.
Then $u_*:Y[I]\to Y[J]$ is a morphism of $\Gamma$-categories,
and it induces a continuous based map
\[ u_*(K)\ : \ |Y[I]|(K)\ \to \ |Y[J]|(K) \]
for every based space~$K$.

\begin{prop}\label{prop:equivariant homotopy}
  Let $Y$ be a $\Gamma$-$\M$-category, $G$ a finite group, $\Uc$ and $\Vc$ universal $G$-sets,
  and $K$ a based $G$-space.
  \begin{enumerate}[\em (i)]
  \item 
    Let $u,v:\Uc\to\Vc$ be two $G$-equivariant injections.
    Then the two $G$-maps
    \[ u_*(K)\ ,\ v_*(K) \ : \ |Y[\Uc]|(K)\ \to \ |Y[\Vc]|(K) \]
    are $G$-equivariantly homotopic,
    with respect to the diagonal $G$-actions on $|Y[\Uc]|(K)$ and $|Y[\Vc]|(K)$. 
  \item Let $u:\Uc\to\Vc$ be a $G$-equivariant injection.
    Then the $G$-map
    \[ u_*(K)\ : \ |Y[\Uc]|(K)\ \to \ |Y[\Vc]|(K) \]
    is a based $G$-homotopy equivalence with respect to the diagonal $G$-actions.    
  \end{enumerate}
\end{prop}
\begin{proof}
  (i) We write $I(\Uc,\Vc)$ for the set of injections from $\Uc$ to $\Vc$,
  and we denote by $E I(\Uc,\Vc)$ the chaotic category with object set $I(\Uc,\Vc)$.
  We let $G$ act on the set $I(\Uc,\Vc)$ by conjugation,
  and on the space $|E I(\Uc,\Vc)|$ by functoriality.
  Taking fixed points commutes with geometric realization and
  with the functor $E$ from sets to categories. So
  \[ |E I(\Uc,\Vc)|^G\ =\ | E  \left( I(\Uc,\Vc)^G \right)  | \ ;  \]
  this space is contractible because the set of $G$-equivariant injections
  from $\Uc$ to $\Vc$ is non-empty.

  The $\M$-action on $Y$ induces a $G$-equivariant action functor
  \[ E I(\Uc,\Vc) \times Y[\Uc]\ \to \ Y[\Vc] \]
  whose restriction to the objects $u$ and $v$ of $E I(\Uc,\Vc)$
  yields the functors $u_*$ and $v_*$, respectively.
  Taking nerves and geometric realization and evaluating at $K$
  provides a continuous  $G$-equivariant action map
  \[ |E I(\Uc,\Vc)| \times |Y[\Uc]|(K)\ \to \ |Y[\Vc]|(K)\ . \]
  The given $G$-equivariant injections $u$ and $v$ are two $G$-fixed points
  in $|E I(\Uc,\Vc)|$.
  Since the $G$-fixed point space is contractible, there is a path
  $\lambda:[0,1]\to|E I(\Uc,\Vc)|^G$ from $u$ to $v$.
  The map
  \[ [0,1]\times |Y[\Uc]|(K)\ \to \ |Y[\Vc]|(K)\ , \quad t\ \longmapsto \ \lambda(t)_*(K) \]
  is the desired $G$-equivariant homotopy from $u_*$ to $v_*$.

  (ii)  Since $\Uc$ and $\Vc$ are universal $G$-sets,
  we can choose a $G$-equivariant injection $v:\Vc\to\Uc$.
  Then 
  \[ v_*(K)\circ u_*(K)\ =\ (v u)_*(K)\ : \ |Y[\Uc]|(K)\ \to \ |Y[\Uc]|(K) \]
  is $G$-homotopic to the identity map, by part (i).
  Similarly, $u_*(K)\circ v_*(K)$ is equivariantly homotopy to the identity.
  So $u_*(K)$ is an equivariant homotopy equivalence.
\end{proof}

We recall that a symmetric spectrum is {\em globally semistable}
if for every finite group $G$, the underlying $G$-symmetric spectrum
is $G$-semistable in the sense of \cite[Definition 3.22]{hausmann:G-symmetric},
i.e., the monoid of equivariant self-injections of a complete $G$-set
acts trivially on the naive $G$-equivariant stable homotopy groups.
One of the main features of global semistability is that for these
symmetric spectra, global equivalences are precisely the morphisms
that induce isomorphisms of equivariant homotopy groups,
see \cite[Proposition 4.13 (vi)]{hausmann:global_finite}.
A globally semistable symmetric spectrum $X$ is {\em globally connective}
if for every finite group $G$, the equivariant homotopy groups $\pi_k^G(X)$
are trivial for negative values of $k$.

\begin{prop}\label{prop:Y<S> globally connective-semistable}
  For every $\Gamma$-$\M$-category $Y$,
  the symmetric spectrum $Y\td{\mS}$ is globally semistable and globally connective.
\end{prop}
\begin{proof}
  To show that the symmetric spectrum $Y\td{\mS}$ is globally semistable
  we have to show that for every finite group $G$,
  every $G$-equivariant injection $u:\Uc_G\to\Uc_G$ and every $k\in\mZ$,
  the induced map $u_*:\pi_k^G(Y\td{\mS})\to \pi_k^G(Y\td{\mS})$
  is the identity. We show this for $k=0$, the general case being similar.
  After unraveling all the definitions, this comes down to the following fact:
  for every finite $G$-set $A$ and every continuous based $G$-map $f:S^A\to |Y[\omega^A]|(S^A)$,
  the two $G$-maps 
  \[ \sigma_{A,A}\circ (S^A\sm f)\ , \ \sigma^{\op}_{A,A}\circ(f\sm S^A)\ : \
    S^A\sm S^A\ \to \ |Y[\omega^{A\amalg A}]|(S^A\sm S^A) \]
  are equivariantly based homotopic.
  The two maps differ by the twist involution of $S^A\sm S^A$ in source and target,
  and by the effect of the involution $t:\omega^{A\amalg A}\to\omega^{A\amalg A}$
  arising from switching the two copies of $A$.
  The effect of $t$ alone is equivariantly homotopic to the identity,
  by Proposition \ref{prop:equivariant homotopy} (i).
  The effect of the twist involution of $S^A\sm S^A$ in source and target
  can be homotoped to the identity by the standard sine-cosine equivariant
  homotopy between the two direct summand embedding $\mR[A]\to\mR[A]\oplus\mR[A]$.

  To show global connectivity we consider a subgroup $H$ of a finite group $G$.
  For a finite $G$-set $A$ we set $d_H=\dim( (\mR[A])^H)$, which is equal to the number of $H$-orbits of $A$.
  The underlying $\Gamma$-$H$-space of $|Y[\omega^A]|$ arises from a $\Gamma$-$H$-simplicial
  set, so it is $H$-cofibrant in the sense of \cite[Definition B.33]{schwede:global},
  by \cite[Example B.34]{schwede:global}.
  So $(|Y[\omega^A]|(S^{\mR^k\oplus \mR[A]}))^H$ is $(k+d_H-1)$-connected by
  \cite[Proposition B.43]{schwede:global}.
  On the other hand, the cellular dimension of $S^A$ at $H$,
  in the sense of \cite[II.2, page 106]{tomDieck:transformation_groups},  is $d_H$.
  So whenever $k$ is positive, the cellular dimension of $S^A$ at $H$ does not exceed the connectivity of
  $(|Y[\omega^A]|(S^{\mR^k\oplus \mR[A]}))^H$. So every based continuous
  $G$-map $S^A\to |Y[\omega^G]|(S^{\mR^k\oplus \mR[A]})$ is equivariantly null-homotopic
  by \cite[II Proposition 2.6]{tomDieck:transformation_groups},
  and the set $[S^A,|Y[\omega^A]|(S^{\mR^k\oplus \mR[A]})]^G$ has only one element.
  Passage to the colimit over finite $G$-invariant subsets of the chosen universal $G$-set
  proves that the homotopy group $\pi_{-k}^G(Y\td{\mS})$ is trivial for all $k\geq 1$.
\end{proof}

The next proposition shows that the passage from $\Gamma$-$\M$-categories
to associated symmetric spectra preserves finite products.
Part (i) of the following proposition ought to be well-known, but I do not know
a reference.

\begin{prop}\label{prop:at spheres preserves products}
  \begin{enumerate}[\em (i)]
  \item
    For all $\Gamma$-spaces $E$ and $F$ and every based space $K$, the map
    \[  (E\times F)(K)\ \to\ E(K)\times F(K) \]
    induced by the projections of $E\times F$ to the two factors is a homeomorphism.
  \item
    For all $\Gamma$-$\M$-categories $X$ and $Y$, the morphism
    \[  (X\times Y)\td{\mS}\ \to\ X\td{\mS}\times Y\td{\mS} \]
    induced by the projections is an isomorphism of symmetric spectra.
  \end{enumerate}
\end{prop}
\begin{proof}
  (i)
  We define a continuous map in the opposite direction.
  To facilitate this, we first rewrite the space $E(K)\times F(K)$.
  We exploit the fact that in the category of compactly generated spaces,
  product with a fixed space is a left adjoint, so it commutes with coends.
  The space $E(K)\times F(K)$ is thus a coend of the functor
  \begin{equation}\label{eq:double coend}
    \Gamma\times \Gamma\times \Gamma^{\op}\times \Gamma^{\op}\ \to \ \bT\ , \quad
    (k_+,l_+,m_+,n_+)\ \longmapsto \ E(k_+)\times F(l_+)\times K^m\times K^n\ .
  \end{equation}
  For given $m,n\geq 0$, we define based maps $i_{m,n}:m_+\to (m+n)_+$ and $j_{m,n}:n_+\to(m+n)_+$
  by
  \[ i_{m,n}(x)\ = \ x \text{\qquad and\qquad} j_{m,n}(x)\ = \ m+x\ . \]
  Then we define continuous maps
  \[ \rho_{m,n}\ : \ E(m_+)\times F(n_+)\times K^m\times K^n\  \to \  (E\times F)(K)\]
  by
  \[ \rho_{m,n}(x,y;\kappa_1,\dots,\kappa_m,\lambda_1,\dots,\lambda_n) \ = \
    [ E(i_{m,n})(x), F(j_{m,n})(y);\kappa_1,\dots,\kappa_m,\lambda_1,\dots,\lambda_n ]\ .\]
  We omit the verification that these maps are compatible when $(m_+,n_+)$ varies over
  the category $\Gamma\times\Gamma$, so that they descend to a continuous map
  \[ \rho \ : \ E(K)\times F(K) \ = \ \int^{(m_+,n_+)\in\Gamma\times\Gamma}
    E(m_+)\times F(n_+)\times K^m\times K^n\ \to \ (E\times F)(K)  \]
  defined on the coend of the functor \eqref{eq:double coend}.
  Now we argue that $\rho$ is inverse to the map induced by the projections.
  We write $p_1:E\times F\to E$ for the projection to the first factor. Then
  for all $(x,y;\kappa_1,\dots,\kappa_m,\lambda_1,\dots,\lambda_n)\in 
  E(m_+)\times F(n_+)\times K^m\times K^n$ we have
  \begin{align*}
    (p_1(K)\circ \rho)[x,y;\kappa_1,\dots,\kappa_m,\lambda_1,&\dots,\lambda_n] \
    = \  [ E(i_{m,n})(x); \kappa_1,\dots,\kappa_m,\lambda_1,\dots,\lambda_n ]\\
    &= \  [ x; i_{m,n}^*(\kappa_1,\dots,\kappa_m,\lambda_1,\dots,\lambda_n) ]\
      = \  [ x; \kappa_1,\dots,\kappa_m]\ . 
  \end{align*}
  Similarly,
  $(p_2(K)\circ \rho)[x,y;\kappa_1,\dots,\kappa_m,\lambda_1,\dots,\lambda_n]=
  [ y; \lambda_1,\dots,\lambda_n]$.
  This proves that the composite of $\rho$ and the canonical map is the
  identity of $E(K)\times F(K)$.
  For the other composite we consider a tuple
  $(x,y;\kappa_1,\dots,\kappa_m)\in (E\times F)(m_+)\times K^m$,
  and we write $\nabla:(m+m)_+\to m_+$ for the based map defined by
  \[ \nabla(a)\ = \
    \begin{cases}
      \  a & \text{ for $0\leq a\leq m$, and}\\
      a-m & \text{ for $m+1\leq a\leq m+m$.}      \end{cases}
  \]
  Then $\nabla\circ i_{m,m}=\nabla\circ j_{m,m}=\Id_{m_+}$, and hence
  \begin{align*}
    (\rho\circ (p_1(K),p_2(K)))[x,y;\kappa_1,\dots,\kappa_m]
    &= \  \rho[x,y;\kappa_1,\dots,\kappa_m,\kappa_1,\dots,\kappa_m]\\
    &= \  [ E(i_{m,m})(x), F(j_{m,m})(y);\kappa_1,\dots,\kappa_m,\kappa_1,\dots,\kappa_m]\\
    &= \  [ E(i_{m,m})(x), F(j_{m,m})(y);\nabla^*(\kappa_1,\dots,\kappa_m) ]\\
    &= \  [ E(\nabla\circ i_{m,m})(x), F(\nabla \circ j_{m,m})(y);\kappa_1,\dots,\kappa_m]\\
    &= \  [x,y;\kappa_1,\dots,\kappa_m] \ .
  \end{align*}
  This proves that the other composite is the identity of $(E\times F)(K)$.
  
  (ii) We let $A$ be a finite set. We specialize part (i)
  to the $\Gamma$-spaces $E=|X[\omega^A]|$ and $F=|Y[\omega^A]|$
  and the based space $K=S^A$. We conclude that the map 
  \begin{align*}
    ( p_1\td{\mS}(A),p_2\td{\mS}(A))\ : \
    (X\times Y)\td{\mS}(A)\
    &= \ |X[\omega^A]\times Y[\omega^A]|(S^A)\   \to \\
    & (|X[\omega^A]|\times |Y[\omega^A]|)(S^A)
      \ = \    ( X\td{\mS}\times Y\td{\mS} )(A)
  \end{align*}
  is a homeomorphism.
\end{proof}

Now we introduce the notion of `global specialness' for $\Gamma$-$\M$-categories,
a global equivariant refinement of Segal's condition \cite[Definition 2.1]{segal:cat coho}
that is nowadays referred to as `specialness'.
As we shall prove in Theorem \ref{thm:Y(S) is global Omega} below,
the symmetric spectrum $Y\td{\mS}$ associated to a globally special $\Gamma$-$\M$-category $Y$
is a restricted global $\Omega$-spectrum.
Our main class of examples arises from parsummable categories:
Theorem~\ref{thm:special G Gamma} below shows
that the $\Gamma$-$\M$-category $\gamma(\Cc)$ associated to a parsummable category $\Cc$
is globally special.

\begin{con}
We let $Y$ be a $\Gamma$-category and $S$ a finite set.
Given $s\in S$ we denote by $p_s:S_+\to 1_+=\{0,1\}$ 
the based map with $p_s^{-1}(1)=\{s\}$.  We denote by
\[ P_S \ : \ Y(S_+) \ \to \ \map(S,Y(1_+))\]
the functor whose $s$-component is $Y(p_s):Y(S_+)\to Y(1_+)$.
Segal's condition \cite[Definition 2.1]{segal:cat coho}
is the requirement that the functor $P_S$
is an equivalence of categories for all finite sets $S$. We need a global
equivariant version of this condition.

If $Y$ is a $\Gamma$-$G$-category, for a finite group $G$,
and if the group $G$ also acts on the finite set $S$, 
then the categories $Y(S_+)$ and $\map(S,Y(1_+))$ have two commuting $G$-actions:
the `external' action is the value at $S_+$ respectively $1_+$ 
of the $G$-action on $Y$; the `internal' action is induced 
by the $G$-action on~$S$.
In this situation the functor $P_S:Y(S_+)\to\map(S,Y(1_+))$
is $(G\times G)$-equivariant.
Below we will consider $Y(S_+)$ and $\map(S,Y(1_+))$ 
endowed with the diagonal $G$-action respectively the conjugation action;
then the functor $P_S:Y(S_+)\to\map(S,Y(1_+))$ is $G$-equivariant.
In particular, the functor restricts to a functor 
\[ (P_S)^G \ : \ (Y(S_+))^G \ \to \ \map^G(S,Y(1_+))\]
on the $G$-fixed subcategories.
\end{con}

If $Y$ is a $\Gamma$-$\M$-category, $G$ a finite group, and $\Uc$ a $G$-set,
then $Y[\Uc]$ becomes a $\Gamma$-$G$-category through the action of $G$ on $\Uc$.

\begin{defn}\label{def:globally special Gamma-cat}
  Let $G$ be a finite group. 
  A $\Gamma$-$G$-category $Y$ is {\em special}
  if for every subgroup $H$ of $G$ and every finite $H$-set $S$ the functor
  \[ (P_S)^H \ : \ Y(S_+)^H\ \to \ \map^H(S,Y(1_+)) \]
  is a weak equivalence of categories.
  A $\Gamma$-$\M$-category $Y$ is {\em globally special} if for every finite group $G$ and
  some (hence any) universal $G$-set $\Uc$, the $\Gamma$-$G$-category $Y[\Uc]$ is special.
\end{defn}

\begin{theorem}\label{thm:Y(S) is global Omega}
  Let $Y$ be a globally special $\Gamma$-$\M$-category.
  \begin{enumerate}[\em (i)]
  \item Let $G$ be a finite group, $\Uc$ a universal $G$-set and
    $A$ a non-empty finite $G$-set. Then for every $G$-representation $V$,
    the adjoint assembly map
    \[ \tilde\alpha\ : \  |Y[\Uc]|(S^A) \ \to  \     \map_*(S^V,\, |Y[\Uc]|(S^A\sm S^V)) \]
    is a $G$-weak equivalence.    
  \item The symmetric spectrum $Y\td{\mS}$ is a restricted global $\Omega$-spectrum.
\end{enumerate}
\end{theorem}
\begin{proof}
  (i)
  Since $\Uc$ is a universal $G$-set, the $\Gamma$-$G$-space $|Y[\Uc]|$
  is special, by hypothesis.
  Since $A$ is non-empty, the permutation representation of $A$ 
  has non-trivial $G$-fixed points;
  so we can apply Shimakawa's theorem~\cite[Theorem B]{shimakawa}
  to $W=\mR[A]$ and the given representation $V$, 
  and we conclude that the adjoint assembly map 
  $\tilde\alpha$ is a $G$-weak equivalence.
  To be completely honest, we cannot literally use Shimakawa's theorem, 
  because he uses a homotopy coend (bar construction) to evaluate a $\Gamma$-$G$-space
  on spheres,  as opposed to the strict coend that we employ.
  So we instead quote \cite[Theorem B.65]{schwede:global},
  using that the $\Gamma$-$G$-space $|Y[\Uc]|$ is $G$-cofibrant
  by \cite[Example B.34]{schwede:global}.

  (ii)
  We let $G$ be a finite group and~$A$ a finite $G$-set with a free orbit.
  Then we let~$B$ be another finite $G$-set.
  The adjoint structure map $\tilde\sigma_{A,B}$ of the symmetric spectrum $Y\td{\mS}$
  is the composite
  \[ |Y[\omega^A]|(S^A)\ \xra{|Y[i_!]|(S^A)} \
    |Y[\omega^{A+B}]|(S^A) \ \xra{\ \tilde\alpha\ }\ 
    \map_*(S^B,\, |Y[\omega^{A+B}]|(S^{A+B}))\ , \]
  where $i_!:\omega^A\to\omega^{A+B}$ is `extension by 0'
  as defined in \eqref{eq:extension_by_zero},
  and $\tilde\alpha$ is the adjoint of the assembly map \eqref{eq:assembly}
  of the $\Gamma$-$G$-space $|Y[\omega^{A+B}]|$.
  The injection $i_!$ is $G$-equivariant, and source and target
  are universal $G$-sets by Proposition \ref{prop:universal G-sets};
  the map $|Y[i_!]|(S^A)$ is thus a $G$-equivariant homotopy equivalence
  by Proposition \ref{prop:equivariant homotopy} (ii).
  The adjoint assembly map is a $G$-weak equivalence by part (i),
  because $\omega^{A+B}$ is a universal $G$-set.
\end{proof}

The symmetric spectrum $Y\td{\mS}$ associated to a $\Gamma$-$\M$-category is
a symmetric spectrum without any additional group actions, and as such it represents a global
stable homotopy type. We will now identify the underlying $G$-symmetric spectrum $Y\td{\mS}_G$
(i.e., $Y\td{\mS}$ endowed with the trivial $G$-action) with the $G$-symmetric spectrum
associated to a $\Gamma$-$G$-category by evaluation on spheres.

\begin{con}\label{con:underlying G comparison}
  We let $Y$ be a $\Gamma$-$\M$-category, $G$ a finite group, and $\Uc$ a universal $G$-set.
  Then $|Y[\Uc]|$ is a $\Gamma$-$G$-space through the action of $G$ on $\Uc$.
  We can thus evaluate $|Y[\Uc]|$ on spheres and obtain
  an orthogonal $G$-spectrum $|Y[\Uc]|(\mS)$.
  We use the same notation $|Y[\Uc]|(\mS)$ for the underlying $G$-symmetric spectrum.

  We relate the $G$-symmetric spectrum $|Y[\Uc]|(\mS)$ to $Y\td{\mS}_G$
  via an intermediate object that is designed to receive morphisms from both.
  We define a $G$-symmetric spectrum $Y\td{\Uc,\mS}$ at a finite set $A$ by
  \[ Y\td{\Uc,\mS}(A) \ = \ |Y[\Uc\amalg\omega^A]|(S^A) \ ,\]
  the value of the $\Gamma$-$G$-space $|Y[\Uc\amalg\omega^A]|$ on the $A$-sphere.
  The $G$-action on $Y\td{\Uc,\mS}(A)$ arises from the $G$-action on $\Uc$,
  just as for $|Y[\Uc]|(\mS)$.
  The $\Sigma_A$-action on $Y\td{\Uc,\mS}(A)$ arises from the action on $\omega^A$ and on $S^A$,
  just as for $Y\td{\mS}$.
  Since the actions of $G$ and $\Sigma_A$ are through disjoint summands,
  they commute with each other.

  The structure map
  \[ i_* \ : \  Y\td{\Uc,\mS}(A)\sm S^{B\setminus i(A)}\ \to\ Y\td{\Uc,\mS}(B)\]
  associated with an injection $i:A\to B$
  is defined as for $Y\td{\mS}$, namely as
  the diagonal composite in the commutative diagram:
  \[ \xymatrix@C=35mm{
      |Y[\Uc\amalg\omega^A]|(S^A)\sm S^{B\setminus i(A)}\ar[r]^-{|Y[\Uc\amalg i_!]|(S^A)\sm S^{B\setminus i(A)}} \ar[d]_{\eqref{eq:assembly}} &
      |Y[\Uc\amalg\omega^B]|(S^A)\sm S^{B\setminus i(A)} \ar[d]^{\eqref{eq:assembly}} \\
      |Y[\Uc\amalg\omega^A]|(S^A\sm S^{B\setminus i(A)})\ar[r]^-{|Y[\Uc\amalg i_!]|(S^A\sm S^{B\setminus i(A)})} \ar[d]_\iso &
      |Y[\Uc\amalg\omega^B]|(S^A\sm S^{B\setminus i(A)}) \ar[d]^\iso \\
      |Y[\Uc\amalg\omega^A]|(S^B)\ar[r]_-{|Y[\Uc\amalg i_!]|(S^B)} &  |Y[\Uc\amalg\omega^B]|(S^B) } \]
  Here $i_!:\omega^A\to\omega^B$ is the `extension by 0'
  injection defined in \eqref{eq:extension_by_zero}.
  The upper vertical maps are assembly maps
  of the $\Gamma$-spaces $|Y[\Uc\amalg\omega^A]|$ and $|Y[\Uc\amalg\omega^B]|$, respectively, 
  and the lower vertical maps are the effects of these two $\Gamma$-spaces
  on the homeomorphism $S^A\sm S^{B\setminus i(A)}\iso S^B$ given by $i$ on $A$.

  The $G$-symmetric spectrum $Y\td{\Uc,\mS}$ is designed so that the $(G\times\Sigma_A)$-maps
  \[ |Y[\Uc]|(S^A) \ \xra{\quad}\ |Y[\Uc\amalg\omega^A]|(S^A) \ \xla{\quad} \ |Y[\omega^A]|(S^A) \]
  induced by the inclusions of $\Uc$ and $\omega^A$ into $\Uc\amalg\omega^A$
  form morphisms of $G$-symmetric spectra
  \[ |Y[\Uc]|(\mS)\ \xra{\ b_G^Y\ }\ Y\td{\Uc,\mS}\ \xla{\ a_G^Y\ }\ Y\td{\mS}_G\ . \]
\end{con}

\begin{theorem}\label{thm:a-b-maps equivalence}
  Let $Y$ be a $\Gamma$-$\M$-category.
  For every finite group $G$ and every universal $G$-set $\Uc$,
  the two morphisms of $G$-symmetric spectra
  \[   |Y[\Uc]|(\mS)\ \xra{\ b_G^Y\ } \ Y\td{\Uc,\mS} \ \xla{\ a^Y_G\ }\ Y\td{\mS}_G\]
  are $G$-stable equivalences.
\end{theorem}
\begin{proof}
  If $A$ is any $G$-set, then the inclusion $\Uc\to\Uc\amalg \omega^A$
  is an equivariant injection between universal $G$-sets.
  The induced $G$-map $|Y[\Uc]|(S^A)\to |Y[\Uc\amalg\omega^A]|(S^A)$
  is thus a based $G$-homotopy equivalence by Proposition \ref{prop:equivariant homotopy} (ii).
  Since the morphism $b_G^Y$ is levelwise a $G$-weak equivalence,
  it is in particular a $\upi_*$-isomorphism,
  and hence a $G$-stable equivalence by \cite[Theorem 3.36]{hausmann:G-symmetric}.

  Now we let $A$ be a $G$-set with a free orbit.
  Then the inclusion $\omega^A\to\Uc\amalg \omega^A$
  is an equivariant injection between universal $G$-sets,
  by Proposition \ref{prop:universal G-sets}.
  The induced $G$-map $|Y[\omega^A]|(S^A)\to |Y[\Uc\amalg\omega^A]|(S^A)$
  is thus a based $G$-homotopy equivalence, again by Proposition \ref{prop:equivariant homotopy} (ii).
  Since the $G$-sets with a free orbit are cofinal in all finite $G$-subsets
  of a given universal $G$-set, this shows that the morphism $a_G^Y$
  is a $\upi_*$-isomorphism,
  and hence a $G$-stable equivalence by \cite[Theorem 3.36]{hausmann:G-symmetric}.
\end{proof}

Global equivalences of $\M$-categories
were introduced in Definition \ref{def:global equiv Gamma-M}.

\begin{prop}\label{prop:globally special yields strict cat}
  Let $\Phi:X\to Y$ be a morphism of globally special $\Gamma$-$\M$-categories
  such that $\Phi(1_+):X(1_+)\to Y(1_+)$ is a global equivalence of $\M$-categories.
  Then the morphism
  \[ \Phi\td{\mS}\ : \ X\td{\mS}\ \to \ Y\td{\mS} \]
  is a global equivalence of symmetric spectra.
\end{prop}
\begin{proof}
  We let $G$ be a finite group, and $A$ a finite $G$-set with a free orbit.
  Then $\omega^A$ is a universal $G$-set
  by Proposition \ref{prop:universal G-sets}.
  So the $\Gamma$-$G$-spaces $|X[\omega^A]|$ and $|Y[\omega^A]|$
  are special by the hypotheses on $X$ and $Y$.
  Moreover, the map
  $|\Phi[\omega^A]|(1_+): |X[\omega^A]|(1_+)\to |Y[\omega^A]|(1_+)$
  is a $G$-weak equivalence because $\Phi(1_+)$ is a global equivalence.
  We claim that the morphism of $\Gamma$-$G$-spaces
  \[   |\Phi[\omega^A]|\ : \ |X[\omega^A]|\ \to \ |Y[\omega^A]| \]
  is a {\em strict equivalence} in the following sense:
  for every subgroup $H$ of $G$ and every finite $H$-set $S$, the map
  \[ ( |\Phi[\omega^A]|(S_+))^H\ :\ ( |X[\omega^A]|(S_+))^H \ \to
    \ ( |Y[\omega^A]|(S_+))^H\]
  is a weak equivalence.
  To see this we choose representatives $s_1,\dots,s_m$ of the $H$-orbits of $S$,
  and we let $K_i$ be the stabilizer group of $s_i$.
  Then we consider the commutative diagram:
  \[ \xymatrix@C=35mm{
     (|X[\omega^A]|(S_+))^H \ar[d]_{(P_S)^H}^\simeq\ar[r]^-{(|\Phi[\omega^A]|(S_+))^H} &
     (|Y[\omega^A]|(S_+))^H\ar[d]^{(P_S)^H}_\simeq  \\
      \map^H(S,|X[\omega^A]|(1_+))\ar[r]^-{\map^H(S,|\Phi[\omega^A]|(1_+))} \ar[d]_\iso &
      \map^H(S,|Y[\omega^A]|(1_+))\ar[d]^\iso\\
      \prod_{i=1}^m \, (|X[\omega^A]|(1_+))^{K_i}\ar[r]_-{\prod (|\Phi[\omega^A]|(1_+))^{K_i}} &
      \prod_{i=1}^m\,  |Y[\omega^A]|(1_+))^{K_i}
    } \]
  The two upper vertical maps are weak equivalences by specialness.
  The two lower vertical maps are evaluation at the representatives $s_1,\dots,s_m$,
  and they are homeomorphisms.
  The lower horizontal map $|\Phi[\omega^A]|(1_+)^{K_i}$ is a weak equivalence
  for each $1\leq i\leq m$, because $|\Phi[\omega^A]|(1_+): |X[\omega^A]|(1_+)\to |Y[\omega^A]|(1_+)$
  is a $G$-weak equivalence.
  So the upper horizontal map is a weak equivalence.

  Now we can finish the proof.  
  Since the $\Gamma$-$G$-spaces $|X[\omega^A]|$ and $|Y[\omega^A]|$
  are realizations of $\Gamma$-$G$-simplicial sets, they are $G$-cofibrant by
  \cite[Example B.34]{schwede:global}. Since the morphism
  $|\Phi[\omega^A]|:|X[\omega^A]|\to |Y[\omega^A]|$ is a strict equivalence of $\Gamma$-$G$-spaces,
  evaluation at the based $G$-CW-complex $S^A$
  yields a $G$-weak equivalence
  \[ \Phi\td{\mS}(A)\ : \  X\td{\mS}(A)\ = \ |X[\omega^A]|(S^A)\ \to \
    |Y[\omega^A]|(S^A)\ =\   Y\td{\mS}(A) \]
  by \cite[Proposition B.48]{schwede:global}.
  Since all $\Gamma$-$G$-spaces in sight arise from $\Gamma$-$G$-simplicial sets,
  and $S^A$ is also the realization of a $G$-simplicial set, we can alternatively
  quote \cite[Lemma 4.8]{ostermayr}.
  
  The $G$-sets with a free orbit are cofinal in the poset of finite $G$-subsets
  of any given universal $G$-set. So the morphism $\Phi\td{\mS}$ induces
  isomorphisms of $G$-equivariant stable homotopy groups, for every finite group $G$.
  In other words, $\Phi\td{\mS}$ is a global $\upi_*$-isomorphism,
  and hence a global equivalence by \cite[Proposition 4.5]{hausmann:global_finite}.
\end{proof}

The $G$-fixed $\M$-category $F^G \Cc$ associated to
a finite group $G$ and an $\M$-category $\Cc$
was introduced in Construction \ref{con:F^G for M-cat}.
We can apply the functor $F^G:\M\cat\to\M\cat$ objectwise (in the $\Gamma$-direction)
to a $\Gamma$-$\M$-category $Y$, and thus obtain another $\Gamma$-$\M$-category $F^G Y$.

\begin{prop}\label{prop:global special 2 fixed}
  Let $Y$ be a globally special $\Gamma$-$\M$-category.
  Then for every finite group $G$, the $\Gamma$-$\M$-category $F^G Y$
  is globally special.
\end{prop}
\begin{proof}
  We must show that for every finite group $K$, every universal $K$-set $\Uc$,
  and every finite $K$-set $S$, the functor
  \[ (P_S)^K \ : \ ((F^G Y)[\Uc](S_+))^K\ \to \ \map^K(S,(F^G Y)[\Uc](1_+)) \]
  is a weak equivalence of categories.
  The $K$-set $\omega^K$ is universal by Proposition \ref{prop:universal G-sets},
  so we can assume that $\Uc=\omega^K$.
  The $(K\times G)$-equivariant isomorphism \eqref{eq:reparametrize}
  restricts to a $K$-equivariant isomorphism on $G$-fixed subcategories, identifying
  \[ (F^G Y)[\omega^K]\ =\ (Y[\omega^G][\omega^K])^G  \]
  with the $K$-category $Y[(\omega^K)^G]^G$.
  We make the $K$-set $S$ into a $(K\times G)$-set by giving it the trivial $G$-action.
  Then the functor $(P_S)^K$ becomes the functor
  \[ (P_S)^{K\times G} \ : \ ( Y[(\omega^K)^G](S_+))^{K\times G}\ \to \
    \map^{K\times G}(S,(Y[(\omega^K)^G](1_+))) \ .\]
  This functor is a weak equivalence of categories because $Y$ is globally special and
  $(\omega^K)^G$ is a universal $(K\times G)$-set.
\end{proof}

The previous Theorem \ref{thm:a-b-maps equivalence}
provides an interpretation of the underlying $G$-symmetric spectrum of $Y\td{\mS}$.
Theorem \ref{thm:G-fixed Gamma-cat} below is of a similar spirit:
we describe the $G$-fixed point spectrum
of $Y\td{\mS}$ in terms of the $G$-fixed $\Gamma$-$\M$-category $F^G Y$,
as long as $Y$ is globally special.

\begin{con}
  Let $G$ be a finite group and $Y$ a $\Gamma$-$\M$-category.
  We now introduce a morphism of symmetric spectra
  \[ \psi_Y^G\ : \ (F^G Y)\td{\mS}  \ \to \ F^G(Y\td{\mS})\ ,\]
  where the target is the $G$-fixed point symmetric spectrum
  as defined in Construction \ref{con:G fixed point}.
  We show in Theorem \ref{thm:G-fixed Gamma-cat} below that this
  morphism is a global equivalence whenever $Y$ is globally special.

  The morphism $\psi^G_Y$ arises from the
  $\Gamma$-$(G\times \Sigma_A)$-space $|Y[\omega^{G\times A}]|$,
  where $A$ is a finite set.
  In much the same way as \eqref{eq:reparametrize}, the intertwiner
  \[ \Im \ : \   \omega \ \xra{\kappa_{\omega^G}} \ \omega^G \ \xra{(\kappa_{\omega^A})^G}\ (\omega^A)^G \
    \xra{\kappa_{(\omega^A)^G}^{-1}} \ \omega    \]
  induces a $(G\times \Sigma_A)$-equivariant isomorphism of $\Gamma$-$(G\times\Sigma_A)$-categories
  \[  \Im_* \ : \ Y[\omega^G][\omega^A] \  \iso \ Y[ (\omega^A)^G]  \ . \]
  In combination with the isomorphism
  induced by the $(G\times\Sigma_A)$-equivariant bijection 
  \[ (\omega^A)^G\ \iso \ \omega^{G\times A}\ , \quad f \ \longmapsto \{ (g,a)\mapsto f(g)(a )\}\ , \]
  we arrive at a $(G\times\Sigma_A)$-equivariant isomorphism
  \[   Y[\omega^G][\omega^A] \  \iso \ Y[ \omega^{G\times A}]  \ . \]
  The $G$-fixed $\Gamma$-$\Sigma_A$-space of $|Y[\omega^{G\times A}]|$ can then be rewritten as
  \begin{align*}
    |Y[\omega^{G\times A}]|^G\ \iso \  |Y[\omega^G][\omega^A]|^G\
    & \iso\       | \left(Y[\omega^G][\omega^A]\right)^G|\\
    &= \  | (Y[\omega^G])^G[\omega^A]|\ = \ | F^G(Y)[\omega^A]|\ .
  \end{align*}
  The second isomorphism is the fact that $G$-fixed points commute with nerve and geometric realization.
  The third equality uses that $G$ acts trivially on $A$.
  Evaluation on $S^A$ thus yields an isomorphism of based $\Sigma_A$-spaces
  \begin{equation}\label{eq:first_isomorphism}
    F^G(Y)\td{\mS}(A) \ = \ 
    |F^G(Y)[\omega^A]|(S^A) \ 
    \iso \   |Y[\omega^{G\times A}]|^G(S^A)  \ 
    \iso \   \left(|Y[\omega^{G\times A}]|(S^A)\right)^G  \ .
  \end{equation}
  As a second input we consider the $(G\times\Sigma_A)$-equivariant 
  assembly map~\eqref{eq:assembly} 
  \begin{align*}
    |Y[\omega^{G\times A}]|(S^A)\sm S^{\bar\rho_G\tensor\mR[A]}\
    \to \  &|Y[\omega^{G\times A}]|(S^A\sm S^{\bar\rho_G\tensor\mR[A]})\\
    \iso \quad  &|Y[\omega^{G\times A}]|(S^{G\times A})\ = \ Y\td{\mS}(G\times A)
  \end{align*}
  of the $\Gamma$-$(G\times \Sigma_A)$-space $|Y[\omega^{G\times A}]|$.
  The homeomorphism above is induced by the $(G\times\Sigma_A)$-linear isometry 
  \[  \mR[A]\oplus (\bar\rho_G\tensor\mR[A])\ \iso \
    (\mR\oplus \bar\rho_G)\tensor\mR[A]\ \iso_\eqref{eq:fix_of_rho} \ 
    \rho_G\tensor \mR[A]\ \iso \ \mR[ G\times A]\ . \]
  We form the adjoint and restrict to $G$-fixed points to obtain
  a $\Sigma_A$-equivariant map
  \begin{equation}\label{eq:adjoint assembly fixed}
    (|Y[\omega^{G\times A}]|(S^A))^G\ \to \ 
    \map_*^G(S^{\bar\rho_G\tensor\mR[A]}, Y\td{\mS}(G\times A)) \ = \ F^G(Y\td{\mS})(A) \ .
  \end{equation}
  So the composition of the homeomorphism~\eqref{eq:first_isomorphism}
  and the map \eqref{eq:adjoint assembly fixed}
  is a continuous map of based $\Sigma_A$-spaces
  \[ \psi_Y^G(A) \ : \   (F^G Y)\td{\mS}(A) \ \to \  
    \map_*^G\left(S^{\bar\rho_G\tensor\mR[A]},\ Y\td{\mS}(G\times A)\right)\ = \ F^G(Y\td{\mS})(A) \ . \]
  As $A$ varies, the maps $\psi^G_Y(A)$ form a morphism of symmetric spectra
  $\psi^G_Y:(F^G Y)\td{\mS}\to F^G(Y\td{\mS})$.
\end{con}

\begin{theorem}\label{thm:G-fixed Gamma-cat} 
  For every globally special $\Gamma$-$\M$-category $Y$ and every finite group $G$,
  the morphism of symmetric spectra
  \[ \psi_Y^G\ : \ (F^G Y)\td{\mS}  \ \to \ F^G(Y\td{\mS})\]
  is a global equivalence.
\end{theorem}
\begin{proof}
  We let $K$ be a finite group and $A$ a finite $K$-set with a free orbit.
  The map $\psi^G_Y(A)$ was defined as the composite of two homeomorphisms and
  the effect on $G$-fixed points  \eqref{eq:adjoint assembly fixed} of the adjoint assembly map
  of the $\Gamma$-$(G\times K)$-space $|Y[\omega^{G\times A}]|$.
  Since $A$ has a free $K$-orbit, the $(G\times K)$-set $G\times A$ has a free $(G\times K)$-orbit.
  So $\omega^{G\times A}$ is a universal $(G\times K)$-set by
  Proposition \ref{prop:universal G-sets}.
  Since $Y$ is globally special, the adjoint assembly map
  \[ \tilde\alpha\ : \ |Y[\omega^{G\times A}]|(S^A)\ \to \ 
    \map_*(S^{\bar\rho_G\tensor\mR[A]}, |Y[\omega^{G\times A}]|(S^A\sm S^{\bar\rho_G\tensor\mR[A]})) \]
  is a $(G\times K)$-weak equivalence by Theorem \ref{thm:Y(S) is global Omega} (i).
  So the restriction to $G$-fixed points  \eqref{eq:adjoint assembly fixed}
  of $\tilde\alpha$ is a $K$-weak equivalence.
  This concludes the proof that the map $\psi^G_Y(A)$ is a $K$-weak equivalence whenever
  $A$ has a free $K$-orbit.
   
  The $K$-sets with a free orbit are cofinal in the poset of finite $K$-subsets
  of any given universal $K$-set. So the morphism $\psi^G_Y$ induces
  isomorphisms of $K$-equivariant stable homotopy groups, for every finite group $K$.
  In other words, $\psi^G_Y$ is a global $\upi_*$-isomorphism,
  and hence a global equivalence by \cite[Proposition 4.5]{hausmann:global_finite}.
\end{proof}

\section{Global K-theory of parsummable categories}\label{sec:K of I}

In this section we introduce the global K-theory construction
and establish some of its formal properties.
We start by discussing the categorical data that our global 
K-theory construction accepts as input. The data is a
{\em parsummable category}, i.e., an $\M$-category equipped with
a unital, associative and commutative multiplication with respect to
the box product of $\M$-categories.
The prototypical example is the parsummable category of finite sets,
compare Examples \ref{eg:Fc as M-category} and \ref{eg:Fc}.
In the later sections we provide further examples of this categorical input data,
including free parsummable categories (Section \ref{sec:K of free I}),
the parsummable category of finite $G$-sets (Section~\ref{sec:G-sets}),
the parsummable category of finitely generated projective modules over a ring (Section~\ref{sec:K(R)}), 
and parsummable categories associated with permutative categories (Section~\ref{sec:permutative category}).

The global K-theory spectrum of a parsummable category $\Cc$
is then obtained in two steps:
the $\boxtimes$-powers of $\Cc$ assemble into a $\Gamma$-$\M$-category
$\gamma(\Cc)$ as explained in Construction \ref{con:Gamma from I};
the previous Construction \ref{con:spectrum from Gamma-M}
turns this $\Gamma$-$\M$-category into a symmetric spectrum $\bK_{\gl}\Cc=\gamma(\Cc)\td{\mS}$.
The $\Gamma$-$\M$-category $\gamma(\Cc)$ is globally special
by Theorem \ref{thm:special G Gamma}, so the global K-theory spectrum is
a restricted global $\Omega$-spectrum, see Theorem \ref{thm:K_gl C is global Omega}.
Global K-theory has two fundamental invariance properties:
by Theorem \ref{thm:equivalence2equivalence}, 
global equivalences of parsummable categories are taken to global equivalences
of K-theory spectra;
by Theorem \ref{thm:K_gl additive}, global K-theory takes box products
and products of parsummable categories to products of K-theory spectra,
up to global equivalence.

\begin{defn}\label{def:parsummable_category}
A {\em parsummable category} is a tame $\M$-category
equipped with the structure of a commutative monoid under the box product.
\end{defn}

In the following, we will write
\[ +\ :\ \Cc\boxtimes\Cc\ \to\ \Cc \]
for the structure functor of a parsummable category $\Cc$.
The adjective `parsummable' stands for `partially summable',
and it reflects the fact that the sum functor is not defined
on all of $\Cc\times\Cc$, but only on its globally equivalent
full $\M$-subcategory $\Cc\boxtimes\Cc$.
Every parsummable category has a distinguished object:
the unit morphism is an $\M$-equivariant functor $\ast\to\Cc$
from the terminal $\M$-category with only one object and its identity morphism.
We write 0 for the image in $\Cc$ of the unique object of $\ast$,
because it behaves a lot like the neutral element in an abelian monoid.
Since the unit functor is $\M$-equivariant, the object 0 has empty support,
and it is $\M$-fixed.

The support of objects in a parsummable category $\Cc$ is sub-additive in the following sense.
If $a$ and $b$ are two disjointly supported objects,
then $(a,b)$ is an object of $\Cc\boxtimes\Cc$ with support
\[ \supp(a,b) \ = \ \supp(a)\cup \supp(b)\ .\]
Because the structure functor $+:\Cc\boxtimes\Cc\to\Cc$ is $\M$-equivariant,
Proposition \ref{prop:finite support} (v) lets us conclude that
\begin{equation}\label{eq:supp_of_sum}
  \supp(a+b)\  \subseteq \ \supp(a,b)\ = \ \supp(a)\cup \supp(b)\ .
\end{equation}

A parsummable category may be thought of as a `partial strict symmetric monoidal category'.
The adjective `partial' refers to the feature that the product $+$
is not defined for all pairs of objects, but only for those with disjoint support.
The adjective `strict' refers to the feature that
whenever the monoidal product is defined, it is strictly unital, associative
and commutative (and not just up to specified coherence data).
We give several examples of parsummable categories after the following construction;
more examples will follow in Sections~\ref{sec:K of free I}--\ref{sec:permutative category}.

\begin{con}[(From parsummable categories to $\Gamma$-$\M$-categories)]\label{con:Gamma from I}
  We recall that $\Gamma$ denotes the category with objects the finite based sets 
  $n_+=\{0,1,\dots,n\}$, with basepoint 0; morphisms in $\Gamma$ are all based maps.
  We associate to a parsummable category $\Cc$ a $\Gamma$-parsummable category $\gamma(\Cc)$,
  i.e., a functor from $\Gamma$ to the category of parsummable categories
  and strict morphisms.
  For $n\geq 0$ we define 
  \[ \gamma(\Cc)(n_+)\ = \ \Cc^{\boxtimes n}\ \subseteq \ \Cc^n \]
  as the full parsummable subcategory 
  of the product category with objects those $n$-tuples  $(a_1,\dots,a_n)$
  that are disjointly supported, i.e.,
  the supports $\supp(a_i)$ are pairwise disjoint.
  In particular, $\gamma(\Cc)(1_+)$ is the category $\Cc$ itself.
  By convention, $\gamma(\Cc)(0_+)$ is the terminal parsummable category
  with only one object~$0$ and only the identity morphism.
  We use the notation $\Cc^{\boxtimes n}$ because this parsummable category
  is canonically isomorphic to the $n$-fold iterated box product of copies of $\Cc$,
  with any bracketing.
  
  To define the effect of a based map $\lambda:m_+\to n_+$
  we use the universal property of $\Cc^{\boxtimes m}$ as an
  $m$-fold coproduct in the category of parsummable categories,
  see Example \ref{eg:coproduct_parsumcat} below.
  For $1\leq k\leq m$ we let $i_k:\Cc\to \Cc^{\boxtimes m}$
  denote the morphism given by
  \[ i_k(f)\ = \ (0,\dots,0,f,0,\dots,0)\ , \]
  where $f$ sits in the $k$-th slot.
  We define
  \[ \gamma(\Cc)(\lambda)\ :\ \gamma(\Cc)(m_+)= \Cc^{\boxtimes m}\ \to\ \Cc^{\boxtimes n}= \gamma(\Cc)(n_+) \]
  as the unique morphism of parsummable categories such that
  \[ \gamma(\Cc)(\lambda)\circ i_k \ = \ i_{\lambda(k)}\ : \
    \Cc\ \to\ \Cc^{\boxtimes n} \]
  for all $1\leq k\leq m$. Here $i_0$ is to be interpreted as the zero morphism,
  i.e., the constant functor with value the distinguished object $(0,\dots,0)$ of $\Cc^{\boxtimes n}$.
  If $\kappa:n_+\to p_+$ is another based map, then
  \[ \gamma(\Cc)(\kappa)\circ\gamma(\Cc)(\lambda)\circ i_k\ = \
    \gamma(\Cc)(\kappa)\circ i_{\lambda(k)}\ = \ i_{\kappa(\lambda(k))}\ = \
    \gamma(\Cc)(\kappa\circ\lambda)\circ i_k 
  \]
  for all $1\leq k\leq m$.
  The universal property of an $m$-fold coproduct then forces
  $\gamma(\Cc)(\kappa)\circ\gamma(\Cc)(\lambda) = \gamma(\Cc)(\kappa\circ\lambda)$,
  so we have really obtained a functor from $\Gamma$ to the category
  of parsummable categories and strict morphisms.

  The functor $\gamma(\Cc)(\lambda):\gamma(\Cc)(m_+)\to\gamma(\Cc)(n_+)$
  is effectively given on objects and morphisms by  summing up, i.e.,
  \[ \left( \gamma(\Cc)(\lambda)(f_1,\dots,f_m) \right)_j\ = \ \sum_{\lambda(i)=j} f_i\ . \]
  If $(a_1,\dots,a_m)$ is an object of $\Cc^{\boxtimes m}$ and
  $(b_1,\dots,b_n)=\gamma(\Cc)(\lambda)(a_1,\dots,a_m)$, then
  \[ \supp(b_j)\ = \ \supp\left({\sum}_{\lambda(i)=j}\ a_i\right) \
    \subseteq \ {\bigcup}_{\lambda(i)=j}\ \supp(a_i) \]
  by~\eqref{eq:supp_of_sum}. So the tuple $(b_1,\dots,b_n)$ is again
  disjointly supported, i.e., it lies in $\gamma(\Cc)(n_+)$.
\end{con}

It is time to discuss examples of parsummable categories.

\begin{eg}[(Abelian monoids)]\label{eg:abelian monoid}
  We let $A$ be an abelian monoid that we consider as a discrete category
  (i.e., the object set is $A$, and there are only identity morphisms).
  We give $A$ the trivial $\M$-action: every object and
  morphism of the category $\M$ acts on $A$ as the identity.
  For this trivial $\M$-action, every object of $A$
  is supported on the empty set, so $A\boxtimes A=A\times A$.
  The structure functor
  \[  +\ : \ A\boxtimes A \ = \ A\times A \ \to \ A \]
  is the given monoid structure on objects.
  Because every tuple of elements is disjointly supported,
  the category $\gamma(A)(n_+)=A^n$ is the entire product category.
  The $\Gamma$-category $\gamma(A)$ 
  is thus discrete and coincides with the usual construction~\cite[page 293]{segal:cat coho}
  of a $\Gamma$-set from an abelian monoid. As we discuss in Example \ref{eg:HA is EM}
  below, the global K-theory spectrum of the parsummable category $A$ is
  a global Eilenberg-Mac\,Lane spectrum of the constant global functor with value
  the group completion of $A$.
\end{eg}

\begin{eg}[(The parsummable category of finite sets)]\label{eg:Fc}
  In Example \ref{eg:Fc as M-category} we introduced the tame $\M$-category $\Fc$
  of finite subsets of $\omega=\{0,1,2,\dots\}$; morphisms in $\Fc$ are all bijections of sets.
  This $\M$-category has a preferred parsummable structure, as follows.
  Since the support of an object of $\Fc$ is the set itself,
  objects of the $\M$-category $\Fc\boxtimes\Fc$ are pairs of finite disjoint
  subsets of $\omega$.
  The structure functor
  \[ + \ : \ \Fc\boxtimes \Fc \ \to \ \Fc \]
  is given on objects by the union inside the set $\omega$.
  If $f:P\to Q$ and $f':P'\to Q'$ are bijections between finite subsets of $\omega$,
  and moreover $P\cap P'=\emptyset=Q\cap Q'$,
  then $f+ f':P+ P'\to Q+ Q'$ is the `union' of $f$ and $f'$, i.e.,
  \[
    (f+ f')(x)\ = \
    \begin{cases}
      f(x) & \text{ for $x\in P$, and}\\
      f'(x) & \text{ for $x\in P'$.}
    \end{cases}
    \]
  We omit the straightforward verification that $+$
  is indeed a functor, that it is compatible with the $\M$-actions,
  and that it is associative, commutative and unital.
  So the sum functor $+$  makes the $\M$-category $\Fc$
  into a parsummable category.

  The parsummable category $\Fc$ has a universal property:  it is a free
  parsummable category generated by an object supported on the set $\{0\}$.
  More formally, for every parsummable category $\Cc$ the map
  \[ \parsumcat(\Fc,\Cc)\ \to \ \{x\ | \ \text{$x$ is supported on $\{0\}$}\} \ , 
    \quad F\ \longmapsto \ F(\{0\})\]
  is bijective. 
  We indicate what goes into the proof. We let $x$ be an object of $\Cc$ 
  supported on the set $\{0\}$; we construct a morphism of parsummable categories
  $F:\Fc\to \Cc$ that takes the set $\{0\}$ to $x$. 
  To define $F$ on objects we let  $P\subset \omega$ be a finite subset
  of cardinality $n$. We choose injections $u^1,\dots,u^n\in M$ such that
  \[ P \ = \ \{ u^1(0),\dots,u^n(0)\}  \ ; \]
  then the object $u^1_*(x),\dots,u^n_*(x)$ of $\Cc$ are disjointly supported,
  so we can set
  \[ F(P)\ = \ u^1_*(x)+\dots + u^n_*(x)\ . \]
  This is independent of the choices of $u^i$ by
  Proposition \ref{prop:finite support} (ii)  and the commutativity
  of the sum functor. Moreover, $F(P)$ is supported on $P$ by
  the additivity relation \eqref{eq:supp_of_sum}.

  Given a bijection $f:P\to Q$ between
  finite subsets of $\omega$,
  we choose an injection $w\in M$ that coincides with $f$ on $P$.
  Then $w u^1,\dots,w u^n\in M$ can be used to define $F(Q)$.
  We define $F(f):F(P)\to F(Q)$ as 
  \begin{align*}
    F(f)\ = \ [w u^1,u^1]^x &+\dots +[w u^n,u^n]^x \ : \\ 
    F(P)\ &=\ u^1_*(x)+\dots + u^n_*(x)\ \to\
            (w u)^1_*(x)+\dots + (w u)^n_*(x)\  = \ F(Q) \ .
  \end{align*}
  We omit the straightforward verification that $F:\Fc\to\Cc$ is a functor,
  compatible with the action of the monoidal category $\M$,
  that it takes the empty set to the distinguished object 0 of $\Cc$,
  and that it is additive on disjoint unions.
  The uniqueness is a consequence of the fact that the identity of the
  object $\{0\}$ of generates $\Fc$ under the $\M$-action and the sum operation.
\end{eg}

\begin{eg}[(Free parsummable categories)]\label{eg:free_parsummable}
  The category of tame $\M$-categories is cocomplete, with colimits
  created in the underlying category $\cat$ of small categories, see Example \ref{eg:(co)limits tame Mcat}.
  In particular, for every tame $\M$-category $\Bc$,
  we can form the symmetric algebra with respect to the box product
  \[  \mP\Bc\ = \ \coprod_{m\geq 0} \, (\Bc^{\boxtimes m})/\Sigma_m  \ . \]
  We refer to $\mP\Bc$ as the {\em free parsummable category} generated by the tame $\M$-category $\Bc$.
  This free functor is left adjoint to the forgetful functor $\parsumcat\to\M\cat^\tau$,
  with the embedding $\eta:\Bc\to\mP\Bc$ as the summand indexed by $m=1$ being the unit of the adjunction.
  In other words, for every parsummable category $\Cc$ and every morphism of
  $\M$-categories $f:\Bc\to\Cc$, there is a unique morphism of parsummable categories
  $f^\sharp:\mP\Bc\to\Cc$ such that $f^\sharp\circ\eta=f$.
\end{eg}

Now we record that parsummable categories are closed under various kinds of constructions.

\begin{eg}[(Full subcategories)]
  We let $\bar\Cc$ be a full subcategory of a parsummable category
  that contains the unit object 0 and is closed under the $\M$-action
  and under the structure functor $+:\Cc\boxtimes\Cc\to\Cc$.
  Then $\bar\Cc$ is a parsummable category in its own right, by restriction of all the structure.
  The inclusion $\bar\Cc\to\Cc$ is a morphism of parsummable categories.

  Here are two specific examples of this situation.
  For every $n\geq 1$, we can take $\Fc^{[n]}$ 
  to be the full subcategory of $\Fc$ of those finite subsets
  of $\omega$ whose cardinality is divisible by $n$. 
  For every ring $R$, another example is the full subcategory $\Pc^\text{fr}(R)$ 
  of free modules inside the parsummable category $\Pc(R)$ of 
  finitely generated projective splittable $R$-submodules of $R\{\omega\}$,
  defined in Construction \ref{con:P(R)} below.
\end{eg}

\begin{eg}
  For every parsummable category $\Cc$ we can consider the full
  subcategory $\Cc^\emptyset$ of $\Cc$
  spanned by the objects that are supported on the empty set.
  By parts (iii) and (iv) of Proposition~\ref{prop:finite support},
  the $\M$-action is trivial on $\Cc^\emptyset$;
  in particular, the subcategory $\Cc^\emptyset$ is closed under the $\M$-action.
  The subcategory $\Cc^\emptyset$ is moreover
  closed under the sum functor by \eqref{eq:supp_of_sum}.
  So the subcategory $\Cc^\emptyset$ is a parsummable category in its own right.
  
  The parsummable categories $\Cc$ in which all objects
  are supported on the empty set are very rigid.
  For example, we have $\Cc\boxtimes\Cc=\Cc\times\Cc$,
  so the sum functor is defined for all pairs of objects and morphisms.
  Moreover, the sum operation makes $\Cc$ into a strict symmetric monoidal category
  in which the symmetry isomorphisms are identities.
  The nerve of $\Cc$ then becomes a simplicial abelian monoid,
  and its K-theory spectrum is a product of Eilenberg-Mac Lane spectra.  
\end{eg}

\begin{eg}[(Opposite parsummable categories)]
  If $\Cc$ is a parsummable category, then the opposite category $\Cc^{\op}$
  inherits a canonical structure of a parsummable category as follows.
  As we explained in Example \ref{eg:opposite M-version},
  the opposite of an $\M$-category has a preferred $\M$-action.
  For this $\M$-action, the support of an object is the same
  in $\Cc$ and in $\Cc^{\op}$. So in particular $\Cc^{\op}$ is again tame,
  and $\Cc^{\op}\boxtimes \Cc^{\op} = (\Cc\boxtimes\Cc)^{\op}$.
  The structure functor of~$\Cc^{\op}$ is then defined as the composite
  \[ \Cc^{\op}\boxtimes \Cc^{\op}\ = \ (\Cc\boxtimes\Cc)^{\op}\
    \xra{\ +^{\op}} \ \Cc^{\op}\ . \]
  So the sum operation in $\Cc$ and $\Cc^{\op}$ is the same on objects,
  and is the `opposite sum' on morphisms. Moreover,
  \[ \gamma(\Cc^{\op}) \ = \ \gamma(\Cc)^{\op}\]
  as $\Gamma$-parsummable categories.
\end{eg}

\begin{eg}[(Coproducts of parsummable categories)]\label{eg:coproduct_parsumcat}
  It is a general feature that the commutative monoid objects
  in a symmetric monoidal category have coproducts, and these
  are given by the underlying monoidal product.
  This applies to parsummable categories, as these are the commutative monoid objects in
  tame $\M$-categories under the box product.
  In more detail: given two parsummable categories $\Cc$ and $\Dc$, the box product
  $\Cc\boxtimes\Dc$ becomes a parsummable category with distinguished zero object $(0,0)$
  and addition functor
  \[
    ( \Cc\boxtimes\Dc)\boxtimes( \Cc\boxtimes\Dc)
    \ \xra[\iso]{\Cc\boxtimes\text{twist}\boxtimes \Dc} \
    ( \Cc\boxtimes\Cc)\boxtimes( \Dc\boxtimes\Dc)\ \xra{+\boxtimes +}\  \Cc\boxtimes\Dc\ .
  \]
  The two functors
  \[ i_1\ =\ (-,0)\ : \ \Cc \ \to \Cc\boxtimes\Dc\text{\qquad and\qquad}
    i_2\ =\ (0,-)\ : \ \Dc \ \to \Cc\boxtimes\Dc    
  \]
  are morphisms of parsummable categories that enjoy the universal property of a coproduct.
  Moreover, the unique morphism of $\Gamma$-parsummable categories
  \[ \gamma(\Cc)\boxtimes \gamma(\Dc)\ \to \ \gamma(\Cc\boxtimes \Dc) \]
  that restricts to $\gamma(i_1):\gamma(\Cc)\to \gamma(\Cc\boxtimes \Dc)$
  and $\gamma(i_2):\gamma(\Dc)\to \gamma(\Cc\boxtimes \Dc)$, respectively,
  is an isomorphism.
\end{eg}

\begin{eg}[(Limits of parsummable categories)]
  As we explained in Example \ref{eg:(co)limits tame Mcat},
  the category of tame $\M$-categories is complete.
  Moreover, finite limits are created on underlying categories,
  and an infinite product of tame $\M$-categories is
  given by the full subcategory of finitely supported objects inside the product
  of categories with the induced $\M$-action.
  Since parsummable categories are the commutative algebras for a symmetric monoidal
  structure on $\M\cat^\tau$, the category of parsummable categories is also
  complete, and limits are created in $\M\cat^\tau$.
\end{eg}

\begin{eg}[(Objects with an action)]\label{eg:objects with action}
  Let $G$ be a monoid. A {\em $G$-object} in a category $\Dc$ 
  is an object $x$ of $\Dc$ equipped
  with a $G$-action, i.e., a monoid homomorphism $\rho:G\to\Dc(x,x)$
  to the endomorphism monoid. 
  We denote by $G\Dc$ the category of $G$-objects in $\Dc$ 
  with $G$-equivariant $\Dc$-morphisms. We will mostly be interested in the
  special case when $G$ is a finite group (whence the letter `$G$').
  
  We note that parsummable structures lift to objects with a monoid action.
  In other words: a parsummable structure on a category $\Cc$ gives rise
  to a preferred parsummable structure on the category $G\Cc$.
  Indeed, because $G$-objects are `the same as' functors from the category with one object
  and $G$ as endomorphism monoid, the category $G\Cc$ inherits a `pointwise'
  $\M$-action as explained in Example \ref{eg:(co)limits Mcat}.
  For this $\M$-action, the support of a $G$-object is the support of
  the underlying $\Cc$-object. So $G\Cc$ is again tame, and
  \[ (G\Cc)\boxtimes(G\Cc)\ = \ G(\Cc\boxtimes\Cc) \ .\]
  The structure functor of $G\Cc$ is then defined as the composite
  \[
    (G\Cc)\boxtimes(G\Cc)\ = \ G(\Cc\boxtimes\Cc)\ \xra{\ G+\ } \ G\Cc\ .
  \]
  Moreover, $\gamma(G\Cc) = G(\gamma(\Cc))$ as $\Gamma$-parsummable categories.
\end{eg}

Every $\Gamma$-parsummable category has an underlying $\Gamma$-$\M$-category.
We introduced the concept of  {\em global specialness}  for $\Gamma$-$\M$-categories
in Definition \ref{def:globally special Gamma-cat}.

\begin{theorem}\label{thm:special G Gamma}
  Let $\Cc$ be a parsummable category.
  Then the $\Gamma$-parsummable category $\gamma(\Cc)$ is globally special.
\end{theorem}
\begin{proof}
  We will show that for every finite group $G$, every universal $G$-set $\Uc$
  and every finite $G$-set $S$ the functor
  \[ (P_S)^G\ : \ (\gamma(\Cc)[\Uc](S_+))^G\ \to \   \map^G(S,\gamma(\Cc)[\Uc](1_+)) \]
  is an equivalence of categories.  
  The $G$-category $\gamma(\Cc)[\Uc](S_+)$ 
  is a full $G$-subcategory of $\map(S,\Cc[\Uc])$, and the functor  
  \[ P_S\ : \ \gamma(\Cc)[\Uc](S_+)\ \to \ \map(S,\gamma(\Cc)[\Uc](1_+))
    = \map(S,\Cc[\Uc]) \]
  is the inclusion, and hence fully faithful.
  For all $G$-fixed objects $x$ and $y$ in $\gamma(\Cc)[\Uc](S_+)$, the map 
  \[ P_S\ : \ \gamma(\Cc)[\Uc](S_+)(x,y) \ \to \ \map(S,\Cc[\Uc])(P_S(x),P_S(y))\]
  of morphism sets is thus a $G$-equivariant bijection. 
  So the restriction to $G$-fixed points
  \[ (P_S)^G\ : \ \gamma(\Cc)[\Uc](S_+)^G(x,y) \ \to \ \map^G(S,\Cc[\Uc])(P_S(x),P_S(y))\]
  is also a bijection. This means that the $G$-fixed functor
  $(P_S)^G$ is also fully faithful.
  
  It remains to show that the functor $(P_S)^G$ is dense, i.e., every object
  in the target category is isomorphic to an object in the image of $(P_S)^G$.
  Since $\Uc$ is a universal $G$-set,
  we can choose a $G$-equivariant injection $\psi:S\times\Uc\to \Uc$.
  For $s\in S$ we define $\psi^s:\Uc\to\Uc$ by
  \[ \psi^s(j)\ = \ \psi(s,j)\ . \]
  Now we let $(x_s)_{s\in S}$ be a $G$-fixed object in the product category $\map(S,\Cc[\Uc])$,
  i.e., such that $l^g_*(x_s)=x_{g s}$ for all $(g, s)\in G\times S$,
  where $l^g:\Uc\to\Uc$ is left multiplication by $g$.
  Since the injections $\psi^s$ have disjoint images, the tuple
  \[    (\psi^s_*(x_s))_{s\in S} \ \in \ \map(S,\Cc[\Uc]) \]
  consists of objects with pairwise disjoint supports, so it 
  belongs to the subcategory $\gamma(\Cc)[\Uc](S_+)$.
  The $G$-equivariance of $\psi$ translates into the relation
  $l^g\circ\psi^s =\psi^{g s}\circ l^g$
  for all $(g,s)\in G\times S$. So we have
  \[    l^g_*( [\psi^s,1]^{x_s})\ =\ [ l^g \psi^s, l^g]^{x_s}\ = \
    [\psi^{g s}l^g, l^g]^{x_s} \ =\  [\psi^{g s},1]^{l^g_*(x_s)} \ =\  [\psi^{g s},1]^{x_{g s}} \ .  \]
  This means that the $S$-tuple of isomorphisms $[\psi^s,1]^{x_s}:x_s\to\psi^s_*(x_s)$
  is $G$-fixed.
  So the original tuple $(x_s)_{s\in S}$ is isomorphic to the tuple $(\psi^s_*(x_s))_{s\in S}$
  in the $G$-fixed category $(\gamma(\Cc)[\Uc](S_+))^G$.
  This completes the proof that the functor $(P_S)^G$ is essentially surjective
  on objects, and thus an equivalence of categories.
\end{proof}

Now we come to the main construction of this paper.

\begin{defn}\label{def:define K_gl}
  The {\em global algebraic K-theory spectrum} of a parsummable category~$\Cc$
  is the symmetric spectrum
  \[ \bK_{\gl}\Cc \ = \ \gamma(\Cc)\td{\mS} \ ,\]
  as defined in Construction \ref{con:spectrum from Gamma-M},
  associated with the $\Gamma$-parsummable category $\gamma(\Cc)$.
\end{defn}

Since the global algebraic K-theory spectrum is the central construction of this paper,
we take the time to expand the definition of $\bK_{\gl}\Cc$.
The value of this symmetric spectrum at a non-empty finite set $A$ is 
\[ (\bK_{\gl}\Cc)(A) \ = \ \gamma(\Cc)\td{\mS}(A) \ = \
  |\gamma(\Cc)[\omega^A]|(S^A)\ ,\]
the value of the $\Gamma$-$\Sigma_A$-space $|\gamma(\Cc)[\omega^A]|$
on the $A$-sphere. The symmetric group $\Sigma_A$ acts diagonally, through
the reparameterization action on $\gamma(\Cc)[\omega^A]$,
and the action on $S^A$ by permuting coordinates.
The value at the empty set is
\[ (\bK_{\gl}\Cc)(\emptyset) \ = \ \gamma(\Cc)\td{\mS}(\emptyset) \ = \
  |\Cc^{\supp=\emptyset}|\ ,\]
the realization of the full subcategory of $\Cc=\gamma(\Cc)(1_+)$
of objects with empty support (also known as the $M$-fixed subcategory of $\Cc$).
The structure maps of the symmetric spectrum $\bK_{\gl}\Cc$
are defined in Construction \ref{con:spectrum from Gamma-M}.
\medskip

The following theorem records two properties
of the global algebraic K-theory spectrum $\bK_{\gl}\Cc$
that are special cases of results from Section \ref{sec:Gamma-M_to_sym}.
As before, $(\bK_{\gl}\Cc)_G$ denotes the underlying $G$-symmetric spectrum
of the symmetric spectrum $\bK_{\gl}\Cc$ (i.e., with $G$ acting trivially).
The $G$-symmetric spectrum $\gamma(\Cc)\td{\omega^G,\mS}$
and the morphisms of $G$-symmetric spectra
\[
  a^{\gamma(\Cc)}_G\colon (\bK_{\gl}\Cc)_G = (\gamma(\Cc)\td{\mS})_G \ \to\ \gamma(\Cc)\td{\omega^G,\mS}
  \text{\quad and\quad}
  b_G^{\gamma(\Cc)}\colon |\gamma(\Cc)[\omega^G]|(\mS)\ \to\ \gamma(\Cc)\td{\omega^G,\mS}
\]
were defined in Construction \ref{con:underlying G comparison}.

\begin{theorem}\label{thm:K_gl C is global Omega}
  Let $\Cc$ be a parsummable category.
  \begin{enumerate}[\em (i)]
  \item 
    The global K-theory spectrum $\bK_{\gl} \Cc$
    is globally connective and a restricted global $\Omega$-spectrum.
  \item
    For every finite group $G$,
    the two morphisms of $G$-symmetric spectra
    \[
      (\bK_{\gl}\Cc)_G\ \xra{\ a^{\gamma(\Cc)}_G\ }\ \gamma(\Cc)\td{\omega^G,\mS} 
      \ \xla{\ b_G^{\gamma(\Cc)}\ } \   |\gamma(\Cc)[\omega^G]|(\mS)
    \]
    are $G$-stable equivalences.
  \end{enumerate}
\end{theorem}
\begin{proof}
  The $\Gamma$-parsummable category $\gamma(\Cc)$
  is globally special by Theorem \ref{thm:special G Gamma}.
  The two claims are thus special cases of Proposition \ref{prop:Y<S> globally connective-semistable}
  and Theorem \ref{thm:Y(S) is global Omega},
  and Theorem \ref{thm:a-b-maps equivalence}, respectively,
  applied to the globally special $\Gamma$-$\M$-category $\gamma(\Cc)$.
\end{proof}

Algebraic K-theory usually takes equivalent categorical input data
to equivalent homotopical output.
Our next theorem is a version of this principle
for global algebraic K-theory.
We call a morphism of parsummable categories a {\em global equivalence}
if the underlying morphism of $\M$-categories is a global equivalence
in the sense of Definition \ref{def:global equiv Gamma-M}.

\begin{theorem}\label{thm:equivalence2equivalence} 
  For every  global equivalence of parsummable categories $\Phi:\Cc\to\Dc$,
  the morphism
  \[ \bK_{\gl}\Phi \ : \  \bK_{\gl}\Cc \ \to \  \bK_{\gl}\Dc \]
  is a global equivalence of symmetric spectra.
\end{theorem}
\begin{proof}
  The $\Gamma$-parsummable categories $\gamma(\Cc)$ and  $\gamma(\Dc)$
  are globally special by Theorem \ref{thm:special G Gamma}.
  The value of the morphism of $\Gamma$-parsummable categories
  $\gamma(\Phi):\gamma(\Cc)\to \gamma(\Dc)$ at $1_+$ is the original morphism $\Phi$,
  and hence a global equivalence.
  Proposition \ref{prop:globally special yields strict cat}
  thus applies to the morphism $\gamma(\Phi):\gamma(\Cc)\to \gamma(\Dc)$,
  and yields the desired result.
\end{proof}

K-theory tends to take finite products of input data to products of spaces or spectra.
The next proposition proves a version of this principle in our context.
The box product $\Cc\boxtimes \Dc$ of two parsummable categories
is a full parsummable subcategory of the product $\Cc\times \Dc$,
and it comes with distinguished morphisms 
\[ i_1\ =\ (-,0)\ :\ \Cc\ \to\ \Cc\boxtimes \Dc \text{\qquad and\qquad}
  i_2\ =\ (0,-)\ :\ \Dc\ \to\ \Cc\boxtimes \Dc\]
that express $\Cc\boxtimes \Dc$ as a coproduct
in the category of parsummable categories, see Example \ref{eg:coproduct_parsumcat}.

\begin{theorem}\label{thm:K_gl additive}
  Let $\Cc$ and $\Dc$ be parsummable categories. Then the morphisms
  \begin{align*}
    (\bK_{\gl}\Cc)\vee (\bK_{\gl}\Dc)\
    &\xra{(\bK_{\gl}i_1)+(\bK_{\gl}i_2)} \  \bK_{\gl}(\Cc\boxtimes\Dc)  \\
    &\xra{\bK_{\gl}(\text{\em incl})} \    
    \bK_{\gl}(\Cc\times\Dc)\ \xra{(\bK_{\gl} p_1,\bK_{\gl} p_2)} \ (\bK_{\gl}\Cc)\times(\bK_{\gl}\Dc)  
  \end{align*}
  are global equivalences of symmetric spectra,
  where $p_1:\Cc\times\Dc\to\Cc$ and $p_2:\Cc\times\Dc\to\Dc$ are the projections.
\end{theorem}
\begin{proof}
  We start with the morphism
  $(\bK_{\gl} p_1,\bK_{\gl} p_2):\bK_{\gl}(\Cc\times\Dc)\to(\bK_{\gl}\Cc)\times(\bK_{\gl}\Dc)$.
  This morphism factors as the composite
  \begin{align*}
    \bK_{\gl}(\Cc\times\Dc)\
    = \ \gamma(\Cc\times\Dc)\td{\mS}\ &\xra{(\gamma(p_1),\gamma(p_2))\td{\mS}} \
      (\gamma(\Cc)\times\gamma(\Dc))\td{\mS}\\
    &\xra{(p_1\td{\mS},p_2\td{\mS})} \
    \gamma(\Cc)\td{\mS}\times \gamma(\Dc)\td{\mS}\ = \ (\bK_{\gl}\Cc)\times(\bK_{\gl}\Dc)\ .
  \end{align*}
  The $\Gamma$-parsummable categories $\gamma(\Cc)$, $\gamma(\Dc)$ and $\gamma(\Cc\times\Dc)$
  are globally special by Theorem \ref{thm:special G Gamma}.
  Global specialness is inherited by products, so 
  the $\Gamma$-parsummable category $\gamma(\Cc)\times \gamma(\Dc)$ is also globally special.
  The morphism of $\Gamma$-parsummable categories
  \[ (\gamma(p_1),\gamma(p_2))\ : \ \gamma(\Cc\times\Dc)\ \to \ \gamma(\Cc)\times\gamma(\Dc) \]
  is the identity of  $\gamma(\Cc\times\Dc)(1_+)=\Cc\times\Dc=(\gamma(\Cc)\times\gamma(\Dc))(1_+)$
  at the object $1_+$. 
  Proposition \ref{prop:globally special yields strict cat}
  thus applies to the morphism $(\gamma(p_1),\gamma(p_2))$
  and shows that the first morphism
  $(\gamma(p_1),\gamma(p_2))\td{\mS}$ is a global equivalence of symmetric spectra.
  The morphism $(p_1\td{\mS},p_2\td{\mS})$
  is an isomorphism of symmetric spectra by Proposition \ref{prop:at spheres preserves products}.
  So altogether this shows that the morphism
  $(\bK_{\gl} p_1,\bK_{\gl} p_2):\bK_{\gl}(\Cc\times\Dc)\to(\bK_{\gl}\Cc)\times(\bK_{\gl}\Dc)$
  is a global equivalence of symmetric spectra.
    
  The inclusion $\Cc\boxtimes \Dc\to\Cc\times \Dc$
  is a morphism of parsummable categories, and a global equivalence
  by Theorem \ref{thm:box2times};
  so the induced morphism of global K-theory spectra is a global equivalence by 
  Theorem \ref{thm:equivalence2equivalence}. 
  The composite $(\bK_{\gl}\Cc)\vee(\bK_{\gl}\Dc)\to   (\bK_{\gl}\Cc)\times(\bK_{\gl}\Dc)$
  is the canonical morphism from the coproduct to the product of two symmetric spectra;
  it is a global equivalence by \cite[Proposition 4.6 (3)]{hausmann:global_finite}.
  This concludes the proof.
\end{proof}

The next theorem generalizes the additivity of global K-theory
to infinitely many factors.

\begin{eg}[(Infinite box products)]\label{eg:infinite box}
  The construction of the box product of parsummable categories directly generalizes
  to more than two factors.
  We let $J$ be an indexing set, possibly infinite, and $\{\Cc_j\}_{j\in J}$
  a family of parsummable categories. The {\em box product}
  \[ \boxtimes_{j\in J} \, \Cc_j\  \subseteq \   \prod_{j\in J} \, \Cc_j\]
  is the full subcategory of the product consisting of those objects
  $(x_j)_{j\in J}$ such that the supports of the objects $x_j$ are pairwise disjoint,
  and moreover  $x_j=0$ for almost all $j\in J$. 
  As for the box product with two factors in Example \ref{eg:coproduct_parsumcat},
  this subcategory is a parsummable category in its own right,
  and it is a coproduct, in $\parsumcat$, of the parsummable categories $\Cc_j$.
\end{eg}

\begin{theorem}\label{thm:K infinite additive}
  Let $J$ be a set and let $\{\Cc_j\}_{j\in J}$ be a family of parsummable categories.
  Then the canonical morphism
  \[     {\bigvee}_{j\in J}\ \bK_{\gl}(\Cc_j)\ \to \  \bK_{\gl}(\boxtimes_{j\in J}\,\Cc_j)  \]
  is a global equivalence of symmetric spectra.
\end{theorem}
\begin{proof}
  For finite indexing sets, the claim follows from Theorem \ref{thm:K_gl additive}
  by induction over the cardinality.
  If the indexing set $J$ is infinite, we let $s(J)$ denote the filtered poset, under inclusion,
  of finite subsets of $J$.
  The canonical morphism factors as the composite
  \[    {\bigvee}_{j\in J}\ \bK_{\gl}(\Cc_j)\ \to \  
    \colim_{K\in s(J)}  \bK_{\gl}(\boxtimes_{k\in K}\,\Cc_k)  \ \to\
      \bK_{\gl}(\boxtimes_{j\in J}\,\Cc_j)  \ . \]
  The wedge over $J$ is the colimit of the finite wedges over $K\in s(J)$,
  and filtered colimits of global equivalences of symmetric spectra
  arising from simplicial sets are homotopical.
  For every $K\in s(J)$, the morphism
  $\bigvee_{k\in K}\bK_{\gl}(\Cc_k)\to \bK_{\gl}(\boxtimes_{k\in K}\,\Cc_k)$
  is a global equivalence, hence so is the first map in the factorization.

  For the second map we observe that for every finite set $A$, the map
  \[  \colim_{K\in s(J)}  |\gamma(\boxtimes_{k\in K}\,\Cc_k)[\omega^A]|(S^A) \ \to\
    |\gamma(\boxtimes_{j\in J}\,\Cc_j)[\omega^A]|(S^A) \]
  is a homeomorphism, because all individual steps involved commute with filtered colimits.
  So the second morphism is even an isomorphism of symmetric spectra.
\end{proof}

\begin{rk}[(Parsummable categories as global $E_\infty$-categories)]\label{rk:I as global E_infty}
  As we indicate now, the structure of a parsummable category on a category $\Cc$
  is equivalent to an action of the {\em injection operad},
  a specific categorical operad whose category of
  unary operations is the monoidal category $\M$.
  We will not use this perspective on parsummable categories,
  so we will be brief.
  In the simpler context of sets (as opposed to categories),
  the analogous comparison is discussed in detail in \cite[Appendix A]{sagave-schwede}:
  Theorem A.13 of that paper provides an isomorphism of categories
  from the category of algebras over the set-version of the injection operad
  to the category of commutative monoids, under box product,
  in the category of tame $M$-sets.
  The present remark is about the analogous statement one category level higher.

  The {\em injection operad} $\Ic$ is an operad in the category of small categories
  with respect to cartesian product; it contains the monoidal category $\M$
  as its category of unary operations. 
  For $m\geq 0$ we define $\mathbf m=\{1,\dots,m\}$
  and we let $I(m)$ denote the set of injective maps
  from the set $\mathbf m\times\omega$ to the set $\omega=\{0,1,2,\dots\}$.
  The symmetric group $\Sigma_m$ acts on $I(m)$ by permuting the first
  coordinate in $\mathbf m\times\omega$.
  The collection of sets $\{I(m)\}_{m\geq 0}$ becomes an operad 
  via `disjoint union and composition'. The functor
  \[ E \ : \ \bset \ \to \ \cat \]
  from sets to small categories is right adjoint to taking the set of objects;
  so $E$ preserves products and hence it takes set operads to categorical operads.
  So we obtain an operad $\Ic=\{E I(m)\}_{m\geq 0}$ in the category of
  small categories under cartesian product.

  Now we sketch why parsummable categories `are' algebras over the categorical operad $\Ic$.
  We let $\Cc$ be a parsummable category and $\varphi:\mathbf m\times\omega\to\omega$
  an injection.
  For $i=1,\dots,m$, the injections $\varphi^i=\varphi(i,-)$
  have disjoint images, so the functor
  \[ \prod_{i=1}^m \varphi^i_* \ : \ \Cc^m \ \to \ \Cc^m \]
  takes values in the full subcategory $\Cc^{\boxtimes m}$
  of $m$-tuples of disjointly supported objects.
  We can thus define a functor $\varphi_* :\Cc^m\to \Cc$
  as the composite
  \[  \Cc^m\  \xra{\prod_{i=1}^m \varphi^i_*}\ \Cc^{\boxtimes m}\ \xra{\ + \ }\ \Cc\ ,  \]
  where the second functor is the iterated sum functor.
  Given another injection $\psi:\mathbf m\times\omega\to\omega$,
  we define a natural isomorphism
  \[ [\psi,\varphi]\ : \ \varphi_* \ \Longrightarrow \ \psi_* \]
  of functors $\Cc^m\to\Cc$ at a tuple $(x_1,\dots,x_m)$ of objects as
  \begin{equation}\label{eq:define [psi,varphi]}
    [\psi,\varphi]^{x_1,\dots,x_m}\ = \ \sum_{i=1}^m \ [\psi^i,\varphi^i]^{x_i} \ :
    \sum_{i=1}^m \ \varphi^i_*(x_i) \ \to \ \sum_{i=1}^m \ \psi^i_*(x_i) \ .
  \end{equation}
  We omit the verification that this data defines an action
  of the category operad $\Ic$ on the category $\Cc$.
  Clearly, for $m=1$, the action of $\Ic(1)$ coincides with the given action
  of the monoidal category $\M$.
  We claim without proof that for {\em tame} $\M$-categories, the process
  can be reversed, and leads to an an isomorphism of categories
  between parsummable categories in the sense of Definition \ref{def:parsummable_category}
  and algebras over the injection operad $\Ic$ whose underlying $\M$-categories are tame. 
  As we already mentioned, the argument is similar to the
  proof of \cite[Theorem A.13]{sagave-schwede}, but with categories instead of sets.

  We also justify why we think of the injection operad as a `global $E_\infty$-operad'.
  We let $G$ be a finite group. Replacing the set $\omega$ by the universal $G$-set
  $\omega^G$ yields a categorical $G$-operad $\Ic_G$, i.e., an operad
  in the cartesian monoidal category of small $G$-categories.
  The $G$-category of $m$-ary operations is $\Ic_G(m)=E I(\mathbf m\times\omega^G,\omega^G)$,
  where $I(\mathbf m\times\omega^G,\omega^G)$ is the set of injections
  from $\mathbf m\times\omega^G$ to $\omega^G$, with $G$-action by conjugation.
  A moment's thought shows that $\Ic_G$ is an {\em $E_\infty$-operad of $G$-categories}
  in the sense of Guillou and May \cite[Definition 3.11]{guillou-may};
  hence the nerve of every $\Ic_G$-$G$-category is an $E_\infty$-$G$-space.

  The connection to our present discussion is that every parsummable category $\Cc$
  has an {\em underlying $E_\infty$-$G$-category} in the sense of
  \cite[Definition 4.10]{guillou-may},
  namely the $G$-category $\Cc[\omega^G]$ with a specific
  action of the categorical $E_\infty$-$G$-operad $\Ic_G$ arising from
  the parsummable structure.
  The formalism of Guillou and May assigns to the
  $E_\infty$-$G$-category $\Cc[\omega^G]$ an orthogonal $G$-spectrum $\mathbb K_G(\Cc[\omega^G])$,
  see \cite[Definition 4.12]{guillou-may}.
  It seems highly plausible that the Guillou-May $G$-spectrum $\mathbb K_G(\Cc[\omega^G])$
  is equivalent to the equivariant K-theory $G$-spectrum obtained by the
  Segal-Shimakawa delooping machine from the $\Gamma$-$G$-category
  that extends $\Cc[\omega^G]$; however, I have not attempted to formally prove this.
  Assuming this equivalence,  our Theorem \ref{thm:K_gl C is global Omega} (ii)
  shows that the underlying genuine $G$-homotopy type of $\bK_{\gl}\Cc$
  agrees with the $G$-homotopy type of the $E_\infty$-$G$-category $\Cc[\omega^G]$
  as defined by Guillou and May.
\end{rk}

In the final part of this section we identify the $G$-fixed point spectrum
of the global homotopy type of $\bK_{\gl}\Cc$, for $G$ a finite group:
loosely speaking, global K-theory `commutes with $G$-fixed points'.
More precisely, the {\em $G$-fixed category} $F^G\Cc$
and a parsummable category $\Cc$ is naturally again a parsummable category.
We show in Corollary \ref{cor:G-fixed}
that the $G$-fixed point spectrum $F^G(\bK_{\gl}\Cc)$
receives a natural equivalence from the K-theory spectrum of $F^G\Cc$.

\begin{con}\label{con:F^G C}
  Given a parsummable category $\Cc$ and a finite group $G$,
  we define a new parsummable category $F^G\Cc$, the $G$-fixed point parsummable category, 
  by extending Construction \ref{con:F^G for M-cat}
  from $\M$-categories to parsummable categories.
  We recall that the underlying category is given by
  \[ F^G\Cc\ =  \ \Cc[\omega^G]^G\ , \]
  the $G$-fixed category of the $G$-category $\Cc[\omega^G]$,
  where $\omega^G$ is the $G$-set of maps from $G$ to $\omega$.
  Tameness is a property of the underlying $\M$-category,
  so Proposition \ref{prop:F^G_preserves_tame} shows that
  $F^G\Cc$ is tame because $\Cc$ is.
    
  We record how the passage to $G$-fixed categories interacts with the box product.
  So we consider two tame $\M$-categories $\Cc$ and $\Dc$.
  Since $\Cc\boxtimes\Dc$ is a full subcategory of $\Cc\times\Dc$,
  the fixed category $F^G(\Cc\boxtimes\Dc)$ is a full subcategory of
  $F^G(\Cc\times\Dc)\iso (F^G\Cc)\times(F^G\Dc)$,
  and so is $(F^G\Cc)\boxtimes(F^G\Dc)$.
  As we now explain, there is a subtle difference in the support conditions that define
  $F^G(\Cc\boxtimes\Dc)$ and $(F^G\Cc)\boxtimes(F^G\Dc)$.
  Indeed, the support condition that defines the subcategory
  $\Cc[\omega^G]\boxtimes \Dc[\omega^G]$ refers to support as a subset of $\omega$,
  and it arises from the $M$-action through the action on $\omega^G$ by postcomposition.
  The support condition that singles out the subcategory $(\Cc\boxtimes \Dc)[\omega^G]$
  refers to support as a subset of $\omega^G$, and uses the action of the monoid $M_G=I(\omega^G,\omega^G)$.
  The connection between these kinds of support is as follows:
  for a finite subset $T$ of $\omega^G$, we set
  \[ I(T)\ = \ \bigcup_{\alpha\in T}\text{image}(\alpha) \ , \]
  which is a finite subset of $\omega$.
  If $S,T\subset \omega^G$ are such that $I(S)$ and $I(T)$ are disjoint subsets of $\omega$,
  then $S$ and $T$ must in particular be disjoint.
  However, the converse implication need not hold in general.
  Moreover, if $(c,d)$ is an object of $\Cc[\omega^G]\times\Dc[\omega^G]$
  with $\supp(c)\cap \supp(d)=\emptyset$ as subsets of $\omega^G$,
  but $I(\supp(c))\cap I(\supp(d))\ne\emptyset$ as subsets of $\omega$,
  then $(c,d)$ belongs to $(\Cc\boxtimes\Dc)[\omega^G]$,
  but {\em not} to $\Cc[\omega^G]\boxtimes \Dc[\omega^G]$.
  In summary, the category $\Cc[\omega^G]\boxtimes\Dc[\omega^G]$ is contained in
  $(\Cc\boxtimes\Dc)[\omega^G]$, but it is typically strictly smaller.
  We write
  \begin{equation}\label{eq:fix and box before G-fix}
    \epsilon \ : \ \Cc[\omega^G]\boxtimes \Dc[\omega^G]\ \to \ (\Cc\boxtimes\Dc)[\omega^G]
  \end{equation}
  for the fully faithful inclusion.
  Restricting to $G$-fixed subcategories provides another fully faithful inclusion
  \begin{equation}\label{eq:fix and box}
    \epsilon^G \ : \ (F^G\Cc)\boxtimes (F^G\Dc)\ = \
   (\Cc[\omega^G]\boxtimes \Dc[\omega^G])^G\ \to \ ((\Cc\boxtimes\Dc)[\omega^G])^G \ = \ F^G(\Cc\boxtimes\Dc)\ .
  \end{equation}
\end{con}

\begin{prop}
  Let $\Cc$ and $\Dc$ be tame $\M$-categories.
  Then for every finite group $G$, the morphism   
  $\epsilon^G :(F^G \Cc)\boxtimes( F^G\Dc) \to F^G(\Cc\boxtimes\Dc)$
  is a global equivalence of $\M$-categories.
\end{prop}
\begin{proof}
  The composite
  \[ 
      (F^G \Cc)\boxtimes( F^G\Dc) \ \xra{\ \epsilon^G\ } \
      F^G(\Cc\boxtimes\Dc)\ \xra{F^G(\text{incl})} \
     F^G(\Cc\times\Dc)\ \xra{(F^G p_1,F^G p_2)} \  (F^G\Cc)\times (F^G \Dc)  \]
   is the inclusion, and hence a global equivalence by Theorem \ref{thm:box2times}.
   Reparameterization by $\omega^G$ and taking $G$-fixed points commute with products,
   so the third functor $(F^G p_1,F^G p_2)$ is an isomorphism of $\M$-categories.
   The inclusion $\Cc\boxtimes\Dc\to\Cc\times\Dc$
   is a global equivalence by Theorem \ref{thm:box2times},
   so the second morphism $F^G(\text{incl})$ is a global equivalence by
   Proposition \ref{prop:F^G preserves global M}.
   Since the second and third morphism, as well as the composite,
   are global equivalences, so is the morphism $\epsilon^G$.
\end{proof}

Now we let $\Cc$ be a parsummable category and $G$ a finite group.
We can make the $G$-fixed category $F^G\Cc$ into a parsummable category by
endowing it with the structure morphism defined as the composite
\[  (F^G\Cc)\boxtimes (F^G\Cc)\ \xra{\ \epsilon^G\ } \ F^G(\Cc\boxtimes\Cc) \
  \xra{F^G(+)}\  F^G(\Cc) \ .\]
We will now compare the global K-theory of the $G$-fixed point parsummable category $F^G\Cc$
to the $G$-fixed point spectrum of the global K-theory of $\Cc$,
via two global equivalences of symmetric spectra
\[
  \bK_{\gl}(F^G\Cc) \ =\   \gamma(F^G(\Cc))\td{\mS}
 \ \xra{\lambda_\Cc^G\td{\mS}}\ F^G(\gamma(\Cc))\td{\mS}
  \ \xra{\ \psi_{\gamma(\Cc)}^G\ }\  F^G(\gamma(\Cc)\td{\mS}) \ = \ F^G(\bK_{\gl} \Cc)\ .
\]
The second morphism is a special case of the one analyzed
in Theorem \ref{thm:G-fixed Gamma-cat}, for the globally special $\Gamma$-parsummable category $\gamma(\Cc)$.
The morphism $\lambda_\Cc^G\td{\mS}$
arises from a preferred morphism of $\Gamma$-parsummable categories
$\lambda_\Cc^G: \gamma(F^G\Cc)\to F^G(\gamma(\Cc))$ that we define next.

\begin{con}
  We let $\Cc$ be a parsummable category and $G$ a finite group.
  Applying the functor $F^G:\parsumcat\to\parsumcat$
  objectwise to the $\Gamma$-parsummable category $\gamma(\Cc)$
  yields a new $\Gamma$-parsummable category $F^G(\gamma(\Cc))$.
  There is then a unique morphism of $\Gamma$-parsummable categories
  \[ \lambda^G_\Cc \ : \ \gamma(F^G\Cc)\ \to \ F^G(\gamma(\Cc))  \]
  such that $\lambda^G_\Cc(1_+):\gamma(F^G\Cc)(1_+)=F^G\Cc\to F^G\Cc=F^G(\gamma(\Cc))(1_+)$
  is the identity.
  The morphism 
  \[ \lambda^G_\Cc(n_+) \ : \ (F^G\Cc)^{\boxtimes n}\ \to \ F^G(\Cc^{\boxtimes n})  \]
  is the iteration of the morphism \eqref{eq:fix and box};
  equivalently $\lambda^G_\Cc(n_+)$ is the unique morphism such that
  $\lambda^G_\Cc(n_+)\circ i_k=F^G(i_k):F^G\Cc\to F^G(\Cc^{\boxtimes n})$
  for all $1\leq k\leq n$.
\end{con}

\begin{theorem}\label{thm:F^G vs gamma} 
For every parsummable category $\Cc$ and every finite group $G$,
the morphism of symmetric spectra
\[ \lambda^G_{\Cc}\td{\mS}\ : \
  \bK_{\gl}(F^G\Cc) \ = \ \gamma(F^G\Cc)\td{\mS}\ \to \ F^G(\gamma(\Cc))\td{\mS}\]
is a global equivalence.
\end{theorem}
\begin{proof}
  The $\Gamma$-parsummable category $\gamma(F^G\Cc)$
  is globally special by Theorem \ref{thm:special G Gamma}.
  The $\Gamma$-parsummable category $F^G(\gamma(\Cc))$ is globally special by 
  Proposition \ref{prop:global special 2 fixed}.
  The morphism $\lambda^G_{\Cc}(1_+)$ is the identity of
  \[  \gamma(F^G\Cc)(1_+)\ = \ F^G\Cc \ = \ F^G(\gamma(\Cc))(1_+) \ ; \]
  so $\lambda^G_{\Cc}(1_+)$ is in particular a global equivalence.
  Proposition \ref{prop:globally special yields strict cat}
  thus applies to the morphism $\lambda^G_\Cc:\gamma(F^G\Cc)\to F^G(\gamma(\Cc))$,
  and yields the desired result.
\end{proof}

For every parsummable category $\Cc$,
the $\Gamma$-parsummable category $\gamma(\Cc)$
is globally special by Theorem \ref{thm:special G Gamma}. 
So for every finite group $G$,
the morphism of symmetric spectra
$\psi^G_{\gamma(\Cc)}: \bK_{\gl}(F^G\Cc) \ \to \  F^G(\gamma(\Cc)\td{\mS})=F^G(\bK_{\gl}\Cc)$
is a global equivalence by Theorem \ref{thm:G-fixed Gamma-cat}.
In combination with Theorem \ref{thm:F^G vs gamma}, this yields:

\begin{cor}\label{cor:G-fixed} 
For every parsummable category $\Cc$ and every finite group $G$,
the morphism of symmetric spectra
\[ \psi^G_{\gamma(\Cc)}\circ \lambda^G_{\Cc}\td{\mS}\ : \
  \bK_{\gl}(F^G\Cc) \ \to \ F^G(\bK_{\gl}\Cc)\]
is a global equivalence.
\end{cor}

\begin{eg}\label{eg:HA is EM}
In Example~\ref{eg:abelian monoid} we discussed the discrete parsummable category
associated with an abelian group $A$.
The $\Gamma$-category $\gamma(A)$ 
is discrete and coincides with the usual construction of a $\Gamma$-set
from an abelian group, compare~\cite[page 293]{segal:cat coho};
the $\M$-action on $A$ and on $\gamma(A)$ is trivial. So the associated symmetric spectrum
\[ \bK_{\gl}A\ =\ \gamma(A)\td{\mS} \ = \ HA \] 
specializes to the standard construction
of a symmetric Eilenberg-Mac\,Lane spectrum (compare~\cite[Example 1.2.5]{HSS}): 
the $n$-th level is
\[ (\bK_{\gl}A)_n \ = \ A[S^n] \ , \]
the reduced $A$-linearization of the $n$-sphere.

Corollary \ref{cor:G-fixed} lets us identify 
the $G$-fixed point spectrum $F^G(H A)$ as the K-theory spectrum
associated with the parsummable category $F^G A=A[\omega^G]^G$. 
Since the $\M$-action on $A$ is trivial, the $G$-action on $A[\omega^G]$
is trivial as well, so 
\[ F^G A\ = \ A[\omega^G]^G  \ = \ A[\omega^G]\ = \ A \ .  \]
This equality is as parsummable categories, 
so the  $G$-fixed point spectrum $F^G(HA)$ is also an
Eilenberg-Mac\,Lane spectrum for the group $A$, independent of $G$.
We conclude that the global homotopy type of the 
global $\Omega$-spectrum $HA$ is that 
of the Eilenberg-Mac\,Lane spectrum of the constant global functor with value $A$.
\end{eg}

\begin{rk}[(Parsummable categories and symmetric monoidal $G$-categories)]\label{rk:parsumcat and GSymMonCat}
  The above fixed point construction for parsummable categories
  is underlying a slightly more refined structure.
  Indeed, for a parsummable category $\Cc$ and a finite group $G$,
  the category $\Cc[\omega^G]$ obtained by reparameterization
  is naturally a $G$-parsummable category,
  i.e., a parsummable category equipped with a
  strict $G$-action through morphisms of parsummable categories.
  More concretely, the $G$-action and the $\M$-action on $\Cc[\omega^G]$
  commute with each other, the distinguished zero object of $\Cc[\omega^G]$
  is $G$-fixed, and the sum functor
  \[    \Cc[\omega^G]\boxtimes \Cc[\omega^G]\ \xra{\ \epsilon\ } \
    (\Cc\boxtimes\Cc)[\omega^G] \ \xra{\ +[\omega^G]}\ \Cc[\omega^G]\]
  is $G$-equivariant.
  The $G$-fixed subcategory of every $G$-parsummable category
  is naturally a parsummable category, and this procedure
  turns $\Cc[\omega^G]$ into $F^G\Cc=\Cc[\omega^G]^G$.

  Since $G$-parsummable categories are just $G$-objects
  in the category of parsummable categories, every functor
  defined on $\parsumcat$ takes  
  $G$-parsummable categories to $G$-objects in the target category.
  In Proposition \ref{prop:symmon_from_parsumcat} below we discuss a functor
  \[ \varphi^* \ : \ \parsumcat \ \to \ \text{SymMonCat}^{\text{strict}} \]
  that turns parsummable categories into symmetric monoidal categories
  with the same underlying category,
  and morphisms of parsummable categories into strict symmetric monoidal functors.
  So $\varphi^*$ extends to a functor from
  $G$-parsummable categories to symmetric monoidal $G$-categories.
  In particular, this construction enhances the $G$-category $\Cc[\omega^G]$
  to a symmetric monoidal $G$-category $\varphi^*(\Cc[\omega^G])$.  
  Shimakawa \cite{shimakawa} defines a $G$-spectrum $\bK_G\Dc$
  for every symmetric monoidal $G$-category $\Dc$.
  I expect that for every parsummable category $\Cc$,
  the underlying genuine $G$-homotopy type of $\bK_{\gl}\Cc$
  agrees with the $G$-homotopy type obtained by applying Shimakawa's construction
  to the symmetric monoidal $G$-category $\varphi^*(\Cc[\omega^G])$.
  Our Theorem \ref{thm:compare 2K} below verifies that this is indeed the
  case when the group $G$ is trivial, and moreover,
  the two constructions have equivalent fixed point spectra for all subgroups of $G$.
  A strategy to compare $(\bK_{\gl}\Cc)_G$ to $\bK_G(\varphi^*(\Cc[\omega^G]))$
  could be to first compare the special $\Gamma$-$G$-category
  $B(\varphi^*(\Cc[\omega^G]))$ used by Shimakawa in his equivariant K-theory construction
  to the special $\Gamma$-$G$-category $\gamma(\Cc)[\omega^G]$,
  and then exploit our Theorem \ref{thm:K_gl C is global Omega} (ii).
\end{rk}

\section{Parsummable categories versus symmetric monoidal categories}
\label{sec:parsumcat_versus_symmetric}

In this section we show that parsummable categories are very closely related
to symmetric monoidal categories, in a way that links the underlying non-equivariant
homotopy type of our global algebraic K-theory spectrum
to the traditional K-theory spectrum of a symmetric monoidal category.
As we explain in Proposition \ref{prop:symmon_from_parsumcat},
every injection $\varphi:\{1,2\}\times\omega\to\omega$
gives rise to a symmetric monoidal structure on the underlying category
of a parsummable category $\Cc$. While the construction involves a choice,
the resulting symmetric monoidal structure is fairly canonical:
different choices of injections lead to strongly  equivalent structures,
and the choices can be parameterized by a contractible category.

Armed with this construction, we perform an important reality check.
We show in Theorem \ref{thm:compare 2K} that
the symmetric spectrum~$\bK_{\gl}\Cc$ is non-equivariantly stably equivalent
to the K-theory spectrum of the symmetric monoidal category $\varphi^*(\Cc)$.
In other words, the underlying non-equivariant homotopy type of $\bK_{\gl}\Cc$
is `the usual K-theory', i.e., the canonical
infinite delooping of the group completion of the classifying space of~$\varphi^*(\Cc)$.

\begin{con}
  Let $\Cc$ be a parsummable category and let $\varphi:\mathbf m\times\omega\to\omega$
  be an injection, for some $m\geq 0$. We define a functor
  \[ \varphi_*\ : \ \Cc^m \ \to \ \Cc  \]
  on objects and morphisms by
  \begin{equation}\label{eq:define_varphi_*}
    \varphi_*(f_1,\dots,f_m)\ = \ \sum_{i=1}^m \ \varphi^i_*(f_i) \ ,    
  \end{equation}
  where $\varphi^i=\varphi(i,-):\omega\to\omega$.
  This definition makes sense because the injections $\varphi^1,\dots,\varphi^m$
  have pairwise disjoint images, so the morphism $(\varphi^1_*(f_1),\dots,\varphi^m_*(f_m))$
  lies in the full subcategory $\Cc^{\boxtimes m}$ of $\Cc^m$.
  If $\psi:\mathbf m\times\omega\to\omega$ is another injection, then a natural isomorphism
 $[\psi,\varphi] :\varphi_*\Longrightarrow  \psi_*$ was defined in \eqref{eq:define [psi,varphi]}.
\end{con}

The next proposition says that 
the natural transformation $[\psi,\varphi]:\varphi_*\Longrightarrow \psi_*$
is distinguished by enjoying an extra layer of functoriality,
namely by its additional naturality for strict morphisms of parsummable categories.

\begin{defn}
  Let $\varphi,\psi:\mathbf m\times\omega\to\omega$ be injections, for some $m\geq 0$.
  A {\em universally natural transformation} from $\varphi$ to $\psi$
  consists of natural transformations $\alpha^\Cc:\varphi^\Cc_*\to\psi^\Cc_*$
  of functors $\Cc^m\to\Cc$, for every parsummable category $\Cc$, with the following property:
  for every morphism of parsummable categories $F:\Dc\to\Cc$ the relation
  \[  \alpha^\Cc \circ F^m \ = \ F\circ \alpha^\Dc\]
  holds as natural transformations from the functor $\varphi_*^\Cc\circ F^m=F\circ\varphi^\Dc_*$
  to the functor $\psi_*^\Cc\circ F^m=F\circ\psi_*^\Dc$.
\end{defn}

The following proposition will be used several times
to establish equalities of certain natural transformations.

\begin{prop}\label{prop:universal naturality}
  For every $m\geq 0$ and
  every pair of injections $\varphi,\psi:\mathbf m\times\omega\to\omega$,
  there is a unique universally natural transformation
  from $\varphi$ to $\psi$, namely $[\psi,\varphi]$ as defined in \eqref{eq:define [psi,varphi]}.
\end{prop}
\begin{proof}
  We only have to show the uniqueness, for which we use a representability argument.
  We let $A$ be a finite subset of $\omega$, and we write $\Ic_A=E I(A,\omega)$
  for the chaotic category with object set $I(A,\omega)$, the set of injections
  from $A$ to $\omega$; the monoidal category $\M$ acts on $\Ic_A$ by postcomposition.
  Every injection $A\to \omega$ is supported on its image, so the $\M$-category $\Ic_A$ is tame.

  We let $\iota_A:A\to \omega$
  denote the inclusion, which is an object of $\Ic_A$;
  then for every $\M$-category $\Cc$, evaluation at $\iota_A$ is a bijection
  \[ \M\cat(\Ic_A,\Cc) \ \xra{\ \iso \ } \
    \Cc^{\leq A}\ = \ \{x\in \text{ob}(\Cc)\ : \ \supp(x)\subseteq A\}\ .\]
  In other words: the $\M$-category $\Ic_A$ represents the functor of taking
  the set of objects that are supported on $A$. 

  Now we consider a parsummable category $\Cc$ and
  an $m$-tuple $(x_1,\dots,x_m)$ of objects of $\Cc$.
  We choose finite subsets $A_1,\dots,A_m$ of $\omega$ such that $x_i$ is supported on $A_i$.
  The category of tame $\M$-categories has coproducts, given by disjoint unions.
  So there is unique morphism of $\M$-categories 
  \[ \bar x\ : \ \Ic_{A_1}\amalg \dots \amalg \Ic_{A_m}\ \to \ \Cc \]
  such that $\bar x(\iota_j)=x_j$ for all $j=1,\dots,m$, where 
  $\iota_j:A_j\to\omega$ is the inclusion, sitting in the $j$-th summand of the disjoint union.
  We let
  \[ x_\sharp\ : \ \mP[A_1,\dots,A_m]\ = \ \mP(\Ic_{A_1}\amalg \dots \amalg \Ic_{A_m})\ \to \ \Cc \]
  be the unique extension of $\bar x$ to a morphism of parsummable categories,
  where the source is the free parsummable category generated by
  $\Ic_{A_1}\amalg \dots \amalg \Ic_{A_m}$, as discussed in Example \ref{eg:free_parsummable}.

  We claim that there is a unique morphism in the category $\mP[A_1,\dots,A_m]$
  from the object $\varphi_*(\iota_1,\dots,\iota_m)$
  to the object $\psi_*(\iota_1,\dots,\iota_m)$,
  namely the morphism $[\psi,\varphi]^{\iota_1,\dots,\iota_m}$.
  Indeed, the objects $\varphi_*(\iota_1,\dots,\iota_m)$
  and $\psi_*(\iota_1,\dots,\iota_m)$ belong to the homogeneous summand of the free parsummable category
  of degree $m$. Since the box product of tame $\M$-categories distributes
  over disjoint unions, the homogeneous degree $m$ summand also breaks up as a disjoint union,
  and $\varphi_*(\iota_1,\dots,\iota_m)$ and $\psi_*(\iota_1,\dots,\iota_m)$ lie in the summand
  \[ (\Sigma_m\times (\Ic_{A_1}\boxtimes\dots\boxtimes\Ic_{A_m}))/\Sigma_m \ = \
    \Ic_{A_1}\boxtimes\dots\boxtimes\Ic_{A_m}  \ .\]
  This particular summand of $\mP[A_1,\dots,A_m]$ is a chaotic category,
  which proves the claim.
  
  Now we let $\alpha=\{\alpha^\Cc:\varphi_*\Longrightarrow\psi_*\}_\Cc$ be
  a universally natural transformation.
  As explained above, for every $m$-tuple $(x_1,\dots,x_m)$ of objects of
  a parsummable category $\Cc$, there is a morphism of parsummable categories
  $x_\sharp: \mP[A_1,\dots,A_m]\to \Cc$
  for suitable finite subsets $A_1,\dots,A_m$ of $\omega$,
  satisfying $x_\sharp(\iota_j)=x_j$ for all $j=1,\dots,m$.
  By the uniqueness property of the previous paragraph, we must have
  $\alpha^{\mP[A_1,\dots,A_m]}_{\iota_1,\dots,\iota_m}=[\psi,\varphi]^{\iota_1,\dots,\iota_m}$.
  The universal naturality of $\alpha$ implies that
  \[  \alpha^\Cc \circ x_\sharp^m \ = \ x_\sharp\circ \alpha^{\mP[A_1,\dots,A_m]}\ : \
    x_\sharp\circ \varphi_* \ \Longrightarrow \ x_\sharp\circ \psi_*\ : \
    \mP[A_1,\dots,A_m]^m \ \to \ \Cc\ .\] 
    In particular,
  \begin{align*}
    \alpha^\Cc_{x_1,\dots,x_m}\
    &= \  \alpha^\Cc_{x_\sharp(\iota_1),\dots,x_\sharp(\iota_m)}\ = \
      (\alpha^\Cc\circ x_\sharp^m)_{\iota_1,\dots,\iota_m}  \\
    &= \ (x_\sharp\circ\alpha^{\mP[A_1,\dots,A_m]})_{\iota_1,\dots,\iota_m}\
    = \ x_\sharp(\alpha^{\mP[A_1,\dots,A_m]}_{\iota_1,\dots,\iota_m})\\
    &= \ x_\sharp([\psi,\varphi]^{\iota_1,\dots,\iota_m})\
    = \  [\psi,\varphi]^{x_\sharp(\iota_1),\dots,x_\sharp(\iota_m)}\ = \ [\psi,\varphi]^{x_1,\dots,x_m}\ .
  \end{align*}
  This completes the proof that the given universal natural transformation $\alpha$
  coincides with $[\psi,\varphi]$.
\end{proof}

\begin{con}[(From parsummable categories to symmetric monoidal categories)]\label{con:parsum2symmetric}
We construct a symmetric monoidal category from a parsummable category.
The construction depends on a choice of injection $\varphi:\mathbf 2\times\omega\to\omega$;
for different choices of~$\varphi$, these structures are different, but equivalent,
as we explain below.
We introduce associativity, symmetry and unit isomorphisms that enhance the functor
\[ \varphi_* \ : \ \Cc\times \Cc \ \to \ \Cc \]
defined in \eqref{eq:define_varphi_*} to a symmetric monoidal structure
with unit object 0, the distinguished object of $\Cc$.
The associativity isomorphism for three objects~$x,y$ and $z$ of $\Cc$ is given by
\[ \alpha_{x,y,z} \ =\ [ \varphi(1+\varphi),\varphi(\varphi+1)]^{x,y,z}\
   \ : \  \varphi_*(\varphi_*(x,y),z)\ \xra{\ \iso \ } \ \varphi_*(x,\varphi_*(y,z))  \ ,\]
 an instance of the natural isomorphism  \eqref{eq:define [psi,varphi]}.

 We let $t=(1\ 2)\times\omega$ be the involution of $\mathbf 2\times \omega$ 
 that interchanges 1 and 2 in the first factor. 
 The symmetry isomorphism for two objects $x$ and $y$ of $\Cc$ is then defined as
 \[ \tau_{x,y}\ =\ [\varphi t,\varphi]^{x,y}  \ : \ \varphi_*(x,y)\  \xra{\ \iso\ } \
   (\varphi t)_*(x,y)\  = \ \varphi_*(y,x) \ .\]
 Finally, the right unit isomorphism for $x$ is given by
 \[ [\varphi^1,1]^x\ : \ x \ \xra{\ \iso\ } \ \varphi^1_*(x)
   \ = \ \varphi_*(x,0) \ ,\]
where $\varphi^1=\varphi(1,-)$.\end{con}

Proposition \ref{prop:universal naturality} shows that in fact,
$\alpha$ is the only universally natural transformation from
$\varphi(\varphi+1)$ to $\varphi(1+\varphi)$,
that $\tau$ is the only universally natural transformation 
from $\varphi$ to  $\varphi t$,
and that $[\varphi(1,-),1]$ is the only universally natural transformation 
from the identity of $\omega$ to $\varphi(1,-)$.
As we show in the proof of the next proposition,
the uniqueness of universally natural transformations will also take care
of the coherence constraints of a symmetric monoidal structure.

\begin{prop}\label{prop:symmon_from_parsumcat} 
Let $\Cc$ be a parsummable category.
  \begin{enumerate}[\em (i)]
  \item 
    For every injection $\varphi:\mathbf 2\times\omega\to\omega$,
    the functor $\varphi_*$ and the coherence isomorphisms defined above
    make the underlying category of~$\Cc$ into a symmetric monoidal category 
    with unit object~0. We denote this symmetric monoidal category by $\varphi^*(\Cc)$.  
  \item For all injections $\varphi,\psi:\mathbf 2\times\omega\to\omega$, 
    the natural isomorphism $[\psi,\varphi]:\varphi_*\Longrightarrow\psi_*$
    makes the identity functor of~$\Cc$ into a strong symmetric monoidal functor
    from the symmetric monoidal category~$\varphi^*(\Cc)$
    to the symmetric monoidal category~$\psi^*(\Cc)$.
  \end{enumerate}
\end{prop}
\begin{proof}
  (i) We need to verify the commutativity of various coherence diagrams.
  Given four objects~$w,x,y$ and $z$ of $\Cc$, we consider the following pentagon:
  \[
    \xymatrix@C=-8mm{
      && \varphi_*(\varphi_*(\varphi_*(w,x), y),z) \ar[dll]_{\varphi_*(\alpha_{w,x,y},z)}
      \ar[drr]^{\alpha_{\varphi_*(w,x),y,z}}\\
      \varphi_*(\varphi_*(w,\varphi_*(x,y)),z) \ar[dr]_{\alpha_{w,\varphi_*(x,y),z}\ }
      &&&& \varphi_*(\varphi_*(w,x),\varphi_*(y,z))\ar[dl]^{\quad \alpha_{w,x,\varphi_*(y,z)}}\\
      & \varphi_*(w,\varphi_*(\varphi_*(x,y),z)) \ar[rr]_{\varphi_*(w, \alpha_{x,y,z})} &&
      \varphi_*(w,\varphi_*(x,\varphi_*(y,z)))
    }  \]
  As the objects and the parsummable category vary, both composites around the pentagon
  are universally natural transformations from 
  $\varphi(\varphi+1)(\varphi+1+1)$ to $\varphi(1+\varphi)(1+1+\varphi)$;
  so they coincide by Proposition \ref{prop:universal naturality}. 
  
  The symmetry condition is the relation $\tau_{y,x}=\tau_{x,y}^{-1}$
  for all objects $x$ and $y$ of $\Cc$. Again, for varying $x, y$ and $\Cc$,
  these two morphisms are universally natural transformations with the same
  source and target, so they agree by Proposition \ref{prop:universal naturality}. 

  Coherence between associativity and symmetry isomorphisms means that
  the two composites from the top to the bottom of the hexagon
  \[ 
    \xymatrix@C=18mm@R=7mm{
      & \varphi_*(\varphi_*(x,y),z) \ar[dr]^-{\alpha_{x,y,z}}
      \ar[dl]_{\varphi_*(\tau_{x,y},z)}  \\ 
      \varphi_*(\varphi_*(y,x),z) \ar[d]_-{\alpha_{y,x,z}} && 
      \varphi_*(x,\varphi_*(y,z)) \ar[d]^-{\tau_{x,\varphi_*(y,z)}}\\ 
      \varphi_*(y,\varphi_*(x,z)) \ar[dr]_-{\varphi_*(y,\tau_{x,z})}& &
      \varphi_*(\varphi_*(y,z),x) \ar[ld]^{\alpha_{y,z,x}} \\
      & \varphi_*(y,\varphi_*(z,x))}
  \]
  must be equal; the composite once around the entire hexagon is
  a universally natural self-transformation of $\varphi(\varphi+1)$,
  so it must be the identity by Proposition \ref{prop:universal naturality}. 

The coherence condition for the unit morphisms requires the following square
to commute:
\[ \xymatrix@C=25mm{ 
\varphi_*(x,y) \ar[r]^-{\varphi_*( [\varphi^1,1]^x,y)}
\ar[d]_{\varphi_*(x,[\varphi^1,1]^y)} & \varphi_*(\varphi_*(x,0),y)\ar[d]^{\alpha_{x,0,y}} \\
 \varphi_*(x,\varphi_*(y,0))\ar[r]_-{\varphi_*(x,\tau_{y,0})} & \varphi_*(x,\varphi_*(0,y))
} \]
Both composites are universally natural transformations from 
$\varphi$ to  $\varphi(1+\varphi(2,-))$,
so they coincide by Proposition \ref{prop:universal naturality}. 

(ii) We have to show the compatibility of $[\psi,\varphi]$
with the unit, associativity and symmetry isomorphism of the two symmetric
monoidal structures. For the symmetry isomorphisms this means the commutativity
of the diagram
\[ \xymatrix@C=25mm{ 
    \varphi_*(x,y) \ar[d]_{[\psi,\varphi]^{x,y}}
    \ar[r]^-{[\varphi t,\varphi]^{x,y}} & 
(\varphi t)_*(x,y) =  \varphi_*(y,x) \ar[d]^{[\psi,\varphi]^{y,x}} \\
\psi_*(x,y) \ar[r]_-{[\psi t,\psi]^{x,y}} & (\psi t)_*(x,y) = \psi_*(y,x)
} \]
for all objects~$x$ and $y$ of~$\Cc$.
As the objects and the parsummable category vary, both composites around the square
are universally natural transformations from $\varphi$ to $\psi t$;
so they coincide by Proposition \ref{prop:universal naturality}. 
At this point, the reader probably got the message
that all relevant coherence diagrams commute by uniqueness
of universal natural transformations;
this principle also yields the compatibility
of $[\psi,\varphi]$ with the unit and associativity isomorphism,
and we omit the remaining details.
\end{proof}

\begin{rk}\label{rk:constractible choice}
Proposition \ref{prop:symmon_from_parsumcat} 
can be interpreted as saying that while the passage from parsummable categories
to symmetric monoidal categories involves a choice, the choices live in a contractible category.
Let us write SymMonCat$^\text{strong}$ for the category of small symmetric monoidal categories
and strong symmetric monoidal functors. Then Construction \ref{con:parsum2symmetric}
provides a functor
\begin{equation}\label{eq:functorial_varphi}
 E I(\mathbf 2\times\omega,\omega)\times \parsumcat\ \to \ \text{SymMonCat}^{\text{strong}}   
\end{equation}
that sends the object $(\varphi,\Cc)$ on the left
to the symmetric monoidal category $\varphi^*(\Cc)$.
Since the chaotic category $E I(\mathbf 2\times\omega,\omega)$ is a contractible
groupoid, this is almost as a good as an honest functor
from parsummable categories to symmetric monoidal categories.

In one of the two input variables, the construction is in fact a little better:
for every injection $\varphi:\mathbf 2\times\omega\to\omega$ and
every morphism of parsummable categories $\Phi:\Cc\to\Dc$,
the induced functor $\varphi^*(\Phi):\varphi^*(\Cc)\to\varphi^*(\Dc)$
is even {\em strictly} symmetric monoidal (as opposed to strongly monoidal).
On the other hand, one can generalize the notion of morphism of parsummable categories
to a lax version that only commutes with the $\M$-action, the zero object and the
sum functor up to specified and suitably coherent natural isomorphisms.
The functor $\varphi^*$ can then be extended to turn such lax morphisms
into strong symmetric monoidal functors.

Also, the functor \eqref{eq:functorial_varphi} can be slightly extended to a strict 2-functor
between 2-categories, where  $E I(\mathbf 2\times\omega,\omega)$ has only identity 2-morphisms,
$\text{SymMonCat}^{\text{strong}}$ has the monoidal transformations as 2-morphisms,
and the 2-morphisms of parsummable categories are
those natural transformations $\alpha:\Phi\Longrightarrow \Psi$ between
morphisms of parsummable categories that preserve the $\M$-actions,
the zero objects and the sum functors in the following sense:
the equality
\[ \alpha \diamond \text{act}_\Cc = \text{act}_\Dc\circ (\M\times \alpha)
  \ : \ \Phi\circ\text{act}_\Cc =\text{act}_\Dc\circ (\M\times \Phi)
\ \Longrightarrow \ \Psi\circ\text{act}_\Cc =\text{act}_\Dc\circ (\M\times \Psi) \]
holds as natural transformations between functors $\M\times\Cc\to\Dc$;
the morphism $\alpha_0:\Phi(0)\to\Psi(0)$ is the identity of $\Phi(0)=\Psi(0)=0$;
and the relation
\[ \alpha \diamond +_\Cc = +_\Dc\circ (\alpha\times \alpha)
  \ : \ \Phi\circ +_\Cc = +_\Dc\circ (\Phi\times \Phi)
\ \Longrightarrow \ \Psi\circ +_\Cc = +_\Dc\circ (\Psi\times \Psi) \]
holds as natural transformations between functors $\Cc\times\Cc\to\Dc$.
\end{rk}

\begin{con}[(Comparing additions)]
  We now have two essentially different ways to `add' objects and morphisms
  in a parsummable category $\Cc$:
  \begin{itemize}
  \item By definition, a parsummable category comes with a sum operation for
    pairs of disjointly supported objects and morphisms of $\Cc$. This operation is only
    partially defined; but whenever it is defined, the sum operation
    is strictly unital, associative and commutative.
  \item Every injection $\varphi:\mathbf 2\times\omega\to\omega$
    gives rise to a symmetric monoidal structure on the category $\Cc$
    with monoidal product $\varphi_*:\Cc\times\Cc\to\Cc$, compare Proposition \ref{prop:symmon_from_parsumcat}.
    This kind of sum operations is everywhere defined, but is only unital, associative
    and commutative up to specific natural isomorphisms arising from the parsummable structure.
  \end{itemize}
  We will now argue that whenever both kinds of sum operations are defined,
  they are coherently isomorphic in a specific way, also arising from the parsummable structure.
  This fact is crucial for showing that
  the underlying non-equivariant homotopy type of the global K-theory spectrum $\bK_{\gl}\Cc$
  is stably equivalent to the K-theory spectrum
  of the symmetric monoidal category $\varphi^*(\Cc)$,
  compare Theorem \ref{thm:compare 2K} below.
  Given two disjointly supported objects $x$ and $y$ of $\Cc$,
  we define a preferred isomorphism by
  \begin{equation}\label{eq:define_varphi_sharp}
    \varphi^\sharp_{x,y} \ = \ [1,\varphi^1]^x+[1,\varphi^2]^y
    \ : \ \varphi_*(x,y) \ = \  \varphi^1_*(x)+\varphi^2_*(y)\ \to \ x + y\ .
  \end{equation}
\end{con}

\begin{prop}\label{prop:coherence for sharp}
  Let $x, y$ and $z$ be objects of a parsummable category $\Cc$ with pairwise disjoint supports.
  \begin{enumerate}[\em (i)]
  \item The morphism
    \[ \varphi^\sharp_{x,0}\ : \ \varphi^1_*(x)\ = \ \varphi_*(x,0) \ \to \ x+0 \ = \ x\]
    is the isomorphism $[1,\varphi^1]^x$.
  \item The composite
    \[  \varphi_*(y,x)\ \xra{\tau_{y,x}}\  \varphi_*(x,y)\
      \xra{\varphi_{x,y}^\sharp}\  x+y \]
    agrees with $\varphi^\sharp_{y,x}:\varphi_*(y,x)\to y+x$.
  \item 
    The following diagram commutes:
    \[ \xymatrix@C=8mm{ \varphi_*(\varphi_*(x,y),z)
        \ar[rr]^-{\alpha_{x,y,z}} \ar[d]_{\varphi_*(\varphi^\sharp_{x,y},z)} &&
        \varphi_*(x,\varphi_*(y,z)) \ar[d]^{\varphi_*(x,\varphi^\sharp_{y,z})} \\
        \varphi_*(x+y,z) \ar[dr]_-{\varphi^\sharp_{x+y,z}} && \varphi_*(x, y + z) \ar[dl]^{\varphi^\sharp_{x,y+z}}\\
        & x+y+z }
    \]
  \item Let $f:x\to x'$ and $g:y\to y'$ be morphisms such that the supports of
    $x$ and $y$ are disjoint, and the supports of $x'$ and $y'$ are disjoint.
    Then the following square commutes:
    \[ \xymatrix@C=18mm{ \varphi_*(x,y) \ar[r]^-{\varphi_*(f,g)} \ar[d]_{\varphi^\sharp_{x,y}} &
        \varphi_*(x',y') \ar[d]^{\varphi^\sharp_{x',y'}} \\
       x+y \ar[r]_-{f+g} & x'+y' }
    \]
  \end{enumerate}
\end{prop}
\begin{proof}
  We deduce all the coherence properties from the uniqueness of universally natural transformations
  (see Proposition \ref{prop:universal naturality}).
  We spell out the argument in detail in the most complicated case of the associativity
  relation (iii), and we leave the other arguments to the reader.

  The functors and natural transformations occurring in the diagram of part (iii)
  are {\em not} defined on the category $\Cc^3$,
  but rather on its full subcategory $\Cc^{\boxtimes 3}$;
  so we cannot directly apply uniqueness of universally natural transformations.
  The trick is to apply  Proposition \ref{prop:universal naturality}
  to injections that depend on the given triple of disjointly supported objects, as follows.  
  Given objects $x, y$ and $z$ of $\Cc$ with pairwise disjoint supports,
  we choose an injection $\mu:\mathbf 3\times\omega\to\omega$
  such that $\mu^1=\mu(1,-)$ is the identity on $\supp(x)$,
  $\mu^2$ is the identity on $\supp(y)$, and $\mu^3$ is the identity on $\supp(z)$.
  We define $\mu^{12}$ as the restriction of $\mu$ to $\mathbf 2\times\omega$,
  and we set $\mu^{23}(i,j)=\mu(i+1,j)$.
  Then the following diagram of functors $\Cc^3\to \Cc$ and natural transformations
  commutes by uniqueness of universally natural transformations:
    \[ \xymatrix@C=10mm{ (\varphi(\varphi+1))_*
        \ar@{=>}[rr]^-{[\varphi(1+\varphi),\varphi(\varphi+1)]} \ar@{=>}[d]_{[\varphi(\mu^{12}+1),\varphi(\varphi+1)]} &&
        (\varphi(1+\varphi))_* \ar@{=>}[d]^{[\varphi(1+\mu^{23}),\varphi(1+\varphi)]} \\
        (\varphi(\mu^{12}+1))_*\ar@{=>}[dr]_-{ [\mu,\varphi(\mu^{12}+1)]\ } &&
        (\varphi(1+\mu^{23}))_* \ar@{=>}[dl]^{\ [\mu,\varphi(1+\mu^{23})]}\\
        & \mu_*}
    \]
    We claim that evaluating this commutative diagram at the object $(x,y,z)$ of $\Cc^3$
    yields the desired diagram of part (iii).
    Indeed,
  \[
    \mu_*(x,y,z)\ = \ \mu^1_*(x)+\mu^2_*(y)+\mu^3_*(z)\ = \ x+y+z
  \]
  by Proposition \ref{prop:finite support} (ii).
  Similarly, $\mu^{12}_*(x,y)=x+y$ and  $\mu^{23}_*(y,z)=y+z$. Also,
  \[ \varphi^\sharp_{x,y}\ = \ [1,\varphi^1]^x+[1,\varphi^2]^y \ = \
    [\mu^1,\varphi^1]^x+[\mu^2,\varphi^2]^y \ = \ [\mu^{1 2},\varphi]^{x,y}\ ,
  \]
  again by Proposition \ref{prop:finite support} (ii),
  and similarly for the remaining morphisms.
\end{proof}
  
Our next aim is to show that the underlying non-equivariant homotopy type of $\bK_{\gl}\Cc$
agrees with the K-theory of the symmetric monoidal category $\varphi^*(\Cc)$.
We recall the reference model.

\begin{con}[(K-theory of symmetric monoidal categories)]\label{con:Gamma from symmon}
  We let $\Dc$ be a symmetric monoidal category with monoidal product $\oplus$, unit object 0,
  associativity isomorphism $\alpha$, symmetry isomorphism $\tau$, and right unit isomorphism $\rho$.
  We shall now recall how this data gives rise to a $\Gamma$-category $\Gamma(\Dc)$,
  and hence to a K-theory spectrum.
  The construction goes back to Segal who introduces it for symmetric monoidal
  structures given by a categorical coproduct in \cite[page 294]{segal:cat coho};
  the definition for general symmetric monoidal categories appears
  in \cite[Definition 2.1]{shimada-shimakawa:delooping}.
  
  An object of the category  $\Gamma(\Dc)(n_+)$ consists of the following data:
  \begin{itemize}
  \item an object $x_A$ of $\Dc$ for every subset $A$ of $\mathbf n=\{1,\dots,n\}$, and
  \item an isomorphism
    \[ \psi_{A,B} \ : \ x_A\oplus x_B \ \xra{\iso}   x_{A\cup B}\]
    for every pair of disjoint subsets $A$ and $B$ of $\mathbf n$.
  \end{itemize}
  This data is required to satisfy the following coherence conditions:
  \begin{itemize}
  \item the object $x_\emptyset$ is the unit object 0 of the monoidal structure,
    and for every subset $A$ of $\mathbf n$ the morphism
    \[ \psi_{A,\emptyset}\ : \ x_A\oplus 0 \ =\ x_A\oplus x_\emptyset \ \to\    x_A\]
    is the right unit isomorphism $\rho_A$;
  \item for every pair of disjoint subsets $A$ and $B$ of $\mathbf n$, we have
    \[ \psi_{B,A}\ = \ \psi_{A,B}\circ \tau_{x_B,x_A}\ ; \]
  \item for every triple $A, B, C$ of pairwise disjoint subsets of $\mathbf n$,
    the following diagram commutes:
    \[ \xymatrix@C=18mm{ (x_A\oplus x_B)\oplus x_C \ar[r]^-{\alpha_{x_A,x_B,x_C}} \ar[d]_{\psi_{A,B}\oplus x_C} &
        x_A\oplus (x_B\oplus x_C) \ar[r]^-{x_A\oplus\psi_{B,C}} & x_A\oplus x_{B\cup C} \ar[d]^{\psi_{A,B\cup C}}\\
        x_{A\cup B}\oplus x_C \ar[rr]_-{\psi_{A\cup B,C}} && x_{A\cup B\cup C} }
    \]
  \end{itemize}  
  A morphism in $\Gamma(\Dc)(n_+)$ from an object $(x_A,\psi_{A,B})$ to an object $(x'_A,\psi'_{A,B})$
  consists of morphisms $f_A:x_A\to x'_A$ for all subsets $A$ of $\mathbf n$,
  such that $f_\emptyset$ is the identity, and for every pair of disjoint subsets $A$ and $B$,
  the following square commutes:
  \[ \xymatrix@C=18mm{ x_A\oplus x_B \ar[r]^-{f_A\oplus f_B} \ar[d]_{\psi_{A,B}} &
      x'_A \oplus x'_B \ar[d]^{\psi'_{A,B}}\\
      x_{A\cup B} \ar[r]_-{f_{A\cup B}} & x'_{A\cup B} }
  \]
  If $\lambda:m_+\to n_+$ is a based map, then the functor $\Gamma(\Dc)(\lambda):\Gamma(\Dc)(m_+)\to \Gamma(\Dc)(n_+)$
  sends an object  $(x_A,\psi_{A,B})$ to the object $(y_C,\nu_{C,D})$ defined as
  \[ y_C \ = \ x_{\lambda^{-1}(C)}\text{\qquad and\qquad}
    \nu_{C,D} \ = \ \psi_{\lambda^{-1}(C),\lambda^{-1}(D)}\ .    \]
Because of the compatibility requirement in the definition for morphisms in $\Gamma(\Dc)(n_+)$,
all the morphisms $f_A$ are in fact completely determined by
the morphisms $f_{\{i\}}:x_{\{i\}}\to y_{\{i\}}$ for $1\leq i\leq n$, which can be freely chosen.
Hence for every $n\geq 1$, the functor
\[ P_n \ : \ \Gamma(\Dc)(n_+) \ \to \ \Dc^n \]
that forgets all data except the objects $x_{\{i\}}$ and the morphisms $f_{\{i\}}$ is an equivalence of categories.
So the $\Gamma$-category $\Gamma(\Dc)$ is special.
The $\Gamma$-category $\Gamma(\Dc)$ gives rise
to a $\Gamma$-space $|\Gamma(\Dc)|$ by taking nerves and geometric realization.

The {\em K-theory spectrum} of the symmetric monoidal category $\Dc$ 
is the spectrum  $|\Gamma(\Dc)|(\mS)$ obtained by evaluating
the $\Gamma$-space $|\Gamma(\Dc)|$ on spheres.
Equivalently, we can regard $\Gamma(\Dc)$ as a $\Gamma$-$\M$-category
with trivial $\M$-action, and then apply Construction \ref{con:spectrum from Gamma-M}.
\end{con}

A convenient feature of this K-theory construction is its functoriality
for strong symmetric monoidal functors that preserve the unit object
(and not just for {\em strict}
symmetric monoidal functors). Indeed, suppose that $F:\Cc\to\Dc$ is a
strong symmetric monoidal functor
between symmetric monoidal categories, where
$\eta:0\to F(0)$ is the unit isomorphism and
\[ \mu_{x,y}\ : \ F(x)\oplus F(y)\ \to \ F(x\oplus y) \]
is the structure isomorphism that satisfies the symmetry, associativity and
unit constraints of a strong symmetric monoidal functor \cite[Section XI.2]{maclane-working}.
We say that $F$ {\em strictly preserves the zero object}
if  $F(0)=0$ and $\eta$ is the identity.
In this case, an induced morphism of $\Gamma$-categories
\[ \Gamma(F,\mu)\ : \   \Gamma(\Cc)\ \to \ \Gamma(\Dc) \]
is defined as follows. For $n\geq 0$, the functor
\[ \Gamma(F,\mu)(n_+)\ : \ \Gamma(\Cc)(n_+)\ \to \ \Gamma(\Dc)(n_+) \]
sends an object $(x_A,\psi_{A,B})$ 
to the object $(F(x_A), F(\psi_{A,B})\circ\mu_{x_A,x_B})$.
On morphisms, $\Gamma(F,\mu)(n_+)$ is componentwise application of the functor $F$.

If the functor $F$ is an equivalence of underlying categories,
then $\Gamma(F,\mu)(n_+)$ is again an equivalence by specialness.
So in this case the induced maps $|\Gamma(F,\mu)|(S^m):
|\Gamma(\Cc)|(S^m)\to |\Gamma(\Dc)|(S^m)$ are weak equivalences,
and hence $(F,\mu)$ induces a stable equivalence of K-theory spectra.

\begin{con} 
  We let $\Cc$ be a parsummable category.
  In Construction \ref{con:Gamma from I} we defined a $\Gamma$-category $\gamma(\Cc)$
  with values $\gamma(\Cc)(n_+)=\Cc^{\boxtimes n}$,
  the full subcategory of $\Cc^n$ on the $n$-tuples of disjointly supported objects.
  On the other hand, we showed in Proposition \ref{prop:symmon_from_parsumcat}
  that every injection $\varphi:\mathbf 2\times\omega\to\omega$
  gives rise to a symmetric monoidal structure on the category $\Cc$
  with monoidal product $\varphi_*:\Cc\times\Cc\to\Cc$.
  This symmetric monoidal category $\varphi^*(\Cc)$
  gives rise to another $\Gamma$-category $\Gamma(\varphi^*(\Cc))$
  via Segal's Construction \ref{con:Gamma from symmon}.
  We now construct an equivalence of  $\Gamma$-categories
  \[ \Psi\ : \ \gamma(\Cc) \ \xra{\ \simeq\ } \ \Gamma(\varphi^*(\Cc)) \ . \]
  For $n\geq 0$, the functor $\Psi(n_+):\gamma(\Cc)(n_+)\to\Gamma(\varphi^*(\Cc))(n_+)$
  is essentially given by
  summing up objects and morphisms. More precisely, the functor is defined on objects by
  \[ \Psi(n_+)(x_1,\dots,x_n) \ = \ (y_A)_{A\subset \mathbf n} \ ,\]
  where
  \[ y_A \ = \ {\sum}_{i\in A} x_i    \ . \]
  For every pair of disjoint subsets $A$ and $B$, the objects $y_A$ and $y_B$ then have disjoint supports.
  We extend the collection of $y_A$'s to an object of the category
  $\Gamma(\varphi^*(\Cc))(n_+)$ via the isomorphisms
  \[ \psi_{A,B} \ = \ \varphi_{y_A,y_B}^\sharp\ : \varphi_*(y_A,y_B) \ \xra{\ \iso\ } \
    y_A + y_B \ = \ y_{A\cup B}  \]
  defined in \eqref{eq:define_varphi_sharp}.
  Parts (i)--(iii) of Proposition \ref{prop:coherence for sharp} provide 
  the coherence conditions of Construction \ref{con:Gamma from symmon},
  required to make $(y_A,\psi_{A,B})$ an object of the category $\Gamma(\varphi^*(\Cc))(n_+)$.

  On morphisms, the functor $\Psi(n_+)$ is given by summing, i.e.,
  \[ \Psi(n_+)(f_1,\dots,f_n) \ = \ ( {\sum}_{i\in A} f_i)_{A\subset \mathbf n} \ .\]
  Proposition \ref{prop:coherence for sharp} (iv) shows that these maps
  indeed form a morphism in the category $\Gamma(\varphi^*(\Cc))(n_+)$.
  Since the sum operation is a functor $+:\Cc\boxtimes\Cc\to\Cc$,
  addition of disjointly supported morphisms is compatible with
  composition and identities;
  so we have indeed defined a functor $\Psi(n_+):\gamma(\Cc)(n_+)\to\Gamma(\varphi^*(\Cc))(n_+)$. 
\end{con}

\begin{theorem}\label{thm:compare 2K}
  Let $\Cc$ be a parsummable category, and let $\varphi:\mathbf 2\times\omega\to\omega$ be an injection. 
  \begin{enumerate}[\em (i)]
  \item 
    The functors $\Psi(n_+)$ form a morphism of $\Gamma$-categories
    $\Psi: \gamma(\Cc) \to\Gamma(\varphi^*(\Cc))$.
  \item
    The symmetric spectrum~$\bK_{\gl}\Cc$
    is non-equivariantly stably equivalent to the K-theory spectrum
    of the symmetric monoidal category $\varphi^*(\Cc)$.
  \item
    For every finite group $G$, the fixed point spectrum $F^G(\bK_{\gl}\Cc)$
    is non-equivariantly stably equivalent to the K-theory spectrum
    of the symmetric monoidal category $\varphi^*(F^G\Cc)$.
  \end{enumerate}
\end{theorem}
\begin{proof}
  (i) The functoriality of $\Psi$ for based maps $\lambda:m_+\to n_+$ in $\Gamma$
  ultimately boils down to the associativity and commutativity of the sum operation,
  as this provides the relation
  \[ {\sum}_{j\in B}{\sum}_{\lambda(i)=j} f_i \ = \ {\sum}_{i\in\lambda^{-1}(B)} f_i\]
  for every subset $B$ of $\mathbf n$. So $\Psi$ is indeed a morphism of $\Gamma$-categories.
  
  (ii) The functor  $\Psi(1_+):\gamma(\Cc)(1_+)\to\Gamma(\varphi^*(\Cc))(1_+)$ is an isomorphism of categories
  (both sides are isomorphic to the category $\Cc$).
  Since both $\gamma(\Cc)$ and $\Gamma(\varphi^*(\Cc))$
  are special $\Gamma$-categories, the functor
  $\Psi(n_+):\gamma(\Cc)(n_+)\to\Gamma(\varphi^*(\Cc))(n_+)$ is an
  equivalence of categories for every $n\geq 0$.   
  Every equivalence of $\Gamma$-categories becomes a strict equivalence between cofibrant
  $\Gamma$-spaces upon taking nerves and geometric realization.
  This in turn induces homotopy equivalences after evaluation at any based CW-complex,
  see for example \cite[Proposition B.48]{schwede:global}. 
  So the map
  \[ |\Psi|(S^m) \ : \ |\gamma(\Cc)|(S^m) \ \to\ |\Gamma(\varphi^*(\Cc))|(S^m) \]
  is a homotopy equivalence for every $m\geq 0$.
  The spectrum obtained by evaluating the $\Gamma$-space $|\gamma(\Cc)|$ on spheres
  is thus level equivalent, hence stably equivalent, to the K-theory spectrum
  of the symmetric monoidal category $\varphi^*(\Cc)$.
  
  Theorem \ref{thm:K_gl C is global Omega} (ii), applied to the trivial group,
  provides a chain of two non-equivariant stable equivalences
  between the spectrum $\bK_{\gl}\Cc$ and the spectrum obtained
  by evaluating the $\Gamma$-space $|\gamma(\Cc)|$ on spheres.
  
  (iii)
  Corollary \ref{cor:G-fixed} provides a global equivalence
  to the fixed point spectrum $F^G(\bK_{\gl}\Cc)$
  from the symmetric spectrum $\bK_{\gl}(F^G\Cc)$.
  This global equivalence is in particular a non-equivariant stable equivalence.
  Part (ii) then provides a non-equivariant stable
  equivalence between $\bK_{\gl}(F^G \Cc)$ and the K-theory spectrum
  of the symmetric monoidal category $\varphi^*(F^G\Cc)$.
\end{proof}

\section{Equivariant homotopy groups and Swan K-theory}
\label{sec:Swan}

The aim of this section is to provide a highly structured isomorphism between
the 0th equivariant homotopy groups of the global K-theory spectrum $\bK_{\gl}\Cc$
and the combinatorially defined `Swan K-groups' of the parsummable category $\Cc$.
For a finite group $G$, we write $\bK(\Cc,G)$ for the group completion of the abelian monoid $\pi_0(F^G\Cc)$,
with addition induced by the parsummable structure.
For varying finite groups, the groups $\pi_0^G(\bK_{\gl}\Cc) $
and the groups $\bK(\Cc,G)$ are connected by transfer maps and restriction homomorphisms,
and they form a global kind of Mackey functor.
This structure has been given different names in the literature:
it is called an {\em inflation functor} by Webb \cite{webb},
a {\em global $(\emptyset,\infty)$-Mackey functor} 
by  Lewis \cite{lewis-projective not flat},
and a {\em $\Fin$-global functor} by the author \cite{schwede:global}.
Moreover, these objects are special cases of the more general {\em biset functors},
see for example \cite{bouc:foncteurs} or \cite[Section 5]{lewis-projective not flat}.
The main result of this section, Theorem \ref{thm:upi_0 tilde beta},
provides an isomorphism of global functors between
$\upi_0(\bK_{\gl}\Cc)$ and $\bK(\Cc)=\{\bK(\Cc,G)\}_{G\text{ finite}}$.

\begin{rk}
  In the introduction we had announced
that our global K-theory spectrum $\bK_{\gl}\Cc$
of a parsummable category $\Cc$ combines the K-theory spectra
of the categories of $G$-objects in $\Cc$ into one global object. 
The rigorous statement is Corollary \ref{cor:G-fixed},
saying in particular that for every finite group $G$,
the $G$-fixed point spectrum $F^G(\bK_{\gl}\Cc)$
is stably equivalent to the K-theory spectrum of the $G$-fixed
parsummable category $F^G\Cc$.

For varying finite groups $G$, the $G$-fixed categories $F^G\Cc$
can be related by restriction functors along group homomorphisms,
and by `transfer' functors (also called `induction' or `norm' functors)
for subgroup inclusions;
moreover, these functors can be arranged to induce morphisms of K-theory spectra.
We claim that this K-theoretic restriction and transfer information is encoded in
the global homotopy type of $\bK_{\gl}\Cc$.
A full justification of this claim would require us to
\begin{itemize}
\item realize the restriction and transfer information by morphisms
(in the stable homotopy category, at least)
between the fixed point spectra $F^G(\bK_{\gl}\Cc)$,
\item and then prove that the identifications of
$F^G(\bK_{\gl}\Cc)$ with the K-theory of $F^G\Cc$ as non-equivariant stable homotopy types
match up the restriction and transfer data, at least 
as a commutative diagram in the stable homotopy category:
\begin{equation}\begin{aligned}\label{eq:non-existing diagram}
    \xymatrix@C=12mm{
  \bK(F^H\Cc) \ar[r]^-{\bK(\tr_H^G)} \ar[d]^\simeq_{\text{Cor.\,\ref{cor:G-fixed}}}&
  \bK(F^G\Cc) \ar[r]^-{\bK(\alpha^*)}\ar[d]^\simeq_{\text{Cor.\,\ref{cor:G-fixed}}}&
  \bK(F^K\Cc) \ar[d]_\simeq^{\text{Cor.\,\ref{cor:G-fixed}}}  \\
      F^H(\bK_{\gl}\Cc)  \ar[r]_-{\tr_H^G} &
      F^G(\bK_{\gl}\Cc)  \ar[r]_-{\alpha^*} &
      F^K(\bK_{\gl}\Cc)
    }
\end{aligned} \end{equation}
Here $\alpha:K\to G$ is a homomorphism between finite groups,
and $H$ is a subgroup of $G$.
\end{itemize}
All ways of rigorously implementing this that I could think of were either
long and involved, or very technical, or both, and I refrain from
fully justifying this claim.
The complications mostly arise from the transfer information
(and not so much from the restriction maps), because applying K-theory to
the categorical transfer functor $F^H\Cc\to F^G\Cc$ a priori has very little to do
with the Thom-Pontryagin construction arising from an embedding of $G/H$ into a $G$-representation.
Also, a proper comparison should take the concomitant
higher structure into account; to capture this higher structure,
a spectral Mackey functor approach {\`a} la Barwick \cite{barwick:spectral_mackey_I}
is probably the right framework, but this requires very different techniques
from the ones we employ here.

Theorem \ref{thm:upi_0 tilde beta} compares the transfer and restriction
maps on the level of Grothendieck groups to the homotopy theoretic
transfer and restriction homomorphisms on 0-th equivariant homotopy groups.
So informally speaking, it verifies that hypothetical diagram \eqref{eq:non-existing diagram}
in the stable homotopy category (which we do {\em not} construct)
commutes on the level of 0-th stable homotopy groups.
\end{rk}

Now we review the concepts of {\em global functors},
and of {\em pre-global functors},
a slight variation where the values are only required to be abelian monoids
(but not necessarily abelian groups).

\begin{con}[(Effective Burnside category)]
  We define a category $\bA^+$, the {\em effective Burnside category}.
  The objects of $\bA^+$ are all finite groups.
  For two finite groups $G$ and $K$, the morphism set $\bA^+(G,K)$
  is the set of isomorphism classes of finite $K$-$G$-bisets where the right $G$-action is free.
  Composition
  \[ \circ \ : \ \bA^+(K,L)\times \bA^+(G,K)\ \to \ \bA^+(G,L)\]
  in the category $\bA^+$ is defined by balanced product over $K$:
  if $T$ is a finite $K$-free $L$-$K$-biset and
  $S$ is a finite $G$-free $K$-$G$-biset, then
  \[ [T]\circ [S]\ = \ [T\times_K S]\ , \]
  where square brackets denote isomorphism classes.
  This composition is clearly associative, and the
  class of the $G$-$G$-biset $_G G_G$, with $G$ acting by left and right translation,
  is an identity for composition.
\end{con}

\begin{defn}\label{def:pre-global functor}
  A {\em pre-global functor} is a functor $M:\bA^+\to\Ab\M on$
  from the effective Burnside category to the category of abelian monoids
  that satisfies the following additivity condition:
  for all finite groups $G$ and $K$ and all finite $G$-free $K$-$G$-bisets
  $S$ and $T$, the relation
  \[ M[S\amalg T]\ = \ M[S] + M[T] \]
  holds as homomorphisms $M(G)\to M(K)$, where the plus is pointwise addition of homomorphisms.
  A {\em global functor} is a pre-global functor all of whose values are abelian groups.
\end{defn}

For all finite groups $G$ and $K$, the morphism set $\bA^+(G,K)$
is an abelian monoid under disjoint union of $K$-$G$-bisets.
As explained in \cite[Remark 4.2.16]{schwede:global},
the preadditive category $\bA$ obtained by group completing the morphism monoids
is equivalent to the $\Fin$-global Burnside category
in the sense of \cite{schwede:global}, i.e., the full subcategory of
the global Burnside category of \cite[Construction 4.2.1]{schwede:global}
spanned by the finite groups.
By the additivity relation, a global functor in the sense of Definition \ref{def:pre-global functor}
extends uniquely to an additive functor from the 
$\Fin$-global Burnside category to abelian groups.
So the category of global functors in the sense of Definition \ref{def:pre-global functor}
is equivalent to the category of $\Fin$-global functors in the
sense of \cite{schwede:global}.

\begin{rk}[(Restrictions and transfers)]\label{rk:res and tr}
  The data of a pre-global functor can be presented in a different way in terms of restriction
  and transfer homomorphisms. This passage is classical and we will not give full details here.
  We let $M$ be a pre-global functor in the sense of Definition \ref{def:pre-global functor}.
  We let $\alpha:K\to G$ be a homomorphism between finite groups.
  We write $_\alpha G_G$ for the $K$-$G$-biset with underlying set $G$,
  $K$-action by left translation through the homomorphism $\alpha$,
  and right $G$-action by translation.
  The {\em restriction homomorphism} along $\alpha$ is the monoid homomorphism
  \[ \alpha^*\ = \ M[_\alpha G_G] \ : \ M(G)\ \to \ M(K)\ .\]
  An important special case is the inclusion $\iota:H\to G$ of a subgroup
  into the ambient group; the associated restriction homomorphism is traditionally denoted
  \[ \res^G_H \ = \ \iota^*\  : \ M(G)\ \to \ M(H)\ .\]
  If $H$ is a subgroup of a finite group $G$,
  we write $_G G_H$ for the $G$-$H$-biset with underlying set $G$,
  $G$-action by left translation, and $H$-action by right translation.
  The {\em transfer homomorphism} is the monoid homomorphism
  \[ \tr_H^G \ = \ M[_G G_H] \ : \ M(H)\ \to \ M(G)\ .\]
  The transfer and restriction homomorphisms satisfy a couple of relations:
  \begin{itemize}
  \item the restriction homomorphisms are contravariantly functorial;
  \item restriction along an inner automorphism is the identity;
  \item transfers are transitive and $\tr_G^G$ is the identity;
  \item transfer along an inclusion $H\leq G$
    interacts with restriction along an epimorphism $\alpha:K\to G$ according to
    \[ \alpha^*\circ \tr_H^G \ = \
    \tr_L^K \circ (\alpha|_L)^* \ : \ M(H)\ \to \ M(K) \ ,\]
    where $L=\alpha^{-1}(H)$;
  \item for all subgroups $H$ and $K$ of $G$,
    the {\em double coset formula} holds:
    \[  \res^G_K\circ \tr_H^G  \ = \  \sum_{K g H\in K\backslash G/H}
      \ \tr^K_{K\cap ^g H}\circ g_\star \circ \res^H_{K^g\cap H}\ ;
    \]
    here $g$ runs over a set of representatives for the $K$-$H$-double cosets in $G$,
    and $g_\star:M(K^g\cap H)\to M(K\cap{^g H})$
    is the restriction map associated to the conjugation homomorphism
    $K\cap{^g H}\to K^g\cap H$ sending $\gamma$ to $g^{-1}\gamma g$.
  \end{itemize}
  Conversely, the operations $M[S]:M(G)\to M(K)$ can be recovered from the restriction and
  transfer homomorphisms. Indeed, the additivity property means that a pre-global functor
  is determined by the operations $M[S]$ for all {\em transitive} $G$-free $K$-$G$-bisets $S$.
  Every such transitive $K$-$G$-bisets is isomorphic to
  \[ K\times_{L,\alpha} G\ = \ (K\times G) / (k l,g)\sim (k,\alpha(l)g) \]
  for some subgroup $L$ of $K$ and some homomorphism $\alpha:L\to G$.
  Moreover,
  \[  M[K\times_{L,\alpha}G] \ = \ M[_K K_L]\circ M[_\alpha G_G]\ = \ \tr_L^K\circ \alpha^* \ . \]
\end{rk}

\begin{eg}[(Free pre-global functors)]\label{eg:free global functor}
  For a finite group $G$, we introduce a pre-global functor $\bA^+_G$ and a global functor $\bA_G$.  
  The pre-global functor $\bA^+_G$ is simply the functor represented by $G$.
  Indeed, for a finite group $K$, the set $\bA^+_G(K)=\bA^+(G,K)$ is an abelian monoid
  under disjoint union of $K$-$G$-bisets.
  Balanced product preserves disjoint union in both variables, so
  $\bA^+_G=\bA^+(G,-)$ becomes a functor to abelian monoids
  that satisfies the additivity relation.
  The pre-global functor $\bA_G^+$ is `freely generated at $G$' by
  $\Id_G=[_G G_G]$, the class of $_G G_G$ in $\bA^+(G,G)$:
  for every pre-global functor $M$ and every element $x\in M(G)$,
  there is a unique morphism of pre-global functors $x^\sharp:\bA_G^+\to M$
  such that $x^\sharp(G)(\Id_G)=x$.
  
  We define $\bA_G$ as the group completion of the pre-global functor $\bA^+_G$.
  Group completions of pre-global functors are `pointwise', i.e.,
  \begin{itemize}
  \item for every finite group $K$, the abelian group
    $\bA_G(K)$ is the group completion (Grothendieck group) of the abelian monoid $\bA^+_G(K)$, and
  \item the operations $\bA_G[S]$ are defined in the unique way so
  that for varying $K$, the group completion maps $\bA^+_G(K)\to \bA_G(K)$
  form a morphism of pre-global functors $i:\bA^+_G\to\bA_G$.
  \end{itemize}
  This morphism $i:\bA^+_G\to\bA_G$ is then initial among morphism
  of pre-global functors from $\bA^+_G$ to global functors.
  In particular, $\bA_G$ inherits the representability property, but now in the category
  of global functors (as opposed to pre-global functors):
  for every global functor $M$ and every element $y\in M(G)$,
  there is a unique morphism of global functors $y^\sharp:\bA_G\to M$
  such that $y^\sharp(G)(\Id_G)=y$.
\end{eg}

As already indicated,
the goal of this section is to construct an isomorphism of global functors \eqref{eq:tilde beta}
from the `Swan K-theory' global functor $\bK(\Cc)$ of $\Cc$
to the homotopy group global functor $\upi_0(\bK_{\gl}\Cc)$.
Our construction uses an intermediate pre-global functor $\upi_0(\Cc)$
whose values are $\upi_0(\Cc)(G)=\pi_0(F^G \Cc)$;
we then construct a morphism of pre-global functors
\[  \beta\ : \  \upi_0(\Cc)\ \to \ \upi_0(\bK_{\gl}\Cc)  \]
that is objectwise a group completion of abelian monoids,
see Theorem \ref{thm:upi_0 beta}.
This means that when we pass to the group completion of
the pre-global functor $\upi_0(\Cc)$,
the morphism $\beta$ extends to an isomorphism of global functors
\begin{equation}\label{eq:tilde beta}
  \tilde\beta\  :\ \bK(\Cc)\ =\ \left( \upi_0(\Cc) \right)^\star \
  \xra{\ \iso \ }\ \upi_0(\bK_{\gl}\Cc)\ .  
\end{equation}

\begin{con}\label{con:addition on pi_0}
  We let $\Cc$ be a parsummable category.
  The addition functor $+:\Cc\boxtimes\Cc\to\Cc$ gives rise to
  the structure of an abelian monoid on the set $\pi_0(\Cc)$ of components of $\Cc$ 
  as follows. The inclusion $\Cc\boxtimes\Cc\to \Cc\times\Cc$ is an equivalence of categories
  by Theorem \ref{thm:box2times}, so the two projections of $\Cc\boxtimes\Cc$ induce a bijection
  \[  (\pi_0(p_1),\pi_0(p_2)) \ : \ \pi_0(\Cc\boxtimes \Cc)
    \ \xra{\ \iso \ } \ \pi_0(\Cc)\times\pi_0(\Cc) \ .\]
  We obtain a binary operation on the set $\pi_0(\Cc)$ as the composite
  \[
    \pi_0(\Cc)\times\pi_0(\Cc) \ \xra[\iso]{(\pi_0(p_1),\pi_0(p_2))^{-1}}\
    \pi_0(\Cc\boxtimes \Cc) \ \xra{\ \pi_0(+)\ } \ \pi_0(\Cc)      \ .\]
  In concrete terms, this means that we represent two given elements of $\pi_0(\Cc)$
  by disjointly supported objects $x$ and $y$, and then the assignment
  \[  [x]+[y]\ = \ [x+y] \]
  is well-defined.
  The associativity and commutativity of the sum functor on $\Cc$
  ensures that the binary operation $+$ on $\pi_0(\Cc)$ is associative and
  commutative. Moreover, the class of the distinguished object 0 is a neutral element.
  The monoid structure on $\pi_0(\Cc)$ is clearly natural for morphisms of parsummable categories.

  An alternative description of the addition on $\pi_0(\Cc)$ is as follows.
  We can choose any injection $\varphi:\mathbf 2\times\omega\to\omega$,
  and let $\varphi^1,\varphi^2\in M$ be the restrictions to the two summands of $\omega$.
  Then $[\varphi^1,1]^x:x\to \varphi^1_*(x)$ and
  $[\varphi^2,1]^y:y\to \varphi^2_*(y)$ are isomorphisms in $\Cc$,
  and the two objects $\varphi^1_*(x)$ and $\varphi^2_*(y)$ are disjointly supported.
  So for all objects $x$ and $y$, not necessarily disjointly supported, we have
  \[  [x]+[y]\ = \ [\varphi^1_*(x)]+[\varphi^2_*(y)]\ = \
    [\varphi^1_*(x)+\varphi^2_*(y)]\ . \]
\end{con}

\begin{con}[(The pre-global functor of a parsummable category)]
  We let $\Cc$ be a parsummable category.
  We introduce a pre-global functor $\upi_0(\Cc)$ whose value at
  a finite group $G$ is
  \[ \upi_0(\Cc)(G) \ = \ \pi_0( F^G\Cc )\ \ = \ \pi_0( \Cc[\omega^G]^G)\ , \]
  the set of components of the $G$-fixed point category introduced
  in Construction \ref{con:F^G C}.
  Since $F^G\Cc$ inherits a parsummable structure as in Construction \ref{con:F^G C},
  its component set $\pi_0(F^G\Cc)$ inherits an abelian monoid structure
  as explained in Construction \ref{con:addition on pi_0}.
  In more down-to-earth terms, this monoid structure works as follows.
  We choose an injection $\varphi:\mathbf 2\times\omega\to\omega$.
  Then $(\varphi^1)^G,(\varphi^2)^G:\omega^G\to\omega^G$ are
  $G$-equivariant injections with disjoint images, and
  \[ + \ :  \ \pi_0( \Cc[\omega^G]^G)\times \pi_0( \Cc[\omega^G]^G)\ \to \ \pi_0( \Cc[\omega^G]^G) \]
  is given by $[x]+[y]= [(\varphi^1)^G_*(x)+(\varphi^2)^G_*(y)]$.

  We let $K$ and $G$ be two finite groups, and we let $S$ be a finite $K$-$G$-biset
  such that the right $G$-action is free.
  The set $S\times_G\omega^G$ is then a countable $K$-set via the left $K$-action on $S$.
  So we can choose a $K$-equivariant injection $\psi:S\times_G\omega^G\to\omega^K$.
  Because $G$ acts freely on $S$, the map
  \[ [s,-]\ : \ \omega^G\ \to \ S\times_G \omega^G \]
  is injective for every $s\in S$.
  Hence the map 
  \[ \psi^s\ = \ \psi[s,-]\ :\ \omega^G \to \omega^K \]
  is injective, too. These maps satisfy
  \[ \psi^{s g}\ = \ \psi^s\circ l^g \]
  for all $g\in G$. 
  So for every  $G$-fixed object $x\in F^G\Cc$ we have
  \[ \psi^{s g}_*(x)\ = \ \psi^s_*(l^g_*(x)) \ = \ \psi^s_*(x)\ , \]
  i.e., $\psi^s_*(x)$ only depends on the orbit $s G$ (and not on the orbit representative $s$).
  We observe that for every $k\in K$ the relation
  \begin{align*}
    l^k_*\left( {\sum}_{s G\in S/G}\ \psi^s_*(x) \right)\
    &=\   \sum_{s G\in S/G} l^k_*(\psi^s_*(x))\ =\   \sum_{s G\in S/G} \psi^{k s}_*(x)\ =
      \sum_{s G\in S/G} \psi^s_*(x) 
  \end{align*}
  holds, i.e., the $\Cc[\omega^K]$-object $\sum_{s G\in S/G} \psi^s_*(x)$ is $K$-fixed.
  The analogous calculation for morphisms shows that summing $\psi^s_*$ over $S/G$ is a functor
  \[  \sum_{s G\in S/G} \psi^s_*\ : \ F^G\Cc \ \to \ F^K\Cc\ .\]
  We define the operation $\upi_0(\Cc)[S]$ as the effect of this functor on path components:
  \begin{equation}\label{eq:define_td(S)}
    \upi_0(\Cc)[S] \ = \ \pi_0\left( {\sum}_{s G\in S/G}\ \psi^s_* \right)\ : \
    \pi_0(F^G\Cc) \ \to \ \pi_0(F^K\Cc)\ .  
  \end{equation}
  Now we argue that the map $\upi_0(\Cc)[S]$ is independent of the choice
  of $K$-equivariant injection $\psi:S\times_G\omega^G\to\omega^K$,
  and it only depends on the isomorphism class of $S$.
  We let $\beta:S\to \bar S$ be an isomorphism of $K$-$G$-bisets,
  and we let $\bar\psi:\bar S\times_G\omega^G\to\omega^K$ be another $K$-equivariant injection.
  Then for all $k\in K$, all $s\in S$ and every $K$-fixed object $x$, we have
  \[ l^k_*([\bar\psi^{\beta(s)},\psi^s]^x)\ = \   [l^k\bar\psi^{\beta(s)},l^k\psi^s]^x\ = \
    [\bar\psi^{k \beta(s)},\psi^{k s}]^x\ = \ [\bar\psi^{\beta(k s)},\psi^{k s}]^x\ .\]
  Summing over orbit representatives shows that the morphism
  \[ \sum_{s G\in S/G} [\bar\psi^{\beta(s)},\psi^s]^x \ : \
    \sum_{s G\in S/G} \psi^s_*(x)\ \to \  \sum_{\beta(s) G\in \bar S/G} \bar\psi^{\beta(s)}_*(x)
    \ = \  \sum_{\bar s G\in \bar S/G} \bar\psi^{\bar s}_*(x) \]
  is $K$-fixed. So $(S,\psi)$ and $(\bar S,\bar\psi)$ yield the same class in $\pi_0(F^K\Cc)$,
  and the operation $\upi_0(\Cc)[S]$ is independent of the choices.
\end{con}

We will show in Theorem \ref{thm:F^bullet pre-GF} below
that the operations \eqref{eq:define_td(S)} make the collection
of abelian monoids $\pi_0(F^G\Cc)$ into a pre-global functor. The following proposition
facilitates the verification that $\upi_0(\Cc)[S]:\pi_0(F^G\Cc)\to\pi_0(F^K\Cc)$ is additive.
We call a functor $F:\parsumcat\to\Ab\M on$
from the category of parsummable categories to the category of
abelian monoids {\em additive}
if for all parsummable categories $\Cc$ and $\Dc$ the map
\[  (F(p_1),F(p_2)) \ : \  F(\Cc\boxtimes\Dc)\ \to\ F(\Cc)\times F(\Dc) \]
is bijective (and hence an isomorphism of monoids),
where $p_1:\Cc\boxtimes \Dc\to\Cc$ and $p_2:\Cc\boxtimes \Dc\to\Dc$
are the projections to the two factors.

\begin{eg}\label{eg:pi F^G additive}
  We claim that for every finite group $G$, the functor
  \[ \pi_0\circ F^G \ : \ \parsumcat \ \to \ \Ab\M on \]
  is additive. Indeed, for all  parsummable categories $\Cc$ and $\Dc$,
  the inclusion $\Cc\boxtimes\Dc\to \Cc\times\Dc$
  is a global equivalence by Theorem \ref{thm:box2times},
  so it induces an isomorphism
  \[ \pi_0(F^G(\text{incl})) \ : \  \pi_0(F^G(\Cc\boxtimes\Dc))\ \xra{\ \iso \ } \ 
   \pi_0(F^G(\Cc\times\Dc))\ .\]
 The functors $F^G:\parsumcat\to\parsumcat$ and $\pi_0:\parsumcat\to\Ab\M on$
 each preserve products, so the map
  \[ (\pi_0(F^G(p_1)),\pi_0(F^G(p_2)))\ : \ \pi_0(F^G(\Cc\times\Dc))\ \to \ \pi_0(F^G\Cc)\times\pi_0(F^G\Dc)
  \] 
  is bijective. Composing these two bijections shows the claim.
\end{eg}

\begin{prop}\label{prop:additivity prop}
Let
\[ F,G\ :\ \parsumcat\ \to\ \Ab\M on \]
be functors from the category of parsummable categories
to the category of abelian monoids.
Suppose that $F$ and $G$ are reduced, i.e.,
they take every terminal parsummable category to a zero monoid,
and that the functor $G$ is additive.
Then every natural transformation of set-valued functors from $F$ to $G$
is automatically additive.
\end{prop}
\begin{proof}
  We let $\tau:F\to G$ be a natural transformation of set-valued functors.
  We write $i_1,i_2:\Cc\to\Cc\boxtimes\Cc$ for the two embeddings
  given by $i_1=(-,0)$ and $i_2=(0,-)$.
  We consider two classes $x$ and $y$ in $F(\Cc)$; we claim that
  \begin{equation}\label{eq:times y additive}
    \tau_{\Cc\boxtimes\Cc}(F(i_1)(x)+ F(i_2)(y))\  = \  G(i_1)(\tau_\Cc(x)) + G(i_2)(\tau_\Cc(y))
  \end{equation}
  in the abelian monoid $G(\Cc\boxtimes \Cc)$.
  To show this we observe that
  \begin{align*}
    G(p_1) ( \tau_{\Cc\boxtimes \Cc}(F(i_1)(x)+ F(i_2)(y)))\
    &= \ \tau_\Cc( F(p_1) (F(i_1)(x)+ F(i_2)(y)))\\ 
    &= \ \tau_\Cc( F(\Id_\Cc)(x)+ F(0)(y))\ = \ \tau_\Cc(x) \\ 
    &= \ G(\Id_\Cc)(\tau_\Cc(x)) +  G(0)(\tau_\Cc(x))\\
    &= \  G(p_1)( G(i_1)(\tau_\Cc(x)) + G(i_2)(\tau_\Cc(y)))
  \end{align*}
  in $G(\Cc)$. Similarly, 
  \[  G(p_2) ( \tau_{\Cc\boxtimes \Cc}(F(i_1)(x)+ F(i_2)(y)))\ = \ 
    G(p_2)( G(i_1)(\tau_\Cc(x)) + G(i_2)(\tau_\Cc(y)))\ . \]
  Since the morphism $(G(p_1),G(p_2))$ is bijective, this shows the relation \eqref{eq:times y additive}.
  The sum functor $+:\Cc\boxtimes \Cc\to \Cc$ is a morphism of parsummable categories,
  and it satisfies $+\circ i_1=+\circ i_2=\Id_\Cc$. So
  \[ F(+)(F(i_1)(x)+ F(i_2)(y)) \ = \ F(+\circ i_1)(x)+ F(+\circ i_2)(y)\ = \  x + y\ . \]
  So we can finally conclude with the desired relation:
  \begin{align*}
    \tau_\Cc(x+y) \
    &= \ \tau_\Cc( F(+)(F(i_1)(x)+ F(i_2)(y))) \\
    &= \ G(+)(\tau_{\Cc\boxtimes \Cc}(F(i_1)(x)+ F(i_2)(y))) \\
    _\eqref{eq:times y additive}& = \ G(+)( G(i_2)(\tau_\Cc(x)) + G(i_2)(\tau_\Cc(y))) \\
    &= \ G(+\circ i_1)(\tau_\Cc(x)) + G(+\circ i_2)(\tau_\Cc(y)) \ = \ \tau_\Cc(x)  +  \tau_\Cc(y) \ .
  \end{align*}  
\end{proof}

\begin{theorem}\label{thm:F^bullet pre-GF}
 For every parsummable category $\Cc$,
 the operations $\upi_0(\Cc)[S]:\pi_0(F^G\Cc)\to\pi_0(F^K\Cc)$
 defined in \eqref{eq:define_td(S)}
 make the abelian monoids $\pi_0(F^G\Cc)$
 into a pre-global functor $\upi_0(\Cc)$.
\end{theorem}
\begin{proof}
  We simplify the notation by writing $\td{S}$
  for the operation $\upi_0(\Cc)[S]:\pi_0(F^G\Cc)\to\pi_0(F^K\Cc)$.
  Source and target of this operation are reduced additive functors in
  the parsummable category $\Cc$ by Example \ref{eg:pi F^G additive}.
  Since $\td{S}$ is natural for morphisms of parsummable categories in $\Cc$,
  it is additive by Proposition \ref{prop:additivity prop}.

  Now we show the functoriality of the operations $\td{S}$.
  The identity property is straightforward: we use the preferred $G$-equivariant
  bijection $\psi:G\times_G\omega^G\to\omega^G$ defined by $\psi[g,x]=g x$.
  Then $\psi^1:\omega^G\to\omega^G$ is the identity.
  Since the right translation action of $G$ on itself is transitive, we obtain
  \[ \td{_G G_G}[x] \ = \ [\psi^1_*(x)]\ = \ [x]\ . \]
  For compatibility with composition we let $S$ be a finite $G$-free $K$-$G$-biset,
  and we let $T$ be a finite $K$-free $L$-$K$-biset.
  We choose a $K$-equivariant injection $\psi:S\times_G\omega^G\to\omega^K$
  and an $L$-equivariant injection $\tau:T\times_K\omega^K\to\omega^L$.
  Then the composite
  \[ T\times_K S\times_G \omega^G\ \xra{T\times_K \psi }\
    T\times_K \omega^K\ \xra{\ \tau\ }\ \omega^L  \]
  is an $L$-equivariant injection that we can use to calculate the operation $\td{T\times_K S}$.
  We let $A\subset S$ be a set of representatives of the $G$-orbits of $S$,
  and we let $B\subset T$ be a set of representatives of the $K$-orbits of $T$.
  Then $\{ [t,s]\}_{(t,s)\in B\times A}$ is a set of representatives
  of the $G$-orbits of $T\times_K S$.
  So for every object $x$ of $F^G\Cc$ we obtain
  \begin{align*}
    \td{T\times_K S}[x]\
    &= \ \left[ {\sum}_{(t,s)\in B\times A} (\tau\circ(T\times_K \psi))^{[t,s]}_*(x) \right]\\
    &= \ \left[ {\sum}_{t\in B}\ \tau^t_*\left( {\sum}_{s\in A}\psi^s_*(x) \right) \right]  \
    = \ \td{T}(\td{S}[x])   \ .
  \end{align*}
  
  For additivity we consider two finite $G$-free $K$-$G$-sets $S$ and $T$.
  Because $\omega^K$ is a universal $K$-set, we can choose
  $K$-equivariant injections $\psi:S\times_G\omega^G\to\omega^K$ and
  $\tau:T\times_G\omega^G\to\omega^K$ with disjoint images.
  Then the map
  \[ (\psi+\tau)\ : \ (S\amalg T)\times_G\omega^G \ \to \ \omega^K \]
  defined as the `union' of $\psi$ and $\tau$ is again $K$-equivariant and injective.
  Since $(S\amalg T)/G=(S/G)\amalg (T/G)$, we conclude that 
  \begin{align*}
    \td{S\amalg T}[x]\
    &= \ \left[ {\sum}_{u G\in (S\amalg T)/G}\ (\psi+\tau)^u_*(x) \right]\\
    &= \ \left[ {\sum}_{s G\in S/G} \psi^s_*(x) \ + \
      {\sum}_{t G\in T/G} \tau^t_*(x)   \right]\
    =\     \td{S}[x]+\td{T}[x] 
  \end{align*}
  for all objects $x$ of $F^G\Cc$.
  This concludes the proof.
\end{proof}

\begin{defn}
  A morphism $i:M\to N$ of pre-global functors is a {\em group completion}
  if it is initial among morphisms from $M$ to global functors.
\end{defn}

In other words, $i:M\to N$ is a group completion if and only if
\begin{itemize}
\item all the abelian monoids $N(G)$ are groups, and
\item for every morphism $f:M\to R$ of pre-global functors
  such that $R$ is a global functor, there is a unique morphism
  of global functors $f^\flat:N\to R$ such that $f^\flat\circ i=f$.
\end{itemize}
Group completions of pre-global functors are formed `pointwise':
a morphism $i:M\to N$ of pre-global functors is a group completion
if and only if the homomorphism $i(G):M(G)\to N(G)$
is a group completion (Grothendieck construction) of abelian monoids for every finite group $G$.

\begin{defn}
  Let $\Cc$ be a parsummable category. The {\em Swan K-theory global functor}
  $\bK(\Cc)$ is the group completion of the pre-global functor $\upi_0(\Cc)$.
\end{defn}

So $\bK(\Cc,G)$ is the group completion of the abelian monoid $\pi_0(F^G\Cc)$,
for every finite group $G$.
The operations $\bK(\Cc,[S]):\bK(\Cc,G)\to\bK(\Cc,K)$,
the restriction maps and the transfer homomorphisms of $\bK(\Cc)$
are uniquely determined by the requirement that the collection of group completion
maps $\pi_0(F^G\Cc)\to \bK(\Cc,G)$ form a morphism of pre-global functors.

The main objective of this section is to identify
the category-theoretically defined global functor $\bK(\Cc)$
with the homotopy-theoretically defined global functor $\upi_0(\bK_{\gl}\Cc)$;
we achieve this in Theorem \ref{thm:upi_0 tilde beta} below.
To make sense of this, we review how
the equivariant homotopy groups $\pi_0^G(X)$ of a symmetric spectrum $X$
form a global functor as the group $G$ varies.
We will not define the restriction and transfer maps on the
equivariant homotopy group for arbitrary symmetric spectra here;
instead we restrict ourselves to globally semistable symmetric spectra,
a class that includes restricted global $\Omega$-spectra,
and hence the global K-theory spectra $\bK_{\gl}\Cc$ for all parsummable categories.
The definition in full generality can be found in \cite[Section 4]{hausmann:global_finite},
and the fact that the structure forms a global functor
is shown in \cite[Proposition 4.12]{hausmann:global_finite}.

\begin{eg}[(Homotopy group global functor)]\label{eg:homotopy global functor}
  We let $X$ be a globally semistable symmetric spectrum.
  For every finite group $G$ we choose a universal $G$-set $\Uc_G$ and abbreviate
the equivariant homotopy group defined in \eqref{eq:define pi^G} to
\[ \pi_0^G(X) \ = \ \pi_0^{G,\Uc_G}(X)\ .\]
The justification for dropping $\Uc_G$ from the notation is that
for  globally semistable symmetric spectra, $\pi_0^{G,\Uc_G}(X)$ is independent
of the choice of universal $G$-set up to preferred natural isomorphism.
Indeed, if $\bar\Uc_G$ is another universal $G$-set, there are $G$-equivariant injections
$\psi:\Uc_G\to\bar\Uc_G$ and $\varphi:\bar\Uc_G\to\Uc_G$ that induce homomorphisms
\[ \psi_* \ : \ \pi_0^{G,\Uc_G}(X)\ \to \ \pi_0^{G,\bar\Uc_G}(X)
\text{\qquad and\qquad}
\varphi_* \ : \ \pi_0^{G,\bar\Uc_G}(X)\ \to \ \pi_0^{G,\Uc_G}(X)\ , \]
and these assignments are functorial.
Since the underlying $G$-symmetric spectrum of $X$ is $G$-semistable,
every $G$-equivariant self-injection of $\Uc_G$ induces the identity, and similarly for $\bar\Uc_G$.
So $\psi_*$ and $\varphi_*$ are inverse isomorphisms between
$\pi_0^{G,\Uc_G}(X)$ and $\pi_0^{G,\bar\Uc_G}(X)$,
and they are independent of the chosen injections $\psi$ and~$\varphi$.

A homomorphism $\alpha:K\to G$ between finite groups
and an injective $K$-map $\psi:\alpha^*(\Uc_G)\to \Uc_K$
together give rise to a restriction homomorphism
\[ (\alpha,\psi)^* \ : \   \pi_0^{G,\Uc_G}(X)\ \to \ \pi_0^{K,\Uc_K}(X)\ , \]
see \cite[Section 4.4]{hausmann:global_finite}.
We let $f:S^M\to  X(M)$ be a $G$-map that represents an
element of the group $\pi_0^{G,\Uc_G}(X)$, where $M$ is a finite $G$-invariant subset of $\Uc_G$.
Then $\psi(\alpha^*(M))$ is a finite $K$-invariant subset of $\Uc_K$,
the composite
\begin{align*}
  S^{\psi(\alpha^*(M))}\
  &\xra[\iso]{S^{\psi|_{\alpha^*(M)}^{-1}}}\ S^{\alpha^*(M)}= \alpha^*(S^M)\\ 
  &\xra{\ \alpha^*(f)\ }\ \alpha^*(X(M)) = X(\alpha^*(M))\
  \xra[\iso]{X(\psi|_{\alpha^*(M)})}\ X(\psi(\alpha^*(M))) 
\end{align*}
is a $K$-map, and
\[ (\alpha,\psi)^*[f]\ = \ \left[ X(\psi|_{\alpha^*(M)})\circ \alpha^*(f) \circ S^{\psi|_{\alpha^*(M)}^{-1}}\right] \ .\]
Since $X$ is globally semistable, the homomorphism $(\alpha,\psi)^*$
is independent of the injection $\psi$, it only depends on the conjugacy class of $\alpha$
\cite[Lemma 4.8]{hausmann:global_finite},
and it is contravariantly functorial in $\alpha$.
We then simplify the notation and write 
\[  \alpha^*\ :\ \pi_0^G(X)\ \to\ \pi_0^K(X) \]
for the restriction homomorphism.

Given a subgroup $H$ of a finite group $G$, a transfer homomorphism
$\tr_H^G:\pi_0^H(X)\to\pi_0^G(X)$ can be defined in two equivalent ways,
by using an equivariant Thom-Pontryagin construction
as in \cite[Section 4.5]{hausmann:global_finite}, or by exploiting the
{\em Wirthm{\"u}ller isomorphism}
(or rather its incarnation in the context of equivariant symmetric spectra).
For our purposes, the definition
via the Wirthm{\"u}ller isomorphism is more convenient; it says that for
every semistable $G$-symmetric spectrum $X$, the composite
\[ \Wirth_H^G\ : \ \pi_0^G(X\sm G/H_+)\ \xra{\ \res^G_H\ } \
\pi_0^H(X\sm G/H_+)\ \xra{\ (X\sm\Psi)_*\ } \ \pi_0^H(X) \]
is an isomorphism, where $\Psi:G/H_+\to S^0$ is the $H$-equivariant
projection to the distinguished coset, i.e.,
\begin{equation}\label{eq:define Psi}
 \psi(g H) \ = \
\begin{cases}
  0 & \text{ if $g\in H$, and}\\
\infty & \text{ if $g\not\in H$.}
\end{cases}
\end{equation}
The transfer can then be defined as the composite
\[ 
\pi_0^H(X) \ \xra[\iso]{(\Wirth_H^G)^{-1}}\ \pi_0^G(X\sm G/H_+)\ \xra{\ (X\sm p)_*\ } \
\pi_0^G(X) \ ,
\]
where $p:G/H_+\to S^0$ takes $G/H$ to the non-basepoint.
\end{eg}

\begin{con}
  We let $\Cc$ be a parsummable category and $G$ a finite group.
  Every object of the category $\Cc[\omega^G]$
  represents a point in $|\Cc[\omega^G]|$, the geometric realization of the nerve of the
  $G$-category $\Cc[\omega^G]$. If the object is $G$-fixed, so is the corresponding point.
  So stabilizing by $S^G$ and composing with the assembly map yields the $G$-map
  \[ S^G \ \xra{x \sm -} \ |\Cc[\omega^G]|\sm S^G \ \xra{\eqref{eq:assembly}} \ 
    |\gamma(\Cc)[\omega^G]|(S^G)\ =\ (\bK_{\gl}\Cc)(G) \ .\]
  If two $G$-fixed objects in $\Cc[\omega^G]$ are related by a $G$-fixed morphism,
  the corresponding objects can be joined by a path of $G$-fixed points in
  $|\Cc[\omega^G]|$; so the resulting maps from $S^G$ to 
  $(\bK_{\gl}\Cc)(G)$ are $G$-equivariantly homotopic.
  Altogether this construction defines a map
  \begin{equation} \label{eq:F^G_to_K}
    \beta(G)\ : \  \pi_0(F^G\Cc) = \pi_0( \Cc[\omega^G]^G) \ \to \ \pi_0^G( \bK_{\gl}\Cc)\ ,\quad
    [x]\ \longmapsto \ [\text{assembly}\circ(x\sm S^G)]\ .   
  \end{equation}
\end{con}

\begin{theorem}\label{thm:upi_0 beta}
  Let $\Cc$ be a parsummable category.
  For every finite group $G$, the map $\beta(G)$ is a group completion
  of abelian monoids.
  As $G$ varies over all finite groups, the maps $\beta(G)$
  form a morphism of pre-global functors
  $\beta:\upi_0(\Cc)\to\upi_0(\bK_{\gl}\Cc)$.
\end{theorem}
\begin{proof}
  We consider the following diagram:
  \[ \xymatrix@C=20mm{
      \pi_0(F^G\Cc)\ar@{=}[rr]\ar[d]_{\beta(G)} && \pi_0( \Cc[\omega^G]^G)\ar[d]^{\beta(G)} \\
      \pi_0^G(\bK_{\gl}\Cc) \ar[r]^\iso_-{\pi_0^G(a^{\gamma(\Cc)}_G)} &
      \pi_0^G(\gamma(\Cc)\td{\omega^G,\mS}) & 
      \pi_0^G(|\gamma(\Cc)[\omega^G]|(\mS) )\ar[l]_-\iso^-{\pi_0^G(b_G^{\gamma(\Cc)})}  }
  \]
  The lower horizontal maps are the isomorphisms of equivariant homotopy groups
  induced by the $G$-stable equivalences of $G$-symmetric spectra
  $a^{\gamma(\Cc)}_G$ and $b^{\gamma(\Cc)}_G$ discussed in Theorem \ref{thm:K_gl C is global Omega}.
  The diagram commutes because the two $G$-maps
  \[ |\gamma(\Cc)[\omega^G]|(S^G)\ \to \ |\gamma(\Cc)[\omega^G\amalg \omega^G]|(S^G)\]
  arising from the embeddings of the two summands of $\omega^G\amalg \omega^G$
  are equivariantly homotopic by 
  Proposition \ref{prop:equivariant homotopy} (i).
  The right vertical map is a group completion of abelian monoids by 
  Proposition \ref{prop:Gamma compatibilities} (i), applied to
  the $\Gamma$-$G$-space $|\gamma(\Cc)[\omega^G]|$.
  So the left vertical map is also a group completion of abelian monoids.
  
  Now we show that the $\beta$-maps form a morphism of pre-global functors.
  As we explained in Remark \ref{rk:res and tr}, every morphism in
  $\bA^+(G,K)$ is a finite sum of compositions of transfer and restriction homomorphisms.
  So it suffices to show that the $\beta$-maps are compatible with transfers and
  with restriction along group homomorphisms.
  For this purpose it is convenient to introduce a generalization with an extra parameter.
  We let $G$ be a finite group and $A$ a non-empty finite $G$-set.
  Then $\omega^A$ is a countably infinite $G$-set; in this generality,
  $\omega^A$ need not be a universal $G$-set, but that is not relevant for the following construction.
  Every object $x$ of the $G$-category $\Cc[\omega^A]$
  represents a point in $|\Cc[\omega^A]|$, the geometric realization of the nerve of the
  $G$-category $\Cc[\omega^A]$. If the object $x$ is $G$-fixed, so is the corresponding point,
  and we obtain a continuous $G$-map
  \[ S^A \ \xra{x \sm -} \ |\Cc[\omega^A]|\sm S^A \ \xra{\eqref{eq:assembly}} \ 
    |\gamma(\Cc)[\omega^A]|(S^A)\ =\ (\bK_{\gl}\Cc)(A) \ .\]
  As in the special case $A=G$, this construction descends to a well-defined map
  \[ \beta(G;A)\ : \  \pi_0( \Cc[\omega^A]^G) \ \to \ \pi_0^G( \bK_{\gl}\Cc)\ ,\quad
    [x]\ \longmapsto \ [\text{assembly}\circ(x\sm S^A)]\ .    \]
  In the special case when $G$ acts on itself by left translation, the map $\beta(G;G)$
  reduces to the map $\beta(G)$ of \eqref{eq:F^G_to_K}.
  
  The maps $\beta(G;A)$ have the following two properties, both of which
  are straightforward from the definitions:
  \begin{enumerate}[(a)]
  \item Let $i:A\to B$ be an injective $G$-map between non-empty finite $G$-sets.
    Then the composite
    \[ \pi_0( \Cc[\omega^A]^G) \ \xra{\pi_0(\Cc[i_!]^G)} \
      \pi_0( \Cc[\omega^B]^G) \ \xra{\beta(G;B)} \    \pi_0^G( \bK_{\gl}\Cc)\]
    coincides with the map $\beta(G;A)$, where $i_!:\omega^A\to\omega^B$
    is extension by 0 as defined in \eqref{eq:extension_by_zero}.
  \item Let $\alpha:K\to G$ be a homomorphism between finite groups
    and $A$ a non-empty finite $G$-set. Then the following square commutes:
    \[ \xymatrix@C=20mm{
        \pi_0( \Cc[\omega^A]^G) \ar[d]_{\pi_0(\text{incl})} \ar[r]^-{\beta(G;A)} & \pi_0^G( \bK_{\gl}\Cc)\ar[d]^{\alpha^*}\\
        \pi_0( \Cc[\omega^{\alpha^*(A)}]^K) \ar[r]_-{\beta(K;\alpha^*(A))} & \pi_0^K( \bK_{\gl}\Cc)
      } \]
  \end{enumerate}
  Now we can prove the compatibility of the $\beta$-maps
  with restriction along a group homomorphism $\alpha:K\to G$.
  The restriction map $\alpha^*:\pi_0(\Cc[\omega^G]^G)\to\pi_0(\Cc[\omega^K]^K)$
  is based on a choice of $K$-equivariant injection
  $\lambda:\omega^{\alpha^*({G})}=\alpha^*(\omega^G)\to\omega^K$.
  We let $i^1:\alpha^*(G)\to\alpha^*(G)\amalg K$ and
  $i^2:K\to\alpha^*(G)\amalg K$ denote the inclusions of the two summands
  into the $K$-set $\alpha^*(G)\amalg K$.
  We claim that the map $\pi_0(\Cc[i^1_!]^K)$ factors as the composite
  \[ 
    \pi_0( \Cc[\omega^{\alpha^*(G)}]^K) \ \xra{\pi_0(\Cc[\lambda]^K)} \
    \pi_0( \Cc[\omega^K]^K)  \ \xra{\pi_0(\Cc[i^2_!]^K)} \  \pi_0(\Cc[\omega^{\alpha^*(G)\amalg K}]^K) \ .
  \]
  Indeed, if $x$ is a $K$-fixed object of the $K$-category $\Cc[\omega^{\alpha^*(G)}]$,
  then the isomorphism $[i^2_!\circ\lambda,i^1_!]^x:(i^1_*)(x)\to (i^2_!)_*(\lambda_*(x))$
  is $K$-fixed.
  
  The relation $\pi_0(\Cc[i^2_!]^K)\circ \pi_0(\Cc[\lambda]^K)= \pi_0(\Cc[i^1_!]^K) $
  and the properties (a) and (b) of the parameterized $\beta$-maps
  witness that the following diagram commutes:
  \[ \xymatrix@C=25mm{
      \pi_0( \Cc[\omega^G]^G)\ar[dr]^{\text{incl}}   \ar@/^1pc/[drr]^-{\beta(G)=\beta(G;G)}
      \ar[ddd]_{\alpha^*}& &\\
      & \pi_0( \Cc[\omega^{\alpha^*(G)}]^K)\ar@/^1pc/[dr]^(.5){\beta(K;\alpha^*(G))}  \ar[d]^{\pi_0(\Cc[i^1_!]^K)}
      \ar@/_1pc/[ddl]_{\pi_0(\Cc[\lambda]^K)}&
      \pi_0^G(\bK_{\gl}\Cc)\ar[d]^{\alpha^*} \\
      &  \pi_0(\Cc[\omega^{\alpha^*(G)\amalg K}]^K) \ar[r]_(.5){\beta(K;\alpha^*(G)\amalg K)}&
      \pi_0^K(\bK_{\gl}\Cc) \\
      \pi_0(\Cc[\omega^K]^K)  \ar[ur]_{\pi_0(\Cc[i^2_!]^K)} \ar@/_1pc/[urr]_-{\beta(K)=\beta(K;K)}&&
    }  \]
  We conclude that $\beta(K)\circ\alpha^*=\alpha^*\circ\beta(G)$.
  
  Now we show that the maps $\beta(G)$ commute with transfers.
  We reduce this to the more general transfer compatibility property
  of Proposition \ref{prop:Gamma compatibilities},
  applied to the special and $G$-cofibrant $\Gamma$-$G$-space $|\gamma(\Cc)[\omega^G]|$.
  We contemplate the following diagram of abelian monoids:
  \begin{equation}\begin{aligned}\label{eq:transfer diagram}
      \xymatrix@C=18mm{
        \pi_0( \Cc[\omega^H]^H)\ar[r]\ar[d]^{\pi_0(\Cc[i_!]^H)}   \ar@<-4ex>@/_2pc/[dd]_{\tr_H^G}
        \ar@<2ex>@/^3pc/[rrd]^(.2){\beta(H)}&
        \pi_0^H(|\gamma(\Cc)[\omega^H]|(\mS))\ar[d]_{\pi_0^H(|\gamma(\Cc)[i_!]|(\mS))}
        \ar@/^1pc/[dr]_\iso   &\\
        \pi_0( \Cc[\omega^G]^H)\ar[r]\ar[d]^{\tr_H^G} &
        \pi_0^H(|\gamma(\Cc)[\omega^G]|(\mS))\ar[r]_-\iso\ar[d]_{\tr_H^G} &
        \pi_0^H(\bK_{\gl}\Cc)\ar[d]^{\tr_H^G} \\
        \pi_0( \Cc[\omega^G]^G)\ar[r] \ar@/_2pc/[rr]_-{\beta(G)} &
        \pi_0^G(|\gamma(\Cc)[\omega^G]|(\mS))\ar[r]^-\iso &
        \pi_0^G(\bK_{\gl}\Cc)    }
    \end{aligned}\end{equation}
  Here $i_!:\omega^H\to\omega^G$ is the $H$-equivariant injection that extends
  a function by zero on $G\setminus H$. The three left horizontal maps are stabilization maps.
  The lower left transfer map $\tr_H^G:\pi_0(\Cc[\omega^G]^H)\to \pi_0(\Cc[\omega^G]^G)$
  is the transfer map for the special $\Gamma$-$G$-space
  $|\gamma(\Cc)[\omega^G]|$, see Construction \ref{con:addition on Gamma-space}.
  The upper left square in \eqref{eq:transfer diagram}
  commutes by naturality,
  because $|\gamma(\Cc)[i_!]|:|\gamma(\Cc)[\omega^H]|\to|\gamma(\Cc)[\omega^G]|$
  is a morphism of $\Gamma$-$H$-spaces.
  The lower left square in diagram \eqref{eq:transfer diagram} commutes by
  Proposition \ref{prop:Gamma compatibilities} for the $\Gamma$-$G$-space $|\gamma(\Cc)[\omega^G]|$.
  The lower right square commutes by naturality of transfers,
  because $b_G^{\gamma(\Cc)}$ and $a^{\gamma(\Cc)}_G$ are morphisms of $G$-symmetric spectra.
  
  Finally, the upper right triangle in \eqref{eq:transfer diagram}
  is the effect on $H$-equivariant stable homotopy groups
  of the commutative diagram of $H$-symmetric spectra of Theorem \ref{thm:a-b-maps equivalence}:
  \[  \xymatrix@C=20mm{
      |\gamma(\Cc)[\omega^H]|(\mS)\ar[r]_-\simeq^-{b_H^{\gamma(\Cc)}}\ar[d]_{|\gamma(\Cc)[i_!]|(\mS)}&
      \gamma(\Cc)\td{\omega^H,\mS} \ar[d]^{|\gamma(\Cc)\td{i_!,\mS}}
      & (\gamma(\Cc)\td{\mS})_H=(\bK_{\gl}\Cc)_H\ar[l]^-\simeq_{a^{\gamma(\Cc)}_H}\ar@{=}[d]\\
      |\gamma(\Cc)[\omega^G]|(\mS)\ar[r]^-\simeq_-{b_G^{\gamma(\Cc)}} &
      \gamma(\Cc)\td{\omega^G,\mS} & (\gamma(\Cc)\td{\mS})_G \ar[l]_-\simeq^-{a^{\gamma(\Cc)}_G}
    }
  \]
  All horizontal morphisms in this diagram are $H$-stable equivalences of $H$-symmetric spectra by 
  Theorem \ref{thm:a-b-maps equivalence}.
\end{proof}

The  Swan K-theory global functor $\bK(\Cc)$ of a parsummable category $\Cc$
was defined as the group completion of the pre-global functor $\upi_0(\Cc)$.
So the morphism of pre-global functors $\beta:\upi_0(\Cc)\to\upi_0(\bK_{\gl}\Cc)$
discussed in Theorem \ref{thm:upi_0 beta}
extends uniquely to a morphism of global functors $\tilde\beta:\bK(\Cc)\to\upi_0(\bK_{\gl}\Cc)$.
The content of Theorem \ref{thm:upi_0 beta} can then be stated
in an equivalent form as follows:

\begin{theorem}\label{thm:upi_0 tilde beta} 
  For every parsummable category $\Cc$, the unique extension of the morphism
  of pre-global functors $\beta:\upi_0(\Cc)\to\upi_0(\bK_{\gl}\Cc)$
  is an isomorphism of global functors $\tilde\beta:\bK(\Cc)\to\upi_0(\bK_{\gl}\Cc)$.
\end{theorem}

\section{Saturation}\label{sec:saturation}

In this section we study the difference between $G$-fixed objects
and $G$-objects in a parsummable category $\Cc$, where $G$ is a finite group.
As we explained in Proposition \ref{prop:lambda_sharp} above,
the $G$-fixed category $F^G\Cc$ embeds fully faithfully
into the category $G\Cc$ of $G$-objects in $\Cc$.
This embedding is often -- but not always -- essentially surjective,
and hence an equivalence.

We introduce a property of parsummable categories that we call `saturation';
loosely speaking, a parsummable category $\Cc$ is saturated
if it has `enough $G$-fixed objects' for all finite groups $G$.
For saturated parsummable categories, the $G$-fixed point spectrum $F^G(\bK_{\gl}\Cc)$
of the global K-theory of $\Cc$ is equivalent to the K-theory
spectrum of $G$-objects in $\Cc$, see Corollary \ref{cor:saturated implies global}.
Saturation can always be arranged in the following sense:
there is a saturation functor for parsummable categories
and a natural morphism of parsummable categories $s:\Cc\to C^{\sat}$
that is an equivalence of underlying categories, see Theorem \ref{thm:saturation}.
Also by Theorem \ref{thm:saturation}, the saturation construction is idempotent
up to global equivalence, so the morphism $s:\Cc\to C^{\sat}$ is
`globally homotopy initial' among morphisms from $\Cc$ to saturated
parsummable categories.

\begin{con}[(Homotopy fixed category)]
  We let $G$ be a group. We write $E G$ for the translation category,
  i.e., the chaotic category with object set $G$. The group $G$
  acts freely on the category $E G$ by left translation.
  We recall that the {\em homotopy fixed category} of a $G$-category $\Ac$ is 
  \[  \Ac^{h G}\ = \ \cat^G(E G,\Ac)\ ,\]
  the category of $G$-equivariant functors from the translation category $E G$ to $\Ac$.
  The canonical functor 
  \begin{equation}  \label{eq:fix2hfix}
    \kappa\ : \ \Ac^G \ \to \ \Ac^{h G}  
  \end{equation}
  sends a $G$-fixed object $x$ to the constant functor with value $x$.
  We recall that the functor \eqref{eq:fix2hfix} is always fully faithful,
  but not necessarily an equivalence.
  Indeed, the unique functor $p:E G\to\ast$ to the terminal category is $G$-equivariant
  and an equivalence of underlying categories. So the `constant functor'
  \[  \Ac \ \to \ \cat(E G,\Ac)   \]
  is $G$-equivariant and an equivalence of underlying categories.
  The restriction to $G$-fixed categories \eqref{eq:fix2hfix}
  is thus fully faithful.
\end{con}

If $\Cc$ is an $\M$-category and $G$ a finite group, then $\Cc[\omega^G]$ is a $G$-category.
So the previous discussion applies to $\Cc[\omega^G]$. We write 
\[ F^{h G}\Cc \ = \ \Cc[\omega^G]^{h G}\ = \ \cat^G(E G,\Cc[\omega^G])\]
for the {\em homotopy $G$-fixed category} of $\Cc$.

\begin{defn}\label{def:saturated}
  An $\M$-category $\Cc$ is {\em saturated} if for every
  finite group $G$ the functor
  \[      \kappa\ : \ F^G\Cc \ \to \ F^{h G}\Cc  \]
  is an equivalence of categories.
  A parsummable category $\Cc$ is {\em saturated} 
  if the underlying $\M$-category is saturated.
\end{defn}

Corollary \ref{cor:saturation characterizations} below
provides an alternative  characterization of saturation in terms of the comparison functors
$\lambda_\flat:F^G\Cc\to G\Cc$ between $G$-fixed objects and $G$-objects
introduced in \eqref{eq:lambda sharp}.
To establish this characterization, we need another construction.

\begin{con}
  We let $\Cc$ be an $\M$-category and $G$ a finite group.
  We define two functors
  \begin{equation}\label{eq:define_c_sharp}
    c_\sharp \ : \ G\Cc\ \to\  F^{h G}\Cc     \text{\qquad and\qquad}
   \lambda_\sharp \ : \ F^{h G}\Cc \ \to \ G\Cc \ ,
  \end{equation}
  the second one depending on a choice of injection $\lambda:\omega^G\to\omega$.
  Both functors are natural for morphisms of $\M$-categories,
  and both are equivalences of categories, by Proposition \ref{prop:h-fixed versus G-objects} below.
  The functor $\lambda_\sharp$ is an extension from fixed points
  to homotopy fixed points of the functor $\lambda_\flat:F^ G\Cc \to G\Cc$
  defined in \eqref{eq:lambda sharp},
  in the sense that $\lambda_\flat=\lambda_\sharp\circ\kappa$. 
 
  We let $c:\omega\to\omega^G$ denote the `constant function' injection that sends
  $i\in\omega$ to the function defined by $c(i)(g)=i$.
  We will now lift the induced functor $c_*:\Cc\to \Cc[\omega^G]$
  to a functor $c_\sharp: G\Cc\to F^{h G}\Cc$ 
  from the category of $G$-objects in $\Cc$ to the homotopy $G$-fixed point category.
  In a nutshell, the functor $c_\sharp$ is the composite of the preferred isomorphism
  of categories $G\Cc\iso\cat^G(E G,\Cc)$ and the functor induced by
  $c_*:\Cc\to\Cc[\omega^G]$ by applying $\cat^G(E G,-)$.
  We take the time to expand this definition.
  We let $y$ be a $G$-object in $\Cc$, and we write $g_\star:y\to y$
  for the action morphism of an element $g$ of $G$.
  The $G$-equivariant functor $c_\sharp(y):E G\to \Cc[\omega^G]$ is then defined on objects by
  $ c_\sharp(y)(g) = c_*(y)$, and on morphisms by
  \[ c_\sharp(y)(\gamma,g)\ = \ c_*(\gamma_\star^{-1} g_\star)\ :  \ c_*(y)\ \to \ c_*(y)\ .\]
  If $f:x\to y$ is a morphism of $G$-objects in $\Cc$, the
  natural transformation $c_\sharp(f):c_\sharp(x)\Longrightarrow c_\sharp(y)$ is given
  at the object $g\in G$ by $c_\sharp(f)(g)=c_*(f)$.
  
  Because the injection $c:\omega\to\omega^G$ takes values in the constant functions,
  we have $l^k\circ c=c$ for all $k\in G$, where $l^k:\omega^G\to\omega^G$ is the action of $k$.
  So
  \begin{align*}
    l^k_*(c_\sharp(y)(\gamma,g))\
    &= \ l^k_*(c_*(\gamma_\star^{-1} g_\star))\
      = \ c_*(\gamma_\star^{-1} g_\star)\\
    &= \ c_*( (k \gamma)_\star^{-1} (k g)_\star)\
    = \ c_\sharp(y)(k \gamma,k g)\  .  
  \end{align*}
  This shows that the functor $c_\sharp(y):E G\to\Cc[\omega^G]$ is $G$-equivariant.
  The proof that the natural transformation $c_\sharp(f)$ is $G$-equivariant is similar,
  and we omit it. This completes the definition of the functor $c_\sharp$.
  
  Now we define the functor $\lambda_\sharp:F^{h G}\Cc\to G\Cc$, which depends
  on a choice of injection $\lambda:\omega^G\to\omega$.
  We let $x$ be an object of the category $F^{h G}\Cc$, i.e.,
  a $G$-equivariant functor $x:E G\to \Cc[\omega^G]$.
  The functor $\lambda_\sharp$ associates to $x$
  the $\Cc$-object $\lambda_\sharp(x)=\lambda_*(x(1))$,
  endowed with the $G$-action via the composite morphisms
  \[ \lambda_*(x(1)) \ \xra{\ \lambda_*([l^g,1]^{x(1)})\ } \
    \lambda_*(l^g_*(x(1))) \ = \ \lambda_*(x(g))
    \ \xra{\lambda_*(x(1,g))}\ \lambda_*(x(1))\ ,\]
  where $l^g:\omega^G\to\omega^G$ is left translation by $g$.
  If $\gamma\in G$ is another group element, then
  \begin{align*}
    x(1,\gamma)\circ [l^\gamma,1]^{x(1)}\circ x(1,g)\circ [l^g,1]^{x(1)}
    \
    &= \ x(1,\gamma)\circ l^\gamma_*(x(1,g))\circ [l^\gamma,1]^{l^g_*(x(1))} \circ [l^g,1]^{x(1)}\\
    &= \ x(1,\gamma)\circ x(\gamma,\gamma g)\circ [l^\gamma l^g,l^g]^{x(1)} \circ [l^g,1]^{x(1)}\\
    &= \  x(1,\gamma g)\circ[l^{\gamma g},1]^{x(1)}
  \end{align*}
  as endomorphisms of $x(1)$ in the category $\Cc[\omega^G]$.
  Applying the functor $\lambda_*:\Cc[\omega^G]\to\Cc$
  shows that for varying $g$, the morphisms $\lambda_*(x(1,g)\circ[l^g,1]^{x(1)})$
  define a $G$-action on $\lambda_*(x(1))$.
  This completes the definition of the functor $\lambda_\sharp$ on objects.

  The value of the functor $\lambda_\sharp$ on a morphism 
  $\alpha:x\to y$ in $F^{h G}\Cc$ 
  (i.e., a $G$-equivariant natural transformation) is
  \[ \lambda_\sharp(\alpha)\ = \ \lambda_*(\alpha(1))\ : \ \lambda_*(x(1))\ \to \ \lambda_*(y(1))\ . \]
  Then for every $g\in G$,
  \begin{align*}
    y(1,g)\circ [l^g,1]^{y(1)}\circ \alpha(1)\
    &= \   y(1,g)\circ l^g_*(\alpha(1))\circ [l^g,1]^{x(1)}  \\
    &= \   y(1,g)\circ \alpha(g)\circ [l^g,1]^{x(1)}  \
    = \   \alpha(1)\circ x(1,g)\circ [l^g,1]^{x(1)}
  \end{align*}
  as $\Cc[\omega^G]$-morphisms from $x(1)$ to $y(1)$.
  Applying the functor $\lambda_*:\Cc[\omega^G]\to\Cc$
  shows that $\lambda_*(\alpha(1)):\lambda_*(x(1))\to\lambda_*(y(1))$
  is $G$-equivariant for the $G$-actions defined above.
  This completes the definition of the functor $\lambda_\sharp$ on morphisms.
  We omit the routine verification that $\lambda_\sharp$ preserves identities
  and composition, so it is indeed a functor.  
\end{con}

\begin{prop}\label{prop:h-fixed versus G-objects}
  Let $\Cc$ be an $\M$-category and $G$ a finite group.
  \begin{enumerate}[\em (i)]
  \item The functor $c_\sharp:G\Cc\to F^{h G}\Cc$ is an equivalence of categories.
  \item For every injection $\lambda:\omega^G\to\omega$,
    the functor $\lambda_\sharp: F^{h G}\Cc\to G\Cc$ is an equivalence of categories.
  \end{enumerate}
\end{prop}
\begin{proof}
  We prove both statements together.  We show first that the functor $\lambda_\sharp$ is faithful.
  We let $x,y:E G\to\Cc[\omega^G]$ be two $G$-equivariant functors,
  and we let  $\alpha,\beta:x\Longrightarrow y$ be two $G$-equivariant natural transformations
  such that $\lambda_\sharp(\alpha)=\lambda_\sharp(\beta)$.
  This means that $\lambda_*(\alpha(1))= \lambda_*(\beta(1))$,
  by definition of $\lambda_\sharp$.
  The functor $\lambda_*:\Cc[\omega^G]\to\Cc$ is an equivalence of underlying categories,
  so it is in particular faithful. So we deduce that $\alpha(1)=\beta(1)$.
  Because $\alpha$ and $\beta$ are $G$-equivariant, this implies that
  \[ \alpha(g)\ = \ l^g_*(\alpha(1)) \ = \ l^g_*(\beta(1))\ = \ \beta(g) \]
  for every $g\in G$. 
  So $\alpha=\beta$ as natural transformations, and this proves that the functor
  $\lambda_\sharp$ is faithful.

  Now we observe that the composite $\lambda_\sharp\circ c_\sharp$ is equal to
  $G(\lambda c)_*$, the functor induced by $(\lambda c)_*:\Cc\to\Cc$ on $G$-objects.
  The natural transformation $G[\lambda c,1]:\Id_{G\Cc}\Longrightarrow G(\lambda c)_*$
  witnesses that $G(\lambda c)_*=\lambda_\sharp\circ c_\sharp$ is an equivalence of categories.
  Because $\lambda_\sharp\circ c_\sharp$ is fully faithful and
  $\lambda_\sharp$ is faithful, both functors $\lambda_\sharp$ and $c_\sharp$ are in fact
  fully faithful.

  Because $\lambda_\sharp\circ c_\sharp$ is essentially surjective,
  $\lambda_\sharp$  is essentially surjective, and hence an equivalence of categories.
  Because $\lambda_\sharp\circ c_\sharp$ and $\lambda_\sharp$ are
  equivalences of categories, so is $c_\sharp$.
\end{proof}

\begin{cor}\label{cor:saturation characterizations}
  An $\M$-category $\Cc$ is saturated if and only if 
  for every finite group $G$ and every injection $\lambda:\omega^G\to\omega$,
  the functor $\lambda_\flat:F^G\Cc\to G\Cc$
  is an equivalence of categories.
\end{cor}
\begin{proof}
  The functor $\lambda_\sharp:F^{h G}\Cc\to G\Cc$ is an equivalence of categories
  by Proposition \ref{prop:h-fixed versus G-objects} (ii).
  Because $\lambda_\flat\circ\kappa=\lambda_\sharp$,
  we conclude that $\kappa:F^G\Cc\to F^{h G}\Cc$ is an equivalence
  if and only if $\lambda_\flat$ is an equivalence.
  Since this holds for all finite groups, we have proved the claim.
\end{proof}

Many $\M$-categories that we discuss in this paper
are saturated, while some are not.

\begin{eg}\label{eg:F is saturated}
  The $\M$-category $\Fc$ of finite sets was introduced in Example \ref{eg:Fc as M-category}.
  We let $G$ be a finite group.
  Objects of $F^G\Fc=(\Fc[\omega^G])^G$ are finite $G$-invariant subsets of $\omega^G$;
  the comparison functor $\lambda_\flat:F^G\Fc=(\Fc[\omega^G])^G\to G\Fc$
  takes such a $G$-invariant subset $A\subset\omega^G$
  to the image $\lambda(A)\subset \omega$ under the injection $\lambda:\omega^G\to\omega$.
  The $G$-action on $\lambda(A)$ is restricted from $\omega^G$ and conjugated by
  $\lambda|_{A}:A\iso \lambda(A)$; so by design, $\lambda(A)$ is $G$-isomorphic to $A$.
  Because $\omega^G$ is a universal $G$-set (Proposition \ref{prop:universal G-sets}),
  every finite $G$-set is equivariantly isomorphic to a $G$-invariant subset of $\omega^G$.
  So the comparison functor $\lambda_\flat$ is essentially surjective,
  and the $\M$-category $\Fc$ of finite sets is saturated.
\end{eg}

\begin{eg}
  We let $\Cc$ be any category, which we endow with the trivial $\M$-action.
  Then for every finite group $G$, the $G$-fixed category is $\Cc$ itself,
  and the functor $\lambda_\flat:\Cc=F^G\Cc\to G\Cc$ endows a $\Cc$-object
  with the trivial $G$-action. So as soon as there is a non-trivial $G$-action
  on an object of $\Cc$, the functor $\lambda_\flat$ is not an equivalence.
  We conclude that whenever there is a non-trivial action of a finite group
  on a $\Cc$-object, then $\Cc$ with trivial $\M$-action is not saturated.
\end{eg}

\begin{eg}\label{eg:GC inherits saturation}
  We let $\Cc$ be an $\M$-category, and we let $G$ be a group, possibly infinite.
  The category $G\Cc$ of $G$-objects inherits a `pointwise' $\M$-action,
  compare Example \ref{eg:(co)limits Mcat}.
  We claim that the $\M$-category $G\Cc$ is saturated whenever $\Cc$ is.
  To see this, we consider a finite group $K$.
  Then
  \[ F^K(G\Cc)\ = \ ((G\Cc)[\omega^K])^K\ = \ (G(\Cc[\omega^K]))^K\ = \
    G((\Cc[\omega^K]))^K\ = \ G(F^K\Cc)\ .
  \]
  The second equality uses the fact that the $\M$-action on $G\Cc$ is `pointwise',
  with the $G$-action carried along by functoriality.

  Interchanging the orders in which the two groups act provides an isomorphism
  of categories between $K(G\Cc)$ and $G(K\Cc)$. Moreover, for every injection $\lambda:\omega^K\to\omega$,
  the following square of categories and functors commutes:
  \[ \xymatrix@C=12mm{
      F^K(G\Cc)\ar@{=}[d] \ar[r]^-{\lambda_\flat^{G\Cc}} & K(G\Cc)\ar[d]^\iso \\
      G(F^K\Cc)\ar[r]_-{G\lambda_\flat^\Cc} & G(K\Cc)
    } \]
  Since $\Cc$ is saturated, the functor $\lambda_\flat^\Cc:F^K\Cc\to K\Cc$
  is an equivalence of categories by Corollary \ref{cor:saturation characterizations}.
  So the lower horizontal functor $G\lambda_\flat^\Cc$ is an equivalence of categories,
  and hence so is the upper horizontal functor $\lambda_\flat^{G\Cc}$.
  Since this holds for all finite groups $K$,
  Corollary \ref{cor:saturation characterizations} shows that the $\M$-category $G\Cc$
  is saturated.
\end{eg}

\begin{prop}\label{prop:tame2global M-version} 
Let $\Phi:\Cc\to\Dc$ be a morphism of $\M$-categories that is
an equivalence of the underlying categories. Suppose moreover that
the $\M$-category $\Cc$ is saturated.
\begin{enumerate}[\em (i)]
\item The $\M$-category $\Dc$ is saturated as well.
\item The morphism $\Phi$ is a global equivalence of $\M$-categories.
\end{enumerate}
\end{prop}
\begin{proof}
  (i) 
  We consider a finite group $G$. We choose an injection $\lambda:\omega^G\to\omega$
  and contemplate the following commutative square of categories and functors:
  \begin{equation}\begin{aligned}\label{eq:C^G2GC}
      \xymatrix{ 
        F^G\Cc\ar[r]^-{F^G\Phi}\ar[d]_{\lambda_\flat^\Cc} & F^G\Dc\ar[d]^{\lambda_\flat^\Dc}\\
        G\Cc\ar[r]_-{G\Phi} &G\Dc  }   
    \end{aligned}\end{equation}
  The two vertical functors are fully faithful
  by Proposition \ref{prop:lambda_sharp}.
  Since $\Phi$ is an equivalence of underlying categories, so is the induced functor $G\Phi$
  on $G$-objects. The functor $\lambda_\flat^\Cc$ is an equivalence
  by Corollary \ref{cor:saturation characterizations}, because $\Cc$ is saturated.
  So the composite functor $G\Phi\circ\lambda_\flat^\Cc=\lambda_\flat^\Dc\circ F^G\Phi$ is an equivalence.
  Hence the functor $\lambda_\flat^\Dc$ is dense,
  and thus an equivalence. So $\Dc$ is saturated.

  (ii) 
  We let $G$ be a finite group. For every choice of injection $\lambda:\omega^G\to\omega$,
  the vertical functors in the commutative diagram \eqref{eq:C^G2GC}
  are equivalences because $\Cc$  and $\Dc$ are saturated. The lower
  horizontal functor $G\Phi:G\Cc\to G\Dc$ is an equivalence because $\Phi$ is.
  So the functor $F^G\Phi:F^G\Cc\to F^G\Dc$ is also an equivalence of categories,
  and hence in particular a weak equivalence of categories. So $\Phi$ is a global equivalence.
\end{proof}

Now we know that for every saturated parsummable category $\Cc$,
the categories $F^G\Cc$ and $G\Cc$ are equivalent.
Both categories $F^G\Cc$ and $G\Cc$ inherit the structure
of a parsummable category from $\Cc$,
by Construction \ref{con:F^G C} and Example \ref{eg:objects with action}, respectively.
We will now argue that when $\Cc$ is saturated,
$F^G\Cc$ and $G\Cc$ are even globally equivalent as parsummable categories.
Since the K-theory of $F^G\Cc$ calculates the $G$-fixed points
of the global K-theory spectrum $\bK_{\gl}\Cc$ (see Corollary \ref{cor:G-fixed}),
an important consequence is that for all saturated
parsummable categories $\Cc$, the $G$-fixed point spectrum of $\bK_{\gl}\Cc$
is globally equivalent to $\bK_{\gl}(G\Cc)$, the global K-theory of the
parsummable category of $G$-objects in $\Cc$, see Corollary \ref{cor:saturated implies global} below.

We will compare $F^G\Cc$ and $G\Cc$
through the homotopy $G$-fixed point category $F^{h G}\Cc$,
via the functors $\kappa:F^G\Cc\to F^{h G}\Cc$ and
$c_\sharp:G\Cc\to F^{h G}\Cc$, which lift to morphisms of parsummable categories.
We warn the reader that the functors $\lambda_\flat:F^G\Cc\to G\Cc$
and $\lambda_\sharp:F^{h G}\Cc\to G\Cc$
defined in \eqref{eq:lambda sharp}
and \eqref{eq:define_c_sharp} with the help of 
an injection $\lambda:\omega^G\to\omega$ are {\em not}
morphisms of parsummable categories.

\begin{con}\label{con:M-action on F^hG}
  We let $\Cc$ be an $\M$-category and $G$ a finite group.
  We already know that the $\M$-category $G\Cc$ of $G$-objects in $\Cc$
  and the $G$-fixed $\M$-category $F^G\Cc$
  inherit $\M$-actions, by  Example \ref{eg:(co)limits Mcat} and
  Construction \ref{con:F^G for M-cat}, respectively.
  We will now define an $\M$-action on the
  homotopy $G$-fixed $\M$-category $F^{h G}\Cc$
  in such a way that the two functors
  \[
    F^G\Cc\ \xra{\ \kappa\ }\ F^{h G}\Cc \ \xla{\ c_\sharp\ }\ G\Cc
  \]
  defined in \eqref{eq:fix2hfix} and \eqref{eq:define_c_sharp}, respectively, are $\M$-equivariant.
  
  The $\M$-action on $\Cc$ induces an action
  of the monoidal category $\M_G=E I(\omega^G,\omega^G)$ on $\Cc[\omega^G]$.
  We turn this into an $\M$-action through automorphisms of $G$-categories
  by restriction along the monoidal functor
  \[ \M \ \to \ \M_G \ , \quad u \ \longmapsto \ u^G\ .\]
  The functor category $\cat(E G,\Cc[\omega^G])$ then inherits a
  `pointwise' $\M$-action from $\Cc[\omega^G]$
  as described in Construction \ref{eg:objects with action}. 
  This pointwise $\M$-action commutes with the $G$-action
  on $\cat(E G,\Cc[\omega^G])$ by conjugation; so the $\M$-action restricts
  to an action on the fixed point category $F^{h G}\Cc=\cat^G(E G,\Cc[\omega^G])$.
  The functor $\kappa:F^G\Cc\to F^{h G}\Cc$ is clearly
  equivariant for this $\M$-action.
\end{con}

The next proposition is a strengthening of
Proposition \ref{prop:h-fixed versus G-objects} (i),
which says that the functor $c_\sharp:G\Cc\to F^{h G}\Cc$
is an equivalence of underlying categories.

\begin{theorem}\label{thm:saturated implies global}
  Let $\Cc$ be an $\M$-category and $G$ a finite group.
  \begin{enumerate}[\em (i)]
  \item The functor $c_\sharp:G\Cc\to F^{h G}\Cc$
    is a global equivalence of $\M$-categories.
  \item If $\Cc$ is saturated,
    then the $\M$-categories $F^G\Cc$ and $F^{h G}\Cc$ are saturated,
    and the functor $\kappa:F^G\Cc\to F^{h G}\Cc$
    is a global equivalences of $\M$-categories.
  \end{enumerate}
\end{theorem}
\begin{proof}
  (i) We start with parts of the verification
  that the functor $c_\sharp$ is a morphism of $\M$-categories.
  We recall that $c:\omega\to\omega^G$ is the `constant function' injection that sends
  $i\in\omega$ to the function defined by $c(i)(g)=i$; so the relation
  $u^G\circ c=c\circ u$ holds for all $u\in M$.
  We consider a $G$-object $y$ of $\Cc$ with action morphisms $g_\star:y\to y$. Then
  \[ u_*(c_\sharp(y))(g)\ = \ u^G_*( c_\sharp(y)(g))\ = \
    u^G_*( c_*(y))\ = \     c_*(u_*(y))\ = \ c_\sharp(u_*(y))(g)\ ,  \]
  and
  \begin{align*}
    u_*(c_\sharp(y))(\gamma,g)\
    &= \ u^G_*( c_\sharp(y)(\gamma,g))\ = \ u^G_*( c_*(\gamma^{-1}_\star g_\star))\\
    &= \  c_*(u_*(\gamma_\star)^{-1}\circ u_*(g_\star))\ = \ c_\sharp(u_*(y))(\gamma,g)\ .  
  \end{align*}
  This shows that $u_*(c_\sharp(y))=c_\sharp(u_*(y))$ as objects of $F^{h G}\Cc$.
  The verifications that $u_*\circ c_\sharp$ and $c_\sharp\circ u_*$ agree on morphisms of $G\Cc$,
  and that the natural transformations $[v,u]^{c_\sharp(y)}$ and $c_\sharp([u,v]^y)$ coincide
  are similar, and we omit them.
  
  To show that the functor $c_\sharp$ is a global equivalence we let $K$ be another finite group.
  We consider the chain of 
  $(K\times G)$-equivariant isomorphisms of categories:
  \begin{equation} \label{eq:switch_GK}
    \Cc[\omega^G][\omega^K]\ \iso_{\eqref{eq:reparametrize}}
    \ \Cc[(\omega^K)^G] \ \iso_{\text{switch}} \
    \Cc[(\omega^G)^K] \ \iso_{\eqref{eq:reparametrize}} \ \Cc[\omega^G][\omega^K] \ .    
  \end{equation}
  The first and third isomorphisms are instances of \eqref{eq:reparametrize},
  in the last case with the roles of $K$ and $G$ interchanged.
  The second isomorphism is induced by the $(K\times G)$-equivariant bijection
  $(\omega^K)^G\iso (\omega^G)^K$ that interchanges $K$ and $G$ in the argument.
  Unraveling the definition of the intertwining isomorphism and the functoriality of $\Cc[-]$
  reveals that \eqref{eq:switch_GK} is the action of the composite bijection
  \[ \omega\ \xra{\kappa_{\omega^G}}\ \omega^G \xra{(\kappa_{\omega^K})^G}\ (\omega^K)^G\
    \xra{\text{switch}}\ (\omega^G)^K\ \xra{(\kappa_{\omega^G}^{-1})^K} \ \omega^K\
    \xra{\kappa_{\omega^K}^{-1}}\ \omega\ .
  \]
  The isomorphism \eqref{eq:switch_GK} makes the following triangle commute:
  \begin{equation}\label{eq:constant_triangle} 
  \xymatrix@C=15mm{ \Cc[\omega^K]\ar[r]^-{c_*^\Cc[\omega^K]}\ar@/_1pc/[dr]_-{c_*^{\Cc[\omega^K]}} & \Cc[\omega^G][\omega^K]\ar[d]_\iso^{\eqref{eq:switch_GK}} \\
      & \Cc[\omega^K][\omega^G] }
      \end{equation}
  Indeed, unraveling all definitions traces this claim back to the fact that the following
  diagram of injections commutes:
  \[
    \xymatrix@C=15mm{ \omega^K \ar@/_1pc/[dr]_-{ (c_*^\omega)^K} \ar[r]^-{c_*^{\omega^K}} &
      (\omega^K)^G \ar[d]^-{\text{switch}}_-\iso \\
      &      (\omega^G)^K  } \]
  Applying $\cat^{G\times K}(E G,-)$ (with trivial $K$-action on $E G$)
  to the $(K\times G)$-equivariant isomorphism \eqref{eq:switch_GK}
  yields an isomorphism of categories
  \begin{align*}
    F^K(F^{h G}\Cc)\
    &= \  \cat^{K\times G}(E G,\Cc[\omega^G][\omega^K]) \\
    &\iso  \  \cat^{K\times G}(E G,\Cc[\omega^K][\omega^G])\ = \ F^{h G}(F^K\Cc) \ .
  \end{align*}
  The commutativity of the triangle \eqref{eq:constant_triangle} implies that the following
  square of categories and functors commutes:
  \[ \xymatrix@C=15mm{
      F^K(G\Cc)\ar[r]^-{F^K(c_\sharp^\Cc)}\ar@{=}[d] &
      F^K(F^{h G}\Cc)\ar[d]^\iso\\
      G(F^K\Cc)\ar[r]_-{c_\sharp^{F^K\Cc}}&  F^{h G}(F^K\Cc)
    } \]
  The lower horizontal functor $c_\sharp^{F^K\Cc}$ is an equivalence of categories
  by Proposition \ref{prop:h-fixed versus G-objects} (i) for the $\M$-category $F^K\Cc$.
  So the upper horizontal functor $F^K(c_\sharp^\Cc)$ is an equivalence of categories.
  Since $K$ was any finite group, this proves that $c_\sharp^\Cc$ is a global equivalence.

  (ii)  We let $K$ be another finite group.
  We consider the commutative diagram of categories and functors:
  \[ \xymatrix@C=20mm{
      F^K(F^G\Cc)\ar[r]^-{\Im_*^{K\times G}}_-\iso \ar[d]_{\kappa^{F^G\Cc}_K}
      & \Cc[(\omega^K)^G]^{K\times G}
      \ar[dd]^{\kappa_{\Cc[(\omega^K)^G]}}\\
      F^{h K}(F^G\Cc)\ar[d]_{F^{h K}(\kappa^\Cc_G)}  &\\
      F^{h K}(F^{h G}\Cc)\ar[r]_-\iso & (\Cc[(\omega^K)^G])^{h(K\times G)}
    } \]
  The upper horizontal isomorphism is the restriction
  of the $(K\times G)$-equivariant isomorphism $\Im_*:\Cc[\omega^G][\omega^K]\iso \Cc[(\omega^K)^G]$
  from  \eqref{eq:reparametrize} to $(K\times G)$-fixed subcategories.
  The lower horizontal isomorphism is the homotopy fixed analog,
  i.e., the restriction of the $(K\times G)$-equivariant isomorphism
  \begin{align*}
    \cat(E K,\cat(E G,\Cc[\omega^G][\omega^K]))\
    \iso\    &\cat(E (K\times G),\Cc[\omega^G][\omega^K])\\
    \xra{\cat(E(K\times G),\Im_*)} \ &\cat(E (K\times G),\Cc[(\omega^K)^G])  
  \end{align*}
  to $(K\times G)$-fixed categories.
  Because $\Cc$ is saturated, the functors $\kappa^{\Cc}_G$ and  $\kappa_{\Cc[(\omega^K)^G]}$
  are equivalences of categories; for the latter one, this exploits that
  $(\omega^K)^G$ is a universal $(K\times G)$-set.  
  Passage to homotopy fixed categories preserves equivalences,
  see the proof of \cite[Proposition 2.16]{merling};
  so the functor $F^{h K}(\kappa^{\Cc}_G)$ is also an equivalence.
  Hence the functor $\kappa^{F^G\Cc}_K$ is an equivalence for every finite group $K$.
  This shows that the $\M$-category $F^G\Cc$ is saturated.
  
  Since $\Cc$ is saturated, the functor $\kappa:F^G\Cc\to F^{h G}$
  is an equivalence of underlying categories.
  Moreover, $F^G\Cc$ is saturated by the previous paragraph;
  so by Proposition \ref{prop:tame2global M-version},
  $\kappa$ is even a global equivalence, and $F^{h G}\Cc$ is saturated.
\end{proof}

If $\Cc$ is a parsummable category and $G$ a finite group,
then the $\M$-categories $G\Cc$ and $F^G\Cc$ inherit natural parsummable structures,
see Example \ref{eg:objects with action} and Construction \ref{con:F^G C}, respectively.
We will now argue that the homotopy $G$-fixed category $F^{h G}\Cc$
also inherits a parsummable structure,
in such a way that the functors $\kappa:F^G\Cc\to F^{h G}\Cc$
and $c_\sharp:G\Cc\to F^{h G}\Cc$ are morphisms of parsummable categories.

\begin{con}\label{con:I structure on F^hG C}
  Given a parsummable category $\Cc$ and a finite group $G$,
  we define a parsummable structure
  on the homotopy fixed category $F^{h G}\Cc=\cat^G(E G,\Cc[\omega^G])$. 
  We start by considering an $\M$-category $\Cc$.
  Then an $\M$-action on $F^{h G}\Cc$ was already defined in Construction \ref{con:M-action on F^hG}.
  A similar argument as for $F^G\Cc$ in Proposition \ref{prop:F^G_preserves_tame}
  shows that the $\M$-action on $F^{h G}\Cc$ is tame whenever the
  original $\M$-action on $\Cc$ is.

  We define a lax symmetric monoidal transformation
  \begin{equation}\label{eq:hG monoidal}
    (F^{h G}\Cc)\boxtimes (F^{h G}\Dc)\ \to \ F^{h G}(\Cc\boxtimes \Dc) \ ,
  \end{equation}
  where $\Cc$ and $\Dc$ are tame $\M$-categories.
  The definition uses the $G$-equivariant, fully faithful and lax symmetric monoidal embedding
  $\epsilon:\Cc[\omega^G]\boxtimes \Dc[\omega^G]\to(\Cc\boxtimes\Dc)[\omega^G]$ 
  discussed in \eqref{eq:fix and box before G-fix}.
  The $\M$-equivariant functor \eqref{eq:hG monoidal} is then defined as the following composite:
  \begin{align*}
    (F^{h G}\Cc)\boxtimes (F^{hG}\Dc)\
    = \ &\left( \cat(E G,\Cc[\omega^G])\boxtimes \cat(E G,\Dc[\omega^G])\right)^G\\
    \xra{\text{product}} \  &\cat^G(E G,\Cc[\omega^G]\boxtimes \Dc[\omega^G])\\
    \xra{\cat^G(E G,\epsilon)} \
      &\cat^G(E G,(\Cc\boxtimes\Dc)[\omega^G])\ = \ F^{h G}(\Cc\boxtimes \Dc) 
  \end{align*}
  Now we let $\Cc$ be a parsummable category; then the tame $\M$-category $F^{h G}\Cc$
  becomes a parsummable category by
  endowing it with the structure morphism 
  \[  (F^{h G}\Cc)\boxtimes (F^{h G}\Cc)\ \xra{\eqref{eq:hG monoidal}} \ F^{h G}(\Cc\boxtimes\Cc) \
    \xra{F^{h G}(+)}\  F^{h G}(\Cc) \ .\]
\end{con}

\begin{cor}\label{cor:saturated implies global}
  Let $\Cc$ be a saturated parsummable category and $G$ a finite group.
  \begin{enumerate}[\em (i)]
  \item 
    The morphisms
    \[ \bK_{\gl}(G\Cc)\ \xra{\bK_{\gl}(c_\sharp)}\ \bK_{\gl}(F^{h G}\Cc) \
      \xla{\bK_{\gl}(\kappa)}\ \bK_{\gl}(F^G\Cc)\
      \xra{\psi^G_{\gamma(\Cc)}\circ \lambda^G_{\Cc}\td{\mS}}\ F^G(\bK_{\gl}\Cc)
    \]
    are global equivalences of symmetric spectra.
  \item Let $\varphi:\mathbf 2\times\omega\to\omega$ be an injection.
  Then the fixed point spectrum $F^G(\bK_{\gl}\Cc)$
  is non-equivariantly stably equivalent to the K-theory spectrum
  of the symmetric monoidal category of $G$-objects in $\varphi^*(\Cc)$.
  \end{enumerate}\end{cor}
\begin{proof}
  (i) The morphism of symmetric spectra
  $\psi^G_{\gamma(\Cc)}\circ \lambda^G_{\Cc}\td{\mS}:\bK_{\gl}(F^G\Cc) \to F^G(\bK_{\gl}\Cc)$
  is a global equivalence by Corollary \ref{cor:G-fixed}.
  The morphisms of parsummable categories $c_\sharp:G\Cc\to F^{h G}\Cc$ and
  $\kappa:F^G\Cc\to F^{h G}\Cc$ are global equivalences
  by Theorem \ref{thm:saturated implies global}.
  So the morphisms $\bK_{\gl}(c_\sharp)$ and $\bK_{\gl}(\kappa)$
  are global equivalences of symmetric spectra
  by Theorem \ref{thm:equivalence2equivalence}.   

  (ii)
  Theorem \ref{thm:compare 2K} (ii) provides a non-equivariant stable
  equivalence between $\bK_{\gl}(G\Cc)$ and the K-theory spectrum
  of the symmetric monoidal category $\varphi^*(G\Cc)=G(\varphi^*\Cc)$.
  In combination with part (i), this proves the claim.
\end{proof}

Now we discuss a saturation construction for $\M$-categories
and parsummable categories that is due to Tobias Lenz.
For a tame $\M$-category $\Cc$, it provides a morphism of $\M$-categories 
\[ s\ : \ \Cc \ \to \ \Cc^{\sat} \]
that is an equivalence of underlying categories and whose target
is a tame saturated $\M$-category.
In Construction \ref{con:saturation_parsummable} we will lift the procedure from $\M$-categories
to the more highly structured parsummable categories by endowing the saturation
with a lax symmetric monoidal structure with respect to the box product.

\begin{con}[(Saturation)]
  We let $\Cc$ be an $\M$-category. We denote by $\cat(\M,\Cc)$
  the category of functors from $\M$ to $\Cc$, with natural transformations
  as morphisms.
  The right translation $\M$-action on itself and the given left action of $\M$ on
  $\Cc$ induce an $(\M\times\M)$-action on the functor category $\cat(\M,\Cc)$.
  We endow $\cat(\M,\Cc)$ with the diagonal $\M$-action.
  More concretely, the action functor
  \[  \diamond  \ : \     \M\times \cat(\M,\Cc) \ \to \ \cat(\M,\Cc) \]
  is thus given on objects by
  \[ (u\diamond X)(w)\ = \ (u_*(X))(w)\ = \ u_*(X (w u))  \ , \]
  where $u,w\in M$ are objects of $\M$, and $X:\M\to \Cc$ is a functor.
  This construction has three levels of functoriality:
  \begin{itemize}
  \item a morphism $(w,v):v\to w$ in $\M$ is taken to  
    \[ (u_*(X))(w,v)\ = \ u_*(X(w u,v u))\ ;  \]
  \item a morphism in $\cat(\M,\Cc)$ 
    (i.e., a natural transformation $\Psi:X\to Y$)
    is taken to  
    \[ (u_*(\Psi))(w)\ = \ u_*(\Psi(w u))\ ;  \]
  \item and a morphism $(u',u):u\to u'$ in $\M$ is taken to  
    \begin{equation}\label{eq:[u',u]_at_w}
    [u',u]^X(w)\
      = \ u'_*(X(w u',w u))\circ[u',u]^{X(w u)}\
      = \ [u',u]^{X(w u')}\circ  u_*(X(w u',w u)) \ .        
    \end{equation}
  \end{itemize}
  We emphasize that the present $\M$-action on $\cat(\M,\Cc)$ is {\em not}
  a specialization of the `pointwise' $\M$-action on a functor category
  discussed in Example \ref{eg:(co)limits Mcat}.
\end{con}

While $\cat(\M,\Cc)$ is an $\M$-category, it will {\em not} be tame,
even if $\Cc$ is tame, except in degenerate cases.
The following proposition gives us access to the support
for objects in the functor $\M$-category $\cat(\M,\Cc)$.
An equivalent way to reformulate the content of the next proposition 
is as follows. A functor $X:\M\to\Cc$ is supported on $A\subset\omega$
if and only if it factors as a composite
\[ \M \ \xra{E\res^\omega_A}\ E I(A,\omega)\ \to \ \Cc^{\supp\subseteq A}\ \xra{\text{incl}}  \ \Cc\ .\]
Here $\res^\omega_A:M=I(\omega,\omega)\to I(A,\omega)$
is restriction of an injection from $\omega$ to the subset $A$.

\begin{prop}\label{prop:support in cat(M,C)}
  Let $\Cc$ be a tame $\M$-category and $A$ a finite subset of $\omega$.
  Then a functor $X:\M\to\Cc$ is supported on $A$ if and only if the following conditions hold:
  \begin{enumerate}[\em (a)]
  \item for every $v\in M$, the $\Cc$-object $X(v)$ is supported on $A$, and
  \item if $v,w\in\M$ agree on the set $A$, then $X(v)=X(w)$, and $X(w,v)=\Id_{X(v)}=\Id_{X(w)}$.
  \end{enumerate}
\end{prop}
\begin{proof}
  We start with a functor $X:\M\to\Cc$ that satisfies conditions (a) and (b).
  We let $u\in M$ be an injection that is the identity on the set $A$.
  Then for every injection $v\in M$,
  we have $(u_*(X))(v)=u_*(X(v u))=X(v u)$,
  because $X(v u)$ is supported on $A$, by hypothesis (a).
  Also, $v u$ and $v$ agree on $A$, so $X(v u)=X(v)$ and $X(v u, v)=\Id_{X(v)}$, by hypothesis (b).
  In particular, $(u_*(X))(v)=X(v)$, i.e., the functors $u_*(X)$ and $X$ agree on objects.
  To see that  $u_*(X)$ and $X$ agree on morphisms, we consider another injection $w\in M$.
  Then
  \begin{align*}
    (u_*(X))(v,w)\
    &= \ u_*(X(v u,w u))\ = \ X(v u,w u) \\
    &= \ X(v u,v)\circ X(v,w)\circ X(w,w u)\  = \       X(v,w) \ .
  \end{align*}
  The second equation is an application of Proposition \ref{prop:finite support} (iv),
  exploiting that $X(v u)$ and $X(w u)$ are supported on $A$.
  Altogether, this proves that $X$ is supported on the set $A$.

  Now we suppose that $X:\M\to\Cc$ is supported on the set $A$,
  and we establish conditions (a) and (b).
  We consider an injection $u\in M$ that is the identity on $A$.
  Then $u_*(X)=X$, and so the relation
  \[   X(v) \ = \ (u_*(X))(v)\ = \ u_*(X(v u))\]
  holds for all $v\in M$. Because $\Cc$ is tame, the object $X(v u)$
  is finitely supported, and hence
  \[ \supp(X(v))\ =\ \supp(u_*(X(v u)))\ =\  u(\supp(X(v u))) \]
  by Proposition \ref{prop:finite support} (iii). 
  Now suppose, by contradiction, that $\supp(X(v))$ were not contained in $A$.
  Then we could choose an element $m\in\supp(X(v))\setminus A$
  and an injection $u\in M$ that is the identity on $A$, and such that $m$
  is not in the image of $u$.
  This contradicts the relation $\supp(X(v))= u(\supp(X(v u)))$.
  So we conclude that $\supp(X(v))\subset A$, i.e., condition (a) holds.

  Now we suppose that $v,w\in M$ are two injections that agree on $A$;
  we let $u\in M$ be any injection such that
  $u v$ and $u w$ are the identity on $A$.
  Because $X$ is supported on $A$, we then have $v_*(X)=w_*(X)$.
  Because $X(v)$ and $X(w)$ are supported on $A$ by the previous paragraph, 
  we conclude that
  \[    
    X(v)\ = \ (u v)_*(X(v)) \ = \  u_*(v_*(X)(1))\ = \ u_*(w_*(X)(1))\ = \ (u w)_*(X(w))
    \ = \ X(w)\ .
  \]
  Because $X(v)$ is supported on $A$, we have $[w,v]^{X(v)}=\Id_{v_*(X(v))}$
  by Proposition \ref{prop:finite support} (ii). 
  Moreover, $[w,v]^X=\Id_{v_*(X)}$ as endo-transformations of the functor $X$,
  because $X$ is supported on $A$.
  Evaluating this equality at the object $1$ of $\M$ yields
  \[ w_*(X(w,v))\ = \ w_*(X(w,v))\circ [w,v]^{X(v)}\ _\eqref{eq:[u',u]_at_w} = \
    [w,v]^X(1)\ = \ (\Id_{v_*(X)})(1)\ = \ \Id_{v_*(X(v))}\ .\]
  We choose an injection $u\in M$ such that $u v$ and $u w$ are
  the identity on $A$. 
  We conclude that
  \begin{align*}
    X(w,v)\
    &= \  (u w)_*(X(w,v))\ = \
      u_*(\Id_{v_*(X(v))})\ = \  \Id_{ (u v)_*(X(v))}\ = \ \Id_{X(v)}\ .
  \end{align*}
  The first relation is Proposition \ref{prop:finite support} (iv).
\end{proof}

Like any $\M$-category, $\cat(\M,\Cc)$ has a maximal tame $\M$-subcategory $\cat(\M,\Cc)^\tau$,
the full subcategory spanned by the finitely supported objects.

\begin{defn}
  The {\em saturation} of a tame $\M$-category $\Cc$
  is the tame $\M$-category $\Cc^{\sat}=\cat(\M,\Cc)^\tau$.
\end{defn}

The `constant functor'
\[ s\ :\ \Cc\ \to\ \cat(\M,\Cc)\  . \]
sends a $\Cc$-object $x$ to the functor $s(x):\M\to\Cc$
that is given on objects and morphisms by
\[ s(x)(w) \ = \ x\text{\qquad and\qquad} s(x)(w,v) \ = \ 1_x\ ,\]
for $v,w\in M$.
A morphism $f:x\to y$ in $\Cc$ is taken to the natural transformation
$s(f):s(x)\to s(y)$ whose value at $w\in M$ is $s(f)(w)=f$.

\begin{theorem}\label{thm:saturation M-version} 
  Let $\Cc$ be an $\M$-category.
  \begin{enumerate}[\em (i)]
  \item The functor $s:\Cc \to \cat(\M,\Cc)$ is a morphism of $\M$-categories.
  \item  The $\M$-category $\cat(\M,\Cc)$ is saturated.
  \item If $\Cc$ is tame, then the inclusion
    \[  \Cc^{\sat}\ = \ \cat(\M,\Cc)^\tau \ \to \ \cat(\M,\Cc)\]
    is a global equivalence of $\M$-categories.
 \item If $\Cc$ is tame, then the $\M$-category $\Cc^{\sat}$ is saturated,
   the functor $s:\Cc\to\cat(\M,\Cc)$ takes values in  $\Cc^{\sat}$, and the restriction
    \[ s\ : \ \Cc \ \to \ \Cc^{\sat} \]
    is an equivalence of underlying categories.
  \item If $\Cc$ is tame and saturated, then 
    $s: \Cc \to\Cc^{\sat}$ is a global equivalence of $\M$-categories.
  \end{enumerate}
\end{theorem}
\begin{proof}
  (i) We must show the equality
  \[ s\circ \diamond^\Cc\ = \ \diamond^{\cat(\M,\Cc)}\circ(\M\times s)\ : \
  \M\times\Cc \ \to \ \cat(\M,\Cc) \ .\]
  We check this on objects:
  given $\varphi\in M$ and an object $x$ of $\Cc$, we have
  \begin{align*}
    \varphi_*(s(x))(v,u) \
    &= \  \varphi_*(1_x)\ = \ 1_{\varphi_*(x)} \
    = \ s(\varphi_*(x))(v,u)\ .
  \end{align*}
  The verification for morphisms is similar, and we omit it.

  (ii)
  The $\M$-action on $\cat(\M,\Cc)$ was defined diagonally
  from the right translation action on the source, and the given action on the target.
  So
  \[ \cat(\M,\Cc)[\omega^G]\ = \ \cat(E I(\omega,\omega),\Cc)[\omega^G]\ = \
    \cat(E I(\omega^G,\omega),\Cc[\omega^G]) \ .
  \]
  Moreover, a group element $g\in G$ acts on a functor
  $X:E I(\omega^G,\omega)\to\Cc[\omega^G]$ by
 \[ g\cdot X\ = \ l^g_*\circ X\circ E I(l^g,\omega)\ , \]
 and similarly for morphisms in $\cat(E I(\omega^G,\omega),\Cc[\omega^G])$,
 i.e., natural transformations.
  
  Since the left $G$-action on $\omega^G$ is faithful,
  the induced right $G$-action on the set $I(\omega^G,\omega)$ is free.
  So we can choose a right $G$-equivariant map
  \[ q \ : \ I(\omega^G,\omega) \ \to \ G\ .\]
  Passing to the associated contractible groupoids
  and applying $\cat^G(-,\Cc[\omega^G])$ yields a functor
  \begin{align*}
     \cat^G(E q,\Cc[\omega^G])\ : \
    F^{h G}\Cc\ = \ &\cat^G(E G,\Cc[\omega^G])\\
    \to \ &\cat^G(E I(\omega^G,\omega),\Cc[\omega^G])\ = \ F^G(\cat(\M,\Cc)) \ .
  \end{align*}
  Now we let $\lambda:\omega^G\to \omega$ be any injection such that $q(\lambda)=1$.
  We write $\varepsilon:\cat(\M,\Cc)\to\Cc$ for the functor that evaluates
  functors and natural transformations at the object 1 of $\M$.
  The composite
  \[ F^{h G}\Cc\ \xra{\cat^G(E q,\Cc[\omega^G])} \
    F^G(\cat(\M,\Cc)) \ \xra{\ \lambda_\flat^{\cat(\M,\Cc)}}\ G\cat(\M,\Cc)
    \ \xra{G\varepsilon}\ G\Cc \]
  equals the functor $\lambda_\sharp$ defined in \eqref{eq:define_c_sharp}.
  The functor $\lambda_\sharp$ is an equivalence
  of categories by Proposition \ref{prop:h-fixed versus G-objects};
  the evaluation functor $\varepsilon:\cat(\M,\Cc)\to\Cc$ is an equivalence of categories,
  hence so is the induced functor $G\varepsilon$ on the associated
  categories of $G$-objects.
  So the functor
  \[
    \lambda_\flat^{\cat(\M,\Cc)}\circ\cat^G(E q,\Cc[\omega^G]) \ : \ 
    F^{h G}\Cc\ \to\  G\cat(\M,\Cc)
  \]
  is an equivalence. Hence the fully faithful (by Proposition \ref{prop:lambda_sharp})
  functor
  $\lambda_\flat^{\cat(\M,\Cc)}: F^G(\cat(\M,\Cc))\to G\cat(\M,\Cc)$
  is essentially surjective, and thus an equivalence.
  Since $G$ was any finite group, Corollary \ref{cor:saturation characterizations}
  shows that the $\M$-category $\cat(\M,\Cc)$ is saturated.
  
  (iii)
  We let $G$ be a finite group. We must show that the inclusion
  \[  F^G (\cat(\M,\Cc)^\tau)\ \to\ F^G (\cat(\M,\Cc))  \]
  is an equivalence of categories.
  As we discussed in part (ii), the objects of the category $F^G (\cat(\M,\Cc))$
  are the $G$-equivariant functors $X:E I(\omega^G,\omega)\to\Cc[\omega^G]$.
  We claim that $X$ is naturally isomorphic
  to a functor in the subcategory $F^G (\cat(\M,\Cc)^\tau)$.
  To this end we choose a finite faithful $G$-invariant subset $S$ of the universal $G$-set $\omega^G$.
  Then the action of $G$ on the set $I(S,\omega)$ of injections from $S$ to $\omega$
  is free, so we can choose a $G$-equivariant map
  $r:I(S,\omega)\to I(\omega^G,\omega)$ with finite image $J\subset I(\omega^G,\omega)$.
  We define 
  \[ A \ = \ S\ \cup \ \bigcup_{\lambda\in J} \supp(X(\lambda))\ , \]
  which is a finite subset of $\omega^G$.
  We define a $G$-equivariant functor $Y:E I(\omega^G,\omega)\to \Cc[\omega^G]$
  as the composite
  \[ E I(\omega^G,\omega)\ \xra{E(r\circ \res_S)}\
    E I(\omega^G,\omega)\ \xra{\ X\ }\  \Cc[\omega^G]\ , \]
  where $\res_S:I(\omega^G,\omega)\to I(S,\omega)$ is the restriction of a function to $S$.
  Proposition \ref{prop:support in cat(M,C)} then shows that the functor
  $Y$ is supported on the finite subset $A$ of $\omega^G$.
  So $Y$ belongs to the subcategory $F^G(\cat(\M,\Cc)^\tau)$.
  Moreover, for varying $\lambda\in I(\omega^G,\omega)$,
  the isomorphisms
  \[ X(r(\res_S(\lambda)),\lambda)\ : \ X(\lambda)\ \to \ X(r(\res_S(\lambda)))\ = \ Y(\lambda) \]
  form a $G$-equivariant natural isomorphism from $X$ to $Y$.
  This concludes the proof.
  
  (iv)
  Every morphism of $\M$-categories takes finitely supported objects
  to finitely supported objects. Since all objects of $\Cc$ are finitely supported,
  the functor $s:\Cc\to\cat(\M,\Cc)$ lands in the full subcategory $\Cc^{\sat}=\cat(\M,\Cc)^\tau$.
  Since the $\M$-category $\cat(\M,\Cc)$ is saturated by part (ii)
  and the inclusion $\Cc^{\sat}\to \cat(\M,\Cc)$ is a global equivalence by part (iii),
  the $\M$-category $\Cc^{\sat}$ is also saturated.
  Since the inclusion $\Cc^{\sat}\to \cat(\M,\Cc)$
  and the functor $s:\Cc\to\cat(\M,\Cc)$ are equivalences,
  so is the restricted functor $s:\Cc\to\Cc^{\sat}$.
  
  Given part (iv), the final part (v) is now an application of 
  Proposition \ref{prop:tame2global M-version} (ii). 
\end{proof}

As we shall now explain, the saturation has a natural extension
from $\M$-categories to parsummable categories.

\begin{con}[(Saturation for parsummable categories)]\label{con:saturation_parsummable}
  Above we introduced the saturation functor
  \[ (-)^{\sat}\ = \ \cat(\M,-)^\tau \ : \ \M\cat^\tau \ \to \ \M\cat^\tau \]
  for tame $\M$-categories.
  We let $\Cc$ and $\Dc$ be tame $\M$-categories.
  We will now define a natural $\M$-equivariant functor
  \begin{equation}\label{eq:saturation_monoidal}
    \Cc^{\sat}\boxtimes \Dc^{\sat} \ \to\ (\Cc\boxtimes \Dc)^{\sat}
  \end{equation}
  that makes saturation into a lax symmetric monoidal functor for the box
  product of tame $\M$-categories.

  We consider functors $X:\M\to\Cc$ and $Y:\M\to\Dc$ that are supported,
  as objects of the $\M$-categories $\cat(\M,\Cc)$ and $\cat(\M,\Dc)$, 
  on disjoint finite subset $A$ and $B$ of $\omega$.
  Proposition \ref{prop:support in cat(M,C)} shows
  that then the values of $X$ and $Y$ are in particular objectwise disjointly supported. 
  So the functor
  \[ (X,Y)\ : \ \M \ \to\ \Cc\times \Dc \]
  takes values in the full subcategory $\Cc\boxtimes\Dc$.
  Moreover, as an object of the $\M$-category $\cat(\M,X\boxtimes Y)$,
  the functor $(X,Y)$ is supported on the finite set $A\cup B$,
  again by Proposition \ref{prop:support in cat(M,C)}.
  So the corresponding functor 
  \[  \Cc^{\sat}\boxtimes \Dc^{\sat}\ = \
    \cat(\M,\Cc)^\tau \boxtimes\cat(\M,\Dc)^\tau \ \to\ \cat(\M,\Cc\times \Dc)   \]
  takes values in the full subcategory $\cat(\M,\Cc\boxtimes \Dc)^\tau=(\Cc\boxtimes\Dc)^{\sat}$.
  This defines the natural functor \eqref{eq:saturation_monoidal},
  which is clearly $\M$-equivariant, associative, commutative and unital.

  Now we suppose that $\Cc$ is a parsummable category.
  Then the saturation $\Cc^{\sat}$ of the underlying $\M$-category
  inherits the structure of a parsummable category, with addition functor
  defined as the composite
  \[
    \Cc^{\sat}\boxtimes \Cc^{\sat} \ \xra{\eqref{eq:saturation_monoidal}} \
    (\Cc\boxtimes\Cc)^{\sat}\ \ \xra{\ +^{\sat}\ }\  \Cc^{\sat}\ .
  \]
  In more down-to-earth terms, this means
  that the parsummable structure is entirely defined `objectwise'.
  For example, if $X,Y:\M\to\Cc$ are functors that have disjoint finite supports
  with respect to the $\M$-action on $\cat(\M,\Cc)$,
  then the functor $X+Y:\M\to\Cc$ is defined on objects and morphisms by
  \[ (X+Y)(u)\ = \ X(u)+Y(u)\text{\qquad and\qquad}  (X+Y)(v,u)\ = \ X(v,u)+Y(v,u)\ .\]
\end{con}

\begin{theorem}\label{thm:saturation} 
  Let $\Cc$ be a parsummable category.
  \begin{enumerate}[\em (i)]
  \item The parsummable category $\Cc^{\sat}$ is saturated.
  \item The functor $s:\Cc \to \Cc^{\sat}$ is a morphism of parsummable categories
    and an equivalence of underlying categories.
  \item If $\Cc$ is saturated, then $s:\Cc\to\Cc^{\sat}$ is a global equivalence of parsummable categories.
  \end{enumerate}
\end{theorem}
\begin{proof}
  The constant functor $\M\to\Cc$ with value 0 is the distinguished zero object
  of the parsummable category $\cat(\M,\Cc)^\tau$, so the constant functor $s:\Cc\to\cat(\M,\Cc)^\tau$ is unital.
  For all tame $\M$-categories $\Cc$ and $\Dc$, the composite
  \[   \Cc\boxtimes \Dc\ \xra{s\boxtimes s} \
    \cat(\M,\Cc)^\tau \boxtimes\cat(\M,\Dc)^\tau \ \xra{\eqref{eq:saturation_monoidal}}\
    \cat(\M,\Cc\boxtimes \Dc)^\tau       
  \]
  coincides with the constant functor $s:\Cc\boxtimes \Dc\to \cat(\M,\Cc\boxtimes \Dc)^\tau$.       
  So for every parsummable category $\Cc$, the functor
  $s:\Cc\to \cat(\M,\Cc)^\tau$ respects the addition functors.       
  `Saturation' and `global equivalence' are properties of the underlying $\M$-categories,
  so the remaining claims are special cases of Theorem \ref{thm:saturation M-version}.   
\end{proof}

A consequence of Theorem \ref{thm:saturation} is that the $G$-fixed category
$F^G(\Cc^{\sat})$ of the saturation of a parsummable category $\Cc$
is equivalent to the category of $G$-objects in $\Cc$:
because $s:\Cc\to\Cc^{\sat}$ is an equivalence of categories and
because $\Cc^{\sat}$ is saturated, the three functors
\[ G \Cc \ \xra[\iso]{\ G s\ } \ G\Cc^{\sat}\ \xra[\iso]{\ c_\sharp\ }\ F^{h G}(\Cc^{\sat})
  \ \xla[\iso]{\ \kappa\ }\ F^G(\Cc^{\sat}) \]
are equivalences of categories.

\section{Global K-theory of free parsummable categories}
\label{sec:K of free I}

As we explained in Example \ref{eg:free_parsummable}, the forgetful functor from
parsummable categories to tame $\M$-categories has a left adjoint,
yielding {\em free parsummable categories}.
In this section we identify the global K-theory spectrum of
the free parsummable category generated by a tame $\M$-category as a suspension spectrum,
see Theorem \ref{thm:global BPQ}.
Our result can be interpreted as a global equivariant generalization
of the Barratt-Priddy-Quillen theorem;
so before going into details, we review the Barratt-Priddy-Quillen theorem
and its variations and generalizations,
in order to put our result into context.

The original result of Barratt-Priddy and Quillen 
states that a specific map
\[ \mZ\times B\Sigma_\infty \ \to \ \colim_n \Omega^n S^n \ = \ Q(S^0) \]
is a homology isomorphism;
here $B\Sigma_\infty$ is the classifying space of the infinite symmetric group
$\Sigma_\infty=\bigcup_{n\geq 1}\Sigma_n$, and $Q(S^0)$
is the infinite loop space of the sphere spectrum.
The result was announced by Barratt in \cite{barratt:free_group}
and by Priddy \cite{priddy:Omega_infty_S_infty},
who jointly published a detailed proof in \cite{barratt-priddy:homology}.
Barratt and Priddy acknowledge that Quillen had also proved
the same result, but to my knowledge, Quillen never circulated or published his proof in writing.
Sometimes the result is stated in an equivalent formulation,
as a weak homotopy equivalence between $\mZ\times (B\Sigma_\infty)^+$ and $Q(S^0)$,
where $(-)^+$ is the plus construction.
Another interpretation of the Barratt-Priddy-Quillen theorem is to say that
$Q(S^0)$ is the group completion of $\coprod_{m\geq 0}B\Sigma_m$,
the free $E_\infty$-space generated by a point.
The $E_\infty$-space $\coprod_{m\geq 0}B\Sigma_m$, in turn, is weakly equivalent
to the nerve of the category of finite sets and bijections,
with the $E_\infty$-structure arising from disjoint union.
So yet another way to view the Barratt-Priddy-Quillen theorem is
as a stable equivalence between the sphere spectrum and the K-theory spectrum
of the category of finite sets under disjoint union.
We offer a global equivariant refinement of this results in 
Theorem \ref{thm:global F}, saying that the global K-theory of the
parsummable category $\Fc$ of finite sets is globally equivalent to the global sphere spectrum.

A generalization of the Barratt-Priddy-Quillen theorem
is the statement that the group completion
of the free $E_\infty$-space generated by a space $X$ is weakly
equivalent to $Q(X)=\colim_{n\geq 0}\Omega^n(\Sigma^n X_+)$, the infinite loop space
of the unreduced suspension spectrum of $X$.
Equivalently, the spectrum made from the K-theoretic deloopings of
the free $E_\infty$-space on $X$ is equivalent to the unreduced suspension spectrum of $X$.
Two versions of this result for different $E_\infty$-operads
were proved by Barratt-Eccles \cite[Theorem A]{barratt-eccles}
and May \cite[Theorem 2.2]{may:E_infty_spaces};
in the framework of $\Gamma$-spaces, Segal stated the corresponding result in
\cite[Proposition 3.6]{segal:cat coho}.
For a fixed finite group $G$, equivariant versions of the Barratt-Priddy-Quillen theorem
have been established by
Hauschild \cite[Theorem  III.4]{hauschild:konfigurationsraeume},
Carlsson-Douglas-Dundas \cite[Section 5.2]{carlsson-douglas-dundas},
Guillou-May \cite[Section 6]{guillou-may},
Barwick-Glasman-Shah \cite[Theorem 10.6]{barwick-glasman-shah:spectral_mackey_II},
and possibly others that I am not aware of.
Our Theorem \ref{thm:global BPQ} below is a global equivariant version
of the Barratt-Priddy-Quillen theorem.

If every object of the tame $\M$-category $\Bc$ has a non-empty support,
then for every $m\geq 0$, the permutation action
of the symmetric group $\Sigma_m$ on the box power $\Bc^{\boxtimes m}$ is free.
So in this situation, the nerve of the underlying category
of the free parsummable category $\mP\Bc$
is weakly equivalent to the free $E_\infty$-space generated by the nerve
of  the underlying category of $\Bc$.
The Barratt-Priddy-Quillen theorem thus suggest that non-equivariantly,
the K-theory of $\mP\Bc$ ought to be the unreduced suspension spectrum
of the nerve of the underlying category of $\Bc$.
We will show that this is indeed the case, even in a global equivariant form.

\begin{con}[(Word length filtration)]\label{con:word length}
  We let $\Bc$ be a tame $\M$-category.
  As we discussed in Example \ref{eg:free_parsummable}, the free parsummable category generated by
  $\Bc$ has underlying category
  \[ \mP\Bc \ = \ \coprod_{m\geq 0} \, (\Bc^{\boxtimes m})/\Sigma_m\ .\]
  We filter the associated $\Gamma$-$\M$-category $\gamma(\mP\Bc)$ by `word length', as follows.
  Since the free functor $\mP:\M\cat^\tau\to\parsumcat$ is left adjoint
  to the forgetful functor, it takes disjoint unions of tame $\M$-categories
  to coproducts of parsummable categories, which are given by the box product, see
  Example \ref{eg:coproduct_parsumcat}.
  So the value of $\gamma(\mP\Bc)$ at an object $n_+$ of $\Gamma$ can be
  rewritten as
  \[ \gamma(\mP\Bc)(n_+)\ = \ (\mP\Bc)^{\boxtimes n}\ \iso \
    \mP(\Bc\times\mathbf n) \ = \ \coprod_{m\geq 0} \, ((\Bc\times\mathbf n)^{\boxtimes m})/\Sigma_m \ ,  \]
  where $\mathbf n=\{1,\dots,n\}$.
  We write $\gamma_k(\mP\Bc)$ for the $\Gamma$-$\M$-subcategory of
  $\gamma(\mP\Bc)$ whose value at $n_+$ is
  \[ \gamma_k(\mP\Bc)(n_+)\ = \ \coprod_{0\leq m\leq k} \, ((\Bc\times\mathbf n)^{\boxtimes m})/\Sigma_m \ ,  \]
  i.e., the disjoint union runs only up to $k$.
  Passing to the associated symmetric spectra (see Construction \ref{con:spectrum from Gamma-M})
  provides an exhaustive filtration
  \[  \gamma_1(\mP\Bc)\td{\mS}\ \subseteq\
    \gamma_2(\mP\Bc)\td{\mS}\ \subseteq\ \dots \ \subseteq
    \gamma_k(\mP\Bc)\td{\mS}\ \subseteq\ \dots
  \]
  of the symmetric spectrum $\gamma(\mP\Bc)\td{\mS}=\bK_{\gl}(\mP\Bc)$.
\end{con}

\begin{con}[(From $\M$-categories to $\bI$-spaces)]\label{con:M2bI}
  We let $\bI$ denote the category of finite sets and injective maps;
  an {\em $\bI$-space} is a functor from $\bI$ to the category of spaces. 
  As indicated in \cite[Section 6.1]{hausmann:global_finite}
  and explained in detail in \cite[Section 1.4]{lenz:G-global},
  the category of $\bI$-spaces is a model for unstable global homotopy theory
  based on finite groups. 

  We will now associate an $\bI$-space $\rho(\Bc)$ to every $\M$-category $\Bc$.
  As before, we denote by $\omega^A$ the set of maps from $A$ to $\omega$.
  The value of $\rho(\Bc)$ at a non-empty finite set~$A$ is
  \[ \rho(\Bc)(A) \ = \ |\Bc[\omega^A]| \ ,\]
  the geometric realization of the category $\Bc[\omega^A]$.
  For the empty set we set
  $ \rho(\Bc)(\emptyset)  =  |\Bc^{\supp=\emptyset}|$,
  the realization of the full subcategory of $\Bc$
  on the objects supported on the empty set.
  The structure map associated with an injection $i:A\to B$ is the map
  \[ \rho(\Bc)(i)\ = \ |\Bc[i_!]| \ : \  |\Bc[\omega^A]|\ \to\ |\Bc[\omega^B]|\ ,\]
  where $i_!:\omega^A\to\omega^B$
  is extension by zero, see \eqref{eq:extension_by_zero}.
  In the special case where $A=\emptyset$ is empty,
  the map $\rho(\Bc)[i_!]:\rho(\Bc)[\omega^A]\to \rho(\Bc)[\omega^B]$ is to be interpreted as
  the inclusion $|\Bc^{\supp=\emptyset}|\to |\Bc[\omega^B]|$.
\end{con}

\begin{defn}[(Suspension spectrum of an $\bI$-space)]
    The {\em unreduced suspension spectrum} $\Sigma^\infty_+ X$
    of an $\bI$-space $X$ is defined by
    \[  (\Sigma^\infty_+ X)(A)\ = \ X(A)_+\sm S^A\ .\]
    The structure map $i_*:(\Sigma^\infty_+ X)(A)\sm S^{B\setminus i(A)}\to(\Sigma^\infty_+ X)(B)$
    associated to an injection $i:A\to B$ between finite sets is
    the smash product of the structure map $X(i)_+:X(A)_+\to X(B)_+$
    of $X$ with the preferred homeomorphism $S^A\sm S^{B\setminus i(A)}\iso S^B$
    that is given by $i$ on the $A$-coordinates.
\end{defn}

\begin{eg}\label{eg:gamma_1=suspension}
  Let $\Bc$ be a tame $\M$-category.
  We claim that the symmetric spectrum $\gamma_1(\mP\Bc)\td{\mS}$,
  the first term in the word length filtration of $\bK_{\gl}(\mP\Bc)$ introduced
  in Construction \ref{con:word length},
  is isomorphic to the unreduced suspension spectrum of the $\bI$-space $\rho(\Bc)$.

  Indeed, by definition we have
  \[ |\gamma_1(\mP\Bc)|(n_+)\
    = \ |\ast \amalg (\Bc\times\mathbf n)| \ \iso \  |\Bc|_+\sm n_+\ .\]
  The prolongation of this $\Gamma$-space is the functor $|\Bc|_+\sm -:\bT_*\to\bT_*$.
  So
  \[   \gamma_1(\mP\Bc)\td{\mS}(A) \ = \ |\gamma_1(\mP\Bc)|\ = \  |\Bc[\omega^A]|_+\sm S^A
    \ = \ (\Sigma^\infty_+\rho(\Bc))(A)\ .
  \]
  We omit the straightforward verification that the structure maps of the symmetric spectra
  $\gamma_1(\mP\Bc)\td{\mS}$ and $\Sigma^\infty_+\rho(\Bc)$ correspond to each other under this identification.
\end{eg}

Our proof of Theorem \ref{thm:global BPQ} below
is based on connectivity estimates of the subquotients in the word-length filtration;
this kind of connectivity argument goes back,
at least, to Barratt and Eccles \cite[Section 6]{barratt-eccles}.
The following proposition is the key technical ingredient;
we write $A/G$ for the set of $G$-orbits
of a $G$-set $A$, and we write $|A/G|$ for its cardinality.

\begin{prop}\label{prop:connectivity}
  Let $G$ be a finite group and let $X$ be a $(G\times \Sigma_k)$-simplicial set
  such that the $\Sigma_k$-action is free, for $k\geq 2$.
  Let $A$ be a finite $G$-set with $q$ free $G$-orbits.
  Then the $G$-fixed simplicial set
  \[ ((X_+\sm (S^A)^{\sm k})/\Sigma_k)^G \]
  is $( |A/G|+q-1)$-connected.
\end{prop}
\begin{proof}
  Because the $\Sigma_k$-action on $X$ is free,
  the $\Sigma_k$-action on $X_+\sm (S^A)^{\sm k}$ is free away from the basepoint.
  The $G$-fixed points of the $\Sigma_k$-orbits can thus be identified as
  \begin{align*}
    \left((X_+\sm (S^A)^{\sm k})/\Sigma_k\right)^G\ 
    \iso \ \bigvee_{[\alpha:G\to\Sigma_k]}\ (X_+\sm (S^A)^{\sm k})^{\Gamma(\alpha)}/C(\alpha)\ .
  \end{align*}
  The wedge is indexed by conjugacy classes of homomorphisms $\alpha:G\to\Sigma_k$,
  $\Gamma(\alpha)$ denotes the graph of $\alpha$, and $C(\alpha)$ denotes
  the centralizer in $\Sigma_k$ of the image of $\alpha$.
  So to prove the claim, it suffices to show that each of the wedge summands
  is $( |A/G|+q-1)$-connected.
  
  For the rest of the proof we thus fix a particular homomorphism $\alpha:G\to\Sigma_k$.
  We let $i_1,\dots, i_m\in \{1,\dots,k\}$
  be representatives of the $G$-orbits for the action through $\alpha$.
  We let $H_j$ be the stabilizer of $i_j$, a subgroup of $G$.
  Then the $\Gamma(\alpha)$-fixed points of $(S^A)^{\sm k}$ are given by
  \begin{align*}
    ((S^A)^{\sm k})^{\Gamma(\alpha)}\ \iso \ \bigwedge_{j=1,\dots,m} \, (S^A)^{H_j}\ ,
  \end{align*}
  the smash product of the $H_j$-fixed simplicial sets of $S^A$.
  The connectivity of these fixed points is one less than
  the dimension of the sphere $\bigwedge_{j=1,\dots, m}  (S^A)^{H_j}$,
  which is
  \[ \dim(\bigwedge_{j=1,\dots, m}  (S^A)^{H_j})\ = \  \sum_{j=1,\dots, m} |A/H_j|  \ .\]
  Now we distinguish two cases.
  If the $G$-action on $\{1,\dots,k\}$ through $\alpha:G\to\Sigma_k$
  is not transitive, then $m\geq 2$, and hence 
  \[  \sum_{j=1,\dots, m} |A/H_j|\ \geq \ 2\cdot |A/G| \ \geq \ |A/G| + q \ .\]
  If the $G$-action on $\{1,\dots,k\}$ is transitive, then $m=1$ and
  the stabilizer group $H_1$ is a proper subgroup of index $k\geq 2$.
  Because $A$ has $q$ free $G$-orbits, it is $G$-isomorphic to $B\amalg (\mathbf q\times G)$
  for some finite $G$-set $B$. So
  \[ |A/H_1| \ = \ |B/H_1| + q\cdot [G:H_1] \ \geq \ |B/G|+ 2 q \ = \ |A/G|+ q \ . \]
  So in either case, the simplicial set
  $((S^A)^{\sm k})^{\Gamma(\alpha)}$ is $(|A/G|+q-1)$-connected. Hence the simplicial set
  \[   (X_+\sm (S^A)^{\sm k})^{\Gamma(\alpha)}
    \ = \  X^{\Gamma(\alpha)}_+\sm ((S^A)^{\sm k})^{\Gamma(\alpha)}  \]
  is also $(|A/G|+q-1)$-connected.
  Since the $\Sigma_k$-action on $X$ is free,
  the action of its subgroup $C(\alpha)$ on $X^{\Gamma(\alpha)}$ is also free.
  So the $C(\alpha)$-action on $(X_+\sm (S^A)^{\sm k})^{\Gamma(\alpha)}$
  is free away from the base point.
  Thus the $C(\alpha)$-orbits $(X_+\sm (S^A)^{\sm k})^{\Gamma(\alpha)}/C(\alpha)$
  are also $(|A/G|+q-1)$-connected. This concludes the proof.
\end{proof}

Now we come to the main result of this section.
The identification provided by Example \ref{eg:gamma_1=suspension},
followed by the inclusion   $\gamma_1(\mP\Bc)\td{\mS}\to \gamma(\mP\Bc)\td{\mS}=\bK_{\gl}(\mP\Bc)$
is a morphism of symmetric spectra
\begin{equation}\label{eq:eq:define_w}
 w\ : \ \Sigma^\infty_+ \rho(\Bc) \  \to \ \bK_{\gl}(\mP\Bc)  \ .  
\end{equation}

\begin{theorem}\label{thm:global BPQ}
  Let $\Bc$ be a tame $\M$-category without objects with empty support.
  Then the morphism  $w:\Sigma^\infty_+ \rho(\Bc)  \to \bK_{\gl}(\mP\Bc)$
  is a global equivalence of symmetric spectra.
\end{theorem}
\begin{proof}
  Since the word length filtration of $\bK_{\gl}(\mP\Bc)$
  arises from a filtration by monomorphisms of symmetric spectra
  of simplicial sets, the $G$-equivariant homotopy groups of
  $\gamma_{k-1}(\mP\Bc)\td{\mS}$, $\gamma_k(\mP\Bc)\td{\mS}$
  and the $k$-th subquotient participate in a long exact sequence.  
  So it suffices to show that the $k$-th subquotient
  of the word-length filtration of $\bK_{\gl}(\mP\Bc)$
  has trivial $G$-equivariant stable homotopy groups
  for every finite group $G$ and all $k\geq 2$.
  We write $\mu_k$ for the $\Gamma$-$\Sigma_k$-space defined by
  \[ \mu_k(n_+)\ = \ |\Bc^{\boxtimes k}|_+\sm  (n_+)^{\sm k}\ . \]
  The functoriality in $\Gamma$ is only through $n_+$;
  the $\Sigma_k$-action is diagonal, through the permutation actions on
  $\Bc^{\boxtimes k}$ and $(n_+)^{\sm k}$.
  Then 
  \begin{align*}
    |\gamma_k(\mP\Bc)|(n_+) / |\gamma_{k-1}(\mP\Bc)|(n_+)\
    &\iso \ | (\Bc\times\mathbf n)^{\boxtimes k}/\Sigma_k|_+\\
    &\iso \ (| (\Bc\times\mathbf n)^{\boxtimes k}|/\Sigma_k)_+\  \iso \ \mu_k(n_+)/\Sigma_k\ . 
  \end{align*}
  The second homeomorphism exploits the hypothesis that $\Bc$ has no objects with empty support:
  this condition guarantees that the permutation action of $\Sigma_k$ on
  the category $(\Bc\times\mathbf n)^{\boxtimes k}$ is free, and so passing to
  $\Sigma_k$-orbits commutes with taking nerves.
  These homeomorphisms are natural for morphisms in $\Gamma$, so they constitute
  an isomorphism of $\Gamma$-spaces
  \[  |\gamma_k(\mP\Bc)|/|\gamma_{k-1}(\mP\Bc)|\ \iso \ \mu_k/\Sigma_k\ .  \]
  Prolongation of $\Gamma$-spaces is an enriched colimit, so it commutes with colimits
  of $\Gamma$-spaces; in particular, prolongation commutes with quotients,
  and with orbits by group actions.
  So the value of the $k$-subquotient of the word length filtration of
  $\bK_{\gl}(\mP\Bc)$ is isomorphic to
  \[ \left( \mu_k(S^A)/\Sigma_k\right)[\omega^A]\ = \
    (|\Bc^{\boxtimes k}[\omega^A]|_+\sm (S^A)^{\sm k})/\Sigma_k    \ . \]
  Now we fix a finite group $G$, and we let $A$ be a finite $G$-set
  with $q$ free $G$-orbits, for some $q\geq 1$.
  Since the $\Sigma_k$-action on the category $\Bc^{\boxtimes k}$ is free,
  so is the action on its nerve.
  So Proposition \ref{prop:connectivity} applies
  to the nerve of the $(G\times\Sigma_k)$-category $\Bc^{\boxtimes k}[\omega^A]$,
  and shows that the fixed point space 
  \[ \left( \left( \mu_k(S^A)/\Sigma_k\right)[\omega^A]\right)^G \]
  is $(|A/G|+q-1)$-connected.
  
  For every subgroup $H$ of $G$, the underlying $H$-set of $A$ has
  at least $q$ free $H$-orbits. So 
  the fixed point space
  $\left( \left( \mu_k(S^A)/\Sigma_k\right)[\omega^A] \right)^H$
  is $(|A/H|+q-1)$-connected.
  On the other hand, the dimension of the $H$-fixed point sphere $(S^A)^H$ is $|A/H|$.
  So as long as $m$ is smaller than the number of free $G$-orbits of $A$,
  the dimension of the $H$-fixed points of $S^{m+A}$
  is smaller than the connectivity
  of the space $\left( \left( \mu_k(S^A)/\Sigma_k\right)[\omega^A] \right)^H$.
  Since this holds for all subgroups $H$ of $G$, every
  based continuous $G$-map $S^{m+A}\to \left( \mu_k(S^A)/\Sigma_k\right)[\omega^A]$
  is equivariantly null-homotopic, i.e., the set
  \[ [S^{m+ A},\left( \mu_k(S^A)/\Sigma_k\right)[\omega^A]]^G \]
  has only one element.
  The $G$-sets with at least $m+1$ free orbits are cofinal in the
  poset of finite $G$-subsets of a universal $G$-set $\Uc_G$.
  So the homotopy group $\pi_m^{G,\Uc_G}( (\mu_k/\Sigma_k)\td{\mS} )$ is trivial.
  Since $G$ was any finite group, we have shown that for every $k\geq 2$,
  the $k$-th subquotient of the word length filtration of $\bK_{\gl}(\mP\Bc)$
  has trivial equivariant homotopy groups. This concludes the proof.
\end{proof}

\begin{rk}
  A hypothesis of Theorem \ref{thm:global BPQ} is that
  no object of the tame $\M$-category $\Bc$ has empty support.
  The example of the terminal $\M$-category shows that the condition is really necessary,
  and not just an artifact of our proof.
  Indeed, the $\bI$-space $\rho(\ast)$ associated with the terminal $\M$-category
  $\ast$ is constant with values a one-point space,
  so its unreduced suspension spectrum $\Sigma^\infty_+\rho(\ast)$ is isomorphic to the sphere spectrum.
  The free parsummable category $\mP(\ast)$ generated by the terminal $\M$-category is discrete
  and isomorphic to the parsummable category associated to the abelian monoid $\mN$ of
  natural numbers, compare Example \ref{eg:abelian monoid}.
  Its global K-theory spectrum refines the Eilenberg-MacLane spectrum of the integers;
  so the morphism $w:\Sigma^\infty_+\rho(\ast)\to\bK_{\gl}(\mP(\ast))$ is not even
  a non-equivariant stable equivalence.
\end{rk}

The parsummable category $\Fc$ of finite sets was introduced in Example \ref{eg:Fc};
the objects of $\Fc$ are the finite subsets of $\omega=\{0,1,2,3,\dots\}$,
and morphisms in $\Fc$ are all bijections.
The {\em global K-theory of finite sets} is the symmetric spectrum
\[ \bK_{\gl}\Fc \ = \ \gamma(\Fc)\td{\mS} \]
associated with the  parsummable category $\Fc$.
As we shall now explain, our Theorem \ref{thm:global BPQ}
applies to the parsummable category $\Fc$,
and it yields an identification of the global K-theory of finite sets with
the global sphere spectrum.

We write $\Bc$ for the full $\M$-subcategory of $\Fc$ whose objects
are the one-element subsets of $\omega$.
We write $t:\Bc\to\ast$ for the unique functor to the terminal $\M$-category
with one object and its identity morphism.
The $\bI$-space $\rho(\ast)$ is then constant with value the one point space $|\ast|$;
so its suspension spectrum $\Sigma^\infty_+\rho(\ast)$ is uniquely isomorphic
to the symmetric sphere spectrum $\mS$.

\begin{theorem}\label{thm:global F}
  The morphisms of symmetric spectra 
  \[ \mS \ \iso \ \Sigma^\infty_+ \rho(\ast) \ \xla{\Sigma^\infty_+\rho(t)}\
    \Sigma^\infty_+\rho(\Bc)\  \xra{\ w\ } \ \bK_{\gl}\Fc\]
  are global equivalences.
\end{theorem}
\begin{proof}
  The inclusion $\Bc \to \Fc$ is a morphism of $\M$-categories,
  so it extends uniquely to a morphism of parsummable categories $\mP\Bc\to \Fc$
  from the free parsummable category generated by $\Bc$.
  This morphism is in fact an isomorphism of parsummable categories, by direct inspection
  -- or as a special case of Theorem \ref{thm:K of free G} (i) below.
  Theorem \ref{thm:global BPQ} thus shows that the morphism
  $w:\Sigma^\infty_+ \rho(\Bc) \to \bK_{\gl}\Fc$ is a global equivalence of symmetric spectra.

  To complete the proof we show that the morphism
  $\Sigma^\infty_+\rho(t):\Sigma^\infty_+\rho(\Bc)\to \mS$
  is even a global level equivalence in the sense of \cite[Definition 2.2]{hausmann:global_finite}.
  We let $G$ be a finite group and $A$ a finite $G$-set.
  The category $(\Bc[\omega^A])^G$ consists of $G$-invariant subsets of $\omega^A$
  with one element, and their isomorphisms. This category is a connected groupoid
  with trivial automorphism groups, so its nerve is contractible.
  Hence the map
  \[ \left(\Sigma^\infty_+\rho(t)(A)\right)^G\ :\
    \left(\Sigma^\infty_+\rho(\Bc)(A)\right)^G \ = \
    |\Bc[\omega^A]|^G_+\sm (S^A)^G\  \to \ (S^A)^G \ = \ (\mS(A))^G  \]
  is a weak equivalence. The morphism  $\Sigma^\infty_+\rho(t)$
  is thus a global level equivalence by the criterion \cite[Lemma.\,2.3]{hausmann:global_finite},
  and hence a global equivalence by \cite[Example 2.10]{hausmann:global_finite}.
\end{proof}
  
\begin{rk}
  Hausmann and Ostermayr \cite{hausmann-ostermayr}
  establish a result of a very similar flavor as our Theorem \ref{thm:global F}.
  They study an orthogonal spectrum $k\Fin$ that also deserves
  to be called the `global K-theory of finite sets'.
  While similar in spirit, the two constructions are different:
  $k\Fin$ is a commutative orthogonal ring spectrum made from configuration spaces of points in spheres
  labeled by orthonormal systems of vectors.
  Hausmann and Ostermayr show in \cite[Corollary 4.2]{hausmann-ostermayr}
  that the unit morphism $\mS \to k\Fin$ is a global equivalence of
  orthogonal spectra.
  So their theorem lives in the full-fledged global homotopy theory
  of orthogonal spectra, i.e., it also has homotopical content for compact Lie groups,
  while our Theorem \ref{thm:global F} only refers to $\Fin$-global homotopy types.  
  The argument of Hausmann and Ostermayr is based on an analysis
  of the cardinality filtration of  $k\Fin$,
  a filtration analogous to our word length filtration.
\end{rk}

\section{Global K-theory of \texorpdfstring{$G$}{G}-sets}
\label{sec:G-sets}

This section is devoted to the global K-theory of finite $G$-sets,
where $G$ is a discrete group, possibly infinite.
The parsummable category $G\Fc$ of finite $G$-sets is,
by definition, the category of $G$-objects in the parsummable category $\Fc$
of finite sets, see Example \ref{eg:G Fc} below.
The global K-theory spectrum $\bK_{\gl}(G\Fc)$ of finite $G$-sets
can be completely described in terms of global classifying spaces of finite groups.
Indeed, by Theorem \ref{thm:split_K(GF)},
the symmetric spectrum $\bK_{\gl}(G\Fc)$ splits,
up to global equivalence, into summands indexed
by the conjugacy classes of finite index subgroup of $G$.
The summand indexed by a finite index subgroup $H$ of $G$
can be identified with the global K-theory of free $W_G H$-sets,
where $W_G H=(N_G H)/H$ is the Weyl group of $H$ in $G$,
see Proposition \ref{prop:H2W}.
An application of our global Barratt-Priddy-Quillen theorem
then identifies the global K-theory of free $W_G H$-sets
with the unreduced suspension spectrum of the global classifying space of $W_G H$,
see Theorem \ref{thm:K of free G}.

For finite groups $G$, the essential mathematical content of Theorem \ref{thm:split_K(GF)}
could also be obtained by combining the identification of the global K-theory of
finite sets (Theorem \ref{thm:global F}) with Corollary \ref{cor:saturated implies global},
the tom Dieck splitting and the Adams isomorphism.
Our approach below, based on the global Barratt-Priddy-Quillen theorem
(Theorem \ref{thm:global BPQ}), does not rely on the tom Dieck splitting or the Adams isomorphism,
and it works for arbitrary discrete groups.

\begin{eg}[(Global K-theory of finite $G$-sets)]\label{eg:G Fc}
  We let $G$ be a discrete group, possibly infinite.
  For every parsummable category $\Cc$,
  the category $G\Cc$ of $G$-objects in $\Cc$ inherits
  a parsummable structure as discussed in Example \ref{eg:objects with action}.
  In essence, the $\M$-action and the sum functor on $G\Cc$ are the given
  structure on underlying objects, with $G$-actions carried along by functoriality.
  The support of a $G$-object coincides with the support of the underlying $\Cc$-object.

  The parsummable category $\Fc$ of finite sets was introduced in Example \ref{eg:Fc};
  so the category $G\Fc$ of $G$-objects in $\Fc$ forms a parsummable category,
  the {\em parsummable category of finite $G$-sets}.
  By definition, $G\Fc$ is the full subcategory of the category of finite $G$-sets,
  with objects those $G$-sets whose underlying set is a subset of $\omega=\{0,1,2,\dots\}$.
  Since $\Fc$ is saturated by Example \ref{eg:F is saturated}, the parsummable category $G\Fc$
  is saturated by Example \ref{eg:GC inherits saturation}.
  We refer to the symmetric spectrum $\bK_{\gl}(G\Fc)$ as the {\em global K-theory of finite $G$-sets}.
\end{eg}

We will now argue that the parsummable category $G\Fc$ decomposes as a box product,
and its global K-theory decompose as a wedge,
both indexed by conjugacy classes of finite index subgroups.
For a subgroup $H$ of $G$ we denote by $(G\Fc)_{(H)}$ the full subcategory of $G\Fc$
whose objects are finite subsets of $\omega$ equipped with a $G$-action such that all isotropy
groups are conjugate to $H$.
Morphisms in $(G\Fc)_{(H)}$ are the $G$-equivariant bijections.
The category $(G\Fc)_{(H)}$ is closed under the $\M$-action and the addition in $G\Fc$,
hence $(G\Fc)_{(H)}$ is a parsummable category in its own right.
If $H$ has infinite index in $G$, then the $G$-orbit of any point with isotropy group $H$
is infinite; so in this case, the category $(G\Fc)_{(H)}$ has only one
object, the empty set. Hence the parsummable category $(G\Fc)_{(H)}$ is only interesting
if $H$ has finite index in $G$.

The box product is the coproduct of parsummable categories,
see Examples \ref{eg:coproduct_parsumcat} and \ref{eg:infinite box}.
So there is a unique morphism of parsummable categories
\begin{equation}\label{eq:box_splitting_GF}
    \boxtimes_{(H)}\  (G\Fc)_{(H)}\ \to \ G\Fc 
\end{equation}
whose restriction to $(G\Fc)_{(H)}$ is the inclusion.

\begin{theorem}\label{thm:split_K(GF)}
Let $G$ be a group.   
As $H$ runs over a set of representatives of the conjugacy classes of finite index subgroups of $G$,
the morphism \eqref{eq:box_splitting_GF} is an isomorphism of parsummable categories,
and the canonical morphism
\[    {\bigvee}_{(H)} \ \bK_{\gl}((G\Fc)_{(H)})\ \xra{\ \simeq\ } \ \bK_{\gl}(G\Fc)\ .   \]
is a global equivalence of symmetric spectra.
\end{theorem}
\begin{proof}
Every finite subset of $\omega$ equipped with a $G$-action
is the disjoint union of its $(H)$-isotypical summands,
i.e., the $G$-invariant subset of those elements whose isotropy group is conjugate to $H$.
Moreover, isomorphisms of $G$-sets must preserve the isotypical decomposition.
So the morphism \eqref{eq:box_splitting_GF} is an isomorphism of parsummable categories.
Additivity of global K-theory (Theorem \ref{thm:K infinite additive}) then proves the
second claim.
\end{proof}

We will now investigate the summands $\bK_{\gl}((G\Fc)_{(H)})$
appearing in the wedge decomposition of Theorem \ref{thm:split_K(GF)} in more detail.
We can handle the extreme case $H=G$ right away:
$(G\Fc)_{(G)}$ is the category of finite subsets of $\omega$ equipped with
the trivial $G$-action, and the $G$-equivariant bijections between these.
So endowing a set with the trivial $G$-action is an isomorphism
of parsummable categories $\Fc\iso (G\Fc)_{(G)}$.
Hence the summand $\bK_{\gl}((G\Fc)_{(G)})$ is isomorphic to $\bK_{\gl}\Fc$,
and thus globally equivalent to the global sphere spectrum, by Theorem \ref{thm:global F}.

Our next step is to look at the other extreme $H=\{e\}$ of the trivial subgroup,
i.e., to study the global K-theory of free $G$-sets.
Since the isotropy subgroup must have finite index in $G$,
we must now restrict to finite groups.
For every group $G$, the nerve of the category of finitely generated
free $G$-sets is equivalent to the free $E_\infty$-space generated
by $B G$, the classifying space of $G$.
So the generalization of the Barratt-Priddy-Quillen theorem
mentioned in the introduction of Section \ref{sec:K of free I}
says that the K-theory of finitely generated free $G$-sets is stably equivalent
to the unreduced suspension spectrum of $B G$.
In Theorem \ref{thm:K of free G} below
we provide a global equivariant generalization of this result
for finite groups.

\begin{defn}
  We let $G$ be a finite group.
  We denote by $\Bc_{\gl}G$ the full subcategory of $G\Fc$
  on those objects for which the $G$-action is free and transitive.  
\end{defn}

The category $\Bc_{\gl}G$ is thus equivalent to the category with a single object
with $G$ as its endomorphisms.
The category $\Bc_{\gl}G$ is invariant under the $\M$-action on $G\Fc$,
and we write
\[ B_{\gl}G \ = \ \rho(\Bc_{\gl}G) \]
for the $\bI$-space associated with the $\M$-category $\Bc_{\gl}G$,
see Construction \ref{con:M2bI}.

The following proposition shows that
$\Bc_{\gl}G$ is an incarnation in the world of $\M$-categories
of the global classifying space of $G$, as defined in \cite[Definition 1.1.27]{schwede:global};
this also justifies the notation.
To properly compare things, we have to recall orthogonal spaces and
explain how to pass from orthogonal spaces to $\bI$-spaces.
We let $\bL$ denote the topological category whose objects are
all finite-dimensional real inner product spaces, and with morphisms
the linear isometric embeddings, topologized as Stiefel manifolds.
An {\em orthogonal space} in the sense of \cite[Definition 1.1.1]{schwede:global}
is a continuous functor from $\bL$ to the category of spaces.
Orthogonal spaces are also called $\mathscr I$-functors,
$\mathscr I$-spaces or $\mathcal I$-spaces by other authors.
The linearization functor
\[ \mR[-]\ : \ \bI \ \to \ \bL \]
takes a finite set $A$ to the free $\mR$-vector space $\mR[A]$
with $A$ as orthonormal basis; injections between finite sets
are $\mR$-linearly extended to linear isometric embeddings.
The {\em underlying $\bI$-space} of an orthogonal space $X$
is simply the composite functor
\[ \bI \ \xra{\ \mR[-]\ } \ \bL \ \xra{\ X \ }\ \bT\ .\]

\begin{prop}
  For every finite group $G$, the $\bI$-space $\rho(\Bc_{\gl}G)$
  is globally equivalent to the underlying $\bI$-space
  of the global classifying space $B_{\gl}G$ as defined in \cite[Definition 1.1.27]{schwede:global}.
\end{prop}
\begin{proof}
  For the course of the proof we write $u(X)=X\circ\mR[-]$
  for the underlying $\bI$-space of an orthogonal space $X$.
  The $G$-action on $\mR[G]$ induces a $G$-action
  on the represented orthogonal space $\bL_{\mR[G]}=\bL(\mR[G],-)$, and the quotient
  \[ B_{\gl}G\ = \ \bL_{\mR[G]}/G \]
  is a global classifying space of $G$ in the sense of \cite[Definition 1.1.27]{schwede:global}.
  We exhibit a chain of three global equivalences of $\bI$-spaces
  \begin{align}\label{eq:chain of I-equivalences}
    \rho(\Bc_{\gl}G)\
    \xla{\ \simeq \ } \ 
      (\rho(\Ec_{\gl}G)&\times u(\bL_{\mR[G]}\circ\Sym))/G\\
    &\xra{\ \simeq\ } \  u(\bL_{\mR[G]}\circ\Sym)/G\ \xla{\ \simeq\ } \
    u(\bL_{\mR[G]})/G\ = \ u(B_{\gl}G)\ .    \nonumber
  \end{align}
  We start by defining the relevant objects that appear in this chain.
  We define a $G$-$\M$-category by
  \[ \Ec_{\gl} \ = \ E I(G,\omega)\ , \]
  the contractible groupoid whose objects are all injections from $G$ to $\omega$.
  The monoidal category $\M$ acts by postcomposition;
  the group $G$ acts by translation on the source of the injections.
  Applying the functor $\rho:\M\cat\to\bI\bT$ yields a 
  $G$-$\bI$-space $\rho(\Ec_{\gl})$.
  Since this $G$-action on the category $\Ec_{\gl}$ is free,
  taking $G$-orbits commutes with the formation of nerves, so
  \begin{align*}
    \left( \rho(\Ec_{\gl}G)(A)\right)/G\
    &= \ | E I(G,\omega^A )| / G\
    \iso \ | ( E I(G,\omega^A )) /G| \ \iso\ | (\Bc_{\gl}G)[\omega^A]|\ = \ \rho(\Bc_{\gl}G)(A)\ .
  \end{align*}
  These isomorphisms are compatible with the structure maps, so they form
  an isomorphism of $\bI$-spaces
  \[ \rho(\Ec_{\gl}G)/G \ \iso \ \rho(\Bc_{\gl}(G))\ . \]

  We write $\Sym^n(V)=V^{\tensor n}/\Sigma_n$ for $n$-th symmetric power of a real inner product space $V$,
  and we write $\Sym(V)=\bigoplus_{n\geq 0}\Sym^n(V)$ for the symmetric algebra.
  As explained in \cite[Proposition 6.3.8]{schwede:global} (or rather its real analog),
  the symmetric algebra $\Sym(V)$ inherits a specific euclidean inner product from $V$,
  such that the canonical algebra isomorphism $\Sym(V)\tensor\Sym(W)\iso \Sym(V\oplus W)$
  becomes an isometry. Moreover, the inner product on $\Sym(V)$ is natural for
  linear isometric embeddings in $V$. So we obtain a $G$-orthogonal space by
  precomposing the $G$-orthogonal space $\bL_{\mR[G]}$ with the symmetric algebra functor,
  i.e.,
  \[ (\bL_{\mR[G]}\circ\Sym)(V)\ = \ \bL(\mR[G],\Sym(V))\ . \]
  This concludes the definition of the $\bI$-spaces
  that occur in the chain \eqref{eq:chain of I-equivalences}.
  
  Now we let $K$ be another finite group, and we let $A$ be a finite $K$-set with a free $K$-orbit.
  We claim that then 
  \[ \rho(\Ec_{\gl}G)(A) \text{\qquad and\qquad}  u(\bL_{\mR[G]}\circ\Sym)(A)\]
  are universal $(K\times G)$-spaces for the family of graph subgroups of $K\times G$.
  In other words, both are cofibrant as $(K\times G)$-spaces, the $G$-actions are free,
  and the fixed point spaces are contractible for the graphs of all
  homomorphisms $\alpha:L\to G$ defined on some subgroup $L$ of $K$.
  On the one hand,  $\rho(\Ec_{\gl}G)(A)$ is $(K\times G)$-cofibrant as the realization
  of a $(K\times G)$-simplicial set, and the $G$-action is free.
  We let $\alpha:L\to G$ be a homomorphism defined on a subgroup of $K$,
  and we let $\Gamma(\alpha)$ be its graph.
  The $\Gamma(\alpha)$-fixed points of $\rho(\Ec_{\gl}G)(A)$ are then given by
  \[ \left( \rho(\Ec_{\gl}G)(A) \right)^{\Gamma(\alpha)} \ = \
    |E I(G,\omega^A) |^{\Gamma(\alpha)} \ = \
    |E ( I(G,\omega^A)^{\Gamma(\alpha)})| \ = \
    |E ( I^L(\alpha^*(G),\omega^A)) | \ .
  \]
  Because $A$ contains a free $L$-orbit, $\omega^A$ is a universal $L$-set.
  So the set $I^L(\alpha^*(G),\omega^A)$ of $L$-equivariant injections from
  $\alpha^*(G)$ to $\omega^A$ is non-empty, and the above fixed point space is contractible.
  This completes the proof that $\rho(\Ec_{\gl}G)(A)$
  is a universal $(K\times G)$-space for the family of graph subgroups.

  On the other hand, the $(K\times G)$-space $u(\bL_{\mR[G]}\circ\Sym)(A)=\bL(\mR[G],\Sym(\mR[A]))$
  is $(K\times G)$-cofibrant by \cite[Proposition 1.1.19]{schwede:global}, and the $G$-action is free.
  The $\Gamma(\alpha)$-fixed points of $u(\bL_{\mR[G]}\circ\Sym)(A)$ are given by
  \[  \bL(\mR[G],\Sym(\mR[A]))^{\Gamma(\alpha)} \ = \  \bL^L(\alpha^*(\mR[G]),\Sym(\mR[A])) \ ,  \]
  the space of $L$-equivariant linear isometric embeddings from $\alpha^*(\mR[G])$
  into the symmetric algebra of $\mR[A]$.
  Since $A$ has a free $L$-orbit, the $L$-action on $A$ is faithful,
  and $\Sym(\mR[A])$ is a complete $L$-universe by \cite[Remark 6.3.22]{schwede:global}.
  So the space $\bL^L(\alpha^*(\mR[G]),\Sym(\mR[A]))$ is contractible by
  \cite[Proposition 1.1.21]{schwede:global}.
  This completes the proof that $u(\bL_{\mR[G]}\circ\Sym)(A)$
  is a universal $(K\times G)$-space for the family of graph subgroups.

  Because $\rho(\Ec_{\gl}G)(A)$ and $u(\bL_{\mR[G]}\circ\Sym)(A)$
  are universal $(K\times G)$-spaces for the same family of subgroups,
  the two projections from $\rho(\Ec_{\gl}G)(A)\times u(\bL_{\mR[G]}\circ\Sym)(A)$
  to each factor are $(K\times G)$-homotopy equivalences.
  Passing to $G$-orbit spaces thus yields two $K$-homotopy equivalences
  \begin{align*}
    \rho(\Bc_{\gl}G)(A)\ \iso \ \rho(\Ec_{\gl}G)(A)/G\
    &\xla{\ \simeq\ } \
      \left( \rho(\Ec_{\gl}G)(A)\times u(\bL_{\mR[G]}\circ\Sym)(A)\right)/G\\
    &\xra{\ \simeq\ } \  u(\bL_{\mR[G]}\circ\Sym)(A)/G \ .
  \end{align*}
  Since this holds for all finite groups $K$ and all finite $K$-sets
  with a free orbit,
  the left and middle morphism in \eqref{eq:chain of I-equivalences} are global
  equivalences of $\bI$-spaces.

  The right morphism $u(\bL_{\mR[G]})/G\to u(\bL_{\mR[G]}\circ\Sym)/G$
  in \eqref{eq:chain of I-equivalences}
  is induced by the natural linear isometric embedding $V\to \Sym(V)$ as
  the linear summand in the symmetric algebra.
  Because $\bL_{\mR[G]}(\Sym(V))/G$ is the colimit, along closed embeddings,
  of $\bL_{\mR[G]}(\bigoplus\Sym^{\leq n}(V))/G$,
  the embedding $B_{\gl}G=\bL_{\mR[G]}/G\to (\bL_{\mR[G]}\circ\Sym))/G$
  is a global equivalence of orthogonal spaces by Proposition 1.1.9 (ix)
  and Theorem 1.1.10 of \cite{schwede:global}.
\end{proof}

The category $\Bc_{\gl}G$ is contained in the full parsummable subcategory  
\[ (G\Fc)_{\free}\ = \ (G\Fc)_{(e)} \]
of $G\Fc$ consisting of the free $G$-sets.
Since the inclusion $\Bc_{\gl}G\to (G\Fc)_{\free}$
is a morphism of $\M$-categories, it freely extends to a morphism of parsummable categories
\[ \iota^\sharp \ : \ \mP(\Bc_{\gl}G)\ \to \ (G\Fc)_{\free} \ .\]
The morphism of symmetric spectra
\[ w \ :\ \Sigma^\infty_+ B_{\gl}G\ = \ \Sigma^\infty_+ \rho(\Bc_{\gl}G)\
  \to\  \bK_{\gl}(\mP(\Bc_{\gl}G)) \]
was defined in \eqref{eq:eq:define_w}.

\begin{theorem}\label{thm:K of free G}
  Let $G$ be a finite group.
  \begin{enumerate}[\em (i)]
  \item The morphism of parsummable categories $\iota^\sharp:\mP(\Bc_{\gl}G)\to (G\Fc)_{\free}$
    is an isomorphism.
  \item The morphism of symmetric spectra
    \[  \bK_{\gl}(\iota^\sharp)\circ w\ : \ \Sigma^\infty_+ B_{\gl}G\ \to \ \bK_{\gl}((G\Fc)_{\free}) \]
    is a global equivalence.
  \end{enumerate}
\end{theorem}
\begin{proof}
  (i) 
  The support of every object of $\Bc_{\gl}G$ is non-empty, so
  the $\Sigma_m$-action on $(\Bc_{\gl}G)^{\boxtimes m}$ is free.
  Hence the objects of $(\Bc_{\gl}G)^{\boxtimes m}/\Sigma_m$
  are $\Sigma_m$-equivalence classes of objects in $(\Bc_{\gl}G)^{\boxtimes m}$,
  i.e., unordered $m$-tuples of pairwise disjoint subsets of $\omega$,
  each equipped with a free and transitive $G$-action.
  The functor $\iota^\sharp$ sends such an equivalence class $(A_1,\dots,A_m)\cdot\Sigma_m$
  to the set $A_1\cup\dots\cup A_n$, endowed with the induced free $G$-action.
  Every finite free $G$-set is the disjoint union of its $G$-orbits,
  so the functor $i^\sharp$ is bijective on objects.
  
  The source and target of the functor $\iota^\sharp$ are groupoids, so it remains
  to show that $\iota^\sharp$ is an isomorphism of automorphism groups.
  Since the $\Sigma_m$-action on $(\Bc_{\gl}G)^{\boxtimes m}$ is free,
  endomorphisms of $(A_1,\dots,A_m)\cdot\Sigma_m$ in the orbit category
  $(\Bc_{\gl}G)^{\boxtimes m}/\Sigma_m$ are pairs $(\sigma,f)$ consisting of
  a permutation $\sigma\in\Sigma_m$ and a morphism $f:(A_1,\dots,A_m)\to
  (A_{\sigma^{-1}(a)},\dots,A_{\sigma^{-1}(m)})$ in the category
  $(\Bc_{\gl}G)^{\boxtimes m}$.
  Morphisms between any two objects of $\Bc_{\gl}G$ identify with elements of the  group $G$,
  and composition works out in such a way that the automorphism group
  of $(A_1,\dots,A_m)\cdot\Sigma_m$ in   $(\Bc_{\gl}G)^{\boxtimes m}/\Sigma_m$
  is isomorphic to the wreath product $\Sigma_m\wr G$.
  Hence the functor  $\iota^\sharp$ is also fully faithful, and thus an isomorphism of categories.
  This concludes the proof of claim (i).
  
  The morphism of symmetric spectra
  $\bK_{\gl}(\iota^\sharp):\bK_{\gl}(\mP(\Bc_{\gl}B))\to \bK_{\gl}((G\Fc)_{\free})$
  is an isomorphism by part (i).
  The morphism
  $w: \Sigma^\infty_+ B_{\gl}G\to \bK_{\gl}(\mP(\Bc_{\gl}B))$
  is a global equivalence of symmetric spectra by Theorem \ref{thm:global BPQ}.
  Together, this proves claim (ii).
\end{proof}

Now we return to the more general situation of a finite index subgroup $H$
of a group $G$ (which can be infinite).
We will now argue that the parsummable category $(G\Fc)_{(H)}$
of finite $G$-sets with $H$-isotropy is globally equivalent
to the parsummable category  $((W_G H)\Fc)_{\free}$
of free $W_G H$-sets, where $W_G H=(N_G H)/H$ is the Weyl group of $H$.
Since $H$ has finite index in $G$, the Weyl group $W_G H$ is finite,
so Theorem \ref{thm:K of free G} lets us identify the global K-theory
of free $W_G H$-sets with the suspension spectrum of $B_{\gl}(W_G H)$.

\begin{con}
  We let $H$ be a finite index subgroup of a group $G$.
  We define the `$H$-fixed point' functor
  \begin{equation}\label{eq:H-fix}
  (-)^H\ : \ (G\Fc)_{(H)}  \ \to \ ((W_G H)\Fc)_{\free} \ .   
  \end{equation}
  Objects of $(G\Fc)_{(H)}$ are finite subset $A$ of $\omega$ equipped with
  a $G$-action with isotropy groups conjugate to $H$.
  The functor takes such a $G$-set $A$ to the $H$-fixed set $A^H$,
  which is a finite subset of $\omega$ equipped with a free action of the Weyl group $W_G H$.
  On morphisms, the functor restricts a $G$-equivariant map $f:A\to B$
  to the $H$-fixed points $f^H:A^H\to B^H$.
  The functor \eqref{eq:H-fix} is clearly a morphism of parsummable categories.
\end{con}

\begin{prop}\label{prop:H2W}
  Let $H$ be a finite index subgroup of a group $G$.
  The $H$-fixed point functor \eqref{eq:H-fix} is a global equivalence of parsummable categories.
  So the induced morphism of symmetric spectra
  \[ \bK_{\gl}((-)^H)\ : \ \bK_{\gl}((G\Fc)_{(H)})  \ \to \ \bK_{\gl}(((W_G H)\Fc)_{\free}) \]
  is a global equivalence.
\end{prop}
\begin{proof}
  If $K$ is another finite group, then objects of $F^K((G\Fc)_{(H)})$
  are finite $K$-invariant subsets of $\omega^K$, equipped with a commuting $G$-action
  with $(H)$-isotropy; morphisms are $(K\times G)$-equivariant bijections.
  Objects of $F^K(((W_G H)\Fc)_{\free})$ are finite $K$-invariant subsets of $\omega^K$,
  equipped with a commuting free $W_G H$-action; morphisms are $(K\times W_G H)$-equivariant bijections.
  The functor
  \[ F^K((-)^H)\ :\ F^K((G\Fc)_{(H)})\ \to\  F^K(((W_G H)\Fc)_{\free}) \]
  takes $H$-fixed points, and it is an equivalence of categories.
  So $(-)^H$ is a global equivalence of parsummable categories.
  The induced functor of global K-theory spectra is then a global equivalence by
  Theorem \ref{thm:equivalence2equivalence}.
\end{proof}

Combining the splitting of Theorem \ref{thm:split_K(GF)},
the global equivalence of Proposition \ref{prop:H2W}
and the global equivalence of Theorem \ref{thm:K of free G}
yields the following corollary.

\begin{cor}\label{cor:K(GF) final}
  For every group $G$, the global K-theory of finite $G$-sets $\bK_{\gl}(G\Fc)$
  is globally equivalent to the wedge, indexed over conjugacy classes of
  finite index subgroups $H$ of $G$, of the symmetric spectra $\Sigma^\infty_+ B_{\gl}(W_G H)$,
  where $W_G H=(N_G H)/H$ is the Weyl group of $H$.
\end{cor}

Proposition \ref{prop:H2W} implies that for every finite index subgroup
$H$ of a group $G$, the morphism
\[ \pi_0(\bK_{\gl}((-)^H)) \ : \ \upi_0(\bK_{\gl}((G\Fc)_{(H)}))  \ \to \
  \upi_0(\bK_{\gl}(((W_G H)\Fc)_{\free})) \]
is an isomorphism of global functors.
This reduces the calculation of the Swan K-groups of $G$-sets with $H$-isotropy
to the special case of free actions of finite groups.
In this special case, the next theorem gives a purely algebraic description of this global functor.
\medskip

Now we exhibit an isomorphism, for finite groups $G$,
between the homotopy group global functor
$\upi_0(\bK_{\gl}((G\Fc)_{\free}))$ and the free global functor $\bA_G$ represented by $G$,
introduced in Example \ref{eg:free global functor}.
One line of approach would be to exploit the global equivalence
of Theorem \ref{thm:K of free G} between $\bK_{\gl}((G\Fc)_{\free})$
and the unreduced suspension spectrum of the global classifying space of $G$,
and to quote the more general calculation \cite[Proposition 4.2.5]{schwede:global}
of the homotopy group global functor $\upi_0(\Sigma^\infty_+ B_{\gl}G)$
for compact Lie groups $G$.
However, a complete argument would require us to translate the statements
of \cite{schwede:global} from the world of orthogonal spaces and orthogonal spectra 
into statements in the world of $\bI$-spaces and symmetric spectra used here,
using Hausmann's equivalence \cite[Theorem 5.3]{hausmann:global_finite}.
While this can certainly be done,
we opt for a more direct approach, using the results of Section \ref{sec:Swan}.
  
We define a specific equivariant homotopy class $u_G\in \pi_0^G(\bK_{\gl}((G\Fc)_{\free}))$.
For $g\in G$, we let 
\[ \chi_g\ : \ G \ \to \omega \]
by the indicator function, i.e,
\[ \chi_g(h)\ = \
  \begin{cases}
    1 & \text{ if $g=h$, and}\\
    0 & \text{ if $g\neq h$.}
  \end{cases}
\]
We endow the set of all indicator functions
\[ I\ = \ \{\chi_g\colon g\in G\} \ \subset \ \omega^G\]
with a free and transitive $G$-action by $\gamma\cdot \chi_g=\chi_{\gamma g}$.
This makes $I$ an object of the category $(G\Fc)_{\free}[\omega^G]=(G\Fc[\omega^G])_{\free}$.
Moreover, $I$ is also invariant under the $G$-action on 
$(G\Fc)_{\free}[\omega^G]$ induced by the $G$-action on $\omega^G$.
So $I$ is in fact an object of the  $G$-fixed category
\[ F^G((G\Fc)_{\free})\ = \ \left( (G\Fc)_{\free}[\omega^G]\right)^G\ .\]
We let $u_G\in\pi_0^G(\bK_{\gl}\Fc[G])$ be the image
of the class $[I]$ under the homomorphism
\[ \beta(G) \ : \ \pi_0(F^G ((G\Fc)_{\free}))\ \to \ \pi_0^G(\bK_{\gl}((G\Fc)_{\free})) \]
defined in \eqref{eq:F^G_to_K}. This means concretely that $u_G$ is the
homotopy class represented by the based $G$-map defined as the composite
\[ S^G \ \xra{I\sm -} \ |\gamma((G\Fc)_{\free})[\omega^G]|\sm S^G\  \xra{\text{assembly}} 
|\gamma((G\Fc)_{\free})[\omega^G]|(S^G)\ = \ (\bK_{\gl}((G\Fc)_{\free}))(G) \ . \]
As we explained in Example \ref{eg:free global functor},
the global functor $\bA_G$ represents evaluation at $G$,
with the universal element being the class of the $(G\times G)$-action
on $G$ by two-sided translation, for which we write $_G G_G$.

\begin{theorem}\label{thm:A_G and K GF}
For every finite group $G$, the morphism of global functors
\[ u \ : \ \bA_G \ \to \  \upi_0(\bK_{\gl}((G\Fc)_{\free})) \]
that sends the class $[_G G_G]$ in $\bA_G(G)$ to the class $u_G$
is an isomorphism of global functors.
\end{theorem}
\begin{proof}
  We let $K$ be another finite group.
  The objects of the category $F^K((G\Fc)_{\free})=((G\Fc)_{\free}[\omega^K])^K$
  are finite $K$-invariant subsets of the universal $K$-set $\omega^K$
  that are equipped with a commuting free $G$-action.
  Sending such an object to its isomorphism class induces a bijection
  \[ v(K)\ : \ \pi_0(F^K((G\Fc)_{\free})) \ \xra{\ \iso \ }\ \bA^+_G(K)\ .\]
  We omit the straightforward (but somewhat tedious) verification 
  that these maps are additive, and compatible with transfers and restrictions
  as the group $K$ varies.
  So the maps $v(K)$ form an isomorphism of pre-global functors
  $v:\upi_0( (G\Fc)_{\free})\iso \bA^+_G$.
  Moreover, the map  $v(G):\pi_0(F^G(G\Fc)_{\free})\iso \bA^+_G(G)$
  sends the class $[I]$ to the universal class $[_G G_G]$.
  So the pair $(\upi_0( (G\Fc)_{\free}),[I])$
  also represents evaluation at $G$ on the category of pre-global functors.
  The relation
  \[ u(i(v[I]))\ = \ u(i[_G G_G])\ = \  \beta(G)[I] \]
  in $\pi_0^G(\bK_{\gl}((G\Fc)_{\free}))$
  thus shows that the composite morphism of global functors
  \[ \upi_0( (G\Fc)_{\free})\ \xra[\iso]{\ v\ }\  \bA^+_G\ \xra{\ i \ }\ \bA_G
    \ \xra{\ u \ }\ \upi_0(\bK_{\gl}((G\Fc)_{\free}))\]
  coincides with $\beta: \upi_0( (G\Fc)_{\free})\to \upi_0(\bK_{\gl}((G\Fc)_{\free}))$.
  Since $\beta$ is a group completion of pre-global functors by Theorem \ref{thm:upi_0 beta}
  and $v$ is an isomorphism,
  the composite $u\circ i:\bA_G^+\to \upi_0(\bK_{\gl}((G\Fc)_{\free}))$
  is also a group completion of pre-global functors.
  So the extension $u:\bA_G\to \upi_0(\bK_{\gl}((G\Fc)_{\free}))$
  is an isomorphism of global functors.
\end{proof}

\begin{eg}
  The same method of proof as in Theorem \ref{thm:A_G and K GF}
  can also be used to give an algebraic description of
  the homotopy group global functor $\upi_0(\bK_{\gl}(G\Fc))$
  of the global K-theory of finite $G$-sets, for every group $G$ (possibly infinite).
  We sketch the argument, leaving some details to interested readers.

  We introduce a global functor $\bB_G$.  
  For a finite group $K$, we let $\bB_G(K)$ be the
  group completion (Grothendieck group) of the  abelian monoid,
  under disjoint union, of isomorphism classes of finite $(K\times G)$-sets.
  A restriction homomorphism $\alpha^*:\bB_G(H)\to\bB_G(K)$
  along a homomorphism $\alpha:K\to H$ between finite groups
  is induced by restriction of the $H$-action along $\alpha$.
  A transfer homomorphism $\tr_L^K:\bB_G(L)\to\bB_G(K)$
  for a subgroup inclusion $L\leq K$ is given by inducing up the $L$-action to a $K$-action.
  One more time we omit the straightforward verification of the fact
  that this indeed defines a global functor. We now sketch the construction
  of an isomorphism of global functors
  \[  \upi_0(\bK_{\gl}(G\Fc))\ \iso \ \bB_G\ .\]
  Theorem \ref{thm:upi_0 tilde beta} provides an isomorphism
  between the global functor $\upi_0(\bK_{\gl}(G\Fc))$
  and the global functor $\bK(G\Fc)$, the group completion of the pre-global functor $\upi_0(G\Fc)$.
  For a finite group $K$, the abelian monoid $\pi_0(F^K(G\Fc))$
  consists of isomorphism classes of finite $K$-invariant subsets of $\omega^K$
  endowed with a commuting $G$-action.
  So $\pi_0(F^K(G\Fc))$ is isomorphic to the monoid of isomorphism classes
  of finite $(K\times G)$-sets, and $\bB_G(K)$ is isomorphic to $\pi_0^K(\bK_{\gl}(G\Fc))$.
  Modulo the verification that the transfer and restriction homomorphisms
  in $\bB_G$ and $\pi_0(F^K(G\Fc))$ correspond under the previous identification,
  this constructs the desired isomorphism between
  $\upi_0(\bK_{\gl}(G\Fc))$ and $\bB_G$.
\end{eg}

\section{Global K-theory of rings}\label{sec:K(R)}

In this section we introduce and discuss the
global algebraic K-theory spectrum of a ring $R$,
always assumed to be associative and unital, but not necessarily commutative.
We want to advertise the global algebraic K-theory spectrum
as a compact and very rigid way of packaging the information
that is contained in the `representation K-theory' of $R$, i.e.,
in the various K-theory spectra of $R G$-lattices, as $G$ varies through all
finite groups. 

The construction is via a certain saturated parsummable category $\Pc(R)$
of finitely generated projective $R$-modules, introduced in Construction \ref{con:P(R)}.
The underlying non-equivariant homotopy type of $\bK_{\gl}R=\bK_{\gl}\Pc(R)$
is that of the direct sum K-theory 
of finitely generated projective $R$-modules, i.e., $\bK_{\gl}R$ refines
the classical algebraic K-theory spectrum.
More generally, for every finite group $G$, 
the $G$-fixed point spectrum $F^G(\bK_{\gl}R)$
has the homotopy type of the direct sum K-theory spectrum
of the category of $R G$-modules whose underlying $R$-modules
are finitely generated projective, see Theorem \ref{thm:fix of KR} (ii).
 The global functor $\upi_0(\bK_{\gl}R)$
that assigns to a finite group $G$ the 0-th equivariant homotopy group
$\pi_0^G(\bK_{\gl}R)$ is naturally isomorphic to the `Swan K-group` global functor
of $R$, whose value at $G$ is the Grothendieck group, with respect to direct sum, 
of $R G$-modules whose underlying $R$-modules
are finitely generated projective.
  
If the ring $R$ is commutative
then $\bK_{\gl}R$ admits the structure of a commutative symmetric ring spectrum,
but we will not show this here.
The strict commutativity of the symmetric ring spectrum $\bK_{\gl}R$
provides power operations on the global Green functor $\upi_0(\bK_{\gl}R)$;
under the isomorphism with the Swan K-group global functor,
these homotopical power operations coincide with the algebraic power operations
obtained by raising modules to a tensor power (over $R$).
So if $R$ is commutative, then the commutative multiplication on $\bK_{\gl}R$
is a compact and rigid way of packaging the 
multiplicative structure, power operations and norm maps that relate the K-theory
spectra  of $R G$-lattices as $G$ varies.

\begin{con}\label{con:P(R)}
  Given a ring $R$, we define a parsummable category $\Pc(R)$ whose underlying
  category is equivalent to the category of finitely generated projective $R$-modules.
  Given any set $U$ we denote by $R\{U\}$ the set of functions $f:U\to R$
  that are almost always zero, and we give $R\{U\}$ the pointwise $R$-module structure;
  we confuse elements of $U$ with their characteristic functions to view
  $U$ as a subset of $R\{U\}$; then $R\{U\}$ is a free $R$-module with basis $U$.
  
  The objects of the category~$\Pc(R)$ are all finitely generated $R$-submodules 
  $P$ of $R\{\omega\}$ such that the inclusion $P\to R\{\omega\}$
  is $R$-linearly splittable;
  morphisms in $\Pc(R)$ are $R$-module isomorphism.
  Since $R\{\omega\}$ is a free $R$-module,
  all the objects of $\Pc(R)$ are projective $R$-modules.
  
  Given an injection $u\in M$, we define the functor
  $u_*: \Pc(R)\to \Pc(R)$ on objects by
  \[ u_*(P) \ = \ R\{u\}(P)\ , \]
  the image of $P$ under the homomorphism $R\{u\}:R\{\omega\}\to R\{\omega\}$.
  Since $R\{u\}$ is $R$-linearly splittable, the module $R\{u\}(P)$
  is again splittable inside $R\{\omega\}$.
  This clearly defines an action of the injection monoid $M$ on the object set of $\Pc(R)$.
  We define an isomorphism of $R$-modules by
  \[ u_\circ^P\ = \ [u,1]^P = \ R\{u\}|_P \ : \ P \ \to \ R\{u\}(P)\ , \]
  the restriction of the monomorphism $R\{u\}$ to the submodule $P$.
  Then the relation \eqref{eq:basic_relation} holds,
  so there is a unique extension of this data to an $\M$-action on the
  category $\Pc(R)$, compare Proposition \ref{prop:minimal M-axioms}.

  Every object $P$ of $\Pc(R)$ is finitely generated,
  so it is contained in $R\{A\}$ for some finite subset $A$ of $\omega$.
  The object $P$ is then supported on the set $A$.
  So all objects of $\Pc(R)$ are finitely supported,
  and the $\M$-category $\Pc(R)$ is tame.
  
  The addition functor
  \[ + \ : \ \Pc(R)\boxtimes \Pc(R)\ \to \ \Pc(R) \]
  is the internal direct sum.
  Indeed, if two objects $P$ and $Q$ of $\Pc(R)$ are disjointly supported,
  then their internal sum inside $R\{\omega\}$ is direct, and the inclusion
  $P\oplus Q\to R\{\omega\}$ is again splittable.
\end{con}

\begin{defn}\label{def:K(R)}
  The {\em global algebraic K-theory spectrum}  $\bK_{\gl}R$ of a ring $R$ 
  is the global K-theory spectrum of the parsummable category $\Pc(R)$, i.e.,
  \[ \bK_{\gl}R\ = \ \bK_{\gl}\Pc(R)\ . \]
\end{defn}

As a special case of Theorem~\ref{thm:K_gl C is global Omega},
the symmetric spectrum $\bK_{\gl}R$ is a restricted global $\Omega$-spectrum. 

\begin{theorem}\label{thm:fix of KR} 
  Let $R$ be a ring.
  \begin{enumerate}[\em (i)]
  \item
    The parsummable category $\Pc(R)$ is saturated.
  \item
    For every finite group $G$, the $G$-fixed point spectrum
    $F^G(\bK_{\gl}R)$ has the stable homotopy type of the 
    direct sum K-theory spectrum of the category of finitely generated
    $R$-projective $R G$-modules.
  \item
    The homotopy group global functor $\upi_0(\bK_{\gl} R)$
    is isomorphic to the Swan K-group global functor.
\end{enumerate}
\end{theorem}
\begin{proof}
  (i) 
  We let $G$ be a finite group.
  The objects of the category $F^G(\Pc(R))=(\Pc(R)[\omega^G])^G$
  are finitely generated $R G$-sub\-modules $P$ of $R\{\omega^G\}$
  such that the inclusion $P\to R\{\omega^G\}$
  is $R$-linearly splittable. Such modules are projective over $R$,
  but {\em not} generally projective as $R G$-modules.
  Morphisms in $F^G(\Pc(R))$ are $R G$-linear isomorphisms.

  We claim that conversely, every finitely generated and $R$-projective $R G$-module
  $P$ is isomorphic to an object of $F^G(\Pc(R))$.
  Since $\omega^G$ is a universal $G$-set, it contains arbitrarily many distinct free $G$-orbits.
  So $R\{\omega^G\}$ contains a free $R G$-module of arbitrarily high rank as a direct summand.
  So it suffices to show that there is 
  an $R G$-linear map $P\to F$ to a finitely generated free $R G$-module
  that has an $R$-linear retraction.
  We let $P^*=\Hom_R(P,R)$ be the left $R$-dual of~$P$.
  Then $P^*$ is a right $R G$-module whose underlying $R$-module is
  finitely generated projective. So there is an epimorphism
  of right $R G$-modules $\epsilon:F\to P^*$ from a finitely generated free
  $R G$-module~$F$. Since $P^*$ is $R$-projective, $\epsilon$ admits an
  $R$-linear section.
  Taking right $R$-duals again, we arrive at a morphism of left $R G$-modules
  \[ \epsilon^* \ : \ (P^*)^*\ \to \ F^* \]
  that admits an $R$-linear retraction.
  Since $P$ is finitely generated projective over $R$, it is 
  $R G$-linearly isomorphic to its double dual $(P^*)^*$.
  Moreover, $F^*$ is a finitely generated free left $R G$-module.
  So we can embed any $P$ as above into a finitely generated free $R G$-module.
  This completes the proof that the category $F^G(\Pc(R))$ is equivalent
  to the category of finitely generated $R$-projective $R G$-modules
  and $R G$-linear isomorphisms. The latter is equivalent to
  the category $G\Pc(R)$ of $G$-object in $\Pc(R)$,
  so we have shown that the parsummable category $\Pc(R)$ is saturated.

  (ii)
  We let $\varphi:\mathbf 2\times\omega\to\omega$ be an injection.
  The parsummable category $\Pc(R)$ is saturated by part (i);
  so Corollary \ref{cor:saturated implies global} (ii)
  provides a chain of non-equivariant stable equivalences between
  the $G$-fixed symmetric spectrum $F^G(\bK_{\gl}R)$
  and the K-theory spectrum of the symmetric monoidal category
  $\varphi^*(G\Pc(R))=G(\varphi^*(\Pc(R)))$.
  Because $\Pc(R)$ is equivalent to the category of finite generated
  projective $R$-modules and $R$-linear isomorphisms,
  the category $G\Pc(R)$ is equivalent to the category of finite generated
  $R$-projective $R G$-modules and $R G$-linear isomorphisms.
  Under this forgetful equivalence,
  the symmetric monoidal structure $\varphi^*(G\Pc(R))$
  provided by Proposition \ref{prop:symmon_from_parsumcat}
  becomes the direct sum of $R G$-modules.
  This proves the claim.

  (iii) We let $G$ be a finite group.
  As we argued in part (ii), every object of $F^G(\Pc(R))$
  is a finitely generated $R$-projective $R G$-module,
  and conversely every such $R G$-module is isomorphic to an object in $F^G(\Pc(R))$.
  So the forgetful map from $\pi_0(F^G(\Pc(R)))$ to the abelian monoid
  of isomorphism classes of finitely generated $R$-projective $R G$-modules
  is an isomorphism. We omit the verification that the restriction and transfer maps
  in the pre-global functor $\upi_0(\Pc(R))$ correspond to the
  module-theoretic restriction and transfer maps.
  Passing from the pre-global functor $\upi_0(\Pc(R))$ to its group completion
  $\bK(\Pc(R))$ then proves part (iii).
\end{proof}

\begin{rk}
  We let $G$ be a finite group, which we let act trivially on the ring $R$.
  Merling \cite[Definition 5.23]{merling} defined the
  {\em equivariant algebraic K-theory spectrum}
  $\mathbf K_G(R)$ of the $G$-ring $R$ (with trivial $G$-action),
  which is an orthogonal $G$-spectrum.
  I expect that the underlying $G$-symmetric spectrum $(\bK_{\gl}R)_G$
  of our global K-theory spectrum is $G$-stably equivalent
  to the underlying $G$-symmetric spectrum of Merling's $\bK_G(R)$.
\end{rk}

\begin{rk}[(Global algebraic K-theory of free modules)]
The previous discussion admits a variation with free modules instead of 
projective modules. Indeed, for a ring $R$ we let $\Pc^\text{fr}(R)$ 
denote the full subcategory of~$\Pc(R)$ consisting of those finitely generated
splittable $R$-submodules of $R\{\omega\}$ that are free. This subcategory
is preserved by the action of the monoidal category $\M$ and closed
under the sum functor;
so it inherits a parsummable category structure such that the inclusion $\Pc^\text{fr}(R)\to \Pc(R)$
is a morphism of parsummable categories. We let 
\[ \bK_{\gl}^\text{fr}R\  = \ \bK_{\gl}(\Pc^\text{fr}(R))\]
be the associated restricted global $\Omega$-spectrum.
The difference between free and projective global K-theory is invisible 
to higher equivariant homotopy groups. Indeed, for every finite group $G$ the
inclusion $F^G(\Pc^\text{fr}(R))\to F^G(\Pc(R))$
of $G$-fixed parsummable categories is fully faithful and
cofinal in the sense of \cite[Section 2]{staffeldt-fundamental}.
This implies that the morphism of K-theory spectra
$\bK_{\gl}(F^G(\Pc^\text{fr}(R)))\to \bK_{\gl}(F^G(\Pc(R)))$
induces an isomorphism of non-equivariant homotopy groups in positive dimensions,
see for example \cite[Theorem 2.1]{staffeldt-fundamental}.
Corollary \ref{cor:G-fixed} implies that then the map
\[ \pi_k^G(\bK_{\gl}^\text{fr}R)\ \to \  \pi_k^G(\bK_{\gl}R ) \]
is an isomorphism for $k>0$.
\end{rk}

\begin{rk} Given a ring $R$ we can consider the exact sequence K-theory spectra
(as opposed to the {\em direct sum}  K-theory spectra)
of $R$-projective finitely generated 
$R G$-modules. As $G$ varies, these spectra also have contravariant
functoriality in arbitrary group homomorphisms and covariant functoriality
for subgroup inclusions. So they are candidates for the fixed point spectra
of a global homotopy type.
However, our approach relies on Segal's $\Gamma$-space machinery,
which cannot handle input of the kind of Waldhausen's
{\em categories with cofibrations and weak equivalences} \cite{waldhausen}.
I do not know if there is a symmetric spectrum $\bK^\text{ex}_{\gl}R$ (preferably 
a restricted global $\Omega$-spectrum)
such that $F^G(\bK^\text{ex}_{\gl}R)$ has the stable homotopy type
of the above exact sequence K-theory spectrum.
\end{rk}

\begin{rk}[(K-theory of group rings)]
  As we explained, the global algebraic K-theory spectrum $\bK_{\gl}R$   
  keeps track of the algebraic K-theory spectra of the categories
  of $R$-projective finitely generated $R G$-modules for all finite groups $G$.
  There is another prominent family of K-theory spectra associated with a ring $R$:
  the K-theory of the group rings $R G$, i.e.,  finitely generated $R G$-modules 
  that are projective over $R G$ (and not only over $R$).
  However, these K-theory spectra have the wrong kind of functoriality in 
  the group $G$ to form a global homotopy type. 
  
  The issue can already be seen on the level of the 0-th equivariant homotopy groups,
  i.e., for the Grothendieck groups $K_0(R G)$.
  As we explained in Example \ref{eg:homotopy global functor},
  a global homotopy type represented by a restricted global $\Omega$-spectrum $X$
  gives rise to a global functor $\upi_0(X)$,
  consisting of the abelian groups $\pi_0^G (X)$, indexed by finite groups $G$, 
  equipped with restriction maps $\alpha^*:\pi_0^G (X)\to\pi_0^K (X)$ for all group homomorphisms
  $\alpha:K\to G$, and transfer maps  $\tr_H^G:\pi_0^H (X)\to\pi_0^G (X)$ 
  for subgroup inclusions $H\leq G$.
  The collection of Grothendieck groups $K_0(R G)$ does have transfers and
  restrictions, but for {\em different} kinds of homomorphisms.
  Indeed, the assignment $G\mapsto K_0(R G)$ has restriction maps along monomorphisms,
  but {\em not} along arbitrary group homomorphisms.
  The issue is that if $\alpha:K\to G$ has a non-trivial kernel, 
  then restriction of scalars
  along $R\alpha:R K\to R G$ may not preserve projective modules.
  On the other hand,  $G\mapsto K_0(R G)$ has covariant functoriality
  for all group homomorphisms (and not just for monomorphisms).
  So for general rings $R$, there cannot be a restricted global $\Omega$-spectrum 
  realizing the system of Grothendieck groups $K_0(R G)$.
  If $R$ happens to be an algebra over the rational numbers, then `projective over
  $R G$' is equivalent to being projective as an underlying $R$-module,
  and in this special case our construction $\bK_{\gl}R$ realizes the 
  K-theory spectra of the group rings.
\end{rk}

\section{Global K-theory of permutative categories}
\label{sec:permutative category}

We recall that a {\em permutative category}
is a symmetric monoidal category in which the monoidal product is strictly
associative and unital (and the unit and associativity isomorphisms
are identity maps), see \cite[Definition 3.1]{elmendorf-mandell:permutative}.
The monoidal product induces an action of the Barratt-Eccles operad
on the geometric realization of the nerve of the permutative category,
giving it the structure of an $E_\infty$-space.
On the other hand, Segal's Construction \ref{con:Gamma from symmon}
assigns to every permutative category 
a special $\Gamma$-category, and hence a (non-equivariant) K-theory spectrum.
The link between the two constructions
is that the infinite loop space of the K-theory spectrum
is a group completion of the $E_\infty$-space.

In this section we offer a construction that turns a permutative category $\Ec$
into a parsummable category $\Phi(\Ec)$;
our construction is a variation of the `rectification'
of the permutative category in the sense of 
of Schlichtkrull and Solberg \cite[Section 7]{schlichtkrull-solberg},
see Remark \ref{rk:rectify}.
Moreover, the global K-theory spectrum $\bK_{\gl}(\Phi(\Ec)^{\sat})$
rigidifies the K-theory spectra of the permutative categories
of $G$-objects in $\Ec$ into a single global object,
where $\Phi(\Ec)^{\sat}$ is the saturation of $\Phi(\Ec)$
in the sense of Construction \ref{con:saturation_parsummable}.
In particular, the underlying non-equivariant stable homotopy type of 
$\bK_{\gl}(\Phi(\Ec)^{\sat})$ agrees with Segal's K-theory of $\Ec$.

\begin{con}\label{con:permutative2parsummable}
  We let~$\Ec$ be a permutative category with monoidal functor~$\oplus$
  and symmetry isomorphism~$\tau$.
  We denote by 0 the strict unit object, which is uniquely determined. 
  We construct a parsummable category
  \[ \Phi(\Ec)\ = \ \Phi(\Ec,\oplus,\tau)\ . \]
  Objects of the category $\Phi(\Ec)$ are all infinite sequences
  $\un{a}=(a_0,a_1,a_2,\dots)$ of objects of~$\Ec$ such that $a_i=0$
  (the unit object) for almost all $i\geq 0$. For every such sequence~$\un{a}$
  we define an~$\Ec$-object $\Sigma(\un{a})$ by 
  \begin{equation}\label{eq:define sum}
   \Sigma(\un{a})\ = \ a_0\oplus\dots\oplus a_m \ ,  
  \end{equation}
  where~$m$ is chosen large enough so that $a_i=0$ for all $i>m$.
  Since~$0$ is a strict unit object, this definition does not depend on the particular
  choice of~$m$.
  The morphisms in the category~$\Phi(\Ec)$ are then defined by
  \[ \Phi(\Ec)(\un{a},\un{b}) \ = \ 
    \Ec(\Sigma(\un{a}),\Sigma(\un{b}))\ .\]
  Composition in $\Phi(\Ec)$ is given by composition in $\Ec$.
  This finishes the definition of the category~$\Phi(\Ec)$.
  
  Now we define the~$\M$-action on~$\Phi(\Ec)$. Given $u\in M$
  and an object $\un{a}$ of~$\Phi(\Ec)$,
  we define the object $u_*(\un{a})$  of~$\Phi(\Ec)$ by
  \[
    u_*(\un{a})_k \ = \  
    \begin{cases}
      a_j & \text{\ if $k=u(j)$, and}\\
      0 & \text{\ if $k$ is not in the image of $u$.}
    \end{cases}
  \]
  This clearly defines an action of the injection monoid $M$ on the object set of $\Phi(\Ec)$.
  The non-unit objects that occur in the sequence~$u_*(\un{a})$  
  are the same, with multiplicities,
  as the non-unit objects that occur in the sequence $\un{a}$;
  so the objects
  $\Sigma(u_*(\un{a}))$ and $\Sigma(\un{a})$ are isomorphic.
  Even better, there is a specific isomorphism
  \[ u_\circ^{\un{a}} 
    : \  \Sigma(\un{a}) \ \to\ \Sigma(u_*(\un{a}))  \]
  obtained by iterated use of the symmetry isomorphism~$\tau$ in order
  to `shuffle' the letters $a_j$ into the appropriate places.
  In more detail, we choose a number $n\geq 0$ such that $a_i=0=(u_*(\un{a}))_i$
  for all $i>n$. Then we choose a permutation $\sigma$ of the set $\{0,1,\dots,n\}$
  such that $\sigma(i)=u(i)$ for all $0\leq i\leq n$.
  With these choices,
  \[ \Sigma(u_*(\un{a}))\ = \ \Sigma(\sigma_*(\un{a}))\ = \
    a_{\sigma^{-1}(0)}\oplus a_{\sigma^{-1}(1)}\oplus\dots\oplus a_{\sigma^{-1}(n)}\ . \]
  We define
  \[ u_\circ^{\un{a}} \ = \ \sigma_\circ^{\un{a}}\ : \
    a_0\oplus a_1\oplus\dots\oplus a_n\ \xra{\ \iso \ }\
    a_{\sigma^{-1}(0)}\oplus a_{\sigma^{-1}(1)}\oplus\dots\oplus a_{\sigma^{-1}(n)} \]
  as the symmetry isomorphism of the permutative structure
  corresponding to the permutation $\sigma$.
  It is routine -- but somewhat tedious -- to check that the isomorphisms $u_\circ^{\un{a}}$
  satisfy the relation \eqref{eq:basic_relation};
  so Proposition \ref{prop:minimal M-axioms} provides
  a unique extension of this data to an $\M$-action on the category $\Phi(\Ec)$.
  Every object $\un{a}$ of $\Phi(\Ec)$ is supported on 
  \[ \supp(\un{a}) \ = \ \{ i\in\omega \colon a_i\ne 0\}\ ,\]
  the set of non-zero coordinates.
  By definition, this set is finite, so the $\M$-category $\Phi(\Ec)$ is tame.
  
  Now we introduce the sum functor
  \[ + \ : \ \Phi(\Ec)\boxtimes\Phi(\Ec)\ \to \ \Phi(\Ec)\ . \]
  If $\un{a}$ and $\un{a}'$ are disjointly supported objects of $\Ec$,
  we set
  \[ (\un{a}+\un{a}')_i\ = \
    \begin{cases}
      \  a_i & \text{ if $i\in \supp(\un{a})$,}\\
      \  a'_i & \text{ if $i\in \supp(\un{a}')$, and}\\
      \  0 & \text{ otherwise.}
    \end{cases}
  \]
  The sum of two morphisms $f:\un{a}\to\un{b}$ and $f':\un{a}'\to\un{b}'$
  is the composite
  \[ \Sigma(\un{a}+\un{a}')\ \xra[\iso]{\tau_{\un{a},\un{a}'}} \
    \Sigma(\un{a})\oplus\Sigma(\un{a}')\ \xra{f\oplus f'}\
    \Sigma(\un{b})\oplus\Sigma(\un{b}')\ \xra[\iso]{\tau_{\un{b},\un{b}'}^{-1}}\
    \Sigma(\un{b}+\un{b}')\ ;\]
  here $\tau_{\un{a},\un{a}'}$ is the shuffle isomorphism that reorders the summands.
  The distinguished object of~$\Phi(\Ec)$ is the sequence
  $\un{0}=(0,0,0,\dots)$ consisting only of the unit object~$0$.
  The verification that the sum functor is associative and commutative
  uses the coherence properties of the symmetry isomorphism $\tau$.
  More precisely, we must use that the following squares
  of shuffle isomorphisms commute for all
  disjointly supported objects:
  \[ \xymatrix@C=6mm{ 
      \Sigma(\un{a}+\un{a}'+\un{a}'')\ar[rr]^-{\tau_{\un{a},\un{a}'+\un{a}''}}
      \ar[d]_{\tau_{\un{a}+\un{a}',\un{a}''}} &&
      \Sigma(\un{a})\oplus\Sigma(\un{a}'+\un{a}'')\ar[d]^{\Sigma(\un{a})\oplus\tau_{\un{a}',\un{a}''}}&
      \Sigma(\un{a}+\un{a}')\ar@{=}[rr]\ar[d]_{\tau_{\un{a},\un{a}'}} &&
      \Sigma(\un{a}'+\un{a})\ar[d]^-{\tau_{\un{a}',\un{a}}}  \\
      \Sigma(\un{a}+\un{a}')\oplus\Sigma(\un{a}'')\ar[rr]_-{\tau_{\un{a},\un{a}'}\oplus\Sigma(\un{a}'')} &&
      \Sigma(\un{a})\oplus\Sigma(\un{a}')\oplus\Sigma(\un{a}'')&
      \Sigma(\un{a})\oplus\Sigma(\un{a}')\ar[rr]_-{\tau_{\Sigma(\un{a}),\Sigma(\un{a}')}} &&
      \Sigma(\un{a}')\oplus\Sigma(\un{a})
    } \]
  This concludes the definition of the parsummable category  $\Phi(\Ec)$
  associated with a permutative category $(\Ec,\oplus,\tau)$.
\end{con}

\begin{rk}[(Relation to the Schlichtkrull-Solberg rectification)]\label{rk:rectify}
  In \cite[Section 7]{schlichtkrull-solberg},
  Schlichtkrull and Solberg introduce the {\em rectification} of a permutative category.
  This rectification is a functor from the category $\bI$ of finite sets and injections to
  the category of small categories, equipped with a strictly commutative multiplication
  for the Day convolution product.
  The term `rectification' refers to the fact that the `coherently commutative'
  product in the permutative category is turned into a strictly commutative multiplication,
  at the expense of enlarging the category to an $\bI$-category.

  Every $\bI$-category can be turned into an $M$-category by forming the
  colimit over the non-full subcategory of $\bI$ consisting of the
  inclusions $\{1,\dots,n\}\to \{1,\dots,m\}$ for all $0\leq n\leq m$.
  When applied to the rectification of $\Ec$,
  this yields the underlying $M$-category of the $\M$-category $\Phi(\Ec)$.
  Moreover, the colimit over the inclusions turns
  a commutative multiplication for the Day convolution product
  into a partially defined sum functor (i.e., defined on the box product).
  In this sense, our parsummable category $\Phi(\Ec)$
  is the result of taking the colimit over the inclusions
  of the Schlichtkrull-Solberg rectification of $\Ec$.
\end{rk}

\begin{rk}
  For a permutative category~$(\Ec,\oplus,\tau)$,
  we define two functors
  \[ c \ : \ \Ec \ \to\ \Phi(\Ec) \text{\qquad and\qquad}
    \Sigma \ : \ \Phi(\Ec) \ \to\ \Ec \ .\]
  The first functor is given on objects by $c(a)= (a,0,0,\dots)$; then $\Sigma(c(a))=a$.
  So we can define the effect of the functor~$c$ on a morphism $f:a\to b$ by $c(f)=f$.
  The second functor is given on objects by the sum \eqref{eq:define sum},
  and by the identity on morphisms.
  The composite $\Sigma\circ c$ is the identity functor of $\Ec$.
  The composite $c\circ \Sigma$ is naturally isomorphic to the identity of $\Phi(\Ec)$.
  So $c$ and $\Sigma$ are mutually inverse equivalences of categories.
  The square of functors
  \[ \xymatrix@C=10mm{
      \Phi(\Ec)\boxtimes \Phi(\Ec) \ar[d]_-+\ar[r]^-{\Sigma\times\Sigma} &\Ec\times\Ec\ar[d]^-{\oplus}\\
      \Phi(\Ec) \ar[r]_-{\Sigma}  &\Ec
    } \]
  does {\em not} commutes, but the shuffle isomorphisms
  $\tau_{\un{a},\un{a}'}:\Sigma(\un{a}+\un{a}')\iso \Sigma(\un{a})\oplus\Sigma(\un{a}')$
  form a natural isomorphism between the two composites.

  We explained in Proposition \ref{prop:symmon_from_parsumcat} 
  how a parsummable category $\Cc$ gives rise to a symmetric monoidal category $\varphi^*(\Cc)$,
  depending on a choice of injection $\varphi:\mathbf 2\times\omega\to\omega$.
  We invite the reader to perform a reality check, namely that the composite
  \[ \text{(permutative categories)} \ \xra{\ \Phi\ }\
    \parsumcat\ \xra{\ \varphi^*}\
    \ \text{(symmetric monoidal categories)} \]
  gives back the original permutative category, up to a strong symmetric monoidal equivalence.
  Granted this, Theorem \ref{thm:compare 2K} and the invariance of K-theory
  under strong symmetric monoidal equivalences show that the
  underlying non-equivariant stable homotopy type of $\bK_{\gl}\Phi(\Ec)$
  agrees with the K-theory of the permutative category $\Ec$.

  In more detail, the equivalence of categories
  $c : \Ec \to\Phi(\Ec)$ defined above can be extended to a strong symmetric monoidal functor
  from the permutative category $\Ec$ to the symmetric monoidal category $\varphi^*(\Phi(\Ec))$,
  as follows.
  For objects $a$ and $b$ of~$\Ec$, the object $\varphi_*(c(a),c(b))$
  of $\Phi(\Ec)$ is the sequence
  with value $a$ at $\varphi(1,0)$,
  with value $b$ at $\varphi(2,0)$,
  and with value 0 everywhere else.
  So the morphism
  \[  \beta_{a,b}\ : \  \Sigma(\varphi_*(c(a),c(b))) \ \to \   a\oplus b\ = \ \Sigma(c(a\oplus b))   \]
  defined by 
  \[ \beta_{a,b}\ =\
    \begin{cases}
      \Id_{a\oplus b} & \text{ if $\varphi(1,0)<\varphi(2,0)$, and}\\
      \ \tau_{b,a} & \text{ if $\varphi(1,0)>\varphi(2,0)$.}
    \end{cases}
  \]
  is an isomorphism from $\varphi_*(c(a),c(b))$ to $c(a\oplus b)$
  in the category~$\Phi(\Ec)$.
  We omit the verification that the isomorphisms $\beta_{a,b}$
  satisfy the associativity, commutativity and unit constraints for a
  strong symmetric monoidal functor.
\end{rk}

\begin{eg}\label{eg:permutative fin}
  We recall the permutative category~$\bSigma$
  with object set $\omega=\{0,1,2,\dots\}$, the set of natural numbers.
  There are no morphisms between different numbers and the endomorphism monoid of~$n$
  is the symmetric group $\Sigma_n$.
  The monoidal functor $+:\bSigma\times\bSigma\to\bSigma$
  is given by addition on objects and by `concatenation' on morphisms, i.e.,
  for $\sigma\in\Sigma_m$ and $\kappa\in\Sigma_n$ we set
  \[ (\sigma +\kappa)(i)\ = \ 
    \begin{cases}
      \ \sigma(i) & \text{ for $1\leq i\leq m$, and}\\  
      \kappa(i-m) +m & \text{ for $m+1\leq i\leq m+n$.}
    \end{cases} \]
  This monoidal structure is strictly associative and unital with unit object~0,
  and it becomes a permutative structure with respect to the
  symmetry isomorphism $\tau_{m,n}\in\Sigma_{m+n}$ defined by
  \[ \tau_{m,n}(i)\ = \ 
    \begin{cases}
      i+n & \text{ for $1\leq i\leq m$, and}\\  
      i-m & \text{ for $m+1\leq i\leq m+n$.}
    \end{cases} \]
  The category~$\bSigma$ is a skeleton of the category of finite sets
  and bijections, with monoidal structure corresponding to disjoint union. 
  So it should not come as a surprise that the associated~parsummable category
  $\Phi(\bSigma)$ is equivalent, as a parsummable category,
  to the~parsummable category~$\Fc$ of finite sets, introduced in Example \ref{eg:Fc}.
  An explicit equivalence of~parsummable categories
  \[ \epsilon \ : \ \Fc \ \to \ \Phi(\bSigma) \]
  is given as follows. On objects, $\epsilon$ sends a finite subset~$A$ of~$\omega$
  to its characteristic function $\chi_A:\omega\to\omega$ defined by
  \[ \chi_A(i)\ = \
    \begin{cases}
      0 & \text{ for $i\not\in A$, and}\\
      1 & \text{ for $i\in A$.}
    \end{cases}\]
  Then $\bSigma(\chi_A)$ is equal to the cardinality of~$A$; so there is a unique
  order preserving bijection
  \[ \lambda_A \ : \ \mathbf n \ \xra{\ \iso\ } \ A \ .\]
  On morphisms, the functor~$\epsilon$ sends a bijection~$f:A\to B$ between
  finite subsets of~$\omega$ of cardinality~$n$
  to the permutation
  \[  \epsilon(f)\ = \ \lambda_B^{-1}\circ f\circ \lambda_A \ \in\ \Sigma_n\ .\]
  The functor~$\epsilon$ is clearly an equivalence of categories,
  and we omit the verification that it is also a morphism of~parsummable categories.
  Granted this, Proposition \ref{prop:tame2global M-version}
  shows that the parsummable category $\Phi(\bSigma)$ is saturated,
  and that the morphism $\epsilon$ is a global equivalence.
  Theorem \ref{thm:equivalence2equivalence} shows
  that the induced morphism of symmetric spectra
  \[ \bK_{\gl}\epsilon \ : \  \bK_{\gl}\Fc \ \to \  \bK_{\gl}\Phi(\bSigma) \]
  is a global equivalence. Hence $\bK_{\gl}\Phi(\bSigma)$ is also globally
  equivalent to the global sphere spectrum, by Theorem \ref{thm:global F}.
\end{eg}

\begin{con}[(Functoriality for strong monoidal functors)]
  We discuss the functoriality of the assignment $\Ec\mapsto\Phi(\Ec)$.
  We let $F:\Dc\to\Ec$ be a strong monoidal functor between permutative categories,
  and we also require that $F$ strictly preserves the zero object.
  `Strong monoidal' is extra data (and not a property), namely
  an isomorphism $\eta:0\iso F(0)$ and a natural isomorphism
  \[ \mu_{d,d'} \ : \ F(d)\oplus F(d')\ \iso \ F(d\oplus d') \]
  of functors $\Dc\times\Dc\to\Ec$; moreover, these isomorphisms are required to
  be unital, associative, and commutative,
  in the sense spelled out in \cite[Section XI.2]{maclane-working}.
  We say that $F$ {\em strictly preserves the zero object}
  if  $F(0)$ is {\em equal} to $0$, and moreover $\eta$ is the identity.

  We will usually follow the common abuse 
  and drop the natural isomorphism $\mu$ from the notation.
  As we shall now explain, the strong monoidal functor $F$
  that strictly preserves the zero object
  induces a morphism of parsummable categories
  \[ \Phi(F)\ = \ \Phi(F,\mu)\ : \ \Phi(\Dc)\ \to \ \Phi(\Ec) \ .\]
  On objects, this functor is given by
  \[ (\Phi(F)(\un{d}))_i\ = \ F(d_i)\ . \]
  The coherence properties of the natural isomorphism $\mu$
  imply that it extends uniquely to natural isomorphisms
  \[ \mu_{\un{d}} \ : \ \Sigma(\Phi(F)(\un{d}))\ = \
    {\bigoplus}_{i\geq 0} F(d_i)\ \xra{\ \iso \ }\
    F\left( {\bigoplus}_{i\geq 0} d_i \right)\ = \ F(\Sigma(\un{d}))\ .\]
  A morphism $f:\un{d}\to\un{d}'$ in $\Phi(\Dc)$ is a
  $\Dc$-morphism $\Sigma(\un{d})\to\Sigma(\un{d}')$;
  the functor $\Phi(F)$ sends this to the morphism $\Phi(F)(f):\Phi(F)(\un{d})\to \Phi(F)(\un{d}')$
  defined as the following composite
  \begin{align*}
    \Sigma(\Phi(F)(\un{d}))\ \xra{\mu_{\un{d}}}\  F(\Sigma(\un{d}))\ \xra{F(f)}\
    F(\Sigma(\un{d}'))\ \xra{\mu_{\un{d}'}^{-1}}\ \Sigma(\Phi(F)(\un{d}'))\ .
  \end{align*}
  These definitions clearly make $\Phi(F)$ into a functor.
  We omit the detailed verification that $\Phi(F)$ is a morphism of parsummable categories,
  and just give a few hints about the facts that enter.
  For example, the property that $\Phi(F)$ is compatible with the $\M$-action amounts
  to the relations
  \begin{itemize}
  \item $\Phi(F)(u_*(\un{a}))=u_*(\Phi(F)(\un{a}))$ (obvious from the definitions), and
  \item $\Phi(F)(u_\circ^{\un{a}})=u_\circ^{\Phi(F)(\un{a})}$, which is consequence of
    the commutativity of the natural isomorphism $\mu$.
  \end{itemize}
  The fact that $\Phi(F)$ is additive on disjointly supported objects
  is straightforward from the definitions.
  The most involved step is probably to check that
  $\Phi(F)$ is additive on morphisms between disjointly supported objects.
  After unraveling all definitions, this comes down to the naturality
  of the isomorphism $\mu:F\oplus F\iso F\circ\oplus$ and
  the commutativity of the following diagram: 
  \[ \xymatrix@C=25mm{
      \Sigma(\Phi(F)(\un{a})+\Phi(F)(\un{a}'))\ar[r]^-{\tau_{\Phi(F)(\un{a}),\Phi(F)(\un{a}')}}\ar@{=}[d] &
      \Sigma(\Phi(F)(\un{a}))\oplus\Sigma(\Phi(F)(\un{a}')) \ar[d]^{\mu_{\un{a}}\oplus\mu_{\un{a}'}}\\
      \Sigma(\Phi(F)(\un{a}+\un{a}')) \ar[d]_{\mu_{\un{a}+\un{a}'}} &
      F(\Sigma(\un{a}))\oplus F(\Sigma(\un{a}')) \ar[d]^{\mu_{\Sigma(\un{a}),\Sigma(\un{a}')}}  \\
      F(\Sigma(\un{a}+\un{a}'))\ar[r]_-{F(\tau_{\un{a},\un{a}'})}&
      F(\Sigma(\un{a})\oplus\Sigma(\un{a}')) 
    } \]
  This diagram, finally, commutes because the natural isomorphism $\mu$
  is commutative.
\end{con}

\begin{prop}
  Let $F:\Dc\to\Ec$ be a strong symmetric monoidal functor between permutative categories
  that strictly preserves the zero object.
  If $F$ is an equivalence of underlying categories, then the functor
  $\Phi(F):\Phi(\Dc)\to\Phi(\Ec)$ is a global equivalence of parsummable categories.
\end{prop}
\begin{proof}
  We consider the commutative square of categories and functors
  \[ \xymatrix@C=12mm{
      \Dc \ar[r]^-F\ar[d]_-c &\Ec\ar[d]^c\\
      \Phi(\Dc) \ar[r]_-{\Phi(F)} &\Phi(\Ec)
    } \]
  The two vertical functors are equivalences by the above, and $F$ is an equivalence by
  hypothesis. So the functor $\Phi(F)$ is an equivalence of underlying categories.
  Now we let $G$ be a finite group and $\Uc$ a universal $G$-set.
  Because $\Phi(F)$ is fully faithful and $G$-equivariant, the restriction
  \[  (\Phi(F)[\Uc])^G \ : \  (\Phi(\Dc)[\Uc])^G  \ \to\ (\Phi(\Ec)[\Uc])^G \]
  to $G$-fixed subcategories is fully faithful.
  To see that $(\Phi(F)[\Uc])^G$ is also essentially surjective, we consider
  a $G$-fixed object $\un{a}$ of  $(\Phi(\Ec)[\Uc])^G$.
  Being $G$-fixed means that the coordinates of $\un{a}$ are constant on
  the $G$-orbits of $\Uc$.
  For every $G$-orbit $G i$ we choose an object $b_{G i}$ of $\Dc$ and an isomorphism
  $f_{G i}:a_{G i}\to F(b_{G i})$ in $\Ec$; we insists that $b_{G i}=0$ whenever $a_{G i}=0$.
  This is possible because $F$ is an equivalence of categories.
  We let $\un{b}$ be the object of $\Phi(\Dc)[\Uc]$ that is constant with value $b_{G i}$
  on the orbit $G i$. Then $\un{b}$ is a $G$-fixed object of $\Phi(\Dc)[\Uc]$,
  and
  \[ \bigoplus_{k\geq 0} f_k \ : \ \Sigma(\un{a})\ \to \ \Sigma(\un{b}) \]
  is a $G$-fixed isomorphism from $\un{a}$ to $\Phi(F)[\Uc](\un{b})$.
  This proves that $(\Phi(F)[\Uc])^G$ is essentially surjective, and hence an equivalence of categories.
\end{proof}

While the parsummable category $\Phi(\bSigma)$ discussed in
Example \ref{eg:permutative fin} is saturated, this is rather untypical.
To understand this better, we let $G$ be a finite group and
$\lambda:\omega^G\to\omega$ an injection.
We investigate which kinds of $G$-objects in $\Ec$
are in the image of the fully faithful functor $\lambda_\flat:F^G(\Phi(\Ec))\to G\Phi(\Ec)$
defined in \eqref{eq:lambda sharp}. The functor  $\Sigma:\Phi(\Ec)\to \Ec$
that sums up the objects in a tuple is an equivalence of categories, so we may equivalently
study the essential image of the composite functor
\begin{equation}\label{eq:non-sat}
  F^G(\Phi(\Ec))\ \xra{\ \lambda_\flat\ }\ G\Phi(\Ec) \ \xra{\ G\Sigma\ } \ G\Ec\ .
\end{equation}

\begin{prop}
  Let $\Ec$ be a permutative category, $G$ a finite group
  and $\lambda:\omega^G\to\omega$ and injection.
  Then the essential image of the fully faithful functor
  \[    (G\Sigma)\circ \lambda_\flat\ : \  F^G(\Phi(\Ec))\ \to\ G\Ec \]
  consists of all $G$-objects that are isomorphic to a finite monoidal product
  of $G$-objects of the form $\bigoplus_{G/H} x$ for subgroups $H$ of $G$ and
  objects $x$ of $\Ec$.
\end{prop}
\begin{proof}
  We let $\un{a}=(a_0,a_1,\dots)$ be a $G$-fixed object of $\Phi(\Ec)[\omega^G]$.
  Then $a_f=a_{l^g(f)}$ for all $g\in G$ and $f\in\omega^G$, i.e.,
  the coordinate objects must be constant on the $G$-orbits of $\omega^G$.
  We suppose first that $\un{a}$ is concentrated on a single $G$-orbit $G\cdot f$
  for some $f\in\omega^G$. If $H$ is the $G$-stabilizer group of $f$, then
  $\Sigma(\lambda_\flat(\un{a}))$ is $G$-equivariantly isomorphic to the object $\bigoplus_{G/H} a_i$.
  If $\un{a}$ is concentrated on several $G$-orbits, then the functor
  \eqref{eq:non-sat} takes $\un{a}$ to the monoidal product of $G$-objects
  coming from the different $G$-orbits. This proves the claim.
\end{proof}

As the previous proposition illustrates, the parsummable category $\Phi(\Ec)$
associated with a permutative category $\Ec$
is typically not saturated.
We can fix this by applying the saturation procedure of Theorem \ref{thm:saturation} 
and obtain a saturated parsummable category
\[ \Phi(\Ec)^{\sat} \ = \ \cat(\M,\Phi(\Ec))^\tau \]
and a strict morphism of parsummable categories
\[ s \ : \ \Phi(\Ec) \ \to \  \Phi(\Ec)^{\sat} \ .\]
Moreover, the morphism $s$ is an equivalence of underlying categories.

\oneappendix

\section{Transfers for equivariant \texorpdfstring{$\Gamma$}{Gamma}-spaces}

In this appendix we compare two different transfers that arise from a special equivariant
$\Gamma$-space: the `unstable' transfers arising from the specialness hypothesis,
and the homotopy theoretic transfer for the associated orthogonal $G$-spectrum.
The precise formulation can be found in Proposition \ref{prop:Gamma compatibilities}.
The results in this appendix ought to be well-known to experts, but I do not
know of a reference.
While we only care about {\em finite} groups in the body of this paper,
the arguments of this appendix apply just as well for {\em compact Lie groups},
and we work in this generality.
Also, in contrast to the main part of the paper, we now work with
{\em orthogonal spectra} (as opposed to symmetric spectra).\medskip

We recall that $\Gamma$ denotes the category with objects the finite based sets 
$n_+=\{0,1,\dots,n\}$, with basepoint 0; morphisms in $\Gamma$ are all based maps.
We let $G$ be a compact Lie group.
A {\em $\Gamma$-$G$-space} is a functor $F:\Gamma\to G\bT$ from $\Gamma$
to the category of $G$-spaces which is reduced,
i.e., $F(0_+)$ is a one-point $G$-space.
We may then view $F$ as a functor to {\em based} $G$-spaces,
where $F(n_+)$ is pointed by the image of the map $F(0_+)\to F(n_+)$
induced by the unique morphism $0_+\to n_+$ in $\Gamma$.

We let $S$ be a finite set.
Given $s\in S$ we denote by $p_s:S_+\to 1_+=\{0,1\}$ 
the based map with $p_s^{-1}(1)=\{s\}$.  For a $\Gamma$-$G$-space $F$, we denote by
\[ P_S \ : \ F(S_+) \ \to \ \map(S,F(1_+))\]
the map whose $s$-component is $F(p_s):F(S_+)\to F(1_+)$.

Now we let the Lie group $G$ also act on the finite set $S$.
It goes without saying that actions of compact Lie groups are required to
be continuous and that the use of the term `set' (as opposed to `space')
implies the discrete topology on the set; so the identity path component of $G$
must act trivially on every G-set.
In this situation, the spaces $F(S_+)$ and $\map(S,F(1_+))$ have two commuting $G$-actions:
the `external' action is the value at $S_+$ and at $1_+$ of the $G$-action on $F$;
the `internal' action is induced by the $G$-action on~$S$.
In this situation the map $P_S:F(S_+)\to\map(S,F(1_+))$
is $(G\times G)$-equivariant.
We endow $F(S_+)$ and $\map(S,F(1_+))$ with the diagonal $G$-action
and the conjugation action, respectively; 
then the map $P_S:F(S_+)\to\map(S,F(1_+))$ is $G$-equivariant.

\begin{defn}
  Let $G$ be a compact Lie group. 
  A $\Gamma$-$G$-space $F$ is {\em special}
  if for every closed subgroup $H$ of $G$ and every finite $H$-set $S$ the map
  \[ (P_S)^H \ : \ F(S_+)^H\ \to \ \map^H(S,F(1_+)) \]
  is a weak equivalence.
\end{defn}

\begin{con}\label{con:addition on Gamma-space}
We recall the algebraic structure on the path components of the fixed point spaces
of a special $\Gamma$-$G$-space $F$.
We let $p_1,p_2:2_+\to 1_+$ denote the two projections. The map
\[ (F(p_1)^G, F(p_2)^G) \ :\  F(2_+)^G\ \to \ F(1_+)^G \times F(1_+)^G \]
is a weak equivalence by specialness for the trivial $G$-set $S=\{1,2\}$.
We let $\nabla:2_+\to 1_+$ denote the fold map. We obtain a diagram of set maps
\begin{align*}
  \pi_0(F(1_+)^G)\times\pi_0 F(1_+)^G)
  \ \xla[\iso]{( \pi_0 (F(p_1 )^G),\pi_0(F(p_2)^G))}\
  \pi_0(F(2_+)^G)\ \xra{\pi_0(F(\nabla)^G)}\ \pi_0(F(1_+)^G)  
\end{align*}
the left of which is bijective. So the map
\begin{align*}
+ \ = \ \pi_0(F(\nabla)^G)\circ(\pi_0(F(p_1)^G),&\pi_0(F(p_2)^G))^{-1} \ :\\
&\pi_0(F(1_+)^G)\times\pi_0(F(1_+)^G)\to\pi_0(F(1_+)^G)  
\end{align*}
is a binary operation on the set $\pi_0(F(1_+)^G)$.
If $\tau:2_+\to 2_+$ is the involution that interchanges 1 and 2, then composition
with $\tau$ fixes $\nabla$ and interchanges $p_1$ and $p_2$; this implies that the operation + is
commutative. Contemplating the different ways to fold and project from the
based set $3_ +$ leads to the proof that the operation is also associative, and hence
an abelian monoid structure on the set $\pi_0(F(1_+)^G)$.
For every closed subgroup $H$ of $G$, the underlying $\Gamma$-$H$-space
is again special. So the same argument provides an abelian monoid structure on $\pi_0(F(1_+)^H)$.

If $K\leq H\leq G$ are nested closed subgroups of $G$, then the $H$-fixed points $F(1_+)^H$
are contained in the $K$-fixed points $F(1_+)^K$. The inclusion induces a
{\em restriction map}
\[ \res^H_K\ : \ \pi_0(F(1_+)^H)\ \to \ \pi_0(F(1_+)^K)  \]
which is clearly a monoid homomorphism for the previously defined additions.

For every closed subgroup $H$ of $G$ and every $g\in G$, we write
$H^g=g^{-1}H g$ for the conjugate subgroup.
Left multiplication by $g$ is then a homeomorphism
\[ l^g \ : \ F(1_+)^{H^g} \ \xra{\ \iso \ }\ F(1_+)^H\ ,\quad l^g(x)\ = \ g x\ .\]
The {\em conjugation homomorphism}
\[ g_\star\ = \ \pi_0(l^g) \ :\
\pi_0(F(1_+)^{H^g}) \ \to\ \pi_0(F(1_+)^H) \]
is the induced map on path components. We omit the verification that
the conjugation homomorphism is indeed additive.

Now we consider nested closed subgroups $K\leq H\leq G$, and we also assume
that $K$ has finite index in $H$.
In this situation, there is also a transfer homomorphism from $\pi_0(F(1_+)^K)$ to $\pi_0(F(1_+)^H)$,
defined as follows. Specialness for the finite $H$-set $H/K$ shows that the map
\[ (P_{H/K})^H \ : \ F(H/K_+)^H\ \to \ \map^H(H/K,F(1_+)) \]
is a weak equivalence. Evaluation at the preferred coset $e K$ identifies
the target with the space $F(1_+)^K$. Hence the map $\Psi=p_{e K}:H/K_+\to 1_+$
induces a weak equivalence
\[ F(\Psi)^H\ : \  F(H/K_+)^H\ \xra{\ \simeq \ } \ F(1_+)^K\ . \]
We let $\nabla:H/K_+\to 1_+$ denote the fold map.
We obtain a diagram of set maps
\begin{align*}
  \pi_0(F(1_+)^K)  \ \xla[\iso]{\pi_0(F(\Psi)^H)}\ \pi_0(F(H/K_+)^H)\ \xra{\pi_0(F(\nabla)^H)}\ \pi_0(F(1_+)^H)  
\end{align*}
the left of which is bijective. The {\em transfer} is the map
\[ \tr_K^H \ = \ \pi_0(F(\nabla)^H)\circ\pi_0(F(\Psi)^H)^{-1} \ :\ \pi_0(F(1_+)^K)\ \to\ \pi_0(F(1_+)^H)  \ . \]
We omit the verification that these finite index transfer maps are homomorphisms of abelian monoids,
that they are transitive, compatible with conjugation,
and that they satisfy a double coset formula with respect to restriction maps.
\end{con}

We let $G$ be a compact Lie group and $F$ a $\Gamma$-$G$-space.
We write $F(\mS)$ for the orthogonal $G$-spectrum
obtained by evaluating the prolongation of $F$ on spheres.
So the value of $F(\mS)$ at an inner product space $V$ is the space $F(S^V)$,
the value of $F$ at the sphere $S^V$.
The structure map $\sigma_{V,W}:F(\mS)(V)\sm S^W\to F(\mS)(V\oplus W)$
is the composite of the assembly map \eqref{eq:assembly}
and the effect of the canonical homeomorphism $S^V\sm S^W\iso S^{V\oplus W}$.

The collection of equivariant homotopy groups of an orthogonal $G$-spectrum
also support restriction, transfer and conjugation maps,
see for example Construction 3.1.5, Construction 3.2.22 and Remark 3.1.7 of \cite{schwede:global},
respectively.
We wish to compare these homotopy operations for the orthogonal $G$-spectrum $F(\mS)$
with the previous operations for the $\Gamma$-$G$-space $F$.
We let $H$ be a closed subgroup of $G$.
We identify $1_+\iso S^0$ by the unique isomorphism of based sets.
We write $\beta(H):\pi_0(F(1_+)^H)\to\pi_0^H(F(\mS))$ for the canonical map
\[ \pi_0(F(1_+)^H)\ = \ [S^0,F(S^0)]^H \ \to \ \colim_V\,[S^V,F(S^V)]^H \ = \
  \pi_0^H(F(\mS)) \ ;\]
the colimit is over the poset of finite-dimensional $G$-subrepresentations
of a complete $G$-universe.

In order to gain homotopical control over the values of a prolonged $\Gamma$-$G$-space,
we impose a mild non-degeneracy condition, following \cite[Definition B.33]{schwede:global}.
We say that a $\Gamma$-$G$-space $F$ is {\em $G$-cofibrant}
if for every $n\geq 1$ the latching map
\[  \colim_{U\subsetneq \{1,\dots,n\}} \, F(U_+)\ \to \  F(n_+)  \]
is a $(G\times\Sigma_n)$-cofibration.
Here the colimit is formed over the poset of proper subsets of the set $\{1,\dots,n\}$,
i.e., over an $n$-cube without the terminal vertex.

The following proposition should not be surprising, and must be well-known to the experts;
however, I am unaware of a published reference, so I include a proof here.

\begin{prop}\label{prop:Gamma compatibilities}
  Let $G$ be a compact Lie group, and let $F$ be a $G$-cofibrant, special $\Gamma$-$G$-space.
  \begin{enumerate}[\em (i)]
  \item For every closed subgroup $H$ of $G$,
    the map $\beta(H):\pi_0( F(1_+)^H)\to \pi_0^H( F(\mS))$ is
    additive and a group completion of abelian monoids.
  \item Let $K\leq H$ be nested closed subgroups of $G$.
    Then the following diagram of monoid homomorphisms commutes:
    \[ \xymatrix@C=15mm{
        \pi_0( F(1_+)^H)\ar[r]^-{\beta(H)}\ar[d]_{\res^H_K}&
        \pi_0^H( F(\mS))\ar[d]^{\res^H_K}\\
        \pi_0( F(1_+)^K)\ar[r]_-{\beta(K)}& \pi_0^K( F(\mS)) } \]
      \item Let $K\leq H$ be nested closed subgroups of $G$, such that moreover $K$ has finite index in $H$.
    Then the following diagram of monoid homomorphisms commutes:
    \[ \xymatrix@C=15mm{
        \pi_0( F(1_+)^K)\ar[r]^-{\beta(K)}\ar[d]_{\tr^H_K}& \pi_0^K( F(\mS))\ar[d]^{\tr^H_K} \\
        \pi_0( F(1_+)^H)\ar[r]_-{\beta(H)} &        \pi_0^H( F(\mS))            } \]
  \item  For every closed subgroup $H$ of $G$ and every element $g\in G$,
    the following diagram of monoid homomorphisms
    commutes:
    \[ \xymatrix@C=15mm{
        \pi_0( F(1_+)^{H^g})\ar[r]^-{\beta(H^g)}\ar[d]_{g_\star}& \pi_0^{H^g}( F(\mS))\ar[d]^{g_\star} \\
        \pi_0( F(1_+)^H)\ar[r]_-{\beta(H)} &
        \pi_0^H( F(\mS))            } \]

  \end{enumerate}
\end{prop}
\begin{proof}
  For every closed subgroup $H$ of $G$,
  the underlying $\Gamma$-$H$-space of $F$ is $H$-cofibrant and special,
  compare Proposition B.35 and B.50 of \cite{schwede:global}.
  Moreover, the underlying orthogonal $H$-spectrum of $F(\mS)$ is equal to $(\res^G_H F)(\mS_H)$.
  So for the proof of parts (i), (ii) and (iii), it is no loss of generality to assume that $H=G$.

  (i)
  As a preparation we recall some facts about non-equivariant special $\Gamma$-spaces.
  We let $X:\Gamma\to\bT$ be a cofibrant special $\Gamma$-space.
  The adjoint assembly map
  \[  X(1_+) \ \to \  \Omega X(S^1)     \]
  induces a map
  \begin{equation}\label{eq:group_completion_monoids}
  \pi_0 ( X(1_+) ) \ \to \ \pi_0( \Omega X(S^1) )\ = \ [S^1, X(S^1) ] \ .     
  \end{equation}
  Because $X$ is special, the source carries an abelian monoid structure
  as explained in Construction \ref{con:addition on Gamma-space}.
  The target carries a group structure by concatenation of loops;
  this group structure is abelian because  $X(S^1)$
  is itself a loop space (it is weakly equivalent to $\Omega X(S^2)$).
  The fact that the map \eqref{eq:group_completion_monoids} is
  an algebraic group completion of abelian monoids
  essentially goes back to Segal \cite[Proposition 4.1]{segal:cat coho},
  who required that $\pi_0(X(1_+))$ has a cofinal free abelian submonoid;
  the general case is implicit in many references, and an explicit reference is
  \cite[Proposition 2.12]{may-merling-osorno}.
  
  For showing that the map $\beta(G)$ is a group completion
  of abelian monoids we apply the previous paragraph to
  the (non-equivariant) special $\Gamma$-space $X^G$, defined as the composite
  \[ \Gamma \ \xra{\ F\ }\ G\bT \ \xra{\ (-)^G\ } \ \bT\ .\]
  We obtain that the canonical map
  \[  \pi_0( F(1_+)^G)\ \to \ [S^1, F^G(S^1)] \]
  is an algebraic group completion of abelian monoids.

  Now we claim that for every based space $K$ with trivial $G$-action, the canonical map
  $F^G(K) \to  F(K)^G$  is a homeomorphism.
  Since $F(n_+)^G$ is closed inside $F(n_+)$, 
  Proposition B.26 (ii) of \cite{schwede:global}
  shows that the inclusion $F^G\to F$ induces a closed embedding $(F^G)(K)\to F(K)$.
  The image of this map is contained in $F(K)^G$, 
  so it only remains to show that every $G$-fixed point of $F(K)$
  is the image of a point in $(F^G)(K)$.
  We consider a point of $F(K)$ represented by
  a tuple $(x;k_1,\dots,k_n)$ in $F(n_+)\times K^n$.
  We assume that the number $n$ has been chosen minimally, so that
  $x$ is non-degenerate and the entries $k_i$ are pairwise distinct 
  and different from the basepoint of $K$.
  If the point  $[x;k_1,\dots,k_n]$ of $F(K)$ is $G$-fixed, 
  then for every group element $g$
  the tuple $(g x;k_1,\dots,k_n)$ is equivalent to the original tuple.
  Proposition B.24 (iii) of \cite{schwede:global}
  provides a permutation $\sigma\in \Sigma_n$ such that
  \[ (g x;  k_1,\dots, k_n)\ = \ 
    (F(\sigma^{-1})(x); k_{\sigma(1)},\dots,k_{\sigma(n)} )\ . \]
  Since the $k_i$ are pairwise distinct, this forces $\sigma$
  to be the identity permutation, and hence $g x=x$.
  In other words, the point $x$ is $G$-fixed. This proves the claim.

  By the previous paragraph, the space $F^G(S^1)$ maps homeomorphically onto $F(S^1)^G$;
  so the group of non-equivariant homotopy classes 
  $[S^1, F^G(S^1)]$ is the same as the group of equivariant homotopy classes $[S^1, F(S^1)]^G$.
    Since the $\Gamma$-$G$-space is special and $G$-cofibrant,
  the Segal-Shimakawa equivariant $\Gamma$-space machine thus shows that
  $F(\mS)$ is a positive $G$-$\Omega$-spectrum, compare \cite[Theorem B]{shimakawa}; 
  a proof in the form adapted to our purposes 
  can be found in \cite[Theorem B.65]{schwede:global}.
  So the canonical map
  \[   [S^1,F(S^1) ]^G \ \to \ \pi_0^G( F(\mS) )      \]
  is an isomorphism of abelian groups. By combining all the above we conclude that
  the homomorphism of abelian monoids $\beta(G):\pi_0(F(1_+)^G)\to \pi_0^G(F(\mS))$
  is an algebraic group completion.

  (ii) As we argued at the start of the proof, we may assume that $H=G$.
  The diagram involving the restriction maps decomposes
  into two parts as follows:
  \[ \xymatrix@C=25mm{
      \pi_0( F(1_+)^G)\ar@/^1pc/[rr]^(.3){\beta(G)}\ar[r]_-{\text{assembly}\circ(-\sm S^1)}
      \ar[d]_{\pi_0(\text{incl})}&
      [S^1,F(S^1)]^G\ar[r]\ar[d]_{\res^G_K}&
      \pi_0^G( F(\mS))\ar[d]^{\res^G_K}\\
      \pi_0( F(1_+)^K)\ar@/_1pc/[rr]_(.3){\beta(K)}\ar[r]^-{\text{assembly}\circ(-\sm S^1)}
      & [S^1,F(S^1)]^K\ar[r]&
      \pi_0^K( F(\mS)) } \]
  The two right horizontal maps are the canonical maps to the colimits.
  The left and right squares each commute, hence so does the composite diagram.
  
  (iii)
  Again we may assume without loss of generality that $H=G$.
  As a preparation for the proof, we define a version of `evaluation of $F$ on spheres'
  with coefficients in a finite based $G$-set $T$.
  We define an orthogonal $G$-spectrum $F(\mS;T)$ at an inner product space $V$ by
  \[ F(\mS;T)(V) \ = \ F(S^V\!\!\sm T)\ . \]
  The structure maps are defined in the same way as for $F(\mS)$, with the
  extra $T$ acting as a dummy. 
  For varying $V$, the assembly maps \eqref{eq:assembly}
  \[  F(S^V)\sm T \ \to \  F(S^V\!\!\sm T)\ = \ F(\mS;T)(V)  \]
  provide a morphism of orthogonal $G$-spectra 
  \[  a\ : \  F(\mS)\sm T\  \to \ F(\mS;T)  \]
  that is natural for based $G$-maps in the variable $T$.
  
  Now we let $K$ be a closed subgroup of $G$ of finite index;
  then $G/K$ is a finite $G$-set.
  We let $\Psi:G/K_+\to S^0$ denote the $K$-equivariant
  projection to the distinguished coset, compare \eqref{eq:define Psi}.
  The left vertical composite in the following commutative diagram 
  is the Wirthm{\"u}ller isomorphism:
  \begin{equation}
    \begin{aligned}\label{eq:wirth diagram}
      \xymatrix@C=15mm{ 
        \pi_0^G(F(\mS)\sm G/K_+)\ar[r]^-{a_*} \ar[d]^{\res^G_K} \ar@<-8ex>@/_2pc/[dd]_{\Wirth_K^G}&
        \pi_0^G(F(\mS;G/K_+)) \ar[d]^{\res^G_K}  \\
        \pi_0^K(F(\mS)\sm G/K_+)\ar[r]^-{a_*}\ar[d]^{(F(\mS)\sm\Psi)_*}&
        \pi_0^K(F(\mS;G/K_+))  \ar@/^1pc/[ld]^{F(\mS;\Psi)_*}  \\
        \pi_0^K(F(\mS))\ &
      }    
    \end{aligned}
  \end{equation}
  We claim that the composite
  \[ F(\mS;\Psi)_*\circ \res^G_K \ : \ \pi_0^G(F(\mS;G/K_+)) \ \to \ \pi_0^K(F(\mS)) \]
  is also an isomorphism.
  To see this we let $V$ be a $G$-representation.
  The Wirthm{\"u}ller map 
  \[ \omega_{S^V}\ : \ F(\mS;G/K_+)(V) \ = \ F(S^V\!\!\sm G/K_+)\ \to \  \map^K(G,F(S^V)) \]
  is the $G$-equivariant adjoint of the continuous based $K$-map
  \[  F(S^V\!\!\sm \Psi)\ : \ F(S^V\!\!\sm G/K_+)\ \to \ F(S^V)\ ; \]
  this Wirthm{\"u}ller map is a $G$-weak equivalence 
  by \cite[Proposition B.54 (ii)]{schwede:global}.
  In particular, the Wirthm{\"u}ller map induces a bijection
  \begin{align*}
    [S^V,F(\mS;G/K_+)(V)]^G\ \xra{\ \iso \ }  \  &[S^V,\map^K(G_+,F(S^V))]^G \ 
    \iso \  [S^V,F(S^V)]^K \ .
  \end{align*}
  Upon passage to colimits over finite-dimensional $G$-subrepresentations of
  a complete $G$-universe, this gives an isomorphism
  \begin{align*}
    \pi_0^G(F(\mS;G/K_+))\ = \ \colim_V\, & [S^V,F(\mS;G/K_+)(V)]^G\\
    \xra{\ \iso \ } & \ \colim_V\, [S^V,F(S^V)]^K\ \iso \ \pi_0^K(F(\mS))\ .  
  \end{align*}
  The last isomorphism exploits that the $K$-representations that underlie $G$-representations
  are cofinal in all $K$-representations. The isomorphism agrees with
  the composite $F(\mS;\Psi)_*\circ \res^G_K$, by inspection.
  
  For an orthogonal $G$-spectrum $X$,
  the transfer $\tr_K^G:\pi_0^K(X)\to\pi_0^G(X)$ is the composite 
  \[ \pi_0^K(X)\ \xra[\iso]{(\Wirth_K^G)^{-1}}\ \pi_0^G(X\sm G/K_+)
    \ \xra{(X\sm \nabla)_*}\ \pi_0^G(X)\ , \]
  where $\nabla:G/K_+\to S^0$ is the fold map that takes all of $G/K$
  to the non-basepoint.
  In the special case $X=F(\mS)$, the second map $(F(\mS)\sm\nabla)_*$
  in this composite in turn factors as the composite
  \[ 
    \pi_0^G(F(\mS)\sm G/K_+)\ \xra{\ a_*\ }\  
    \pi_0^G(F(\mS;G/K_+)) \ \xra{F(\mS;\nabla)_*}  \ \pi_0^G(F(\mS)) \ .
  \]
  The commutativity of the diagram \eqref{eq:wirth diagram} lets us thus express
  the transfer as the composite
  \[ \pi_0^K(F(\mS))\ \xra[\iso]{(F(\mS;\Psi)\circ\res^G_K)^{-1}}\ 
    \pi_0^G(F(\mS;G/K_+))\ \xra{ F(\mS;\nabla)_*}\ \pi_0^G(F(\mS))\ . \]
  Now we contemplate the following commutative diagram:
  \[ \xymatrix@C=20mm{
      \pi_0(F(1_+)^K)\ar[d]\ar@/^2pc/[rr]^(.3){\tr_K^G} & \pi_0(F(G/K_+)^G)\ar[d]\ar[l]^-{(F(\Psi)^K)_*\circ\res^G_K}_(.4)\iso\ar[r]_--{\pi_0(F(\nabla))} &
      \pi_0(F(1_+)^G)\ar[d]\\
      \pi_0^K(F(\mS))\ar@/_2pc/[rr]_(.3){\tr^G_K} &
      \pi_0^G(F(\mS;G/K_+)) \ar[l]_-{F(\mS;\Psi)_*\circ\res^G_K}^(.4)\iso  
      \ar[r]^-{F(\mS;\nabla)_*}&\pi_0^G(F(\mS)) } \]
  All vertical maps are stabilization maps.
  The two squares commute by naturality. So the entire diagram commutes,
  and we have shown the compatibility of the stabilization maps with transfers.
  
  The compatibility (iv) of the $\beta$-maps with conjugation 
  is straightforward from the definitions.
\end{proof}

\begin{acknowledgements}
I would like to thank Tobias Lenz for a host of helpful comments and suggestions, 
and for pointing out problems with earlier versions of some statements;
also, the saturation construction for parsummable categories (Construction \ref{con:saturation_parsummable})
is due to him. 
I would like to thank Markus Hausmann, Lars Hesselholt and Mona Merling
for several inspiring discussions about the topics of this paper.
A substantial part of the work was done in the spring of 2018, 
while the author was a long term guest at the
Centre for Symmetry and Deformation at K{\o}benhavns Universitet;
I would like to thank the Center for Symmetry and Deformation 
for the hospitality and stimulating atmosphere during this visit.
Finally, I am grateful to the anonymous referee for an extremely careful reading of the paper,
resulting in various suggestions for improvements.

\end{acknowledgements}

\affiliationone{
  Stefan Schwede\\
  Mathematisches Institut\\
  Universit{\"a}t Bonn\\
  Endenicher Allee 60\\
  D-53115 Bonn\\
  Germany
   \email{schwede@math.uni-bonn.de}}
%
\end{document}